\documentclass[11pt]{amsart}

\usepackage{amsmath}
\usepackage{amsfonts}
\usepackage{amssymb}
\usepackage{amsthm}
\usepackage{graphicx}
\usepackage{pb-diagram}
\usepackage{epstopdf}
\usepackage{amscd}
\usepackage{pdfpages}
\usepackage{bbm}
\usepackage{multirow}
\usepackage[mathscr]{eucal}
\usepackage{epsfig,epsf,epic}
\usepackage{pdfsync}
\usepackage{enumerate}
\usepackage[all]{xy}
\usepackage{fancyref}
\usepackage{mathtools}
\usepackage{scalerel,stackengine}
\usepackage{comment}
\usepackage{hyperref}
\usepackage{subfig}
\usepackage{tabularx}

\stackMath

\newtheorem{theorem}{Theorem}[section]
\newtheorem{prop}[theorem]{Proposition}
\newtheorem{conjecture}[theorem]{Conjecture}

\newtheorem{definition}[theorem]{Definition}
\newtheorem{example}[theorem]{Example}
\newtheorem{lemma}[theorem]{Lemma}
\newtheorem{remark}[theorem]{Remark}
\newtheorem{assum}[theorem]{Assumption}
\newtheorem{condition}[theorem]{Condition}
\newtheorem{notation}[theorem]{Notation}
\numberwithin{equation}{section}

\DeclareMathOperator{\Spec}{Spec}

\newcommand{\real}{\mathbb{R}}
\newcommand{\comp}{\mathbb{C}}

\newcommand{\inte}{\mathbb{Z}}
\newcommand{\bb}[1]{\mathbb{#1}}
\newcommand{\cu}[1]{\mathcal{#1}}

\newcommand{\msc}[1]{\mathscr{#1}}

\newcommand{\pdb}{\bar{\partial}}
\newcommand{\dd}[1]{\frac{\partial}{\partial #1}}

\newcommand{\half}{\frac{1}{2}}
\newcommand{\ddd}[2]{\frac{\partial#1}{\partial{#2}}}

\newcommand{\tanpoly}{\Lambda}
\newcommand{\reint}{\mathrm{int}_\mathrm{re}}
\newcommand{\pcate}{\underline{\mathrm{LPoly}}}
\newcommand{\pdecomp}{\mathscr{P}}
\newcommand{\norpoly}{\mathscr{Q}}
\newcommand{\centerfiber}[1]{\prescript{#1}{}{X}}
\newcommand{\spec}{\mathrm{Spec}_{\mathrm{an}}}

\newcommand{\localmod}[1]{\prescript{#1}{}{\mathbb{V}}}

\newcommand{\moment}{\mu}
\newcommand{\amoeba}{\cu{A}}
\newcommand{\tsing}{\mathscr{S}}
\newcommand{\modmap}{\nu}
\newcommand{\tinclude}{\iota}
\newcommand{\nsf}{\hat{\mathscr{S}}}

\newcommand{\blat}{\mathbf{K}}
\newcommand{\cfr}{R}
\newcommand{\cfrk}[1]{\prescript{#1}{}{\cfr}}
\newcommand{\logs}{S^{\dagger}}
\newcommand{\logsk}[1]{\prescript{#1}{}{S}^{\dagger}}
\newcommand{\logsf}{\hat{S}^{\dagger}}
\newcommand{\logsdrk}[2]{\prescript{#1}{}{\Omega}^{#2}_{S^{\dagger}}}
\newcommand{\logsvfk}[1]{\prescript{#1}{}{\Theta}_{S^{\dagger}}}
\newcommand{\logsdrf}[1]{\hat{\Omega}^{#1}_{S^{\dagger}}}
\newcommand{\logsvff}{\hat{\Theta}_{S^{\dagger}}}
\newcommand{\bva}[1]{\prescript{#1}{}{\mathcal{G}}}
\newcommand{\tbva}[2]{\prescript{#1}{#2}{\mathcal{K}}}
\newcommand{\gmiso}[2]{\prescript{#1}{#2}{\sigma}}
\newcommand{\dpartial}[1]{\prescript{#1}{}{\partial}}
\newcommand{\volf}[1]{\prescript{#1}{}{\omega}}

\newcommand{\bvd}[1]{\prescript{#1}{}{\Delta}}
\newcommand{\gmc}[1]{\prescript{#1}{}{\nabla}}
\newcommand{\rest}[1]{\prescript{#1}{}{\flat}}
\newcommand{\patch}[1]{\prescript{#1}{}{\psi}}

\newcommand{\cocyobs}[1]{\prescript{#1}{}{\mathfrak{o}}}
\newcommand{\bvdobs}[1]{\prescript{#1}{}{\mathfrak{w}}}

\newcommand{\polyv}[1]{\prescript{#1}{}{PV}}
\newcommand{\totaldr}[2]{\prescript{#1}{#2}{\mathcal{A}}}
\newcommand{\glue}[1]{\prescript{#1}{}{g}}
\newcommand{\mdga}{\Omega}
\newcommand{\md}{d}

\newcommand{\hp}{\hslash}
\newcommand{\tform}{\mathscr{T}}
\newcommand{\sfbva}[1]{\prescript{#1}{}{\mathsf{G}}}
\newcommand{\sftvbva}[1]{\prescript{#1}{}{\mathfrak{h}}}
\newcommand{\sftbva}[2]{\prescript{#1}{#2}{\mathsf{K}}}
\newcommand{\sflocmod}[1]{\prescript{#1}{}{\mathsf{V}}}
\newcommand{\sfpolyv}[1]{\prescript{#1}{}{\mathsf{PV}}}
\newcommand{\sftropv}[1]{\prescript{#1}{}{\mathsf{TL}}}
\newcommand{\sftotaldr}[2]{\prescript{#1}{#2}{\mathsf{A}}}
\newcommand{\tvbva}[1]{\prescript{#1}{}{\mathscr{H}}}
\newcommand{\mpatch}[1]{\prescript{#1}{}{\tilde{\psi}}}
\newcommand{\tropv}[1]{\prescript{#1}{}{TL}} 
\newcommand{\wcs}[1]{\prescript{#1}{}{\mathscr{O}}}

\begin{document}
\title[Smoothing, scattering, and a conjecture of Fukaya]{Smoothing, scattering, and a conjecture of Fukaya}
	
\author[Chan]{Kwokwai Chan}
\address{Department of Mathematics\\ The Chinese University of Hong Kong\\ Shatin\\ Hong Kong}
\email{kwchan@math.cuhk.edu.hk}	

\author[Leung]{Naichung Conan Leung}
\address{The Institute of Mathematical Sciences and Department of Mathematics\\ The Chinese University of Hong Kong\\ Shatin\\ Hong Kong}
\email{leung@math.cuhk.edu.hk}
	
\author[Ma]{Ziming Nikolas Ma}
\address{Department of Mathematics\\ Southern University of Science and Technology\\ Nanshan District \\ Shenzhen\\ China}
\email{mazm@sustech.edu.cn}

\begin{abstract}
In 2002, Fukaya \cite{fukaya05} proposed a remarkable explanation of mirror symmetry detailing the SYZ conjecture \cite{syz96} by introducing two correspondences: 
one between the theory of pseudo-holomorphic curves on a Calabi-Yau manifold $\check{X}$ and the multi-valued Morse theory on the base $\check{B}$ of an SYZ fibration $\check{p}\colon \check{X}\to \check{B}$,
and the other between deformation theory of the mirror $X$ and the same multi-valued Morse theory on $\check{B}$. 
In this paper, we prove a reformulation of the main conjecture in Fukaya's second correspondence, where multi-valued Morse theory on the base $\check{B}$ is replaced by tropical geometry on the Legendre dual $B$. 
In the proof, we apply techniques of asymptotic analysis developed in \cite{kwchan-leung-ma, kwchan-ma-p2} to tropicalize the pre-dgBV algebra which governs smoothing of a maximally degenerate Calabi-Yau log variety $\centerfiber{0}^{\dagger}$ introduced in \cite{chan2019geometry}. Then a comparison between this tropicalized algebra with the dgBV algebra associated to the deformation theory of the semi-flat part $X_{\mathrm{sf}} \subseteq X$ allows us to extract consistent scattering diagrams from appropriate Maurer-Cartan solutions.
\end{abstract}

\maketitle

%\tableofcontents

\section{Introduction}

%\subsection{Background}

Two decades ago, in an attempt to understand mirror symmetry using the SYZ conjecture \cite{syz96}, Fukaya \cite{fukaya05} proposed two correspondences:
\begin{itemize}
	\item Correspondence I: between the theory of pseudo-holomorphic curves (instanton corrections) on a Calabi--Yau manifold $\check{X}$ and the multi-valued Morse theory on the base $\check{B}$ of an SYZ fibration $\check{p}\colon \check{X}\to \check{B}$, and
	\item Correspondence II: between deformation theory of the mirror $X$ and the same multi-valued Morse theory on the base $\check{B}$.
\end{itemize}
In this paper, we prove a reformulation of the main conjecture \cite[Conj 5.3]{fukaya05} in Fukaya's Correspondence II, where multi-valued Morse theory on the SYZ base $\check{B}$ is replaced by tropical geometry on the Legendre dual $B$. Such a reformulation of Fukaya's conjecture was proposed and proved in \cite{kwchan-leung-ma} in a local setting; the main result of the current paper is a global version of the main result in {\it loc. cit}. A crucial ingredient in the proof is a precise link between tropical geometry on an integral affine manifold with singularities and smoothing of maximally degenerate Calabi--Yau varieties.

The main conjecture \cite[Conj. 5.3]{fukaya05} in Fukaya's Correspondence II asserts that 
there exists a Maurer--Cartan element of the Kodaira--Spencer dgLa associated to deformations of the semi-flat part $X_{\mathrm{sf}}$ of $X$ that is asymptotically close to a Fourier expansion (\cite[Eq. (42)]{fukaya05}), whose Fourier modes are given by smoothings of distribution-valued 1-forms defined by moduli spaces of gradient Morse flow trees which are expected to encode counting of non-trivial (Maslov index 0) holomorphic disks bounded by Lagrangian torus fibers (see \cite[Rem. 5.4]{fukaya05}). Also, the complex structure defined by this Maurer--Cartan element can be compactified to give a complex structure on $X$.
At the same time, Fukaya's Correspondence I suggests that these gradient Morse flow trees arise as adiabatic limits of loci of those Lagrangian torus fibers which bound non-trivial (Maslov index 0) holomorphic disks. This can be reformulated as a holomorphic/tropical correspondence, and much evidence has been found \cite{Floer88, fukaya-oh, Mikhalkin05, Nishinou-Siebert06, Cho-Hong-Lau17, Cho-Hong-Lau18, Lin21, Cheung-Lin21, BE-C-H-Lin21}.

The tropical counterpart of such gradient Morse flow trees are given by consistent scattering diagrams, which were invented by Kontsevich--Soibelman \cite{kontsevich-soibelman04} and extensively used in the Gross--Siebert program \cite{gross2011real} to solve the reconstruction problem in mirror symmetry, namely, the construction of the mirror $X$ from smoothing of a maximally degenerate Calabi--Yau variety $\centerfiber{0}$. It is therefore natural to replace the distribution-valued 1-form in each Fourier mode in the Fourier expansion \cite[Eq. (42)]{fukaya05} by a distribution-valued 1-form associated to a wall-crossing factor of a consistent scattering diagram. This was exactly how Fukaya's conjecture \cite[Conj. 5.3]{fukaya05} was reformulated and proved in the local case in \cite{kwchan-leung-ma}.

In order to reformulate the global version of Fukaya's conjecture, however, we must also relate deformations of the semi-flat part $X_{\mathrm{sf}}$ with smoothings of the maximally degenerate Calabi--Yau variety $\centerfiber{0}$. This is because consistent scattering diagrams were used by Gross--Siebert \cite{Gross-Siebert-logII} to study the deformation theory of the compact log variety $\centerfiber{0}^{\dagger}$ (whose log structure is specified by \emph{slab functions}), instead of $X_{\mathrm{sf}}$.
For this purpose, we consider the open dense part
$$\centerfiber{0}_{\mathrm{sf}} := \moment^{-1}(W_0) \subset \centerfiber{0},$$
where $\moment\colon \centerfiber{0} \rightarrow B$ is the \emph{generalized moment map} in \cite{ruddat2019period} and $W_0 \subseteq B$ is an open dense subset such that $B\setminus W_0$ contains the tropical singular locus and all codimension $2$ cells of $B$.

Equipping $\centerfiber{0}_{\mathrm{sf}}$ with the \emph{trivial} log structure, there is a \emph{semi-flat dgBV algebra} $\sfpolyv{}^{*,*}$ governing its smoothings, and the general fiber of a smoothing is given by the semi-flat Calabi--Yau $X_{\mathrm{sf}}$ that appeared in Fukaya's original conjecture \cite[Conj. 5.3]{fukaya05}.
However, the Maurer--Cartan elements of $\sfpolyv{}^{*,*}$ cannot be compactified to give complex structures on $X$. 
On the other hand, in our previous work \cite{chan2019geometry} we constructed a \emph{Kodaira--Spencer--type pre-dgBV algebra} $\polyv{}^{*,*}$ which controls the smoothing of $\centerfiber{0}$.
A key observation is that a \emph{twisting} of $\sfpolyv{}^{*,*}$ by slab functions is isomorphic to the restriction of $\polyv{}^{*,*}$ to $\centerfiber{0}_{\mathrm{sf}}$ (Lemma \ref{lem:comparing_sheaf_of_dgbv}).

Our reformulation of the global Fukaya conjecture now claims the existence of a Maurer--Cartan element $\phi$ of this twisted semi-flat dgBV algebra that is asymptotically close to a Fourier expansion whose Fourier modes give rise to the wall-crossing factors of a consistent scattering diagram. This conjecture follows from (the proof of) our main result, stated as Theorem \ref{thm:introduction_theorem} below, which is a combination of Theorem \ref{prop:Maurer_cartan_equation_unobstructed}, the construction in \S \ref{subsubsec:consistent_diagram_from_solution} and Theorem \ref{thm:consistency_of_diagram_from_mc}:
\begin{theorem}\label{thm:introduction_theorem}
	There exists a solution $\phi$ to the classical Maurer--Cartan equation \eqref{eqn:classical_maurer_cartan_equation} giving rise to a smoothing of the maximally degenerate Calabi--Yau log variety $\centerfiber{0}^{\dagger}$ over $\comp[[q]]$, from which a consistent scattering diagram $\mathscr{D}(\phi)$ can be extracted by taking asymptotic expansions.  
\end{theorem}

A brief outline of the proof of Theorem \ref{thm:introduction_theorem} is now in order.
First, recall that the pre-dgBV algebra $\polyv{}^{*,*}$ which governs smoothing of the maximally degenerate Calabi--Yau variety $\centerfiber{0}$ was constructed in \cite[Thm. 1.1 \& \S 3.5]{chan2019geometry}, and we also proved a Bogomolov--Tian--Todorov--type theorem \cite[Thm. 1.2 \& \S 5]{chan2019geometry} showing unobstructedness of the extended Maurer--Cartan equation \eqref{eqn:extended_maurer_cartan_equation}, under the Hodge-to-de Rham degeneracy Condition \ref{cond:Hodge-to-deRham} and a holomorphic Poincar\'{e} Lemma Condition \ref{cond:holomorphic_poincare_lemma} (both proven in \cite{Gross-Siebert-logII, Felten-Filip-Ruddat}). 
In Theorem \ref{prop:Maurer_cartan_equation_unobstructed}, we will further show how one can extract from the extended Maurer--Cartan equation \eqref{eqn:extended_maurer_cartan_equation} a smoothing of $\centerfiber{0}$, described as a solution $\phi \in \polyv{}^{-1,1}(B)$ to the \emph{classical Maurer--Cartan equation} \eqref{eqn:classical_maurer_cartan_equation}
$$
\pdb \phi + \half[\phi,\phi] + \mathfrak{l} = 0,
$$
together with a holomorphic volume form $e^{f} \volf{}$ which satisfies the \emph{normalization condition} 
\begin{equation}\label{eqn:introduction_normalization_of_volume_form}
	\int_T e^f \volf{} = 1,
\end{equation}
where $T$ is a nearby vanishing torus in the smoothing. 

Next, we need to tropicalize the pre-dgBV algebra $\polyv{}^{*,*}$.
However, the original construction of $\polyv{}^{*,*}$ in \cite{chan2019geometry} using the Thom--Whitney resolution \cite{whitney2012geometric, dupont1976simplicial} is too algebraic in nature. 
Here, we construct a geometric resolution exploiting the affine manifold structure on $B$.
Using the generalized moment map $\moment \colon \centerfiber{0} \rightarrow B$ \cite{ruddat2019period} and applying the techniques of asymptotic analysis (in particular the notion of \emph{asymptotic support}) in \cite{kwchan-leung-ma},
we define the sheaf $\tform^*$ of \emph{monodromy invariant tropical differential forms} on $B$ in \S \ref{sec:asymptotic_support}. 
According to Definition \ref{def:sheaf_of_tropical_dga}, a tropical differential form can be regarded as a distribution-valued form supported on polyhedral subsets of $B$. Using the sheaf $\tform^*$, we can take asymptotic expansions of elements in $\polyv{}^{*,*}$, and hence connect differential geometric operations in dgBV/dgLa with tropical geometry. In this manner, we can extract \emph{local} scattering diagrams from Maurer--Cartan solutions as we did in \cite{kwchan-leung-ma}, but we need to glue them together to get a global object.

%We can then proceed to extract a scattering diagram from the asymptotic expansion of the Maurer-Cartan solution $\phi$. 
To achieve this, we need 
the aforementioned comparison between $\polyv{}^{*,*}$ and the semi-flat dgBV algebra $\sfpolyv{}^{*,*}_{\mathrm{sf}}$ which governs smoothing of the semi-flat part $\centerfiber{0}_{\mathrm{sf}} := \moment^{-1}(W_0) \subset \centerfiber{0}$ equipped with the trivial log structure.
The key Lemma \ref{lem:comparing_sheaf_of_dgbv} says that the restriction of $\polyv{}^{*,*}$ to the semi-flat part is isomorphic to $\sfpolyv{}^{*,*}_{\mathrm{sf}}$ precisely after we \emph{twist} the semi-flat operator $\pdb_{\circ}$ by elements corresponding to the \emph{slab functions} associated to the \emph{initial walls} of the form:
$$
\phi_{\mathrm{in}} = - \sum_{v \in \rho} \delta_{v,\rho} \otimes \log(f_{v,\rho}) \partial_{\check{d}_{\rho}};
$$
here the sum is over vertices in codimension one cells $\rho$'s which intersect with the \emph{essential singular locus} $\tsing_e$ (defined in \S \ref{subsec:tropical_singular_locus}), $\delta_{v,\rho}$ is a distribution-valued $1$-form supported on a component of $\rho\setminus \tsing_e$ containing $v$, $\partial_{\check{d}_{\rho}}$ is a holomorphic vector field and $f_{v,\rho}$'s are the slab functions associated to the initial walls. We remark that slab functions were used to specify the log structure on $\centerfiber{0}$ as well as the local models for smoothing $\centerfiber{0}$ in the Gross--Siebert program; see \S \ref{sec:gross_siebert} for a review.

Now, the Maurer--Cartan solution $\phi \in \polyv{}^{-1,1}(B)$ obtained in Theorem \ref{prop:Maurer_cartan_equation_unobstructed} defines a new operator $\pdb_{\phi}$ on $\polyv{}^{*,*}$ which squares to zero. Applying the above comparison of dgBV algebras (Lemma \ref{lem:comparing_sheaf_of_dgbv}) and the gauge transformation from Lemma \ref{lem:vector field}, we show that, after restricting to $W_0$, there is an isomorphism
$$\left(\polyv{}^{-1,1}(W_0), \pdb_{\phi}\right) \cong \left(\sfpolyv{}^{-1,1}_{\mathrm{sf}}(W_0),\pdb_{\circ} + [\phi_{\mathrm{in}}+\phi_{\mathrm{s}},\cdot] \right)$$
for some element $\phi_{\mathrm{s}}$, where `s' stands for {\it scattering terms}. 
From the description of $\tform^*$, the element $\phi_{\mathrm{s}}$, to any fixed order $k$, is written locally as a finite sum of terms supported on codimension one walls/slabs (Definitions \ref{def:walls} and \ref{def:slabs}. 
For the purpose of a brief discussion in this introduction, we will restrict ourselves to a wall $\mathbf{w}$ below, though the same argument applies to a slab; see \S \ref{subsubsec:consistent_diagram_from_solution} for the details.
In a neighborhood $U_{\mathbf{w}}$ of each wall $\mathbf{w}$, the operator $\pdb_{\circ}+[\phi_{\mathrm{in}}+\phi_{\mathrm{s}},\cdot]$ is gauge equivalent to $\pdb_{\circ}$ via some vector field $\theta_{\mathbf{w}} \in \sfpolyv{}^{-1,0}_{\mathrm{sf}}(W_0)$, i.e.
$$
e^{[\theta_{\mathbf{w}},\cdot]}\circ \pdb_{\circ} \circ e^{-[\theta_{\mathbf{w}},\cdot]} = \pdb_{\circ}+ [\phi_{\mathrm{in}}+\phi_{\mathrm{s}},\cdot].
$$
Employing the techniques for analyzing the gauge which we developed in \cite{kwchan-leung-ma, kwchan-ma-p2, matt-leung-ma}, we see that the gauge will jump across the wall, resulting in a wall-crossing factor $\varTheta_{\mathbf{w}}$ satisfying
$$
	e^{[\theta_{\mathbf{w}},\cdot]}|_{ \cu{C}_{\pm}}	= \begin{dcases}
		\varTheta_{\mathbf{w}}|_{ \cu{C}_+}  & \text{on $U_{\mathbf{w}} \cap \cu{C}_+$,}\\
		\mathrm{id}  & \text{on $U_{\mathbf{w}} \cap \cu{C}_-$,}
	\end{dcases}
$$
where $\cu{C}_{\pm}$ are the two chambers separated by $\mathbf{w}$. Then from the fact that the volume form $e^{f}\volf{}$ is normalized as in \eqref{eqn:introduction_normalization_of_volume_form}, it follows that $\phi_{\mathrm{s}}$ is closed under the semi-flat BV operator $\bvd{}$, and hence we deduce that the wall-crossing factor $\varTheta_{\mathbf{w}}$ lies in the {\it tropical vertex group}. This defines a scattering diagram $\mathscr{D}(\phi)$ on the semi-flat part $W_0$ associated to $\phi$. Finally, we prove consistency of the scattering diagram $\mathscr{D}(\phi)$ in Theorem \ref{thm:consistency_of_diagram_from_mc}. We emphasize that the consistency is over the {\it whole} $B$ even though the diagram is only defined on $W_0$, because the Maurer--Cartan solution $\phi$ is globally defined on $B$. %This proves the reformulated Fukaya conjecture, as a global version of the main result in \cite{kwchan-leung-ma}.

\begin{remark}\label{rmk:relaxed_scattering_diagram}
	Our notion of scattering diagrams (Definition \ref{def:scattering_diagram}) is a little bit more relaxed than the usual notion defined in \cite{kontsevich-soibelman04, gross2011real} in two aspects: One is that we do not require the generator of the exponents of the wall-crossing factor to be orthogonal to the wall.\footnote{It seems reasonable to relax this orthogonality condition because one cannot require such a condition in more general settings \cite{bridgeland2016scattering, matt-leung-ma}.}
    The other is that we allow possibly infinite number of walls/slabs approaching strata of the tropical singular locus. See the paragraph after Definition \ref{def:scattering_diagram} for more details.
    In practice, this simply means that we are considering a larger gauge equivalence class (or equivalently, a weaker gauge equivalence), which is natural from the point of view of both the Bogomolov--Tian--Todorov Theorem and mirror symmetry (in the A-side, this amounts to flexibility in the choice of the almost complex structure).
	We also have a different, but more or less equivalent, formulation of the consistency of a scattering diagram; see Definition \ref{def:consistency_of_scattering_diagram} and \S \ref{subsubsec:scattering_diagram}.
\end{remark}

Along the way of proving Fukaya's conjecture, besides figuring out the precise relation between the semi-flat part $X_{\mathrm{sf}}$ and the maximally degenerate Calabi--Yau log variety $\centerfiber{0}^{\dagger}$, we also find the correct description of the Maurer--Cartan solutions near the singular locus, namely, they should be extendable to the local models prescribed by the log structure (or slab functions), as was hinted by the Gross--Siebert program. This is related to a remark by Fukaya \cite[Pt. (2) after Conj. 5.3]{fukaya05}. 

Another important point is that we have established in the global setting an interplay between the differential-geometric properties of the tropical dgBV algebra and the scattering (and other combinatorial) properties of tropical disks, which was speculated by Fukaya as well (\cite[Pt. (1) after Conj. 5.3]{fukaya05}) although he considered holomorphic disks instead of tropical ones. 

Furthermore, by providing a direct linkage between Fukaya's conjecture with the Gross--Siebert program \cite{Gross-Siebert-logI, Gross-Siebert-logII, gross2011real} and Katzarkov--Kontsevich--Pantev's Hodge theoretic viewpoint \cite{KKP08} through $\polyv{}^{*,*}$ (recall from \cite{chan2019geometry} that a semi-infinite variation of Hodge structures can be constructed from $\polyv{}^{*,*}$, using the techniques of Barannikov--Kontsevich \cite{Barannikov-Kontsevich98, Barannikov99} and Katzarkov--Kontsevich--Pantev \cite{KKP08}), we obtain a more transparent understanding of mirror symmetry through the SYZ framework. 

\begin{remark}
	A future direction is to apply the framework in this paper and the works \cite{kwchan-leung-ma, chan2019geometry} to develop a local-to-global approach to understand genus $0$ mirror symmetry.
	In view of the ideas of Seidel \cite{Seidel-ICM} and Kontsevich \cite{Kontsevich-sheaf}, and also recent breakthroughs by Ganatra--Pardon--Shende \cite{Ganatra-Pardon-Shende20, Ganatra-Pardon-Shende18a, Ganatra-Pardon-Shende18b} and Gammage--Shende \cite{Gammage-Shende17, gammage2021homological}, we expect that there is a sheaf of $L_\infty$ algebras on the A-side mirror to (the $L_\infty$ enhancement of) $\polyv{}^{*,*}$ that can be constructed by gluing local models. 
	More precisely, a large volume limit of a Calabi--Yau manifold $\check{X}$ can be specified by removing from it a normal crossing divisor $\check{D}$ which represents the K\"ahler class of $\check{X}$. This gives rise to a Weinstein manifold $\check{X} \setminus \check{D}$, and produces a mirror pair $\check{X} \setminus \check{D} \leftrightarrow \centerfiber{0}$ at the large volume/complex structure limits.
	In \cite{gammage2021homological}, Gammage--Shende constructed a Lagrangian skeleton $\Lambda(\Phi) \subset \check{X} \setminus \check{D}$ from a combinatorial structure $\Phi$ called \emph{fanifold}, which can be extracted from the integral tropical manifold $B$ equipped with a polyhedral decomposition $\pdecomp$ (here we assume that the gluing data $s$ is trivial). They also proved an HMS statement at the large limits.
	We expect that an A-side analogue of $\polyv{}^{*,*}$ can be constructed from the Lagrangian skeleton $\Lambda(\Phi)$ in $\check{X} \setminus \check{D}$, possibly together with a nice and compatible SYZ fibration on $\check{X} \setminus \check{D}$, via gluing of local models.
	A local-to-global comparsion on the A-side and isomorphisms between the local models on the two sides should then yield an isomorphism of Frobenius manifolds. %This program will be taken up in future works.
\end{remark}

\section*{Acknowledgement}
We thank Kenji Fukaya, Mark Gross and Richard Thomas for their interest and encouragement, and also Helge Ruddat for useful comments on an earlier draft of this paper. 
We are very grateful to the anonymous referees for numerous constructive and extremely detailed comments/suggestions which have helped to greatly enhanced the exposition of the whole paper.

K. Chan was supported by grants of the Hong Kong Research Grants Council (Project No. CUHK14301420 \& CUHK14301621) and direct grants from CUHK.
N. C. Leung was supported by grants of the Hong Kong Research Grants Council (Project No. CUHK14301619 \& CUHK14306720) and a direct grant (Project No. 4053400) from CUHK.
Z. N. Ma was supported by National Science Fund for Excellent Young Scholars (Overseas).
These authors contributed equally to this work.

%\newpage
\section*{List of notations}

\begin{center}
\begin{tabular}{l l l}
	$M$, $M_{A}$ & \S \ref{subsec:integral_affine_manifolds} & lattice, $M_{A} := M \otimes_{\inte} A$ for any $\inte$-module $A$\\
	$N$, $N_{A}$ & \S \ref{subsec:integral_affine_manifolds} & dual lattice of $M$, $N_{A} := N \otimes_{\inte} A$ for any $\inte$-module $A$\\
	$(B,\pdecomp)$ & Def. \ref{def:integral_tropical_manifold} & integral tropical manifold equipped with a polyhedral\\
    & & decomposition\\
	$\tanpoly_{\sigma}$ & \S \ref{subsec:integral_affine_manifolds} & lattice generated by integral tangent vectors along $\sigma$\\
	$\reint(\tau)$ & \S \ref{subsec:integral_affine_manifolds} & relative interior of a polyhedron $\tau$\\
	$U_\tau$ & \S \ref{subsec:integral_affine_manifolds} & open neighborhood of $\reint(\tau)$\\
	$\norpoly_{\tau}$ & \S \ref{subsec:integral_affine_manifolds} & lattice generated by normal vectors to $\tau$\\
	$S_\tau \colon U_\tau \rightarrow \norpoly_{\tau,\real}$ & \S \ref{subsec:integral_affine_manifolds} & fan structure along $\tau$\\
	$\Sigma_{\tau}$ & \S \ref{subsec:integral_affine_manifolds} & complete fan in $\norpoly_{\tau,\real}$ constructed from $S_{\tau}$\\
	$K_{\tau}\sigma$ & \S \ref{subsec:integral_affine_manifolds} & $K_{\tau}\sigma=\real_{\geq 0} S_{\tau}(\sigma \cap U_{\tau})$ is a cone in $\Sigma_{\tau}$ corresponding to $\sigma$\\
	$T_{x}$ & \S \ref{subsec:monodromy_data} & lattice of integral tangent vectors of $B$ at $x$\\
	$\Delta_i(\tau)$, $\check{\Delta}_i(\tau)$ & Def. \ref{def:simplicity} & monodromy polytope of $\tau$, dual monodromy polytope of $\tau$\\
	$\cu{A}\mathit{ff}$ & Def. \ref{def:piecewise_linear} & sheaf of affine functions on $B$\\
	$\cu{PL}_{\pdecomp}$ & Def. \ref{def:piecewise_linear} & sheaf of piecewise affine functions on $B$ with respect to $\pdecomp$\\
	$\cu{MPL}_{\pdecomp}$ & Def. \ref{def:strictly_convex_piecewise_affine} & sheaf of multi-valued piecewise affine functions on $B$\\
    & & with respect to $\pdecomp$\\
	$\varphi$ & Def. \ref{def:strictly_convex_multi_valued_function} & strictly convex multi-valued piecewise linear function\\
	$\tau^{-1}\Sigma_v$ & \S \ref{subsec:open_construction} & localization of the fan $\Sigma_v$ at $\tau$\\
	$V(\tau)$ & \S \ref{subsec:open_construction} & local affine scheme associated to $\tau$ used for open gluing\\
	$\mathrm{PM}(\tau)$ & \S \ref{subsec:open_construction} & group of piecewise multiplicative maps on $\tau^{-1}\Sigma_v$\\
	$D(\mu,\rho,v)$ & Def. \ref{def:alternative_description_open_gluing_data} & number encoding the change of $\mu \in \mathrm{PM}(\tau)$ across $\rho$ through $v$\\ 
	$\centerfiber{0}_\tau$ & \S \ref{subsec:open_construction} & closed stratum of $\centerfiber{0}$ associated to $\tau$\\
	$C_{\tau}$ & \S \ref{subsec:log_structure_and_slab_function} & cone defined by the strictly convex function $\bar{\varphi}_{\tau}\colon \Sigma_{\tau} \rightarrow \real$\\
    & & representing $\varphi$\\
	$\bar{P}_{\tau}$ & \S \ref{subsec:log_structure_and_slab_function} & monoid of integral points in $C_{\tau}$\\
	$q = z^{\varrho}$ & \S \ref{subsec:log_structure_and_slab_function} & parameter for a toric degeneration\\
	$\cu{N}_{\rho}$ & \S \ref{subsec:log_structure_and_slab_function} & line bundle on $\centerfiber{0}_{\rho}$ having slab functions $f_{\rho}$ as sections\\
	$f_{v\rho}$ & \S \ref{subsec:log_structure_and_slab_function} & local slab function associate to $\rho$ in the chart $V(v)$\\
	$\varkappa_{\tau,i} \colon \centerfiber{0}_{\tau} \rightarrow \bb{P}^{r_{\tau,i}}$ & \S \ref{subsec:log_structure_and_slab_function} & toric morphism induced from the monodromy polytope $\Delta_i(\tau)$ \\
	$ P_{\tau,x}$ & \S \ref{subsec:log_structure_and_slab_function} & toric monoid describing the local model of toric degeneration\\
    & & near $x \in \centerfiber{0}_{\tau}$ \\
	$Q_{\tau,x}$ & \S \ref{subsec:log_structure_and_slab_function} & toric monoid isomorphic to $P_{\tau,x} /( \varrho + P_{\tau,x} )$ \\
	$\mathscr{N}_{\tau}$ & \S \ref{subsec:log_structure_and_slab_function} & normal fan of a polytope $\tau$\\
	$\moment\colon \centerfiber{0}\rightarrow B$ & \S \ref{subsec:moment_map} & generalized moment map\\
\end{tabular}
\end{center}

\begin{center}
\begin{tabular}{lll}
    $\Upsilon_{\tau}$  & \S \ref{subsubsec:charts_on_B} & coordinate chart on $W(\tau) \subset B$\\
    $\tsing$ (resp. $\tsing_e$) & \S \ref{subsec:tropical_singular_locus} & (resp. essential) tropical singular locus in $B$\\
    $\modmap \colon \centerfiber{0} \rightarrow B $ &Def. \ref{def:modified_moment_map} & surjective map with $\modmap(Z) \subset \tsing_{e}$\\
    $\mathcal{W} = \{W_{\alpha}\}_{\alpha}$ & \S \ref{sec:deformation_via_dgBV} & good cover (Condition \ref{cond:good_cover_of_B}) of $B$ with $V_{\alpha}:= \modmap^{-1}(W_{\alpha})$ being Stein\\
    $\localmod{k}_{\alpha}^{\dagger}$ & \S \ref{sec:deformation_via_dgBV} & $k^{\text{th}}$-order local smoothing model of $V_{\alpha}$\\
	$\bva{k}_{\alpha}^*$ & Def. \ref{def:higher_order_thickening_data_from_gross_siebert} & sheaf of $k^{\text{th}}$-order holomorphic relative log polyvector fields on $\localmod{k}_{\alpha}^{\dagger}$\\
    $\tbva{k}{}_{\alpha}^*$ & Def. \ref{def:higher_order_thickening_data_from_gross_siebert} & sheaf of $k^{\text{th}}$-order holomorphic log de Rham differentials on $\localmod{k}_{\alpha}^{\dagger}$\\
    $\tbva{k}{\parallel}_{\alpha}^*$ & \S \ref{subsubsec:local_deformation_data} & sheaf of $k^{\text{th}}$-order holomorphic relative log de Rham differentials on $\localmod{k}_{\alpha}^{\dagger}$\\
    $\volf{k}_{\alpha}$ & Def. \ref{def:higher_order_thickening_data_from_gross_siebert} & $k^{\text{th}}$-order relative log volume form on $\localmod{k}_{\alpha}^{\dagger}$\\
	$\bvd{k}_{\alpha}$ & \S \ref{subsubsec:local_deformation_data} & BV operator on $\bva{k}_{\alpha}$\\
	$\polyv{k}^{*,*}_{\alpha}$ & Def. \ref{def:local_dgBV_from_resolution} & local sheaf of $k^{\text{th}}$-order polyvector fields\\
	$\totaldr{k}{}^{*,*}_{\alpha}$ & Def. \ref{def:local_dga_from_resolution} & local sheaf of $k^{\text{th}}$-order de Rham forms\\
    $\polyv{k}^{*,*}$ & Def. \ref{def:global_polyvector_and_de_rham} & global sheaf of $k^{\text{th}}$-order polyvector fields from gluing of $\polyv{k}^{*,*}_{\alpha}$'s \\
	$\totaldr{k}{}^{*,*}$ & Def. \ref{def:global_polyvector_and_de_rham} & global sheaf of $k^{\text{th}}$-order de Rham forms from gluing of $\totaldr{k}{}^{*,*}_{\alpha}$'s\\
    $\tform^*$ & Def. \ref{def:global_sheaf_of_monodromy_invariant_tropical_forms} & global sheaf of tropical differential forms on $B$\\
    $W_0$ & \S \ref{subsubsec:semi-flat} & semi-flat locus\\
    $\sfbva{k}^*_{\mathrm{sf}}$ & \S \ref{subsubsec:semi-flat} & sheaf of $k^{\text{th}}$-order semi-flat holomorphic relative vector fields \\
    $\sftbva{k}{}^*_{\mathrm{sf}}$ & \S \ref{subsubsec:semi-flat} & sheaf of $k^{\text{th}}$-order semi-flat holomorphic log de Rham forms \\
	$\sftvbva{k}$ & eqt. \eqref{eqn:tropical_vertex_lie_algebra} & sheaf of $k^{\text{th}}$-order semi-flat holomorphic tropical vertex Lie algebras \\
    $\sfpolyv{k}^{*,*}_{\mathrm{sf}}$ & Def. \ref{def:sheaf_of_sf_polyvector} & sheaf of $k^{\text{th}}$-order semi-flat polyvector fields\\
    $\sftotaldr{k}{}^{*,*}_{\mathrm{sf}}$ & Def. \ref{def:sheaf_of_sf_polyvector} & sheaf of $k^{\text{th}}$-order semi-flat log de Rham forms \\
    $\sftropv{k}^*_{\mathrm{sf}}$ & Def. \ref{def:semi_flat_tropical_vertex_lie_algebra} & sheaf of $k^{\text{th}}$-order semi-flat tropical vertex Lie algebras \\
    $(\mathbf{w},\Theta_{\mathbf{w}})$ & Def. \ref{def:walls} & wall equipped with a wall-crossing factor \\
    $(\mathbf{b},\Theta_{\mathbf{b}})$ & Def. \ref{def:slabs} & slab equipped with a wall-crossing factor \\
    $\mathscr{D}$ & Def. \ref{def:scattering_diagram} & scattering diagram \\
    $W_0(\mathscr{D})$ & \S \ref{subsubsec:scattering_diagram} & complement of joints in the semi-flat locus \\
    $\mathfrak{i}$ & \S \ref{subsubsec:scattering_diagram} & the embedding $\mathfrak{i}\colon W_0(\mathscr{D}) \rightarrow B$ \\
    $\wcs{k}_{\mathscr{D}}$ & \S \ref{subsubsec:scattering_diagram} & $k^{\text{th}}$-order wall-crossing sheaf associated to $\mathscr{D}$ \\
\end{tabular}
\end{center}

\vspace{3mm}

\begin{notation}\label{not:universal_monoid}
	We usually fix a rank $s$ lattice $\blat$ together with a strictly convex $s$-dimensional rational polyhedral cone $Q_\real \subset \blat_\real = \blat\otimes_\inte \real$. We call $Q := Q_\real \cap \blat$ the {\em universal monoid}.
	We consider the ring $\cfr:=\comp[Q]$, a monomial element of which is written as $q^m \in \cfr$ for $m \in Q$, and the maximal ideal $\mathbf{m}:= \comp[Q\setminus \{0\}]$. 
	Then $\cfrk{k}:= \cfr / \mathbf{m}^{k+1}$ is an Artinian ring, and we denote by $\hat{\cfr}:= \varprojlim_{k} \cfrk{k}$ the completion of $\cfr$. We further equip $\cfr$, $\cfrk{k}$ and $\hat{\cfr}$ with the natural monoid homomorphism $Q \rightarrow \cfr$, $m \mapsto q^m$, which gives them the structure of a {\em log ring} (see \cite[Definition 2.11]{gross2011real}); the corresponding log analytic spaces are denoted as $\logs$, $\logsk{k}$ and $\logsf$ respectively.
	
	Furthermore, we let $\logsdrk{}{*} := \cfr \otimes_{\comp} \bigwedge^*\blat_{\comp}$, $\logsdrk{k}{*}:=   \cfrk{k} \otimes_{\comp} \bigwedge^*\blat_{\comp}$ and $\logsdrf{*} := \hat{\cfr} \otimes_{\comp} \bigwedge^*\blat_{\comp}$ (here $\blat_\comp = \blat \otimes_\inte \comp$) be the spaces of log de Rham differentials on $\logs$, $\logsk{k}$ and $\logsf$ respectively, where we write $1 \otimes m = d \log q^m$ for $m \in \blat$; these are equipped with the de Rham differential $\partial$ satisfying $\partial(q^m) = q^m d\log q^m$. We also denote by $\logsvfk{}:= \cfr \otimes_{\comp} \blat_{\comp}^{\vee}$, $\logsvfk{}$ and $\logsvff$, respectively, the spaces of log derivations, which are equipped with a natural Lie bracket $[\cdot,\cdot]$. We write $\partial_n$ for the element $1\otimes n$ with action $\partial_n (q^m) = (m,n) q^m$, where $(m,n)$ is the natural pairing between $\blat_\comp$ and $\blat^{\vee}_\comp$.
\end{notation}
\section{Gross--Siebert's cone construction of maximally degenerate Calabi--Yau varieties}\label{sec:gross_siebert}

This section is a brief review of Gross--Siebert's construction of the maximally degenerate Calabi--Yau variety $\centerfiber{0}$ from the affine manifold $B$ and its log structures from slab functions \cite{Gross-Siebert-logI, Gross-Siebert-logII, gross2011real}. 

\subsection{Integral tropical manifolds}\label{subsec:integral_affine_manifolds}

We first recall the notion of integral tropical manifolds from \cite[\S 1.1]{gross2011real}.
Given a lattice $M$ of rank $n$, a \textit{rational convex polyhedron} $\sigma$ is a convex subset in $M_{\real}$ given by a finite intersection of rational (i.e. defined over $M_{\bb{Q}}$) affine half-spaces.  We usually drop the attributes ``rational'' and ``convex'' for polyhedra. A polyhedron $\sigma$ is said to be \textit{integral} if all its vertices lie in $M$; a \textit{polytope} is a compact polyhedron. The group $\mathbf{Aff}(M):= M \rtimes \mathrm{GL}(M)$ of integral affine transformations acts on the set of polyhedra in $M_{\real}$.
	Given a polyhedron $\sigma \subset M_\real$, let $\tanpoly_{\sigma,\real} \subset M_\real$ be the smallest affine subspace containing $\sigma$, and denote by $\tanpoly_{\sigma} := \tanpoly_{\sigma,\real} \cap M$ the corresponding lattice. The \textit{relative interior} $\reint(\sigma)$ refers to taking the interior of $\sigma$ in $\tanpoly_{\sigma,\real}$. There is an identification $T_{\sigma,x} \cong \tanpoly_{\sigma,\real}$ for the tangent space at $x \in \reint(\sigma)$. Write $\partial \sigma = \sigma \setminus \reint(\sigma)$. Then a \textit{face} of $\sigma$ is the intersection of $\partial \sigma$ with a supporting hyperplane. Codimension one faces are called \textit{facets}.
	
	Let $\pcate$ be the category whose objects are integral polyhedra and morphisms consist of the identity and integral affine isomorphisms onto faces (i.e. an integral affine morphism $\tau \rightarrow \sigma$ which is an isomorphism onto its image and identifies $\tau$ with a face of $\sigma$).  
	An \textit{integral polyhedral complex} is a functor $\mathtt{F}\colon \pdecomp \rightarrow \pcate$ from a finite category $\pdecomp$ to $\pcate$ such that every face of $\mathtt{F}(\sigma)$ still lies in the image of $\mathtt{F}$, and there is at most one arrow $\tau \rightarrow \sigma$ for every pair $\tau ,\sigma \in \pdecomp$. By abuse of notation, we usually drop the notation $\mathtt{F}$ and write $\sigma \in \pdecomp$ to represent an integral polyhedron in the image of the functor. From an integral polyhedral complex, we obtain a topological space $B := \varinjlim_{\sigma \in \pdecomp} \sigma$ via gluing of the polyhedra along faces. We further assume that:
\begin{enumerate}
	\item the natural map $\sigma \rightarrow B$ is injective for each $\sigma \in \pdecomp$, so that $\sigma$ can be identified with a closed subset of $B$ called a \textit{cell}, and a morphism $\tau \rightarrow \sigma$ can be identified with an inclusion of subsets;
	
	\item a finite intersection of cells is a cell; and
	
	\item $B$ is an orientable connected topological manifold of dimension $n$ without boundary which in addition satisfies the condition that $H^1(B,\mathbb{Q}) = 0$. 
	\end{enumerate}

\begin{remark}
	The condition $H^1(B,\mathbb{Q}) = 0$ will be used only in Theorem \ref{prop:Maurer_cartan_equation_unobstructed} to ensure that $H^{1}(\centerfiber{0},\mathcal{O}) = H^1(B,\comp) = 0$, where $\centerfiber{0}$ is the degenerate Calabi--Yau variety that we are going to construct.\footnote{In his recent work \cite{Felten23}, Felten was able to prove Theorem \ref{prop:Maurer_cartan_equation_unobstructed} without assuming that $H^1(B,\mathbb{Q}) = 0$.} This corresponds to the condition that $b_1=0$ for smooth Calabi--Yau manifolds. 
	\end{remark}

The set of $k$-dimensional cells is denoted by $\pdecomp^{[k]}$, and the $k$-skeleton by $\pdecomp^{[\leq k]}$. For every $\tau \in \pdecomp$, we define its \textit{open star} by
$$
U_{\tau}:= \bigcup_{\sigma \supset \tau} \reint(\sigma),
$$
which is an open subset of $B$ containing $\reint(\tau)$. 
A \textit{fan structure along $\tau \in \pdecomp^{[n-k]}$} is a continuous map $S_{\tau} \colon U_{\tau} \rightarrow \real^{k}$ such that
\begin{itemize}
	\item $S^{-1}_{\tau}(0) = \reint(\tau)$,
	\item for every $\sigma \supset \tau$, the restriction $S_{\tau}|_{\reint(\sigma)}$ is an integral affine submersion onto its image (meaning that it is induced by some epimorphism $\tanpoly_{\sigma} \rightarrow W \cap \inte^k$ for some vector subspace $W\subset \real^k$), and
	\item the collection of cones $\{ K_{\tau}\sigma:= \real_{\geq 0} S_{\tau}(\sigma \cap U_{\tau}) \}_{\sigma \supset \tau}$ forms a complete finite fan $\Sigma_{\tau}$.
\end{itemize}
Two fan structures along $\tau$ are \emph{equivalent} if they differ by composition with an integral affine transformation of $\real^{k}$.
If $S_\tau$ is a fan structure along $\tau$ and $\sigma \supset \tau$, then $U_{\sigma} \subset U_{\tau}$ and there is a fan structure along $\sigma$ induced from $S_{\tau}$ via the composition:
$$
U_{\sigma} \hookrightarrow U_{\tau} \rightarrow \bb{R}^{k} \twoheadrightarrow \bb{R}^{l},
$$
where $\bb{R}^{k} \rightarrow \bb{R}^{k}/ \real S_{\tau}(\sigma \cap U_{\tau}) \cong \bb{R}^{l}$ is the quotient map. 

\begin{definition}[\cite{gross2011real}, Def. 1.2]\label{def:integral_tropical_manifold}
	An \emph{integral tropical manifold} is an integral polyhedral complex $(B,\pdecomp)$ together with a fan structure $S_{\tau}$ along each $\tau \in \pdecomp$ such that whenever $\tau \subset \sigma$, the fan structure induced from $S_{\tau}$ is equivalent to $S_{\sigma}$. 
\end{definition}

Taking sufficiently small and mutually disjoint open subsets $W_{v} \subset U_{v}$ for $v \in \pdecomp^{[0]}$ and $\reint(\sigma)$ for $\sigma \in \pdecomp^{[n]}$, there is an integral affine structure on $\bigcup_{v \in \pdecomp^{[0]}} W_v \cup \bigcup_{\sigma \in \pdecomp^{[n]}} \reint(\sigma)$. 
We will further choose the open subsets $W_v$'s and $\reint(\sigma)$'s so that the affine structure is defined outside a closed subset $\Gamma$ of codimension two in $B$, as in \cite[\S 1.3]{Gross-Siebert-logI}. This affine structure allows us to use parallel transport to identify the tangent spaces $T_x B$ for different points $x$ outside the closed subset. For every $\tau$ we choose a maximal cell $\sigma\supset \tau$ and consider the lattice of normal vectors $\norpoly_{\tau}=\tanpoly_{\sigma}/\tanpoly_{\tau}$ (we suppress the dependence on $\sigma$ because we will see that $\tanpoly_{\tau}$ is monodromy invariant under the monodromy transformation given by any two vertices of $\tau$ and any two maximal cells containing $\tau$). We can identify $\norpoly_{\tau}$ with $\inte^{k}$ via $S_{\tau}$, and write the fan structure as $S_{\tau} \colon U_{\tau} \rightarrow \norpoly_{\tau,\real}$. 

    \begin{example}\label{eg:K3_example}
	We take a $2$-dimensional example from \cite[Ex. 6.74]{dbrane} to illustrate the above definitions. Let $\Xi$ be the convex hull of the points 
	$$
	p_0 = \begin{bmatrix} -1 \\ -1 \\ -1  \end{bmatrix}, \ p_1 = \begin{bmatrix} 3\\-1\\-1\end{bmatrix}, \ p_2 = \begin{bmatrix} -1 \\ 3 \\ -1 \end{bmatrix}, \ p_3 = \begin{bmatrix} -1 \\ -1 \\3 \end{bmatrix}, 
	$$
	so $\Xi$ is a $3$-simplex. Take $B$ (as a topological space) to be the boundary of $\Xi$. 
    The polyhedral decomposition $\pdecomp$ is defined so that the integral points are vertices as shown in Figure \ref{fig:k3_polytope}.
	\begin{figure}[h!]
		\includegraphics[scale=0.3]{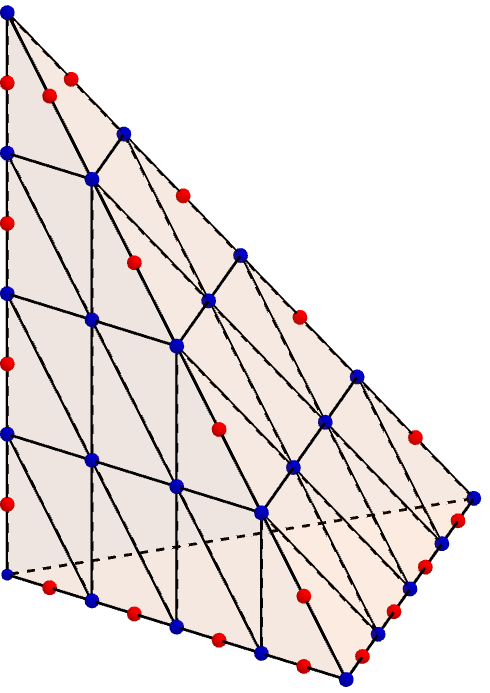}
		\caption{The polyhedral decomposition}\label{fig:k3_polytope}
	\end{figure}
	
	Then we define affine coordinate charts on $\bigcup_{\sigma \in \pdecomp^{[n]}} \reint(\sigma) \cup \bigcup_{v\in \pdecomp^{[0]}} W_v$ as follows. On $\reint(\sigma)$, we take $\psi_{\sigma} \colon \reint(\sigma) \rightarrow \tanpoly_{\sigma,\real}$ which maps homeomorphically onto its image. At a vertex $v$ treated as a vector in $\real^3$, we let $\psi_v\colon W_v \subset \real^3 \rightarrow \real^3 / \real v$, where $\real^3 \rightarrow \real^3/\real v$ is the natural projection onto the quotient. By \cite[Prop. 6.81]{dbrane}, this gives an integral affine manifold with singularities. The affine structure can be extended to the complement of a subset $\Gamma$ consisting of $24$ points lying on the six edges of $\Xi$, with each edge containing $4$ points (colored in red in Figure \ref{fig:k3_polytope}). The fan structure $S_\tau$ can be defined similarly. 
	
	Locally near each singular point $p \in \Gamma$ contained in an edge $\rho$, the affine structure is described as a gluing of two affine charts $U_{\mathrm{I}}\subset \real^2 \setminus \{0\} \times \real_{\geq 0}$ and $U_{\mathrm{II}}\subset \real^2 \setminus 0 \times \real_{\leq 0}$  as in \cite[\S 3.2]{gross2011invitation}. The change of coordinates from $U_{\mathrm{I}}$ to $U_{\mathrm{II}}$ is given by the restriction of the map $\Upsilon$ from $(\real \setminus \{0\}) \times \real$ to itself defined by 
	$$
	(x,y) \mapsto \begin{cases}
		(x,y), & x<0\\
		(x,x+y), & x>0. 
		\end{cases}
	$$
	The fan structure $S_\rho\colon U_{\rho} \rightarrow \real$ is given as $S_{\rho}(x,y) = x$ and the fan $\Sigma_{\rho}$ is the toric fan for $\mathbb{P}^1$. Figure \ref{fig:affine_chart} below illustrates the situation.
	
	\begin{figure}[h!]
	\includegraphics[scale=0.5]{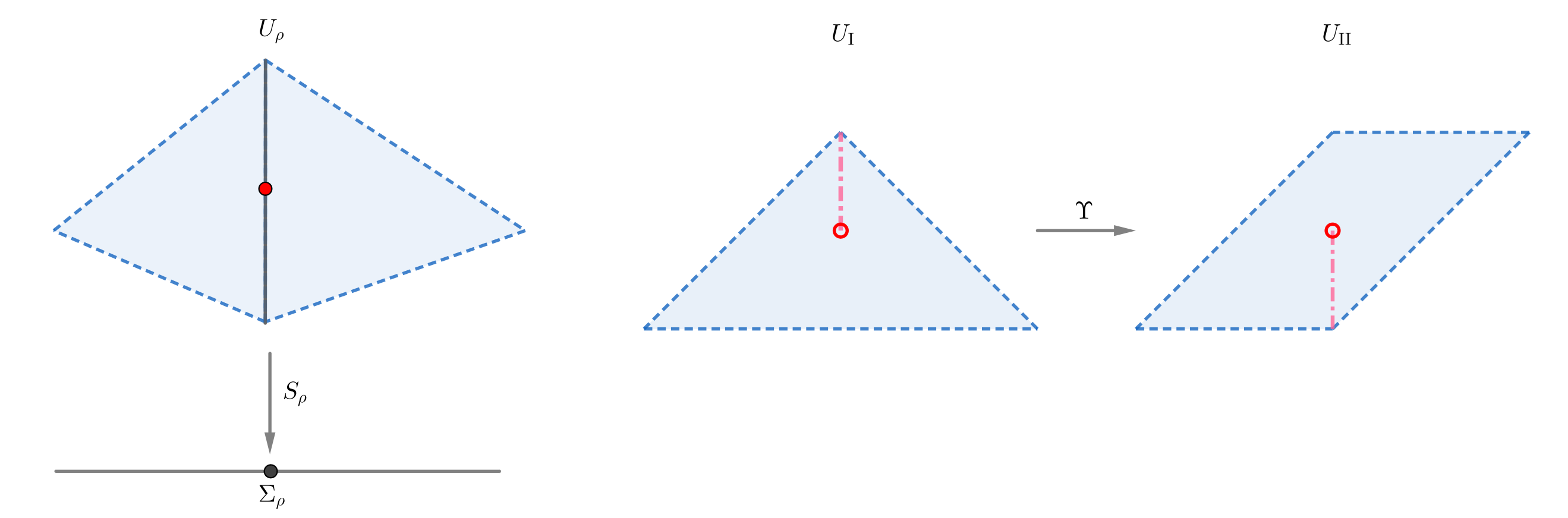}
	\caption{Affine coordinate charts}\label{fig:affine_chart}
	\end{figure}
	
	With the structure of an integral tropical manifold, the corners and edges in Figure \ref{fig:k3_polytope} are flattened via the affine coordinate charts, and we can view $(B,\pdecomp)$ as the 2-sphere equipped with a polyhedral decomposition and with $24$ affine singularities. Such an affine structure with singularities also appears in the base $B$ of an SYZ fibration of a K3 surface.
	\end{example}
	
	\begin{example}\label{eg:3d_example}
		A $3$-dimensional example can be constructed as in \cite[Ex. 6.74]{dbrane}. Take $\Xi$ to be the convex hull of the points 
		$$
		p_0 = \begin{bmatrix} -1 \\ -1 \\ -1 \\-1  \end{bmatrix}, \ p_1 = \begin{bmatrix} 4 \\-1\\-1 \\ -1 \end{bmatrix}, \ p_2 = \begin{bmatrix} -1 \\ 4 \\ -1 \\ -1 \end{bmatrix}, \ p_3 = \begin{bmatrix} -1 \\ -1 \\4 \\ -1 \end{bmatrix}, \ p_4 = \begin{bmatrix} -1 \\ -1 \\-1 \\ 4 \end{bmatrix},
		$$
		which gives a $4$-simplex. Take $B$ (as a topological space) to be the boundary of $\Xi$. There are five $3$-dimensional maximal cells intersecting along ten $2$-dimensional facets. The polyhedral decomposition $\pdecomp$ on each facet is as in Figure \ref{fig:three_d_polyhedral_decomposition}. 
		
		\begin{figure}[h!]
		\includegraphics[scale=0.8]{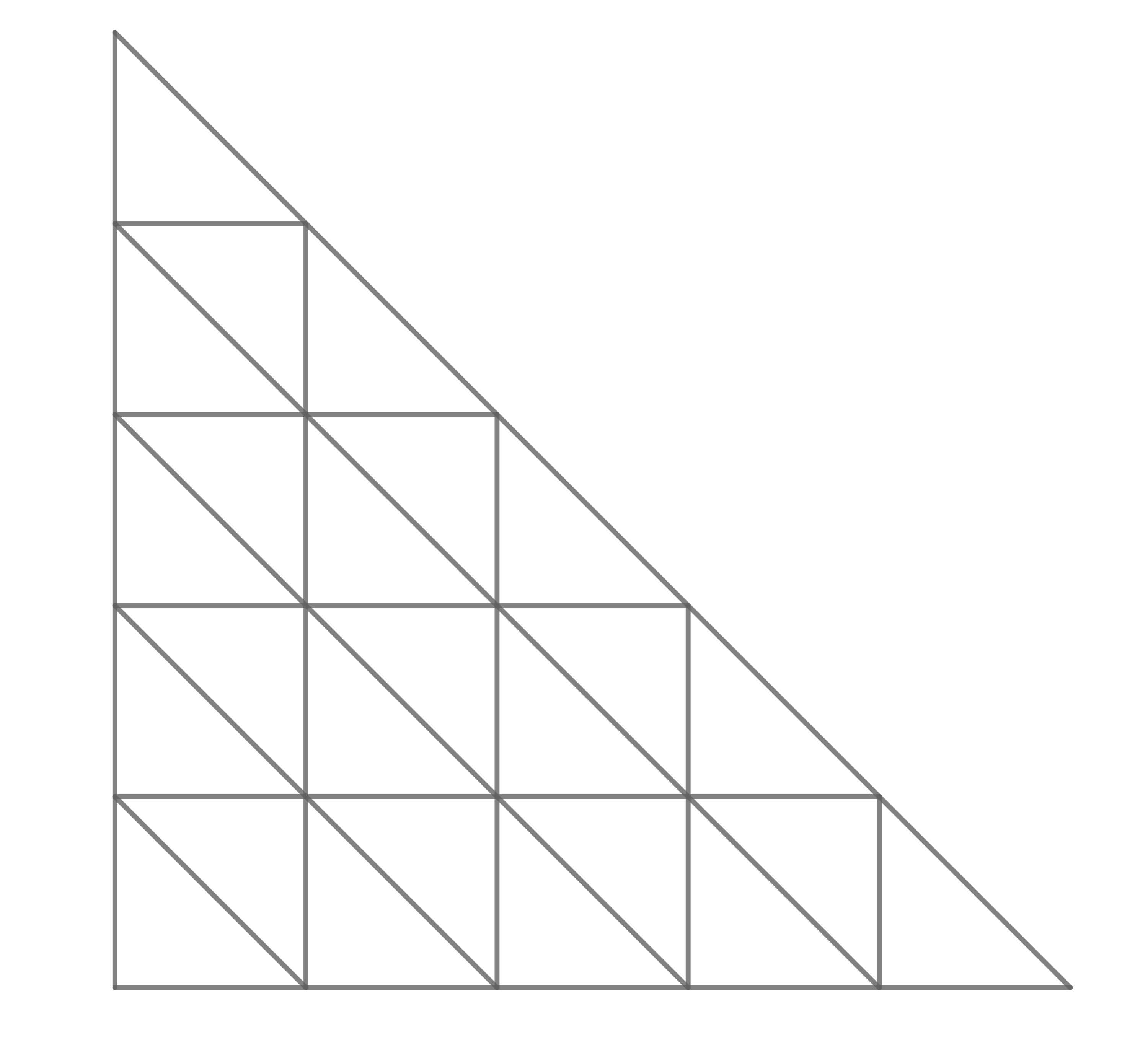}
		\caption{The polyhedral decomposition on a facet}\label{fig:three_d_polyhedral_decomposition}
		\end{figure}
		The affine structure can be extended to the complement of codimension 2 closed subset $\Gamma$ whose intersection with a triangle in Figure \ref{fig:three_d_polyhedral_decomposition} is a $Y$-shaped locus. Locally near each of these triangles, it looks like Figure \ref{fig:3_d_singular_locus_1}.
		\begin{figure}[h!] 
			\centering
			\subfloat[$Y$-vertex of type I]{%
				\includegraphics[width=0.35\textwidth]{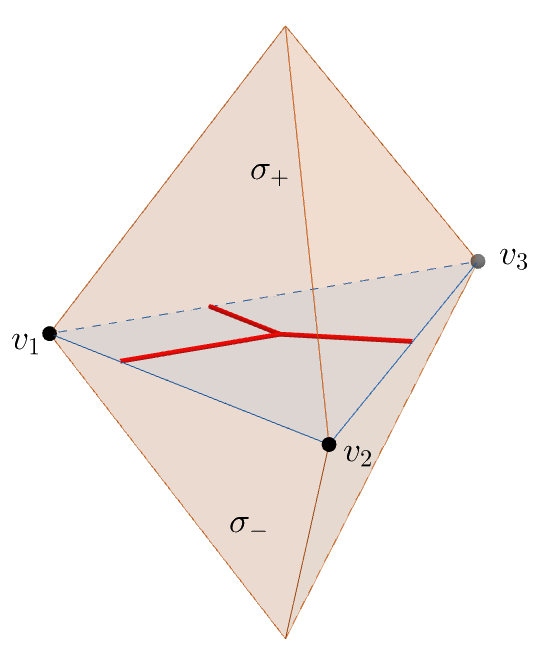}%
				\label{fig:3_d_singular_locus_1}%
			}%
			\hspace{12mm}%
			\subfloat[$Y$-vertex of type II]{%
				\includegraphics[width=0.45\textwidth]{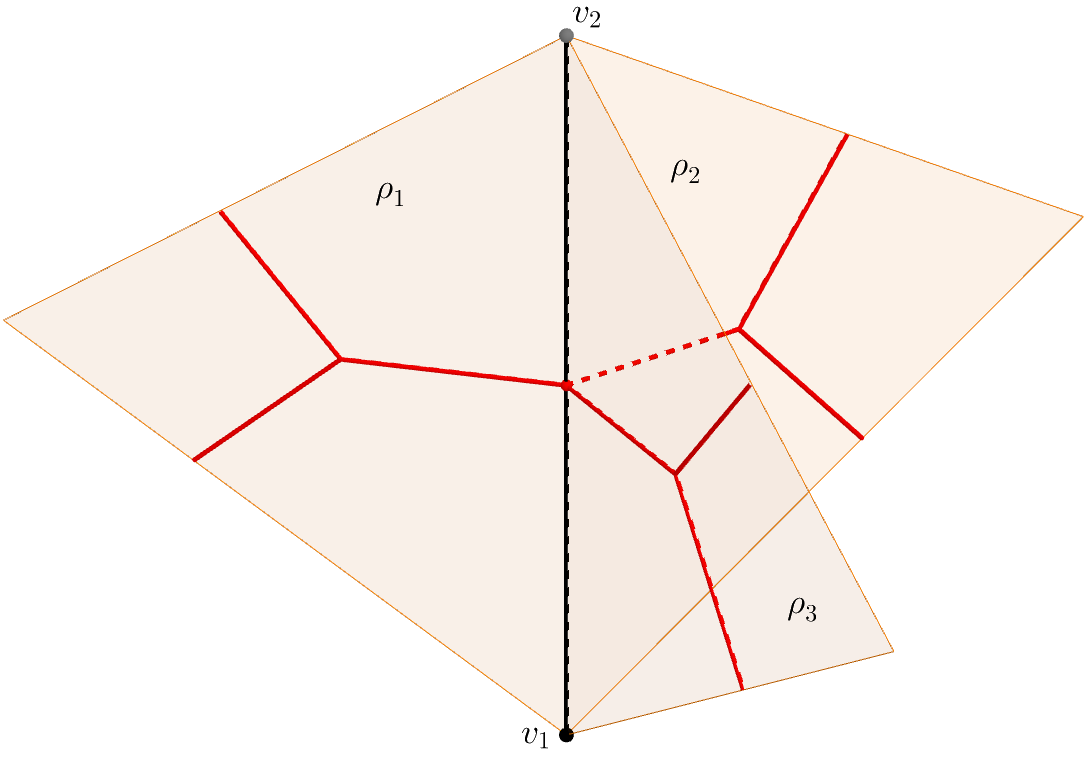}%
				\label{fig:3_d_singular_locus_2}%
			}%
		\end{figure}
		$\Xi$ has ten $1$-dimensional faces, each of which is an edge with affine length $5$. The polyhedral decomposition $\pdecomp$ divides each edge into $5$ intervals as we can see in Figure \ref{fig:three_d_polyhedral_decomposition}. Locally near each of these length $1$ intervals, there are three $2$-cells of $\pdecomp$ intersecting along it. The locus $\Gamma$ on each $2$-cell intersects on the interval as shown in Figure \ref{fig:3_d_singular_locus_2}.
		\end{example}

    \begin{definition}[\cite{Gross-Siebert-logI}, Def. 1.43]\label{def:piecewise_linear}
	An \emph{integral affine function} on an open subset $U \subset B$ is a continuous function $\varphi$ on $U$ which is integral affine on $U \cap \reint(\sigma)$ for $\sigma \in \pdecomp^{[n]}$ and on $U \cap W_v$ for $v \in \pdecomp^{[0]}$. We denote by $\cu{A}\mathit{ff}_{B}$ (or simply $\cu{A}\mathit{ff}$) \emph{the sheaf of integral affine functions on $B$}.
	
	A \emph{piecewise integral affine function} (abbrev.\ as \emph{PA-function}) on $U$ is a continuous function $\varphi$ on $U$ which can be written as $\varphi = \psi + S_{\tau}^*(\bar{\varphi})$ on $U \cap U_{\tau}$ for every $\tau \in \pdecomp$, where $\psi \in \cu{A}\mathit{ff}(U \cap U_\tau)$ and $\bar{\varphi}$ is a piecewise linear function on $\norpoly_{\tau,\real}$ with respect to the fan $\Sigma_{\tau}$. \emph{The sheaf of PA-functions on $B$} is denoted by $\cu{PL}_{\pdecomp}$.
\end{definition}

There is a natural inclusion $\cu{A}\mathit{ff}\hookrightarrow\cu{PL}_{\pdecomp}$, and we let $\cu{MPL}_{\pdecomp}$ be the quotient:
$$0\to \cu{A}\mathit{ff}\to\cu{PL}_{\pdecomp}\to\cu{MPL}_{\pdecomp}\to 0.$$
Locally, an element $\varphi\in\Gamma(B,\cu{MPL}_{\pdecomp})$ is a collection of piecewise affine functions $\{\varphi_U\}$ such that on each overlap $U\cap V$, the difference
$\varphi_U|_{V}-\varphi_V|_{U}$ is an integral affine function on $U\cap V$.

\begin{definition}[\cite{Gross-Siebert-logI}, Def. 1.45 and 1.47]\label{def:strictly_convex_piecewise_affine}
	The sheaf $\cu{MPL}_{\pdecomp}$ is called \emph{the sheaf of multi-valued piecewise affine functions (abbrev.\ as MPA-funtions) of the pair $(B,\pdecomp)$}.
	A section $\varphi\in H^0(B,\cu{MPL}_{\pdecomp})$ is said to be \emph{convex} (resp. \emph{strictly convex}) if for any vertex $\{v\}\in\pdecomp$, there is a convex (resp. strictly convex) representative $\varphi_v$ on $U_v$. (Here, convexity (resp. strict convexity) means if we take any maximal cone $\sigma \subset U_v$ with the affine function $l_{\sigma}\colon U_v\rightarrow \real$ defined by requiring $\varphi_v|_{\sigma} = l_{\sigma}$, we always have $\varphi_v(y)\geq  l_{\sigma}(y)$ (resp. $\varphi_v(y)> l_{\sigma}(y)$) for $y\in U_v \setminus \sigma$). 
\end{definition}
The set of all convex multi-valued piecewise affine functions gives a sub-monoid of $H^0(B,\cu{MPL}_{\pdecomp})$ under addition, denoted as $H^0(B,\cu{MPL}_{\pdecomp},\bb{N})$; we let $Q$ be the dual monoid.

\begin{definition}[\cite{Gross-Siebert-logI}, Def. 1.48]\label{def:strictly_convex_multi_valued_function}
	The polyhedral decomposition $\pdecomp$ is said to be \emph{regular} if there exists a strictly convex multi-valued piecewise linear function $\varphi\in H^0(B,\cu{MPL}_{\pdecomp})$. 
\end{definition}

We always assume that $\pdecomp$ is regular with a fixed strictly convex $\varphi \in H^0(B,\cu{MPL}_{\pdecomp})$.

\subsection{Monodromy, positivity and simplicity}\label{subsec:monodromy_data}

To describe monodromy, we consider two maximal cells $\sigma_{\pm}$ and two of their common vertices $v_{\pm}$. Taking a path $\gamma$ going from $v_+$ to $v_-$ through $\sigma_+$, and then from $v_-$ back to $v_+$ through $\sigma_-$, we obtain a monodromy transformation $T_{\gamma}$. 
As in \cite[\S 1.5]{Gross-Siebert-logI}, we are interested in two cases. The first case is when $v_+$ is connected to $v_-$ via a bounded edge $\omega \in \pdecomp^{[1]}$. Let $d_{\omega} \in \tanpoly_{\omega}$ be the unique primitive vector pointing to $v_-$ along $\omega$. For an integral tangent vector $m \in T_{v_+} := T_{v_+,\inte}B$, the monodromy transformation $T_{\gamma}$ is given by 
\begin{equation}\label{eqn:monodromy_transformation_edge_fixed}
T_{\gamma}(m) = m +  \langle m , n^{\sigma_+ \sigma_-}_{\omega} \rangle d_{\omega}
\end{equation}
for some $n^{\sigma_+ \sigma_-}_{\omega} \in \norpoly_{\sigma_+ \cap \sigma_-}^*\subset T_{v_+}^* $, where $\langle \cdot, \cdot \rangle$ is the natural pairing between $T_{v_+}$ and $T_{v_+}^*$. 
The second case is when $\sigma_+$ and $\sigma_-$ are separated by a codimension one cell $\rho \in \pdecomp^{[n-1]}$. Let $\check{d}_{\rho}\in \norpoly_{\rho}^*$ be the unique primitive covector which is positive on $\sigma_+$. The monodromy transformation is given by 
\begin{equation}\label{eqn:monodromy_transformation_rho_fixed}
T_{\gamma}(m) = m +  \langle m , \check{d}_{\rho} \rangle m^{\rho}_{v_+v_-}
\end{equation}
for some $m^{\rho}_{v_+v_-} \in \tanpoly_{\tau}$, where $\tau \subset \rho$ is the smallest face of $\rho$ containing $v_{\pm}$. 
In particular, if we fix both $v_{\pm} \in \omega \subset \rho \subset\sigma_{\pm}$, one obtains the formula
\begin{equation}\label{eqn:monodromy_transformation_both_fixed}
 T_{\gamma}(m) = m +  \kappa_{\omega\rho}\langle m , \check{d}_{\rho} \rangle d_{\omega}
\end{equation}
for some integer $\kappa_{\omega\rho}$. 
 \begin{definition}[\cite{Gross-Siebert-logI}, Def. 1.54]\label{def:positivity_assumption}
 	We say that $(B,\pdecomp)$ is \emph{positive} if $\kappa_{\omega\rho} \geq 0$ for all $\omega \in \pdecomp^{[1]}$ and $\rho \in \pdecomp^{[n-1]}$ with $\omega \subset \rho$. 
 \end{definition}
 
%\subsubsection{Monodromy polytope and dual monodromy polytope}\label{subsubsec:monodromy_polytope}

Following \cite[Definition 1.58]{Gross-Siebert-logI}, we package the monodromy data into polytopes associated to $\tau \in \pdecomp^{[k]}$ for $1\leq k \leq n-1$.
The simplest case is when $\rho \in \pdecomp^{[n-1]}$, whose \emph{monodromy polytope} is defined by fixing a vertex $v_0 \in \rho$ and setting 
\begin{equation}\label{eqn:monodromy_polytope_for_rho}
 \Delta(\rho):= \mathrm{Conv}\{ m^{\rho}_{v_0 v} \ | \ v \in \rho, \ v \in \pdecomp^{[0]} \} \subset \tanpoly_{\rho,\real},
\end{equation}
where $\mathrm{Conv}$ refers to taking the convex hull. It is well-defined up to translation and independent of the choice of $v_0$. The normal fan of $\rho$ in $\tanpoly_{\rho,\real}^*$ is a refinement of the normal fan of $\Delta(\rho)$. 
Similarly, when $\omega \in \pdecomp^{[1]}$, one defines the \emph{dual monodromy polytope} by fixing $\sigma_0 \supset \omega$ and setting
\begin{equation}\label{eqn:dual_monodromy_polytope_for_omega}
 \check{\Delta}(\omega):= \mathrm{Conv}\{ n^{\sigma_0 \sigma}_{\omega} \ | \ \sigma\supset \omega, \ \sigma \in \pdecomp^{[n-1]} \} \subset \norpoly_{\omega,\real}^*.
\end{equation}
Again, this is well-defined up to translation and independent of the choice of $\sigma_0$. The fan $\Sigma_{\omega}$ in $\norpoly_{\omega,\real}$ is a refinement of the normal fan of $\check{\Delta}(\omega)$.
For $1< \dim_{\real}(\tau) <n-1$, a combination of monodromy and dual monodromy polytopes is needed. We let $\pdecomp_1(\tau) = \{ \omega \ | \ \omega \in \pdecomp^{[1]}, \ \omega \subset \tau \}$ and $\pdecomp_{n-1}(\tau) = \{ \rho \ | \ \rho \in \pdecomp^{[n-1]}, \ \rho \supset \tau \}$.
For each $\rho \in \pdecomp_{n-1}(\tau)$, we choose a vertex $v_0 \in \rho$ and let
$$\Delta_{\rho}(\tau):= \mathrm{Conv}\{ m^{\rho}_{v_0 v} \ | \ v \in \tau, \ v \in \pdecomp^{[0]} \} \subset \tanpoly_{\tau,\real}.$$ Similarly, for each $\omega \in \pdecomp_1(\tau)$, we choose $\sigma_0 \supset \tau$ and let 
$$\check{\Delta}_{\omega}(\tau):= \mathrm{Conv} \{ n^{\sigma_0 \sigma}_{\omega} \ | \ \sigma\supset \tau, \ \sigma \in \pdecomp^{[n-1]} \} \subset \norpoly_{\tau,\real}^*.$$
These are well-defined up to translation and independent of the choices of $v_0$ and $\sigma_0$ respectively. 
 
\begin{definition}[\cite{Gross-Siebert-logI}, Def. 1.60]\label{def:simplicity}
 	We say $(B,\pdecomp)$ is \emph{simple} if, for every $\tau \in \pdecomp$, there are disjoint non-empty subsets
 	$$
 	\Omega_1,\dots,\Omega_p \subset \pdecomp_1(\tau), \quad R_1,\dots, R_p \subset \pdecomp_{n-1}(\tau)
 	$$
 	(where $p$ depends on $\tau$) such that 
 	\begin{enumerate}
 		\item for $\omega \in \pdecomp_1(\tau)$ and $\rho \in \pdecomp_{n-1}(\tau)$, $\kappa_{\omega\rho} \neq 0$ if and only if $\omega \in \Omega_i$ and $\rho \in R_i$ for some $1 \leq i \leq p$;
 		
 		\item $\Delta_{\rho}(\tau)$ is independent (up to translation) of $\rho \in R_i$ and will be denoted by $\Delta_i(\tau)$; similarly, $\check{\Delta}_{\omega}(\tau)$ is independent (up to translation) of $\omega \in \Omega_i$ and will be denoted by $\check{\Delta}_i(\tau)$;
 		
 		\item if $\{e_1,\dots,e_p\}$ is the standard basis in $\inte^{p}$, then
 		$$
 		\Delta(\tau):= \mathrm{Conv} \left\{ \bigcup_{i=1}^{p} \Delta_i(\tau) \times \{e_i\} \right\}, \quad \check{\Delta}(\tau):= \mathrm{Conv} \left\{ \bigcup_{i=1}^{p} \check{\Delta}_i(\tau) \times \{e_i\} \right\}
 		$$
 		are elementary simplices (i.e. a simplex whose only integral points are its vertices) in $\left(\tanpoly_{\tau} \oplus \inte^{p}\right)_{\real}$ and $\left( \norpoly_{\tau}^*\oplus \inte^{p} \right)_{\real}$ respectively. 
   \end{enumerate}
\end{definition}
 
We need the following stronger condition in order to apply \cite[Thm. 3.21]{Gross-Siebert-logII} in a later stage:
\begin{definition}\label{def:strongly simple}
We say $(B,\pdecomp)$ is \emph{strongly simple} if it is simple, and for every $\tau \in \pdecomp$, both $\Delta(\tau)$ and $\check{\Delta}(\tau)$ are standard simplices. 
\end{definition}

\begin{example}\label{eg:2d_monodromy}
	Consider the $2$-dimensional example in Example \ref{eg:K3_example}. Following \cite[Ex. 6.82(1)]{dbrane}, we may choose the two adjacent vertices in Figure \ref{fig:k3_polytope} to be $v_1 = \begin{bmatrix} -1 & -1 & -1 \end{bmatrix}^T$ and $v_2 = \begin{bmatrix} 0 & -1 & -1 \end{bmatrix}^T$ which bound a $1$-cell $\rho$. The two adjacent maximal cells are given by
    $\sigma_+ \subset \{ b \ | \ \langle w_+,b\rangle =1\}$ where $w_+ = \begin{bmatrix} 0 & 0 & -1 \end{bmatrix}^T$ and $\sigma_- \subset \{ b \ | \ \langle w_-,b\rangle =1\}$ where $w_- = \begin{bmatrix} 0 & -1 & 0 \end{bmatrix}^T$. The tangent lattice $T_{v_1}$ can be identified with $\inte^3/\inte \cdot v_1$ equipped with the basis $e_1 =  \begin{bmatrix} 1 & 0 &0 \end{bmatrix}^T$, $e_2 =  \begin{bmatrix} 0 & 1& 0 \end{bmatrix}^T$. If we let $\gamma$ be a loop going from $v_1$ to $v_2$ through $\sigma_+$ and going back to $v_1$ through $\sigma_-$, we have
	$$
	T_{\gamma}(m) = m + \langle \begin{bmatrix} 0 &  1 &  -1 \end{bmatrix}^T, m \rangle e_1
	$$
	for $m \in T_{v_1}$. Therefore, we have $p=1$, $\Delta_1(\rho) = \mathrm{Conv} \{0,e_1\}$ and $\check{\Delta}_{1}(\rho) = \mathrm{Conv} \{ 0 , w_+ - w_-\}$. This is an example of a positive and strongly simple $(B,\pdecomp)$ (Definitions \ref{def:positivity_assumption} and \ref{def:strongly simple}).
	\end{example}

\begin{example}\label{eg:3d_monodromy}
	Next we consider the two types of $Y$-vertex in Example \ref{eg:3d_example}. 
	
	We begin with $Y$-vertex of type $I$ in Figure \ref{fig:3_d_singular_locus_1}. Following \cite[Ex. 6.82(2)]{dbrane}, the three vertices $v_1,v_2,v_3$ can be chosen to be 
	$$
	v_1 = \begin{bmatrix} -1 &  -1&  -1 & -1 \end{bmatrix}^T, \ v_2 = \begin{bmatrix} 0 &  -1 &  -1 &  -1 \end{bmatrix}^T, \ v_3 = \begin{bmatrix} -1 &  0 &  -1 & -1 \end{bmatrix}^T,
	$$
	and $\sigma_+ \subset \{ b \in \real^4 \ | \ \langle w_+ , b \rangle  = 1 \}$, $\sigma_-  \subset \{ b \in \real^4 \ | \ \langle w_- , b \rangle  = 1 \}$ are $3$-cells of $B$ lying in the affine hyperplanes with dual vector $w_+ = \begin{bmatrix} 0 & 0 & -1 & 0 \end{bmatrix}^T$ and $w_- = \begin{bmatrix} 0 & 0 & 0 & -1 \end{bmatrix}^T$ respectively. If we identify $T_v$ with $\tanpoly_{\sigma_+}$ via parallel transport and choose the basis of $\tanpoly_{\sigma_+}$ as
	$$
	e_1 = \begin{bmatrix} 1 & 0 & 0 & 0 \end{bmatrix}^T, \ e_2 = \begin{bmatrix} 0 & -1 & 0 & 0 \end{bmatrix}^T, \ e_3 = \begin{bmatrix} 0 & 0 & 0 & 1 \end{bmatrix}^T,
	$$
	then the monodromy transformations are given by
	$$
	T_{\gamma_1}= \begin{bmatrix} 
		1 & 0 & 1 \\
		0 & 1 & 0 \\
		0 & 0 & 1
		\end{bmatrix}, \ 
		T_{\gamma_2} = \begin{bmatrix} 
			1 & 0 & -1 \\
			0 & 1 & -1 \\
			0 & 0 & 1
		\end{bmatrix}, \
		T_{\gamma_3} = \begin{bmatrix} 
			1 & 0 & 0 \\
			0 & 1 & 1 \\
			0 & 0 & 1
		\end{bmatrix},
	$$
    where $\gamma_i$ is the loop going from $v_i$ to $v_{i+1}$ through $\sigma_+$ and going back to $v_i$ through $\sigma_-$, with indices of $v_i$'s taken modulo $3$.
	In this case, we have $p=1$, $\Delta_1(\rho) = \mathrm{Conv}\{0, e_1,-e_2\}$ is a $2$-simplex and $\check{\Delta}_{1}(\rho) = \mathrm{Conv}\{0, w_+ - w_-\}$ is a $1$-simplex.
	
	For the $Y$-vertex of type II in Figure \ref{fig:3_d_singular_locus_2}, we can choose 
    $$v_1 = \begin{bmatrix} -1 & -1 & -1 & -1 \end{bmatrix}^T,\ v_2 = \begin{bmatrix} 0 & -1 & -1 & -1 \end{bmatrix}^T,$$
    which are the end-points of a $1$-cell $\tau$. We choose the three maximal cells $\sigma_1$, $\sigma_2$ and $\sigma_3$ intersecting at $\tau$ to be the $3$-cells lying in affine hyperplanes defined by $\{b \ | \ \langle w_i, b \rangle = 1\}$, where 
	$$
	w_1 = \begin{bmatrix} 0 & 0&  -1 & 0 \end{bmatrix}^T, \ w_2 = \begin{bmatrix} 0 &  0 &  0 &  -1 \end{bmatrix}^T, \ w_3 = \begin{bmatrix} 0 &  -1 &  0 & 0 \end{bmatrix}^T.
	$$
	Let $\tilde{\gamma}_i$ be the loop going from $v_1$ to $v_2$ through $w_i$ and then going back to $v_1$ through $w_{i+1}$, with indices taken to be modulo $3$. Then the corresponding monodromy transformations are given by
	$$
T_{\gamma_1}= \begin{bmatrix} 
	1 & 0 & 1 \\
	0 & 1 & 0 \\
	0 & 0 & 1
\end{bmatrix}, \ 
T_{\gamma_2} = \begin{bmatrix} 
	1 & 1& 0 \\
	0 & 1 & 0 \\
	0 & 0 & 1
\end{bmatrix}, \
T_{\gamma_3} = \begin{bmatrix} 
	1 & -1& -1 \\
	0 & 1 & 0 \\
	0 & 0 & 1
\end{bmatrix},
	$$
	with respect to the basis 
	$$
	e_1 = \begin{bmatrix} 1 & 0 & 0 & 0 \end{bmatrix}^T, \ e_2 = \begin{bmatrix} 0 & 1 & 0 & 0 \end{bmatrix}^T, \ e_3 = \begin{bmatrix} 0 & 0 & -1 & 0 \end{bmatrix}^T.
	$$
	In this case, $p=1$, $\Delta_1(\tau) = \mathrm{Conv}\{0, v_2 - v_1\}$ is a $1$-simplex and $\check{\Delta}_1(\tau)= \mathrm{Conv}\{0, w_1 - w_2, w_1-w_3\}$ is a $2$-simplex. 
	
	Both examples are positive and strongly simple. 
	\end{example}
 
%Throughout this paper, we always assume that $(B,\pdecomp)$ is positive. 
%Actually, as a consequence of a more general simplicity assumption (\cite[Definition 1.60]{Gross-Siebert-logI} or see \S \ref{subsubsec:monodromy_polytope}) on $(B,\pdecomp)$, we will have $\kappa_{\omega\rho} \in \{0,1\}$. 
 
Throughout this paper, we always assume that $(B,\pdecomp)$ is positive and strongly simple. 
In particular, 
%we will have $\kappa_{\omega\rho} \in \{0,1\}$. Furthermore, 
both $\Delta_i(\tau)$ and $\check{\Delta}_i(\tau)$ are standard simplices of positive dimensions, and $\tanpoly_{\Delta_1(\tau)}\oplus \cdots \oplus \tanpoly_{\Delta_p(\tau)}$ (resp. $\tanpoly_{\check{\Delta}_1(\tau)} \oplus \cdots \oplus \tanpoly_{\check{\Delta}_p(\tau)}$) is an internal direct summand of $\tanpoly_{\tau}$ (resp. $\norpoly_{\tau}^*$).  

%\subsubsection{Piecewise affine functions}\label{subsubsec:piecewise_affine_functions}

\subsection{Cone construction by gluing open affine charts}\label{subsec:open_construction}

In this subsection, we recall the cone construction of the maximally degenerate Calabi--Yau $\centerfiber{0} = \centerfiber{0}(B,\pdecomp,s)$, %twisted by an open gluing data $s$,
following \cite{Gross-Siebert-logI} and \cite[\S 1.2]{gross2011real}. For this purpose, we take $\blat = \inte$ and $Q$ to be the positive real axis in Notation \ref{not:universal_monoid}. 
Throughout this paper, we will work in the category of analytic schemes.

We will construct $\centerfiber{0}$ as a gluing of affine analytic schemes $V(v)$ parametrized by the vertices of $\pdecomp$. For each vertex $v$, we consider the fan $\Sigma_v$ and take the analytic affine toric variety
$$V(v):=\spec(\bb{C}[\Sigma_v]),$$
where $\spec$ means analytification of the algebraic affine scheme given by $\mathrm{Spec}$. 
Here, the monoid structure for a general fan $\Sigma \subset M_{\real}$ is given by
$$p+q=\begin{cases}
	p+q &\text{ if }p,q \in M \text{ are in a common cone of } \Sigma,
	\\\infty & \text{ otherwise},
\end{cases}$$
and we set $z^{\infty} =0$ in taking $\mathrm{Spec}(\bb{C}[\Sigma])$ (by abuse of notation, we use $\Sigma$ to stand for both the fan and the monoid associated to a fan if there is no confusion);
in other words, the ring $\comp[\Sigma]$ is defined explicitly as 
$$\comp[\Sigma]:= \bigoplus_{p \in |\Sigma| \cap M} \comp \cdot z^p, \quad z^p\cdot z^q=\begin{cases}
z^{p+q} &\text{ if }p,q \in M \text{ are in a common cone of } \Sigma,
\\0 & \text{ otherwise},
\end{cases}$$
where $|\Sigma|$ denotes the support of the fan $\Sigma$.

To glue these affine analytic schemes together, we need affine subschemes $\{V(\tau)\}$
associated to $\tau\in \pdecomp$ with $v\in\tau$
and natural open embeddings $V(\tau)\hookrightarrow V(\omega)$ for $v \in \omega\subset\tau$.
First, for $\tau\in \pdecomp$ such that $v\in\tau$, we consider the \emph{localization of $\Sigma_v$ at $\tau$} defined by
$$\tau^{-1}\Sigma_v:=\{K_v\sigma + \tanpoly_{\tau,\bb{R}}\,|\,K_v\sigma\text{ is a cone in } \Sigma_v \text{ such that }\sigma \supset \tau\};$$
here recall that $K_v\sigma = \real_{\geq 0} S_{v}(\sigma \cap U_v)$ is the cone in $\Sigma_v$ (see the definition of a fan structure before Definition \ref{def:integral_tropical_manifold}).
This defines a new complete fan in $T_{v,\real}$ consisting of convex, but not necessarily strictly convex, cones. 
%A maximal cone of $\tau^{-1}\Sigma_v$ is of the form $K_{v}\sigma + \tanpoly_{\tau,\bb{R}}$, where 
If $\tau$ contains another vertex $v'$, we can identify the fans $\tau^{-1}\Sigma_v$ and $\tau^{-1}\Sigma_{v'}$ as follows: for each maximal $\sigma \supset \tau$, we identify the maximal cones $K_{v}\sigma + \tanpoly_{\tau,\bb{R}}$ and $K_{v'}\sigma + \tanpoly_{\tau,\bb{R}}$ by identifying the tangent spaces $T_v\cong T_{v'}$ using parallel transport through $\sigma\supset\tau$.
Patching these identifications for all $\sigma \supset \tau$ together, we get a piecewise linear transformation from $T_{v}$ to $T_{v'}$, identifying the fans $\tau^{-1}\Sigma_v$ and $\tau^{-1}\Sigma_{v'}$ and hence the corresponding monoids. 
This defines the affine analytic scheme
$$V(\tau):=\spec(\bb{C}[\tau^{-1}\Sigma_v]),$$
up to a unique isomorphism.  Notice that $\tau^{-1}\Sigma_v$ can be identified (non-canonically) with the fan $\Sigma_{\tau} \times \tanpoly_{\tau,\real}$ in $\norpoly_{\tau,\real}\times \tanpoly_{\tau,\real}$, so actually 
$$V(\tau) \cong \spec(\comp[\tanpoly_{\tau}]) \times \spec(\comp[\Sigma_{\tau}]),$$
where $\spec(\comp[\tanpoly_{\tau}]) \cong \tanpoly_{\tau}^* \otimes_\inte \comp^* \cong (\comp^{*})^l$ is a complex torus.

For any $v\in \omega\subset\tau$, there is a map of monoids $ \omega^{-1}\Sigma_v \to\tau^{-1}\Sigma_v$ given by
$$p\mapsto\begin{cases}
	p & \text{ if }p\in K_{v}\sigma +\Lambda_{\omega,\bb{R}}\text{ for some }\sigma\supset\tau,
	\\\infty & \text{ otherwise}
\end{cases}$$
(though there is no fan map from $\omega^{-1}\Sigma_v$ to $\tau^{-1}\Sigma_v$ in general),
and hence a ring map 
$$\iota_{\omega\tau}^*\colon \bb{C}[\omega^{-1}\Sigma_v]\to\bb{C}[\tau^{-1}\Sigma_v].$$ 
This gives an open inclusion of affine schemes
$$\iota_{\omega\tau}\colon V(\tau)\hookrightarrow V(\omega),$$
%In particular, $V(\tau)\hookrightarrow V(v)$. This can be summarized by the
and hence a functor $F\colon \pdecomp\to{\bf{Sch}}_{\mathrm{an}}$ defined by
$$
F(\tau):= V(\tau), \quad 
F(e):= \iota_{\omega\tau} \colon V(\tau)\to V(\omega)
$$
for $\omega \subset \tau$.

\begin{comment}
\begin{example}\label{ex:fan_P2}
Let $v_0=(-1,-1),v_1=(1,0),v_2=(0,1)\in M_{\bb{R}}$ and $\Sigma:=\Sigma_{\{0\}}$ be the fan generated by $v_i$'s. Then
$$\bb{C}[\partial P_{\{0\}}]=\bb{C}[X,Y,Z]/(XY,YZ,ZX).$$
Thus $V(\{0\})=\{(X,Y,Z)\in\bb{C}^3:XYZ=0\}$, which is the union of three coordinate planes. Let $\tau=\bb{R}_{\geq 0}\cdot v_0$. Then $\tau^{-1}\Sigma=\{K_+,K_-\}$, where
\begin{align*}
K_+=\{(x,y)\in M_{\bb{R}}:y\geq x\},
\\K_-=\{(x,y)\in M_{\bb{R}}:y\leq x\}.	    
\end{align*}
Then $\bb{C}[\partial P_{\tau}]=\bb{C}[X,Y,Z,Z^{-1}]/(XY)$ and so $V(\tau)=\{(X,Y,Z)\in\bb{C}^3:XY=0,Z\neq 0\}\subset V(\{0\})$, which is the union of $\{X=0,Z\neq 0\}$ and $\{Y=0,Z\neq 0\}$ along the divisor $\{X=Y=0\}$.
\end{example}
\end{comment}

We can further introduce twistings of the gluing of the affine analytic schemes $\{V(\tau)\}_{\tau\in\pdecomp}$. Toric automorphisms $\mu$ of $V(\tau)$ are in bijection with the set of $\bb{C}^{*}$-valued piecewise multiplicative maps on $T_v\cap|\tau^{-1}\Sigma_v|$ with respect to the fan $\tau^{-1}\Sigma_v$. Explicitly, for each maximal cone $\sigma\in \pdecomp^{[n]}$ with $\tau\subset\sigma$, there is a monoid homomorphism $\mu_{\sigma}\colon \tanpoly_{\sigma}\to\bb{C}^{*}$ such that if $\sigma'\in \pdecomp^{[n]}$ also contains $\tau$, then $\mu_{\sigma}|_{\tanpoly_{\sigma\cap\sigma'}}=\mu_{\sigma'}|_{\tanpoly_{\sigma\cap\sigma'}}$. Denote by $\mathrm{PM}(\tau)$ the multiplicative group of $\bb{C}^{*}$-valued piecewise multiplicative maps on $T_v\cap|\tau^{-1}\Sigma_v|$. The group $\mathrm{PM}(\tau)$ a priori depends on the choice of $v \in \tau$; however, for different choices, say $v$ and $v'$, the groups can be identified via the identification $\tau^{-1}\Sigma_v \cong \tau^{-1}\Sigma_{v'}$.  For $\omega \subset \tau$, there is a natural restriction map $|_{\tau} \colon  \mathrm{PM}(\omega) \rightarrow \mathrm{PM}(\tau)$ given by restricting to those maximal cells $\sigma \supset \omega$ with $\sigma \supset \tau$. %We now define the notion of an open gluing data.
 
\begin{definition}[\cite{gross2011real}, Def. 1.18]\label{def:open_gluing_data}
	A choice of \emph{open gluing data} (for the cone construction) for $(B,\pdecomp)$ is a set $s=(s_{\omega\tau})_{\omega \subset \tau}$ of elements $s_{\omega \tau}\in \mathrm{PM}(\tau)$ such that
	\begin{enumerate}
		\item $s_{\tau\tau}=1$ for all $\tau\in\pdecomp$, and
		\item if $\omega \subset \tau \subset \rho$, then
		$$s_{\omega\rho}=s_{\tau\rho}\cdot s_{\omega\tau}|_{\rho}.$$
	\end{enumerate}
	Two choices of open gluing data $s,s'$ are said to be \emph{cohomologous} if there exists a system $\{t_{\tau}\}_{\tau \in \pdecomp}$, with $t_{\tau}\in \mathrm{PM}(\tau)$ for each $\tau\in\pdecomp$,  such that $s_{\omega\tau}=t_{\tau}(t_{\omega}|_{\tau})^{-1}s_{\omega\tau}'$ whenever $\omega\subset \tau$. 
\end{definition}

The set of cohomology classes of choices of open gluing data is a group under multiplication, denoted as $H^1(\msc{P},\msc{Q}_{\msc{P}}\otimes\bb{C}^{\times})$. For $s\in \mathrm{PM}(\tau)$, we will denote also by $s$ the corresponding toric automorphism on $V(\tau)$ which is explicitly given by $s^*(z^m) = s_{\sigma}(m) z^m$ for $m \in \sigma \supset \tau$. 
If $s$ is a choice of open gluing data, then we can define an \emph{$s$-twisted functor} $F_s \colon \pdecomp \to{\bf{Sch}}_{\mathrm{an}}$ by setting $F_s(\tau):=F(\tau)=V(\tau)$ on objects and $F_s(\omega \subset \tau):=F(\omega \subset \tau)\circ s_{\omega\tau}^{-1} \colon V(\tau) \to V(\omega)$ on morphisms.
This defines the analytic scheme
$$\centerfiber{0}=\centerfiber{0}(B,\pdecomp,s):=\lim_{\longrightarrow}F_s.$$
Gross--Siebert \cite{Gross-Siebert-logI} showed that $\centerfiber{0}(B,\pdecomp,s)\cong\centerfiber{0}(B,\pdecomp,s')$ as schemes when $s,s'$ are cohomologous.

\begin{remark}\label{rem:closed_Construction}
Given $\tau \in \pdecomp^{[k]}$, one can define a closed stratum $\iota_{\tau} \colon \centerfiber{0}_{\tau} \rightarrow \centerfiber{0}$ of dimension $k$ by gluing together the $k$-dimensional toric strata $V_\tau(\omega) \subset V(\omega) = \spec(\comp[\omega^{-1}\Sigma_v])$ corresponding to the cones $K_v\tau + \tanpoly_{\omega,\bb{R}}$ in $\omega^{-1}\Sigma_v$, for all $\omega \subset \tau$. Abstractly, it is isomorphic to the toric variety associated to the polyhedron $\tau \subset \tanpoly_{\tau,\real}$. Also, for every pair $\omega \subset \tau$, there is a natural inclusion $\iota_{\omega \tau}\colon \centerfiber{0}_{\omega}\rightarrow \centerfiber{0}_{\tau} $. One can alternatively construct $\centerfiber{0}$ by gluing along the closed strata $\centerfiber{0}_\tau$'s according to the polyhedral decomposition; see \cite[\S 2.2]{Gross-Siebert-logI}.
\end{remark}

We recall the following definition from \cite{Gross-Siebert-logI}, which serves as an alternative set of combinatorial data for encoding $\mu\in \mathrm{PM}(\tau)$.

\begin{definition}[\cite{Gross-Siebert-logI}, Def. 3.25 and \cite{gross2011real}, Def. 1.20]\label{def:alternative_description_open_gluing_data}
	Let $\mu\in \mathrm{PM}(\tau)$ and $\rho\in\pdecomp^{[n-1]}$ with $\tau\subset\rho$. For a vertex $v\in\tau$, we define
	$$D(\mu,\rho,v):=\frac{\mu_{\sigma}(m)}{\mu_{\sigma'}(m')}\in\bb{C}^{\times},$$
	where $\sigma,\sigma'$ are the two unique maximal cells such that $\sigma\cap\sigma'=\rho$, $m \in \tanpoly_{\sigma}$ is an element projecting to the generator in $\norpoly_{\rho} \cong \tanpoly_{\sigma}/\tanpoly_{\rho}\cong\bb{Z}$ pointing to $\sigma'$, and $m'$ is the parallel transport of $m\in\tanpoly_{\sigma}$ to $\tanpoly_{\sigma'}$ through $v$. $D(\mu,\rho,v)$ is independent of the choice of $m$.
\end{definition}

Let $\rho\in\msc{P}^{[d-1]}$ and $\sigma_+,\sigma_-$ be the two unique maximal cells such that $\sigma_+\cap\sigma_-=\rho$. Let $\check{d}_{\rho}\in\norpoly_{\rho}^*$ be the unique primitive generator pointing to $\sigma_+$. For any two vertices $v,v'\in\tau$, we have the formula
\begin{equation}\label{eqn:gluing_data_ratio_under_monodromy}
	D(\mu,\rho,v)=\mu(m_{vv'}^{\rho})^{-1}\cdot D(\mu,\rho,v')
\end{equation}
relating monodromy data to the open gluing data, where $m_{vv'}^{\rho}\in\tanpoly_{\rho}$ is 
%the monodromy of a simple loop based at $v$, 
as discussed in \eqref{eqn:monodromy_transformation_rho_fixed}. 
The formula \eqref{eqn:gluing_data_ratio_under_monodromy} describes the interaction between monodromy and a fixed $\mu \in \mathrm{PM}(\tau)$. We shall further impose the following lifting condition from \cite[Prop. 4.25]{Gross-Siebert-logI} relating $s_{v\tau}, s_{v'\tau} \in \mathrm{PM}(\tau)$ and monodromy data:
\begin{condition}\label{cond:open_gluing_data_lifting_condition}
	We say a choice of open gluing data $s$ satisfies the \emph{lifting condition} if for any two vertices $v, v'\in \tau \subset \rho$ with $\rho \in \pdecomp^{[n-1]}$, we have 
	$D(s_{v\tau},\rho,v) = D(s_{v'\tau},\rho,v')$ whenever $m^{\rho}_{vv'} = 0$. 
\end{condition}

\subsection{Log structures}\label{subsec:log_structure_and_slab_function}
We need to equip the analytic scheme $\centerfiber{0}=\centerfiber{0}(B,\pdecomp,s)$ with log structures. The main reference is \cite[\S 3 - 5]{Gross-Siebert-logI}.

\begin{definition}\label{def:log_structure}
Let $X$ be an analytic space, a \emph{log structure} on $X$ is a sheaf of monoids $\cu{M}_X$ together with a homomorphism $\alpha_X\colon \cu{M}_X \rightarrow \cu{O}_X$ of sheaves of (multiplicative) monoids such that $\alpha_X \colon \alpha^{-1}(\cu{O}_X^*) \rightarrow \cu{O}_X^*$ is an isomorphism. 
The \emph{ghost sheaf} $\overline{\cu{M}}_X$ of a log structure is defined as the quotient sheaf $\cu{M}_X/\alpha^{-1}(\cu{O}_X^*)$, whose monoid structure is written additively.
\end{definition}

\begin{example}\label{example:divisor_log_structure}
Let $X$ be an analytic space and $D\subset X$ be a closed analytic subspace of pure codimension one. We denote by $j \colon X\setminus D \hookrightarrow X $ the inclusion. Then the sheaf of monoids 
$$
\cu{M}_{X}:=j_*(\cu{O}_{X\setminus D}^*) \cap \cu{O}_X,
$$ 
together with the natural inclusion $\alpha_X \colon \cu{M}_{X} \rightarrow \cu{O}_X$ defines a log structure on $X$.
\end{example}

We write $X^{\dagger}$ if we want to emphasize the log structure on $X$. A general way to define a log structure is to take an arbitary homomorphism of sheaves of monoids 
$$
\tilde{\alpha} \colon \cu{P} \rightarrow \cu{O}_X,
$$
and then define the associated log structure by
\begin{equation*}
\cu{M}_X := (\cu{P}\oplus \cu{O}_X^*)/\{(p,\tilde{\alpha}(p)^{-1}) \ | \ p \in \tilde{\alpha}^{-1}(\cu{O}_X^*) \}.
\end{equation*}
In particular, this allows us to define log structures on an analytic space $Y$ by pulling back those on another analytic space $X$ via a morphism $f \colon Y \rightarrow X$. More precisely, given a log structure on $X$, the \emph{pullback log structure} on $Y$ is defined to be the log structure associated to the composition $\tilde{\alpha}_Y \colon f^{-1}(\cu{M}_X) \rightarrow f^{-1}(\cu{O}_X) \rightarrow \cu{O}_Y$. 
For more details of the theory of log structures, readers are referred to, e.g., \cite[\S 3]{Gross-Siebert-logI}.

\begin{example}\label{example:toric_log_structure}
Taking a toric monoid $P$ (i.e. $P= C \cap M$ for a cone $C \subset M_\real$), we can define $\tilde{\alpha} \colon \underline{P} \rightarrow \cu{O}_{\Spec(\comp[P])}$ by sending $m \mapsto z^{m}$, where $\underline{P}$ is the constant sheaf with stalk $P$. From this we obtain a log structure on the analytic toric variety $\spec(\comp[P])$. Note that this is a special case of Example \ref{example:divisor_log_structure}, where we take $X = \spec(\comp[P])$ and $D$ to be the toric boundary divisor.
\end{example}

Before we describe the log structures on $\centerfiber{0}=\centerfiber{0}(B,\pdecomp,s)$, let us first specify a ghost sheaf $\overline{\cu{M}}$ over $\centerfiber{0}$. 
Recall that the polyhedral decomposition $\pdecomp$ is assumed to be regular, namely, there exists a strictly convex multi-valued piecewise linear function $\varphi\in H^0(B,\cu{MPL}_{\pdecomp})$. 
For any $\tau \in \pdecomp$, we take a strictly convex representative $\bar{\varphi}_{\tau}$ of $\varphi$ on $\norpoly_{\tau,\real}$, and define 
$$\Gamma(V(\tau),\overline{\cu{M}}) := \bar{P}_{\tau} = C_{\tau} \cap (\norpoly_{\tau} \oplus \bb{Z}),$$
where $C_\tau:=\{(m,h)\in \norpoly_{\tau,\real} \oplus\bb{R} \,|\, h\geq \bar{\varphi}_\tau(m)\}$. For any $\omega \subset \tau$, we take an integral affine function $\psi_{\omega\tau}$ on $U_{\omega}$ such that $\psi_{\omega\tau}+ S_{\omega}^*( \bar{\varphi}_{\omega})$ vanishes on $K_{\omega}\tau$, and agrees with $S_{\tau}^*(\bar{\varphi}_{\tau})$ on all of $\sigma \cap U_{\tau}$ for any $\sigma \supset \tau$. This induces a map $C_{\omega} \rightarrow C_{\omega \tau} :=\{(m,h)\in \norpoly_{\omega,\real} \oplus\bb{R} \,|\, h\geq \psi_{\omega\tau}(m)+\bar{\varphi}_\omega(m)\}$ by sending $(m,h)\mapsto (m,h+\psi_{\omega\tau}(m))$, whose composition with the quotient map $\norpoly_{\omega,\real}\oplus \real \rightarrow \norpoly_{\tau,\real}\oplus \real$ gives a map $C_{\omega} \rightarrow C_{\tau}$ of cones that corresponds to the monoid homomorphism $\bar{P}_{\omega}\rightarrow \bar{P}_{\tau}$. The $\bar{P}_{\tau}$'s glue together to give the ghost sheaf $\overline{\cu{M}}$ over $\centerfiber{0}$. There is a well-defined section $\bar{\varrho} \in \Gamma(\centerfiber{0},\overline{\cu{M}})$ given by gluing $(0,1) \in C_{\tau}$ for each $\tau$. 
%The pair $(\overline{\cu{M}},\bar{\varrho})$ and the identification $V(v) \cong \spec(\bb{C}[P_v]/q)$ for each $v \in \pdecomp^{[0]}$ define a \emph{ghost structure} on $\centerfiber{0}$ in the sense of \cite[Def. 3.16. and Ex. 3.17]{Gross-Siebert-logI}.

%Next, we define \emph{positivity} (which corresponds to positivity of the integral tropical manifold) of the log structure on $\centerfiber{0}$ with ghost type specified by $\varphi$.
One may then hope to find a log structure on $\centerfiber{0}$ which is log smooth and with ghost sheaf given by $\overline{\cu{M}}$. However, due to the presence of non-trivial monodromies of the affine structure, this can only be done away from a complex codimension $2$ subset $Z \subset \centerfiber{0}$ not containing any toric strata. Such log structures can be described by sections of a coherent sheaf $\cu{LS}^+_{\mathrm{pre}}$ supported on the scheme-theoretic singular locus $\centerfiber{0}_{\mathrm{sing}} \subset \centerfiber{0}$. We now describe the sheaf $\cu{LS}^+_{\mathrm{pre}}$ and some of its sections called \emph{slab functions}; readers are referred to \cite[\S 3 and 4]{Gross-Siebert-logI} for more details. 

For every $\rho \in \pdecomp^{[n-1]}$, we consider $\iota_{\rho} \colon \centerfiber{0}_{\rho} \rightarrow \centerfiber{0}$, where $\centerfiber{0}_{\rho}$ is the toric variety associated to the polytope $\rho \subset \tanpoly_{\rho,\real}$. From the fact that the normal fan $\mathscr{N}_{\rho} \subset \tanpoly_{\rho,\real}^*$ of $\rho$ is a refinement of the normal fan $\mathscr{N}_{\Delta(\rho)}\subset \tanpoly_{\rho,\real}^*$ of the $r_{\rho}$-dimensional simplex $\Delta(\rho)$ (as in \S \ref{subsec:monodromy_data}), we have a toric morphism 
\begin{equation}\label{eqn:toric_map_to_monodromy_projective_spaces}
	\varkappa_{\rho}\colon \centerfiber{0}_{\rho} \rightarrow \bb{P}^{r_\rho}.
\end{equation}
Now, $\Delta(\rho)$ corresponds to $\cu{O}(1)$ on $\bb{P}^{r_\rho}$. We let $\cu{N}_{\rho}:= \varkappa_{\rho}^*(\cu{O}(1))$ on $\centerfiber{0}_{\rho}$, and define
\begin{equation}\label{eqn:line_bundle_parametrizing_log_structure}
	\cu{LS}^+_{\mathrm{pre}} := \bigoplus_{\rho \in \pdecomp^{[n-1]}} \iota_{\rho,*} (\cu{N}_{\rho}). 
\end{equation}

Sections of $\cu{LS}^+_{\mathrm{pre}}$ can be described explicitly. For each $v \in \pdecomp^{[0]}$, we consider the open subscheme $V(v)$ of $\centerfiber{0}$ and the local trivialization
$$
\cu{LS}^+_{\mathrm{pre}}|_{V(v)} = \bigoplus_{\rho: v \in \rho} \cu{O}_{V_{\rho}(v)},
$$
whose sections over $V(v)$ are given by $(f_{v\rho})_{v \in \rho}$. Given $v ,v' \in \tau$ where $\tau$ corresponding to $V(\tau)$, these local sections obey the change of coordinates given by
\begin{equation}\label{eqn:change_of_coordinates_of_line_bundle_for_log_structures}
	D(s_{v'\tau},\rho, v')^{-1} s_{v'\tau}^{-1} (f_{v'\rho}) = z^{-m^{\rho}_{vv'}} D(s_{v\tau},\rho, v)^{-1} s_{v\tau}^{-1} (f_{v\rho}),
\end{equation}
where $\rho \supset \tau$ and $s_{v\tau}, s_{v'\tau}$ are part of the open gluing data $s$. The section $f := (f_{v\rho})_{v \in \rho}$ is said to be \emph{normalized} if $f_{v\rho}$ takes the value $1$ at the $0$-dimensional toric stratum corresponding to a vertex $v$, for all $\rho$. We will restrict ourselves to normalized sections $f$ of $\cu{LS}^+_{\mathrm{pre}}$. The complex codimension $2$ subset $Z \subset \centerfiber{0}$ is taken to be the zero locus of $f$ on $\centerfiber{0}_{\mathrm{sing}}$. 

Only a subset of normalized sections of $\cu{LS}^+_{\mathrm{pre}}$ corresponds to log structures. For every vertex $v \in \pdecomp^{[0]}$ and $\tau \in \pdecomp^{[n-2]}$ containing $v$, we choose a cyclic ordering $\rho_1,\dots,\rho_l$ of codimension one cells containing $\tau$ according to an orientation of $\norpoly_{\tau,\real}$. Let $\check{d}_{\rho_i} \in \norpoly_v^*$ be the positively oriented normal to $\rho_i$.  Then the condition for $f = (f_{v\rho})_{v\in \rho}\in \cu{LS}^+_{\mathrm{pre}}|_{V(v)}$ to define a log structure is given by
\begin{equation}\label{eqn:condition_for_sections_to_define_log_structure}
	\prod_{i=1}^l \check{d}_{\rho_i}  \otimes f_{v\rho_i}|_{V_{\tau}(v)} = 0 \otimes 1,\quad  \text{in } \norpoly_v^* \otimes \Gamma(V_{\tau}(v)\setminus Z, \cu{O}_{V_{\tau}(v)}^*),
\end{equation}
where the group structure on $\norpoly_v^*$ is additive and that on $\Gamma(V_{\tau}(v)\setminus Z, \cu{O}_{V_{\tau}(v)}^*)$ is multiplicative. If $f = (f_{v\rho})_{v\in \rho}$ is a normalized section satisfying this condition, we call the $f_{v\rho}$'s \emph{slab functions}.
%We require $f \in \Gamma(\centerfiber{0},\cu{LS}^+_{\mathrm{pre}})$ to satisfy equation \eqref{eqn:condition_for_sections_to_define_log_structure} when restricted to $V(v)$ for every vertex $v$. 

\begin{theorem}[\cite{Gross-Siebert-logI}, Thm. 5.2]\label{thm:log_structures}
	Suppose that $B$ is compact and the pair $(B,\pdecomp)$ is simple and positive. Let $s$ be a choice of open gluing data satisfying the lifting condition (Condition \ref{cond:open_gluing_data_lifting_condition}). Then there exists a unique normalized section $f \in \Gamma(\centerfiber{0},\cu{LS}^+_{\mathrm{pre}})$ which %satisfies \eqref{eqn:condition_for_sections_to_define_log_structure} and 
	defines a log structure on $\centerfiber{0}$ (i.e. satisfying the condition \eqref{eqn:condition_for_sections_to_define_log_structure}).
\end{theorem}

From now on, we always assume that $B$ is compact. To describe the log structure in Theorem \ref{thm:log_structures}, we first construct some local smoothing models: For each vertex $v \in \pdecomp^{[0]}$, we represent the strictly convex piecewise linear function $\varphi$ in a small neighborhood $U$ of $v$ by a strictly convex piecewise linear $\varphi_v\colon \norpoly_{v,\real}\to\bb{R}$ (so that $\varphi = S_v^*(\varphi_v)$) and set
\begin{align*}
	C_v := \{(m,h)\in \norpoly_{v,\real} \oplus\bb{R} \,|\, h\geq\varphi_v(m)\},\quad P_v := C_v\cap(\norpoly_{v}\oplus\bb{Z}).
\end{align*}
The element $\varrho = (0,1)\in\norpoly_v\oplus\bb{Z}$ gives rise to a regular function $q:=z^{\varrho}$ on $\spec(\bb{C}[P_v])$. We have a natural identification
$$V(v):=\spec(\bb{C}[\Sigma_v]) \cong \spec(\bb{C}[P_v]/q),$$
through which we view $V(v)$ as the toric boundary divisor in $\spec(\bb{C}[P_v])$ that corresponds to the holomorphic function $q$, and
$\pi_v \colon \spec(\bb{C}[P_v]) \rightarrow \spec(\comp[q])$
as a local model for smoothing $V(v)$. 
%To relate these with the local smoothing models of $\centerfiber{0}$, we need to pick a \emph{ghost structure} and so-called \emph{slab functions} to specify a log structure on $\centerfiber{0}$. 

%Recall that the polyhedral decomposition $\pdecomp$ is assumed to be regular, namely, there exists a strictly convex multi-valued piecewise linear function $\varphi\in H^0(B,\cu{MPL}_{\pdecomp})$. This datum enters the picture when one tries to put a \emph{log structure} on the degenerate Calabi-Yau $\centerfiber{0}=\centerfiber{0}(B,\pdecomp,s)$ (see \cite[\S 3 - 5]{Gross-Siebert-logI}). 

Using these local models, we can now describe the log structure around a point $x \in \centerfiber{0} \setminus Z$. 
On a neighborhood $V \subset V(v)\setminus Z$ of $x$, the local smoothing model is given by composing the two inclusions $\rest{} \colon V \hookrightarrow V(v)$ and $V(v) \hookrightarrow \spec(\comp[P_{v}])$. 
%, which is determined by the ghost structure $(\overline{\cu{M}},\overline{\varrho})$. 
The natural monoid homomorphism $P_v \rightarrow \comp[P_v]$ defined by sending $m \mapsto z^m$ determines a log structure on $\spec(\comp[P_v])$ which restricts to one on the toric boundary divisor $V(v) = \spec(\comp[\Sigma_v])$. We further twist the inclusion $\rest{} \colon V \hookrightarrow V(v)$ as
\begin{equation}\label{eqn:zero_order_embedding_encode_log_structure}
	z^{m} \mapsto h_m\cdot z^{m}\text{ for $m \in \Sigma_v$;}
\end{equation}
here, for each $m \in \Sigma_v$, $h_m$ is chosen as an invertible holomorphic function on $V \cap \mathrm{Zero}(z^m;v)$, where we denote $\mathrm{Zero}(z^m;v):=\overline{ \{x \in V(v) \ | \ z^{m} \in \cu{O}_{x}^* \}}$, and such that they satisfy the relations 
\begin{equation}\label{eqn:local_log_structure_from_monoid_homomorphism}
	h_m \cdot h_{m'} = h_{m+m'},\quad \text{on } V \cap \mathrm{Zero}(z^{m+m'};v).
\end{equation}
Then pulling back the log structure on $V(v)$ via $\rest{} \colon V \hookrightarrow V(v)$ produces a log structure on $V$ which is log smooth. 

These local choices of $h_m$'s are also required to %glue into a global log structure on $\centerfiber{0} \setminus Z$ which is log smooth. Such choices of $h_m$'s are 
be determined by the slab functions $f_{v\rho}$'s, up to equivalences. Here, we shall just give the formula relating them; see \cite[Thm. 3.22]{Gross-Siebert-logI} for details. For any $\rho \in \pdecomp^{[n-1]}$ containing $v$ and two maximal cells $\sigma_{\pm}$ such that $\sigma_{+} \cap \sigma_- = \rho$, we take $m_+ \in \norpoly_{v} \cap K_{v} \sigma_+$ generating $\norpoly_{\rho}$ with some $m_0 \in \norpoly_v\cap K_v \rho$ such that $m_0 - m_+ \in \norpoly_v \cap K_v \sigma_-$. Then the required  relation is given by
\begin{equation}\label{eqn:relation_between_log_smooth_structure_and_slab_functions}
	f_{v\rho} = \frac{h_{m_0}^2}{h_{m_0-m_+} \cdot h_{m_0 + m_+}}\Big|_{V_\rho(v)\cap V} \in \cu{O}^*_{V_\rho(v)}( V_\rho(v) \cap V),
\end{equation}  
which is independent of the choices of $m_0$ and $m_+$.

By abuse of notation, we also let $\rest{} \colon V \rightarrow \localmod{k}$ be the $k$-th order thickening of $V$ over $\comp[q]/q^{k+1}$ in the model $\spec(\comp[P_v])$ under the above embedding.
 %and $\rest{}\colon V \rightarrow \localmod{}$ be the corresponding infinitesimal thickening over $\comp[[q]] = \varprojlim_{k}\comp[q]/q^{k+1}$ defined by $\localmod{} = \varinjlim_k \localmod{k}$. 
Then there is a natural divisorial log structure on $\localmod{k}^{\dagger}$ over $\logsk{k} $ coming from restriction of the log structure on $\spec(\comp[P_{v}])^{\dagger}$ over $\logs$ (i.e. Example \ref{example:divisor_log_structure}, which is the same as the one given by Example \ref{example:toric_log_structure} in this case). Restricting to $V$ reproduces the log structure we constructed above, which is the log structure of $\centerfiber{0}^{\dagger}$ over the log point $\logsk{0}$ locally around $x$. We have a Cartesian diagram of log spaces
\begin{equation}\label{eqn:cartesian_diagram_of_log_spaces}
	\xymatrix@1{
		V^{\dagger}\ \ar@{^{(}->}[r] \ar[d] & \localmod{k}^{\dagger} \ar[d]\\
		\logsk{0}\ \ar@{^{(}->}[r] & \logsk{k}
	}
\end{equation}

%We want to construct a log structure on $\centerfiber{0}$ so that the log space $\centerfiber{0}^{\dagger}$ is identified locally with $V^{\dagger}$ over the log point $\logsk{0} = \comp^{\dagger}$. 

Next we describe the log structure around a singular point $x \in Z \cap \left( \centerfiber{0}_{\tau} \setminus \bigcup_{\omega \subset\tau} \centerfiber{0}_{\omega}\right)$ for some $\tau$. Viewing $f = \sum_{\rho \in \pdecomp^{[n-1]}} f_{\rho}$ where $f_{\rho}$ is a section of $\cu{N}_{\rho}$, we let $Z_{\rho} = Z(f_{\rho}) \subset \centerfiber{0}_{\rho} \subset \centerfiber{0}$ and write $Z = \bigcup_{\rho} Z_{\rho}$. For every $\tau \in \pdecomp$, we have the data $\Omega_i$'s, $R_i$'s, $\Delta_i(\tau)$ and $\check{\Delta}_i(\tau)$ described in Definition \ref{def:simplicity} because $(B,\pdecomp)$ is simple. Since the normal fan $\mathscr{N}_\tau \subset \tanpoly_{\tau,\real}^*$ of $\tau$ is a refinement of $\mathscr{N}_{\Delta_i(\tau)} \subset \tanpoly_{\tau,\real}^*$, we have a natural toric morphism 
\begin{equation}\label{eqn:i_th_map_to_toric_variety_of_monodromy}
	\varkappa_{\tau,i} \colon \centerfiber{0}_{\tau} \rightarrow \bb{P}^{r_{\tau,i}},
\end{equation}
and the identification $\iota_{\tau\rho}^*(\cu{N}_{\rho}) \cong \varkappa_{\tau,i}^*(\cu{O}(1))$.  By the proof of \cite[Thm. 5.2]{Gross-Siebert-logI}, $\iota_{\tau\rho}^*(f_{\rho})$ is completely determined by the gluing data $s$ and the associated monodromy polytope $\Delta_i(\tau)$ where $\rho \in R_i$. In particular, we have $\iota_{\tau\rho}^*(f_{\rho}) = \iota_{\tau \rho'}^*(f_{\rho'})$ and $Z_\rho \cap \centerfiber{0}_{\tau} = Z_{\rho'}\cap \centerfiber{0}_{\tau} =: Z_{i}^{\tau}$ for $\rho,\rho' \in R_i$. Locally, if we write $V(\tau) = \spec(\comp[\tau^{-1}\Sigma_v])$ by choosing some $v \in \tau$, then, for each $ 1 \leq i \leq p$, there exists an analytic function $f_{v,i}$ on $V(\tau)$ such that $f_{v,i}|_{V_{\rho}(\tau)} =s_{v\tau}^{-1}( f_{v\rho})$ for $\rho \in R_i$. 

According to \cite[\S 2.1]{Gross-Siebert-logII}, for each $ 1 \leq i \leq p$, we have $\check{\Delta}_i(\tau) \subset \norpoly_{\tau,\real}^*$, which gives
\begin{equation}\label{eqn:piecewise_linear_function_associated_to_dual_monodromy_polytope}
	\psi_i(m) = - \inf \{ \langle m,n \rangle \ |  \ n \in \check{\Delta}_i(\tau)\}. 
\end{equation}
By convention, we write $\psi_0:= \bar{\varphi}_{\tau}$. By rearranging the indices $i$'s, we can assume that $x \in  Z^\tau_{1} \cap \cdots \cap Z^\tau_{r}$ and $x \notin Z^{\tau}_{r+1} \cup \cdots \cup Z^{\tau}_{p} $. We introduce the convention that $\psi_{x,i} = \psi_{i}$ for $0\leq i \leq r$ and $\psi_{x,i} \equiv 0$ for $r<i\leq \dim_{\real}(\tau)$. 
Then the local smoothing model near $x$ is constructed as $\spec(\comp[P_{\tau,x}])$, where 
\begin{equation}\label{eqn:local_monoid_near_singularity}
	P_{\tau,x}:= \{ (m,a_0,\dots,a_{l})\in \norpoly_{\tau} \times \inte^{l+1} \ | \ a_i \geq \psi_{x,i}(m)  \},
\end{equation}
$l = \dim_{\real}(\tau)$, and the distinguished element $\varrho = (0,1,0,\dots,0)$ defines a family $$\spec(\comp[P_{\tau,x}]) \rightarrow \spec(\comp[q])$$
by sending $q \mapsto z^{\varrho}$. The central fiber is given by $\spec(\comp[Q_{\tau,x}])$, where 
\begin{equation}\label{eqn:central_fiber_local_model}
Q_{\tau,x} = \{ (m,a_0,\dots,a_l)  \ | \ a_0 = \psi_{x,0}(m) \} \cong P_{\tau,x}/(\varrho+P_{\tau,x})
\end{equation}
is equipped with the monoid structure 
$$m + m' = 
\begin{cases} m+m' & \text{if } m+m' \in Q_{\tau,x},\\
	\infty & \text{otherwise.}
\end{cases}
$$
We have the ring isomorphism $\comp[Q_{\tau,x}] \cong\comp[\Sigma_{\tau} \oplus \bb{N}^{l}]$ induced by the monoid isomorphism defined by sending $(m,a_0,a_1,\dots,a_l) \mapsto (m,a_1 - \psi_1(m),\dots,a_l - \psi_l(m))$. 

We also fix some isomorphism $\comp[\tau^{-1}\Sigma_v] \cong \comp[\Sigma_{\tau}\oplus \bb{Z}^{l}]$ coming from the identification of $\tau^{-1}\Sigma_v $ with the fan $\Sigma_{\tau} \oplus \real^{l} = \{ \omega \oplus \real^{l} \ | \ \omega \text{ is a cone of } \tau \} $ in $\norpoly_{\tau,\real} \oplus \real^{l}$. Taking a sufficiently small neighborhood $V$ of $x$ such that $Z_\rho \cap V = \emptyset$ if $x \notin Z_{\rho}$, %such that $f_{v,i} \in \cu{O}^*(V)$ for $r<i\leq p$ 
we define a map $V \rightarrow \spec(\comp[Q_{\tau,x}])$ by composing $V \hookrightarrow \spec(\comp[\tau^{-1}\Sigma_v]) \cong \spec(\comp[\Sigma_{\tau}\oplus \bb{Z}^{l}])$ with the map $\spec(\comp[\Sigma_{\tau}\oplus \bb{Z}^{l}]) \rightarrow \spec(\comp[\Sigma_{\tau}\oplus \bb{N}^{l}])$ described on generators by 
\begin{equation}\label{eqn:singular_locus_local_model}
	\begin{cases}
		z^m \mapsto  h_m \cdot z^m & \text{if } m\in \Sigma_{\tau} ;\\
		u_i \mapsto  f_{v,i} & \text{if } 1\leq i \leq r;\\
		u_i \mapsto  z_i-z_i(x) & \text{if } r<i\leq l;
	\end{cases}
\end{equation}
here $u_i$ is the $i$-th coordinate function of $\comp[\bb{N}^{l}]$, $z_i$ is the $i$-th coordinate function of $\comp[\bb{Z}^{l}]$ chosen so that $\left(\ddd{f_{v,i}}{z_j}\right)_{1\leq i \leq r, 1\leq j \leq r}$ is non-degenerate on $V$; also, each $h_m$ is an invertible holomorphic functions on $V \cap \mathrm{Zero}(z^m;v)$, and they satisfy the equations \eqref{eqn:local_log_structure_from_monoid_homomorphism} and \eqref{eqn:relation_between_log_smooth_structure_and_slab_functions} where we replace $f_{v\rho}$ by
$$\tilde{f}_{v\rho} = 
\begin{cases}
	s_{v\tau}^{-1}( f_{v\rho}) & \text{if } x \notin Z_{\rho},\\
	1 & \text{if } x \in Z_{\rho}.
\end{cases}
$$

Letting $\rest{} \colon V \rightarrow \localmod{k}$ be the $k$-th order thickening of $V$ over $\comp[q]/q^{k+1}$ in the model $\spec(\comp[P_{\tau,x}])$ under the above embedding,
% and $\rest{}: V \rightarrow \localmod{}$ be the corresponding infinitesimal thickening over $\comp[[q]]$. 
we have a natural divisorial log structure on $\localmod{k}^{\dagger}$ over $\logsk{k}$ induced from the inclusion $\spec(\comp[Q_{\tau,x}]) \hookrightarrow \spec(\comp[P_{\tau,x}])$ (i.e. Example \ref{example:divisor_log_structure}). Restricting it to $V$ gives the log structure of $\centerfiber{0}^{\dagger}$ over the log point $\logsk{0}$ locally around $x$.

\section{A generalized moment map and the tropical singular locus on $B$}\label{sec:modified_moment_map}

In this section, we recall the construction of a generalized moment map $\moment \colon \centerfiber{0} \rightarrow B$ from \cite[Prop. 2.1]{ruddat2019period}. Then we construct some convenient charts on the base tropical manifold $B$ and study its singular locus.

\subsection{A generalized moment map}\label{subsec:moment_map}

From this point onward, we will assume the vanishing of an obstruction class associated to the open gluing data $s$, namely, $o(s) = 1$, where the obstruction class $o(s)$ is written multiplicatively (see \cite[Thm. 2.34]{Gross-Siebert-logI}). Under this assumption, one can construct an ample line bundle $\cu{L}$ on $\centerfiber{0}$ as follows: For each polytope $\tau \subset \tanpoly_{\tau,\real}$, by identifying $\centerfiber{0}_{\tau}$ (a closed stratum of $\centerfiber{0}$ described in Remark \ref{rem:closed_Construction}) with the projective toric variety associated to $\tau$, we obtain an ample line bundle $\cu{L}_\tau$ on $\centerfiber{0}_{\tau}$. When the assumption holds, then there exists an isomorphism $\mathbf{h}_{\omega\tau}\colon \iota_{\omega\tau}^*(\cu{L}_{\tau}) \cong \cu{L}_{\omega}$, for every pair $\omega \subset \tau$, such that the isomorphisms $\mathbf{h}_{\omega\tau}$'s satisfy the \emph{cocycle condition}, i.e. $\mathbf{h}_{\omega\tau}\circ\iota_{\omega\tau}^*(\mathbf{h}_{\tau\sigma}) = \mathbf{h}_{\omega\sigma}$ for every triple $\omega\subset\tau\subset\sigma$.\footnote{In fact, the vanishing of the obstruction class corresponds exactly to the validity of the cocycle condition.}
In particular, the degenerate Calabi--Yau $\centerfiber{0} = \centerfiber{0}(B,\pdecomp,s)$ is projective.  

Sections of $\cu{L}$ correspond to the lattice points $B_\inte \subset B$. More precisely, given $m\in B_\inte$, there is a unique $\tau \in \pdecomp$ such that $m \in \reint(\tau)$, and this determines a section $\vartheta_{m,\tau}$ of $\cu{L}_{\tau}$ by toric geometry. This section extends uniquely as $\vartheta_{m}$ to $\sigma \supset \tau$ such that $\mathbf{h}_{\tau\sigma}(\vartheta_m) = \vartheta_{m,\tau}$. Further extending $\vartheta_m$ by $0$ to other cells gives a section of $\cu{L}$ corresponding to $m$, called a \emph{($0^{\text{th}}$-order) theta function}. 
Now for a vertex $v \in \pdecomp^{[0]}$, we can trivialize $\cu{L}$ over $V(v)$ using $\vartheta_{v}$ as the holomorphic frame. Then, for $m$ lying in a cell $\sigma$ that contains $v$, $\vartheta_{m}$ is of the form $g \vartheta_v$, where $g$ is a constant multiple of $z^{m}$. 

Under the above projectivity assumption, one can define a \emph{generalized moment map} 
\begin{equation}
\moment \colon \centerfiber{0} \rightarrow B
\end{equation}
following \cite[Prop. 2.1]{ruddat2019period}: First of all, the theta functions $\{\vartheta_m\}_{m \in B_\inte}$ defines an embedding of $\Phi \colon \centerfiber{0}\hookrightarrow \bb{P}^{N}$. 
Restricting to each closed toric stratum $\centerfiber{0}_{\tau} \subset \centerfiber{0}$, the only non-zero theta functions are those corresponding to $m \in B_{\inte} \cap \tau$. Also, there is an embedding $\mathfrak{j}_{\tau}\colon\mathtt{T}_{\tau} := \tanpoly_{\tau,\real}^*/\tanpoly_{\tau,\inte}^* \hookrightarrow \mathtt{U}(1)^{N}$ of real tori such that the composition $\Phi_{\tau}\colon\centerfiber{0}_{\tau} \rightarrow \bb{P}^N$ of $\Phi$ with the inclusion $\centerfiber{0}_{\tau} \hookrightarrow \centerfiber{0}$ is equivariant. The map $\moment$ is then defined by setting
\begin{equation}\label{eqn:moment_map_equation_on_maximal_cell}
\moment|_{\centerfiber{0}_{\tau}} (z):=  \frac{1}{\sum_{m \in B_{\inte} \cap \tau} |\vartheta_m(z) |^2} \sum_{m \in B_{\inte} \cap \tau} |\vartheta_m(z)|^2 \cdot m,
\end{equation}
which can be understood as a composition of maps 
$$
\xymatrix{
\centerfiber{0}_{\tau} \ar[r]^{\Phi_{\tau}}& \bb{P}^N \ar[r]^{\moment_{\bb{P}}} & (\real^{N})^*  \ar[r]^{d \mathfrak{j}^*_{\tau}} & \tanpoly_{\tau,\real},  
}
$$
where $\moment_{\bb{P}}$ is the standard moment map for $\bb{P}^{N}$ and $d\mathfrak{j}_{\tau}\colon \tanpoly_{\tau,\real}^* \rightarrow \real^{N}$ is the Lie algebra homomorphism induced by $\mathfrak{j}_{\tau}\colon\mathtt{T}_{\tau} \rightarrow \mathtt{U}(1)^N$. 

Fixing a vertex $v \in \pdecomp^{[0]}$, we can naturally embed $\tanpoly_{\tau,\real} \hookrightarrow T_{v,\real}$ for all $\tau$ containing $v$. Furthermore, we can patch the $d\mathfrak{j}^*_{\tau}$'s into a linear map $d\mathfrak{j}^* \colon (\real^{N})^* \rightarrow T_{v,\real}$ so that $\moment_{\tau} = d\mathfrak{j}^* \circ \moment_{\bb{P}} \circ \Phi_{\tau}$ for each $\tau$ which contains $v$. In particular, on the local chart $V(\tau)=\spec(\comp[\tau^{-1}\Sigma_v])$ associated with $v \in \tau$, we have the local description $\moment|_{V(\tau)} = d \mathfrak{j}^*\circ \moment_{\bb{P}}\circ \Phi|_{V(\tau)}$ of the generalized moment map $\moment$. 

%To specify the singular locus of the affine structure of $B$,
We consider the \emph{amoeba} $\amoeba:= \moment(Z)$.
%As mentioned after Definition \ref{def:integral_tropical_manifold}, we should specify the singular locus of the affine structure of $B$. Through the moment map $\moment \colon \centerfiber{0} \rightarrow B$, we let it to be the image $\amoeba:= \moment(Z)$. 
As $\centerfiber{0}_{\tau} \cap Z = \bigcup_{i=1}^p Z^{\tau}_i$, where $Z^{\tau}_i$ is the zero set of a section of $\varkappa_{\tau,i}^*(\cu{O}(1))$ (see the discussion right after equation \eqref{eqn:i_th_map_to_toric_variety_of_monodromy}), we can see that $\amoeba \cap \tau = \bigcup_{i=1}^{p} \moment_{\tau}(Z^{\tau}_i)$ is a union of amoebas $\amoeba^{\tau}_i:= \moment_{\tau}(Z^{\tau}_i)$. It was shown in \cite{ruddat2019period} that the affine structure defined right after Definition \ref{def:integral_tropical_manifold} extends to $B \setminus \amoeba$.

\subsection{Construction of charts on $B$}\label{subsubsec:charts_on_B}
For any $\tau \in \pdecomp$, we have 
$$\moment(V(\tau)) =  \bigcup_{\tau \subset \omega } \reint(\omega) =: W(\tau).$$
For later purposes, we would like to relate sufficiently small open convex subsets $W \subset W(\tau)$ with \emph{Stein} (or \emph{strongly $1$-completed}, as defined in \cite{demailly1997complex}) open subsets $U \subset V(\tau)$. To do so, we need to introduce a specific collection of (non-affine) charts on $B$. 

Recall that there are natural maps $\tanpoly_{\tau} \hookrightarrow \tau^{-1} \Sigma_v $ and $\tau^{-1} \Sigma_v \twoheadrightarrow \Sigma_{\tau}$. By choosing a piecewise linear splitting $\mathsf{split}_\tau\colon \Sigma_{\tau} \rightarrow \tau^{-1} \Sigma_v$, we have an identification of monoids $\tau^{-1} \Sigma_v \cong \Sigma_{\tau} \times \tanpoly_{\tau}$, which induces the biholomorphism 
$$V(\tau) = \spec(\comp[\tau^{-1} \Sigma_v]) \cong \spec(\comp[\tanpoly_{\tau}]) \times \spec(\comp[\Sigma_{\tau}]),$$
where $\tanpoly_{\tau,\comp^*}^*:=\spec(\comp[\tanpoly_{\tau}]) \cong \tanpoly_{\tau}^* \otimes_\inte \comp^* \cong (\comp^{*})^l$ is a complex torus.
Fixing a set of generators $\{m_i \}_{i \in \mathtt{B}_{\tau}}$ of the monoid $\Sigma_{\tau}$, which is not necessarily a minimal set, we can define an embedding $\spec(\comp[\Sigma_{\tau}])\hookrightarrow \comp^{|\mathtt{B}_{\tau}|}$ as an analytic subset using the functions $z^{m_i}$'s. 
We consider the real torus $\mathtt{T}_{\tau,\perp} := \norpoly_{\tau,\real}^*/\norpoly_{\tau}^* \cong \mathtt{U}(1)^{n - l}$ and its action on $\spec(\comp[\Sigma_{\tau}])$ defined by $t\cdot z^{m} = e^{2\pi i (t,m)} z^m$, together with an embedding $\mathtt{T}_{\tau,\perp} \hookrightarrow \mathtt{U}(1)^{|\mathtt{B}_{\tau}|}$ of real tori via $t\mapsto (e^{2\pi i (t,m_i)})_{i \in \mathtt{B}_\tau}$, so that $\spec(\comp[\Sigma_{\tau}])\hookrightarrow \comp^{|\mathtt{B}_{\tau}|}$ is $\mathtt{T}_{\tau,\perp}$-equivariant.

We consider the moment map $\hat{\moment}_{\tau} \colon \spec(\comp[\Sigma_{\tau}]) \rightarrow \norpoly_{\tau,\real}$ defined by 
\begin{equation}\label{eqn:local_moment_map}
\hat{\moment}_{\tau} := \sum_{i \in \mathtt{B}_{\tau}} \half |z^{m_i}|^2  \cdot m_i,
\end{equation} 
which is obtained by composing the standard moment map $\comp^{|\mathtt{B}_{\tau}|} \rightarrow \real^{|\mathtt{B}_{\tau}|}_{\geq 0}$, $(z_i)_{i\in \mathtt{B}_{\tau}} \mapsto (\half |z_i|^2)_{i\in \mathtt{B}_{\tau}}$ with the projection $\real^{|\mathtt{B}_{\tau}|} \rightarrow \norpoly_{\tau,\real}$, $e_i \mapsto m_i$. By \cite[\S 4.2]{Fulton_toric_book}, $\hat{\mu}_{\tau}$ induces a homeomorphism between the quotient $\spec(\comp[\Sigma_{\tau}])/\mathtt{T}_{\tau,\perp}$ and $\norpoly_{\tau,\real}$.
%To see that, we choose a cone $\sigma \in \Sigma_{\tau}$ and consider the big torus orbit in the corresponding $\spec(\comp[\sigma \cap \norpoly_{\tau}])$, which is biholomorphic to $(\comp^*)^{\dim_{\real}\sigma}$. Restricting to $(\comp^*)^{\dim_{\real}\sigma}$, $\hat{\mu}_{\tau}$ agrees with the moment map for $\mathtt{T}^{n-l}$ action with respect to the restriction of the standard symplectic form onto $(\comp^*)^{\dim_{\real}\sigma}$. After quotienting by the subtorus $\sigma^{\perp}/\sigma^{\perp} \cap \norpoly_{\tau}$ of $\mathtt{T}^{n-l} \cong \norpoly_{\tau,\real}/\norpoly_{\tau}$, we see that the action becomes free and the moment maps is a submersion onto its image $\reint(\sigma)$. 
Taking product with the log map $\log\colon \tanpoly_{\tau,\comp^*}^*  \rightarrow \tanpoly_{\tau,\real}^*$ (which is induced from the standard log map $\log\colon \comp^* \rightarrow \real$ defined by $\log(e^{2\pi (x+i\theta)}) = x$), we obtain a map $\moment_{\tau} := (\log, \hat{\moment}_{\tau}) \colon V(\tau) \rightarrow \tanpoly_{\tau,\real}^* \times \norpoly_{\tau,\real}$,\footnote{It depends on the choices of the splitting $\mathsf{split}_{\tau}\colon \Sigma_{\tau} \rightarrow \tau^{-1} \Sigma_v$ and the generators $\{m_i \}_i$, but we omit these dependencies from our notations.} and the following diagram 
\begin{equation}\label{eqn:local_chart_for_stein_subsets}
\xymatrix@1{ & V(\tau) \ar[d]^{\moment} \ar[dl]_{\moment_{\tau}}\\
	\tanpoly_{\tau,\real}^* \times \norpoly_{\tau,\real} \ar[r]^{\Upsilon_{\tau}} & W(\tau),
}
\end{equation}
where $\Upsilon_{\tau}$ is a homeomorphism which serves as a chart. 

The homeomorphism $\Upsilon_{\tau}$ exists because if we fix a vertex $v \in \tau$, then we can equip $V(\tau)$ with an action by the real torus $\mathtt{T}^{n} := T^*_{v, \mathbb{R}} / T^*_v$ such that both $\moment$ and $\moment_\tau$ induce homeomorphisms from the quotient $V(\tau)/\mathtt{T}^{n}$ onto the images. 
The restriction of $\Upsilon_{\tau}$ to $\tanpoly_{\tau,\real}^* \times\{o\}$, where $\{o\}$ is the zero cone, is a homeomorphism onto $\reint(\tau) \subset W(\tau)$, which is nothing but (a generalized version of) the Legendre transform (see \cite[\S 4.2]{Fulton_toric_book} for the explicit formula); also, this homeomorphism is independent of the choices of the splitting $\mathsf{split}_{\tau}$ and the generators $\{m_i \}_{i \in \mathtt{B}_{\tau}}$. 

%can be described by descending the restriction $\mu|_{\tanpoly_{\tau,\comp^*}^*}$ to $\tanpoly_{\tau,\real}^*$ via the log map $\log \colon \tanpoly_{\tau,\comp^*}^*  \rightarrow \tanpoly_{\tau,\real}^*$. This map has image lying in $\reint(\tau)$ and it can be viewed as a version of the Legendre transform identifying $\tanpoly_{\tau,\real}^* $ with $\reint(\tau)$. Also, this map is independent of the choices of the splitting $\mathsf{split}_{\tau}$ and the generators $\{m_i \}_{i \in \mathtt{B}_{\tau}}$. 

The dependences of the chart $\Upsilon_{\tau}$ on the choices of the splitting $\mathsf{split}_{\tau}\colon \Sigma_{\tau} \rightarrow \tau^{-1} \Sigma_v$ and the generators $\{m_i \}_i$ can be described as follows.
First, if we choose another piecewise linear splitting $\widetilde{\mathsf{split}}_{\tau}\colon \Sigma_{\tau} \rightarrow \tau^{-1} \Sigma_v$, then there is a piecewise linear map $b \colon \Sigma_{\tau} \rightarrow \tanpoly_{\tau,\real}$ recording the difference between $\mathsf{split}_\tau$ and $\widetilde{\mathsf{split}}_{\tau}$. The two corresponding coordinate charts $\Upsilon_{\tau}$ and $\tilde{\Upsilon}_{\tau}$ are then related by a homeomorphism $\gimel$ 
such that
$$
\gimel\left(x,\sum_{i} y_i m_i\right) = \left(x , \sum_{i}y_i e^{4\pi \langle b(m_i),x\rangle } m_i \right),
$$
where $y_i= \half |z^{m_i}|^2$ for some point $z\in  \spec(\comp[\Sigma_{\tau}])$ and $i$ runs through $m_i \in \sigma$,
via the formula $\tilde{\Upsilon}_{\tau} = \Upsilon_{\tau} \circ \gimel$.
Second, if we choose another set of generators $\tilde{m}_j$'s, then the corresponding maps $\hat{\moment}_{\tau}, \tilde{\moment}_{\tau} \colon \spec(\comp[\Sigma_{\tau}]) \rightarrow \norpoly_{\tau,\real}$ are related by a continuous map $\hat{\gimel} \colon \norpoly_{\tau,\real} \rightarrow \norpoly_{\tau,\real}$ which maps each cone $\sigma \in \Sigma_{\tau}$ back to itself. This is because both $\hat{\moment}_{\tau}, \tilde{\moment}_{\tau} $ induce a homeomorphism between $\spec(\comp[\Sigma_{\tau}])/\mathtt{T}_{\tau,\perp}$ and $\norpoly_{\tau,\real}$.

Now suppose that $\omega \subset \tau$. We want to see how the charts $\Upsilon_{\omega}$, $\Upsilon_{\tau}$ can be glued together in a compatible manner. We first make a compatible choice of splittings. So we fix a vertex $v \in \omega$ and a piecewise linear splitting $\mathsf{split}_{\omega}\colon \Sigma_{\omega} \rightarrow \omega^{-1} \Sigma_{v}$. 
We then choose a piecewise linear splitting $\mathsf{split}_{\omega\tau}\colon \Sigma_{\tau} \rightarrow \Sigma_{\omega}$ such that $K_\tau \sigma$ is mapped into $K_\omega \sigma$ for any $\sigma \supset \tau$. 
Together with the natural maps $\tanpoly_{\tau}/\tanpoly_{\omega} \hookrightarrow \tau^{-1}\Sigma_{\omega}$ and $\tau^{-1}\Sigma_{\omega} \twoheadrightarrow \Sigma_{\tau}$, we obtain an isomorphism $\tau^{-1}\Sigma_{\omega} \cong (\tanpoly_{\tau}/\tanpoly_{\omega}) \times \Sigma_{\tau}$.
By composing together $\mathsf{split}_{\omega\tau}\colon \Sigma_{\tau} \rightarrow \Sigma_{\omega}$, $\mathsf{split}_{\omega}\colon \Sigma_{\omega} \rightarrow \omega^{-1} \Sigma_{v}$ and the natural monoid homomorphism $ \omega^{-1} \Sigma_{v} \rightarrow \tau^{-1}\Sigma_{v}$, we get a splitting $\mathsf{split}_{\tau}\colon \Sigma_{\tau} \rightarrow \tau^{-1} \Sigma_{v}$. 

Using these choices of splittings, we have a biholomorphism 
$$\spec(\comp[\tau^{-1}\Sigma_{\omega}]) \cong (\tanpoly_{\tau}/\tanpoly_{\omega})^*\otimes_{\inte}\comp^* \times \spec(\comp[\Sigma_{\tau}])$$ which fits into the following diagram
\begin{equation}\label{eqn:change_of_chart_omega_tau}
\xymatrix@1{ \tanpoly_{\omega,\comp^*}^* \times \spec(\comp[\Sigma_{\omega}]) \ar[r]^{\cong} & \spec(\comp[\omega^{-1}\Sigma_v])&  \\
	 \tanpoly_{\omega,\comp^*}^* \times \spec(\comp[\tau^{-1}\Sigma_{\omega}]) \ar@{^{(}->}[u] \ar[d]^{\cong}  &\spec(\comp[\tau^{-1}\Sigma_v]) \ar[l]^{\cong} \ar[d]^{\cong} \ar@{^{(}->}[u]&  \ar[l]^{s_{\omega\tau}^{-1}}  \spec(\comp[\tau^{-1}\Sigma_v]) \ar[d]^{\cong} \ar@{_{(}->}[ul]_{F_s(\omega\subset \tau)}\\
(\tanpoly_{\omega}\oplus \tanpoly_{\tau}/\tanpoly_{\omega})^*\otimes_{\inte}\comp^* \times \spec(\comp[\Sigma_{\tau}])& \ar[l] \tanpoly_{\tau,\comp^*}^*\times \spec(\comp[\Sigma_{\tau}])& \ar[l]^{s_{\omega\tau}^{-1}} \tanpoly_{\tau,\comp^*}^* \times \spec(\comp[\Sigma_{\tau}]).
}
\end{equation}
Here, the bottom left horizontal map is induced from a splitting $(\tanpoly_{\tau}/\tanpoly_{\omega}) \rightarrow \tanpoly_{\tau}$ obtained by composing $\tanpoly_{\tau}/\tanpoly_{\omega} \rightarrow \tau^{-1}\Sigma_{\omega}$ with the splitting $\tau^{-1}\Sigma_{\omega} \rightarrow \tau^{-1}(\omega^{-1} \Sigma_{v})$, and then identifying with the image lattice $\tanpoly_{\tau}$. The appearance of $s_{\omega\tau}$ in the diagram is due to the twisting of $V(\tau)$ by the open gluing data $(s_{\omega\tau})_{\omega\subset\tau}$ when it is glued to $V(\omega)$. 

We also have to make a compatible choice of the generators $\{m_i\}_{i \in \mathtt{B}_{\omega}}$ and $\{m_i\}_{i \in \mathtt{B}_{\tau}}$. 
First note that the restriction of $\hat{\moment}_{\omega}$ to the open subset $ \spec(\comp[\tau^{-1}\Sigma_{\omega}]) \subset \spec(\comp[\Sigma_{\omega}])$ depends only on the subcollection $\{m_i\}_{i \in \mathtt{B}_{\omega\subset \tau}}$ of $\{m_i\}_{i \in \mathtt{B}_{\omega}}$ which contains those $m_i$'s that belong to some cone $\sigma \supset \tau$. 
We choose the set of generators $\{\tilde{m}_i\}_{i \in \mathtt{B}_{\tau}}$ for $\Sigma_{\tau}$, with $\mathtt{B}_{\tau}=\mathtt{B}_{\omega\subset \tau}$, to be the projection of $\{m_i\}_{i \in \mathtt{B}_{\omega\subset \tau}}$ through the natural map $\tau^{-1} \Sigma_{\omega} \rightarrow \Sigma_{\tau}$. Each $m_i$ can be expressed as $m_i =  \mathsf{split}_{\omega\tau}(\tilde{m}_i) + b_i$ for some $b_i \in \tanpoly_{\tau}/\tanpoly_{\omega}$, through the splitting $\mathsf{split}_{\omega\tau}\colon \Sigma_{\tau} \rightarrow \Sigma_{\omega}$. Notice that if $m_i \in K_{\omega} \tau$, then we have $\tilde{m}_i= o$ and hence $b_i \in K_{\omega}\tau $. 
By tracing through the biholomorphism in \eqref{eqn:change_of_chart_omega_tau} and taking either the modulus or the log map, we have a map 
$$\gimel \colon \tanpoly_{\omega,\real}^* \times (\tanpoly_{\tau,\real}/\tanpoly_{\omega,\real})^* \times \norpoly_{\tau,\real} \rightarrow \tanpoly_{\omega,\real}^* \times \norpoly_{\omega,\real}$$
satisfying
\begin{equation}\label{eqn:affine_local_chart_change_of_strata}
\gimel\left(x_1 - c_{\omega\tau,1},x_2-c_{\omega\tau,2},\sum_i y_i |s_{\omega\tau}(\mathsf{split}_{\omega\tau}(\tilde{m}_i))|^{-2} \tilde{m}_i\right) = \left(x_1,\sum_{i} y_i e^{4\pi \langle b_i,x_2 \rangle} m_i\right),
\end{equation}
where $y_i = \half | z^{\tilde{m}_i}|^2$. 
Here, $s_{\omega\tau} \in \mathrm{PM}(\tau)$ is the part of the open gluing data associated to $\omega \subset \tau$, 
and $c_{\omega\tau}=c_{\omega\tau,1}+c_{\omega\tau,2} \in \tanpoly_{\tau,\real}^*$ is the unique element representing the linear map $\log|s_{\omega\tau}| \colon \tanpoly_{\tau,\real} \rightarrow \real$ defined by $\log|s_{\omega\tau}|(b) = \log|s_{\omega\tau}(b)|$.
For instance, the holomorphic function $z^{m_i} \in \comp[\tau^{-1}\Sigma_{\omega}]$ is identified with $z^{b_i}\cdot z^{\tilde{m}_i}$ in $(\tanpoly_{\tau}/\tanpoly_{\omega})^*\otimes_{\inte}\comp^* \times \spec(\comp[\Sigma_{\tau}])$, resulting in the expression $\sum_{i} y_i e^{4\pi \langle b_i,x_2 \rangle} m_i$ on the right hand side. 
We have $\Upsilon_{\tau} = \Upsilon_{\omega} \circ \gimel$, where we use the splitting $(\tanpoly_{\tau}/\tanpoly_{\omega}) \rightarrow \tanpoly_{\tau}$ to obtain an isomorphism $\tanpoly_{\omega,\real}^* \times (\tanpoly_{\tau,\real}/\tanpoly_{\omega,\real})^* \cong \tanpoly_{\tau,\real}^*$ and an identification of the domains of the two maps $\Upsilon_{\tau}$ and $\Upsilon_{\omega} \circ \gimel$.

\begin{lemma}\label{lem:stein_open_subset_lemma}
	There is a base $\mathscr{B}$ of open subsets of $B$ such that the preimage $\moment^{-1}(W)$ is Stein for any $W \in \mathscr{B}$. 
\end{lemma}

\begin{proof}
	First of all, it is well-known that analytic spaces associated to affine varieties are Stein. So $V(\tau)$ is Stein for any $\tau$. 
	Now we fix a point $x \in \reint(\tau) \subset B$. It suffices to show that there is a local base $\mathscr{B}_x$ of $x$ such that the preimage $\moment^{-1}(W)$ is Stein for each $W \in \mathscr{B}_x$. 
	We work locally on $\moment|_{V(\tau)} \colon V(\tau) \rightarrow W(\tau)$. 
	Consider the diagram \eqref{eqn:local_chart_for_stein_subsets} and write $\Upsilon^{-1}(x) = (\underline{x},o)$, where $o \in \norpoly_{\tau,\real}$ is the origin. By \cite[Ch. 1, Ex. 7.4]{demailly1997complex}, the preimage $\log^{-1}(W)$ under the log map $\log \colon (\comp^*)^l \rightarrow \tanpoly_{\tau,\real}^*$ is Stein for any convex $W \subset \tanpoly_{\tau,\real}^*$ which contains $\underline{x}$. 
	Again by \cite[Ch. 1, Ex. 7.4]{demailly1997complex}, any subset 
    $$\bigcap_{j=1}^N\{z \in \spec(\comp[\Sigma_{\tau}]) \ | \ |f_j(z)|<\epsilon \},$$
    where $f_j$'s are holomorphic functions, is Stein.
	By taking $f_j$'s to be the functions $z^{m_j}$'s associated to the set of all non-zero generators in $\{m_j\}_{j\in \mathtt{B}_{\tau}}$ and $\epsilon$ sufficiently small, we have a subset 
    $$W  = \left\{  y \ \Big| \ y = \sum_{j} y_j m_j \text{ with } |y_j|<\frac{\epsilon^2}{2}, \ \text{where $y_j = \half |z^{m_j}|^2$ at some point $z \in \spec(\comp[\Sigma_{\tau}])$}\right\}$$
    of $\norpoly_{\tau,\real}$ such that the preimage $\hat{\moment}_{\tau}^{-1}(W)$ is Stein. Therefore, we can construct a local base $\mathscr{B}_{o}$ of $o$ such that the preimage $\hat{\moment}_{\tau}^{-1}(W)$ is Stein for any $W \in \mathscr{B}_o$. Finally, since a product of Stein open subsets is Stein, we obtain our desired local base $\mathscr{B}_x$ by taking the products of these subsets. 
\end{proof}
 
\subsection{The tropical singular locus $\tsing$ of $B$}\label{subsec:tropical_singular_locus}
We now specify a codimension $2$ singular locus $\tsing \subset B$ of the affine structure using the charts $\Upsilon_{\tau}$ introduced in \eqref{eqn:local_chart_for_stein_subsets} for $\tau$ such that $\dim_{\real}(\tau)<n$. Given the chart $\Upsilon_{\tau}$ that maps $\tanpoly_{\tau,\real}^*$ to $\reint(\tau)$, we define the \emph{tropical singular locus} $\tsing$ by requiring that
\begin{equation}\label{eqn:definition_of_tropical_singular_locus}
 \Upsilon_{\tau}^{-1}(\tsing \cap \reint(\tau)) = \bigcup_{\substack{\rho \in  \mathscr{N}_{\tau};\\ \dim_{\real}(\rho) < \dim_{\real}(\tau)}} \big( (\reint(\rho) +c_{\tau}) \times \{o\} \big) ,
 \end{equation}
where $\mathscr{N}_{\tau} \subset \tanpoly_{\tau,\real}^*$ is the normal fan of the polytope $\tau$, and $\{o\}$ is the zero cone in $\Sigma_{\tau} \subset \norpoly_{\tau,\real}$; here, $c_{\tau} = \log|s_{v\tau}|$ is the element in $\tanpoly_{\tau,\real}^*$ representing the linear map $\log|s_{v\tau}|\colon \tanpoly_{\tau,\real} \rightarrow \real$, which is independent of the vertex $v\in \tau$. A subset of the form $\tsing_{\tau,\rho} := (\reint(\rho)+c_{\tau}) \times \{o\}$ in \eqref{eqn:definition_of_tropical_singular_locus} is called a \emph{stratum} of $\tsing$ in $\reint(\tau)$. 
The locus $\tsing$ is independent of the choices of the splittings $\mathsf{split}_{\tau}$'s and generators $\{m_i\}_{i \in \mathtt{B}_{\tau}}$ used to construct the charts $\Upsilon_{\tau}$'s. 

\begin{remark}
    Our definition of the singular locus is similar to those in \cite{Gross-Siebert-logI, gross2011real}; the only difference is that our locus is a collection of polyhedra in $\tanpoly_{\tau,\real}^*$, instead of $\reint(\tau)$. Note that $\tanpoly_{\tau,\real}^*$ is homeomorphic to $\reint(\tau)$ by the Legendre transform. This modification is needed for our construction of the contraction map $\mathscr{C}$ below, where we need to consider the convex open subsets in $\tanpoly_{\tau,\real}^*$, instead of those in $\reint(\tau)$.% (cf. Lemma \ref{lem:stein_open_subset_lemma}). 
\end{remark}

\begin{lemma}\label{lem:stratification_structure_on_singular_locus}
	For $\omega \subset \tau$ and a stratum $\tsing_{\tau,\rho}$ in $\reint(\tau)$, the intersection of the closure $\overline{\tsing_{\tau,\rho}}$ in $B$ with $\reint(\omega)$ is a union of strata of $\tsing$ in $\reint(\omega)$.
\end{lemma}

\begin{proof}
We consider the map $\gimel$ described in equation \eqref{eqn:affine_local_chart_change_of_strata} and take a neighborhood $W = W_1 \times \norpoly_{\omega,\real}$ of a point $(\underline{x},o)$ in $ \tanpoly_{\omega,\real}^* \times \norpoly_{\omega,\real}$, where $W_1$ is some sufficiently small neighborhood of $\underline{x}$ in $\tanpoly_{\omega,\real}^*$. By shrinking $W$ if necessary, we may assume that $\gimel^{-1}(W) = W_1 \times (a -\reint(K_{\omega}\tau^{\vee})) \times \norpoly_{\tau,\real}$, where $a$ is some element in $-\reint(K_{\omega}\tau^{\vee}) \subset (\tanpoly_{\tau,\real}/\tanpoly_{\omega,\real})^*$. 
Writing $c_{\tau} = c_{\tau,1}+c_{\tau,2}$, where $c_{\tau,1},c_{\tau,2}$ are the components of $c_{\tau}$ according to the chosen decomposition $\tanpoly_{\tau,\real}^* \cong \tanpoly_{\omega,\real}^* \times (\tanpoly_{\tau,\real}/\tanpoly_{\omega,\real})^*$, the equality $c_{\tau,1} + c_{\omega\tau,1} = c_{\omega}$ follows from the compatibility of the open gluing data in Definition \ref{def:open_gluing_data}. 
If $\tsing_{\tau,\rho}$ intersects the open subset $\gimel^{-1}(W)$, then $\rho \subset \tanpoly_{\tau,\real}^*$ must be the dual cone of some face $\rho^{\vee} \subset \omega \subset \tau$ in $\tanpoly_{\tau,\real}^*$. The intersection is of the form 
$$(\reint(\underline{\rho})+c_{\tau,1}) \times (a-\reint(K_{\omega}\tau^{\vee})) \times \{o\}$$
for some $\underline{\rho} \in \mathscr{N}_{\omega}$ ($c_{\tau,2}$ is absorbed by $a$), where $\underline{\rho} \subset \tanpoly_{\omega,\real}^*$ is the dual cone of $\rho^{\vee}$ in $\tanpoly_{\omega,\real}^*$, and hence we have $W \cap \tsing_{\tau,\rho} =\gimel((\reint(\underline{\rho})+c_{\tau,1}) \times (a-\reint(K_{\omega}\tau^{\vee})) \times \{o\})$. Therefore, the intersection of $\overline{\tsing_{\tau,\rho}}$ with $\tanpoly_{\omega,\real}^*$ in the open subset $W \subset \tanpoly_{\omega,\real}^* \times \norpoly_{\omega,\real}$ is given by $(\underline{\rho}+c_{\omega}) \times \{o\}$, which is a union of strata.
\end{proof}
 
The tropical singular locus $\tsing$ is naturally equipped with a stratification, where a stratum is given by $\tsing_{\tau,\rho}$ for some cone $\rho \subset \mathscr{N}_{\tau}$ of $\dim_{\real}(\rho) < \dim_{\real}(\tau)$ for some $\tau \in \pdecomp^{[<n]}$. We use the notation $\tsing^{[k]}$ to denote the set of $k$-dimensional strata of $\tsing$. The affine structure on $\bigcup_{v \in \pdecomp^{[0]}} W_v \cup \bigcup_{\sigma \in \pdecomp^{[n]}} \reint(\sigma)$ introduced right after Definition \ref{def:integral_tropical_manifold} in \S \ref{subsec:integral_affine_manifolds} can be naturally extended to $B \setminus \tsing$ as in \cite{gross2011real}.

If we consider $\omega \subset \tau \subset \rho$ for some $\omega \in \pdecomp^{[1]}$ and $\rho \in \pdecomp^{[n-1]}$, the corresponding monodromy transformation $T_{\gamma}$ is non-trivial if and only if $\omega \in \Omega_p$ and $\rho \in R_p$, where $p$ is as in Definition \ref{def:simplicity}. Therefore, the part of the singular locus $\tsing$ lying in $\Upsilon_{\tau}^{-1}(\reint(\tau)) = \tanpoly_{\tau,\real}^* \times \{o\}$ is determined by the subsets $\Omega_p$'s. We may further define the \emph{essential singular locus} $\tsing_e$ to include only those strata contained in $\tsing^{[n-2]}$ with non-trivial monodromy around them. We observe that the affine structure can be further extended to $B \setminus \tsing_e$.  

More explicitly, we have a projection 
$$\mathtt{i}_{\tau} = \mathtt{i}_{\tau,1} \oplus \cdots \oplus \mathtt{i}_{\tau,p} \colon \tanpoly_{\tau}^* \rightarrow \tanpoly_{\Delta_1(\tau)}^* \oplus \cdots \oplus \tanpoly_{\Delta_p(\tau)}^*,$$
in which $\tanpoly_{\Delta_1(\tau)}^* \oplus \cdots \oplus \tanpoly_{\Delta_p(\tau)}^*$ can be treated as a direct summand as in \S \ref{subsec:monodromy_data}. So we can consider the pullback of the fan $\mathscr{N}_{\Delta_1(\tau)} \times \cdots \times \mathscr{N}_{\Delta_p(\tau)}$ via the map $\mathtt{i}_{\tau}$, and realize $\mathscr{N}_{\tau} \subset \tanpoly_{\tau,\real}^*$ as a refinement of this fan. Similarly, we have $\check{\mathtt{i}}_{\tau} =\check{\mathtt{i}}_{\tau,1} \oplus \cdots \oplus \check{\mathtt{i}}_{\tau,p} \colon \norpoly_{\tau}^* \rightarrow \tanpoly^*_{\check{\Delta}_1(\tau)} \oplus \cdots \oplus \tanpoly^*_{\check{\Delta}_p(\tau)}$ and the fan $\mathscr{N}_{\check{\Delta}_1(\tau)} \times \cdots \times \mathscr{N}_{\check{\Delta}_p(\tau)}$ in $\norpoly_{\tau,\real}^*$ under pullback via $\check{\mathtt{i}}_{\tau}$. The intersection $\tsing_e \cap \reint(\tau) $ can be described by replacing $\rho \in \mathscr{N}_{\tau}$ with the condition $\rho \in \mathtt{i}_{\tau}^{-1}(\mathscr{N}_{\Delta_1(\tau)} \times \cdots \times \mathscr{N}_{\Delta_p(\tau)})$, with a stratum denoted by $\tsing_{e,\tau,\rho}$. This gives a stratification on $\tsing_e$.
 
 \begin{lemma}\label{lem:stratification_structure_on_essential_singular_locus}
    For $\omega \subset \tau$ and a stratum $\tsing_{e,\tau,\rho}$ in $\reint(\tau)$, the intersection of the closure $\overline{\tsing_{e,\tau,\rho}}$ in $B$ with $\reint(\omega)$ is a union of strata of $\tsing_e$ in $\reint(\omega)$.
 \end{lemma}

\begin{proof}
Given $\omega \subset \tau$, we take a change of coordinate map $\gimel$ together with a neighborhood $W$ as in the proof of Lemma \ref{lem:stratification_structure_on_singular_locus}. We need to show that $W \cap \tsing_{\tau,\rho} =\gimel((\reint(\rho)+c_{\tau,1}) \times (a-\reint(K_{\omega}\tau^{\vee})) \times \{o\})$ for some cone $\rho \in \mathtt{i}_{\tau}^{-1}(\prod_{i=1}^p \mathscr{N}_{\Delta_i(\tau)})$.
Let $\Delta_1(\tau),\dots,\Delta_r(\tau),\dots,\Delta_p(\tau)$ be the monodromy polytopes of $\tau$, and $\Delta_1(\omega), \dots, \Delta_r(\omega),\dots,\Delta_{p'}(\omega)$ be those of $\omega$ such that $\Delta_j(\omega)$ is the face of $\Delta_j(\tau)$ parallel to $\tanpoly_{\omega}$ for $j=1,\dots,r$. 
Then we have direct sum decompositions $\tanpoly_{\Delta_{1}(\tau)}\oplus \cdots \oplus \tanpoly_{\Delta_{p}(\tau)} \oplus A_{\tau} = \tanpoly_{\tau}$ and $ \tanpoly_{\Delta_{1}(\omega)}\oplus \cdots \oplus \tanpoly_{\Delta_{p'}(\omega)} \oplus A_{\omega} = \tanpoly_{\omega}$. We can further choose an inclusion
$$
\mathsf{a}_{\omega\tau}\colon \tanpoly_{\Delta_{r+1}(\omega)}\oplus \cdots \oplus \tanpoly_{\Delta_{p'}(\omega)} \oplus A_{\omega} \hookrightarrow A_{\tau};
$$
in other words, for every $j=r+1,\dots,p'$, any $f \in R_j \subset \pdecomp_{n-1}(\omega)$ in Definition \ref{def:simplicity} is not containing $\tau$. For every $j=r+1,\dots,p$ and any $f \in R_j \subset \pdecomp_{n-1}(\tau)$, the element $m^{f}_{v_1v_2}$ is zero for any two vertices $v_1,v_2$ of $\omega$. We have the identification
$$
\tanpoly_{\tau}/\tanpoly_{\omega} = \bigoplus_{j=1}^r (\tanpoly_{\Delta_j(\tau)}/\tanpoly_{\Delta_j(\omega)}) \oplus \bigoplus_{l=r+1}^{p} \tanpoly_{\Delta_l(\tau)} \oplus \mathrm{coker}(\mathsf{a}_{\omega\tau}).
$$

As a result, any cone $\mathtt{i}^{-1}_{\tau}(\prod_{j=1}^{p} \rho_j) \in \mathtt{i}^{-1}_{\tau}\big(\prod_{i=1}^p \mathscr{N}_{\Delta_i(\tau)} \big)$ of codimension greater than $0$ intersecting $\gimel^{-1}(W)$ will be a pullback of a cone under the projection to $\tanpoly_{\Delta_1(\tau),\real}^* \oplus \cdots \oplus \tanpoly_{\Delta_r(\tau),\real}^*$. Consider the commutative diagram of projection maps
\begin{equation}\label{eqn:compatibility_of_essential_singular_locus}
\xymatrix@1{
\tanpoly_{\omega,\real}^* \ar[d]^{\mathtt{p}_{\omega}} & & \ar[ll]^{\mathtt{p}_{\omega\subset \tau}} \tanpoly_{\tau,\real}^* \ar[d]^{\mathtt{p}_{\tau}}\\
\prod_{j=1}^r \tanpoly_{\Delta_j(\omega),\real}^*& & \ar[ll]^{\Pi_{\omega\subset \tau}} \prod_{j=1}^r \tanpoly_{\Delta_j(\tau),\real}^*,
}
\end{equation}
we see that, in the open subset $\gimel^{-1}(W)$, every cone of codimension greater than $0$ coming from pullback via $\mathtt{p}_\tau$ is a further pullback via $\Pi_{\omega\subset\tau} \circ \mathtt{p}_{\tau}$. As a consequence, it must be of the form $\gimel((\reint(\rho)+c_{\tau,1}) \times (a-\reint(K_{\omega}\tau^{\vee})) \times \{o\})$ in $W$. 
\end{proof}
 
\subsubsection{Contraction of $\amoeba$ to $\tsing$}
We would like to relate the amoeba $\amoeba = \moment(Z)$ with the tropical singular locus $\tsing$ introduced above. 

\begin{assum}\label{assum:existence_of_contraction}
We assume the existence of a surjective \emph{contraction map} $\mathscr{C} \colon B \rightarrow B$ which is isotopic to the identity and satisfies the following conditions:
\begin{enumerate}
    \item $\mathscr{C}^{-1}(B \setminus \tsing) \subset (B \setminus \tsing)$ and the restriction $\mathscr{C}|_{\mathscr{C}^{-1}(B \setminus \tsing)}\colon \mathscr{C}^{-1}(B \setminus \tsing) \to B \setminus \tsing$ is a homeomorphism.%For $x \in B \setminus \tsing$, $\mathscr{C}^{-1}(x) = \{x'\}$ is a singleton.
    \item $\mathscr{C}$ maps $\amoeba$ into the essential singular locus $\tsing_e$.
    \item For each $\tau \in \pdecomp$, we have $\mathscr{C}^{-1}(\reint(\tau)) \subset \reint(\tau)$. 
    \item For each $\tau \in \pdecomp$ with $0<\dim_{\real}(\tau)<n$, we have a decomposition $$\tau \cap \mathscr{C}^{-1}(B\setminus \tsing) = \bigcup_{v \in \tau^{[0]}} \tau_{v}$$
    of the intersection $\tau \cap \mathscr{C}^{-1}(B\setminus \tsing)$ into connected components $\tau_{v}$'s, where each $\tau_{v}$ is contractible and is the unique component containing the vertex $v \in \tau$.
    \item For each $\tau \in \pdecomp$ and each point $x \in \reint(\tau) \cap \tsing$, $\mathscr{C}^{-1}(x) \subset \reint(\tau)$ is a connected compact subset.
    \item For each $\tau \in \pdecomp$ and each point $x \in \reint(\tau) \cap \tsing$, there exists a local base $\mathscr{B}_x$ around $x$ such that $(\mathscr{C} \circ \moment)^{-1}(W) \subset V(\tau)$ is Stein for every $W \in \mathscr{B}_x$, and for any $U \supset \mathscr{C}^{-1}(x)$, we have $\mathscr{C}^{-1}(W) \subset U$ for sufficiently small $W \in \mathscr{B}_x$.
\end{enumerate}
\end{assum}

Similar contraction maps appear in \cite[Rem. 2.4]{ruddat2019period} (see also \cite{Ruddat-Zharkov21, Ruddat-Zharkov22}). 

When $\dim_{\real}(B) = 2$, we can take $\mathscr{C} = \mathrm{id}$ because from \cite[Ex. 1.62]{Gross-Siebert-logI}, we see that $Z$ is a finite collection of points, with at most one point lying in each closed stratum $\centerfiber{0}_{\tau}$, and the amoeba $\amoeba$ is exactly the image of $Z$ under the generalized moment map $\mu$.

When $\dim_{\real}(B) = 3$, the amoeba $\amoeba$ can possibly be of codimension one and we need to construct a contraction map as shown in Figure \ref{fig:contraction map}.
\begin{center}
	\begin{figure}[h]
	\includegraphics[scale=0.16]{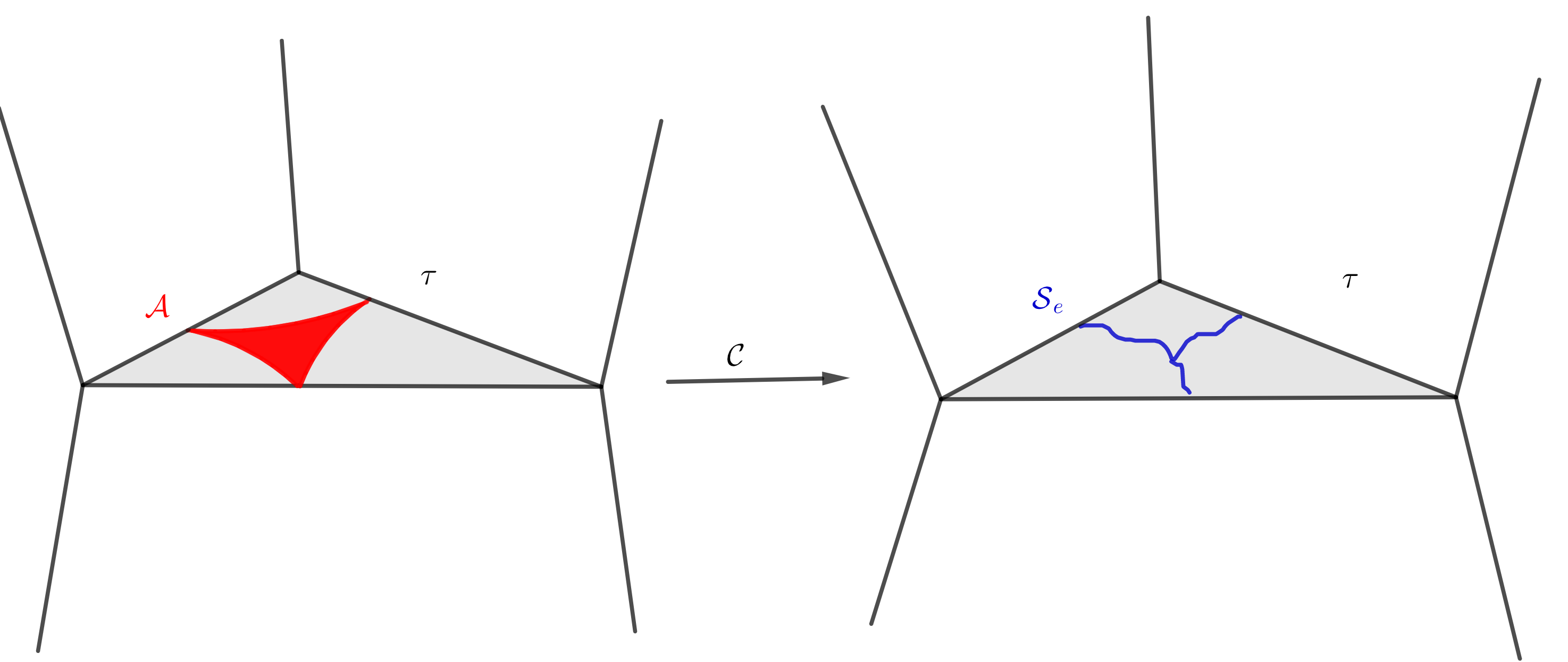}
	\caption{Contraction map $\mathscr{C}$ when $\dim_{\real}(B) = 3$}\label{fig:contraction map}
	\end{figure}
\end{center}
For $\dim_{\real}(\tau) = 1$, again from \cite[Ex. 1.62]{Gross-Siebert-logI}, we see that if $\amoeba \cap \reint(\tau) \neq \emptyset$, then there is exactly one $\Omega_1$ and $R_1$, and $\Delta_{1}(\tau)$ is a line segment of affine length $1$. In this case, $Z \cap \centerfiber{0}_{\tau}$ consists of only one point, given by the intersection of the zero locus $s_{v\tau}^{-1}(f_{v\rho})$ with $\comp^* \cong V_{\tau}(\tau) \subset V(\tau)$. Taking $m$ to be the primitive vector in $\tanpoly_{\tau}$ starting at $v$ that points into $\tau$, we can write $s_{v\tau}^{-1}(f_{v\rho}) = 1 + s_{v\tau}^{-1}(m) z^{m}$. Applying the log map $\log \colon \comp^* \rightarrow \real$, we see that $\amoeba \cap \reint(\tau) = c_{\tau}$. Therefore, for an edge $\tau \in \pdecomp^{[1]}$, we can define $\mathscr{C}$ to be the identity on $\tau$. 
\begin{center}
	\begin{figure}[h]
		\includegraphics[scale=0.6]{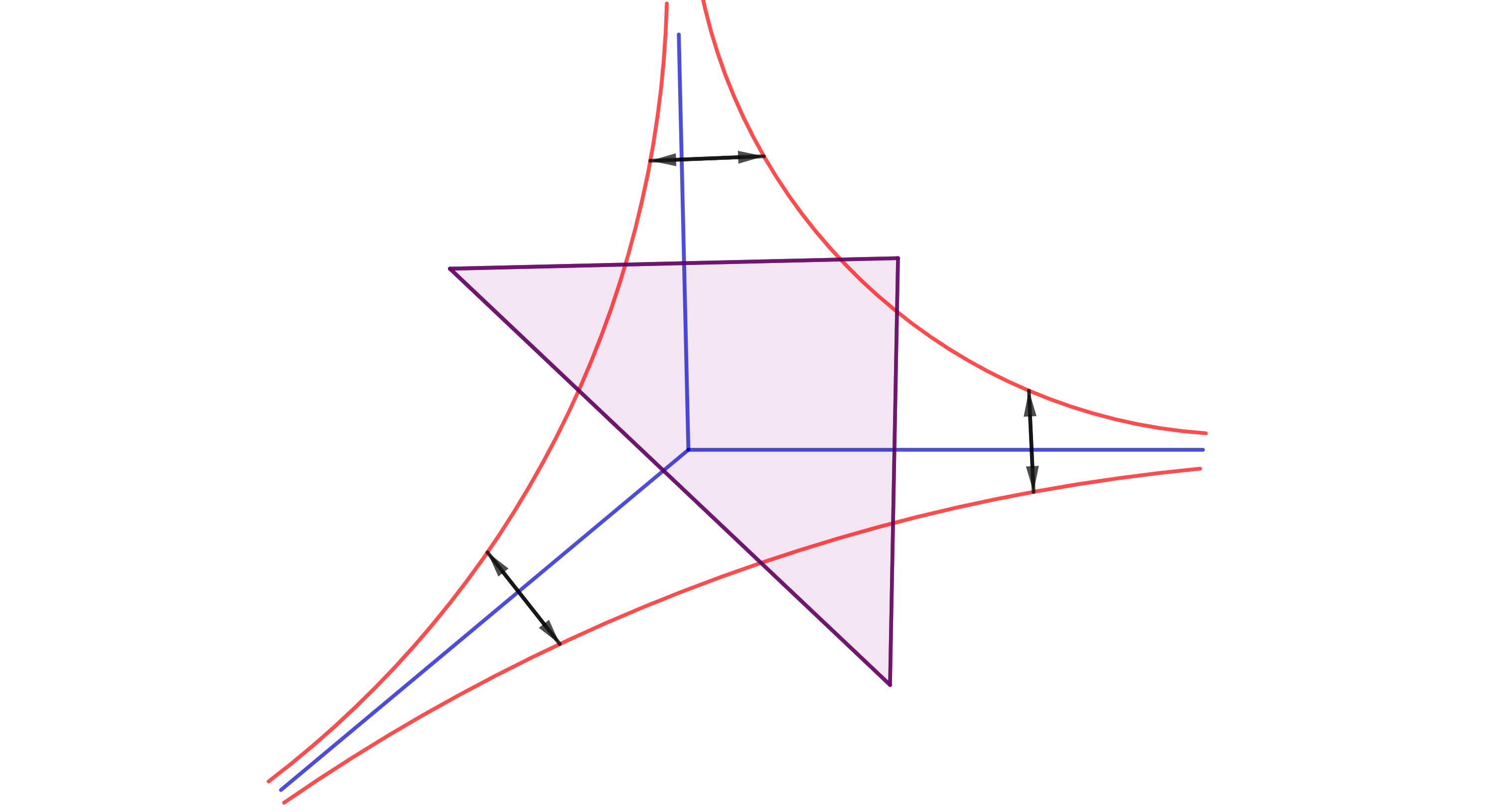}
		\caption{Contraction at $\rho$}\label{fig:contraction_at_rho}
	\end{figure}
\end{center}
On a codimension one cell $\rho$ such that $\reint(\rho) \cap \amoeba \neq \emptyset$ (see Figure \ref{fig:contraction_at_rho}), we consider the log map $\log \colon \spec(\comp[\tanpoly_{\rho}]) \cong (\comp^*)^2 \rightarrow \tanpoly_{\rho,\real}^* \cong \real^2$, and take a sufficiently large polytope $\mathtt{P}$ (colored purple in Figure \ref{fig:contraction_at_rho}) so that $\amoeba \setminus \reint(\mathtt{P})$ is a disjoint union of legs. We first contract each leg to the tropical singular locus (colored blue in Figure \ref{fig:contraction_at_rho}) along the normal direction to the tropical singular locus. Next, we contract the polytope $\mathtt{P}$ to the $0$-dimensional stratum of $\tsing_e$. Notice that the restriction of $\mathscr{C}$ to the tropical singular locus $\tsing$ is not the identity but rather a contraction onto itself. Once the contraction map is constructed for all codimension one cells $\rho$, we can then extend it continuously to the whole of $B$ so that it is a diffeomorphism on $\reint(\sigma)$ for every maximal cell $\sigma$.
The map is chosen such that the preimage $\mathscr{C}^{-1}(x)$ for every point $x \in \reint(\rho)$ is a convex polytope in $\real^2$. Therefore, given any open subset $U\subset \real^2$ which contains $\mathscr{C}^{-1}(x)$, we can find some convex open neighborhood $W_1 \subset U$ of $\mathscr{C}^{-1}(x)$ giving the Stein open subset $\log^{-1}(W_1) \subset (\comp^*)^2$. By taking $W = W_1 \times W_2$ in the chart $\tanpoly_{\rho,\real}^* \times \norpoly_{\rho,\real}$ as in the proof of Lemma \ref{lem:stein_open_subset_lemma}, we have the open subset $W$ that satisfies condition (5) in Assumption \ref{assum:existence_of_contraction}. 

In general, we need to construct $\mathscr{C}|_{\reint(\tau)}$ inductively for each $\tau \in \pdecomp$, so that the preimage $\mathscr{C}^{-1}(x) \subset \reint(\tau)$ is convex in the chart $\tanpoly_{\tau,\real}^*\cong \reint(\tau)$ and the codimension one amoeba $\amoeba$ is contracted to the codimension 2 tropical singular locus $\tsing_e$. The reason for introducing such a contraction map is that we can modify the generalized moment map $\moment$ to one which is more closely related with tropical geometry:
\begin{definition}\label{def:modified_moment_map}
	We call the composition $\modmap :=  \mathscr{C} \circ \moment \colon \centerfiber{0} \rightarrow B$ the \emph{modified moment map}. 
\end{definition}

One immediate consequence of property $(6)$ in Assumption \ref{assum:existence_of_contraction} is that we have $R\modmap_* (\mathcal{F}) = \modmap_*(\mathcal{F})$ for any coherent sheaf $\mathcal{F}$ on $\centerfiber{0}$, thanks to Lemma \ref{lem:stein_open_subset_lemma} and Cartan's Theorem B: 
\begin{theorem}[Cartan's Theorem B \cite{cartan1957varietes}; see e.g. Ch. IX, Cor. 4.11 in \cite{demailly1997complex}]\label{thm:cartan_theorem_B}
	For any coherent sheaf $\mathcal{F}$ over a Stein space $U$, we have
	$
	H^{>0}(U,\mathcal{F}) = 0.
	$
\end{theorem}
%This allows us to work locally on $B$ with a tropical singular locus $\tsing$ in the rest of this paper. 
 
 \subsubsection{Monodromy invariant differential forms on $B$}\label{subsec:derham_for_B}
 
 Outside of the essential singular locus $\tsing_e$, we have a nice integral affine manifold $B \setminus \tsing_e$, on which we can talk about the sheaf $\mdga^*$ of ($\real$-valued) de Rham differential forms. But in fact, we can extend its definition to $\tsing_e$ as well using \emph{monodromy invariant differential forms}.
 
 We consider the inclusion $\tinclude  \colon B_0 :=B \setminus \tsing_e \rightarrow B$ and the natural exact sequence 
 \begin{equation}\label{eqn:affine_function_cotangent_lattice_relation}
 0 \rightarrow \underline{\inte} \rightarrow \cu{A}\mathit{ff} \rightarrow \tinclude_* \tanpoly_{B_0}^* \rightarrow 0,
 \end{equation}
 where $\tanpoly_{B_0}^*$ denotes the sheaf of integral cotangent vectors on $B_0$. For any $\tau \in \pdecomp$, the stalk $(\tinclude_*\tanpoly_{B_0}^*)_x$ at a point $x \in \reint(\tau) \cap \tsing_e$ can be described using the chart $\Upsilon_{\tau}$ in \eqref{eqn:local_chart_for_stein_subsets}. 
 Using the description in \S \ref{subsec:tropical_singular_locus}, we have $x \in \tsing_{e,\tau,\rho} = \reint(\rho) \times \{o\}$ for some $\rho \in \mathtt{i}_{\tau}^{-1}(\mathscr{N}_{\Delta_1(\tau)} \times \cdots \times \mathscr{N}_{\Delta_p(\tau)})$. Taking a vertex $v \in \tau$, we can consider the monodromy transformations $T_{\gamma}$'s around the strata $\tsing_{e,\eta,\rho}$'s that contain $x$ in their closures. We can identify the stalk $\tinclude_*(\tanpoly_{B_0}^*)_{x}$ as the subset of invariant elements of $T_{v}^*$ under all such monodromy transformations. Since $\rho \subset \tanpoly_{\tau,\real}^*$ is a cone, we have $\tanpoly_{\rho} \subset \tanpoly_{\tau}^*$. Using the natural projection map $\pi_{v\tau}\colon T_{v}^* \rightarrow \tanpoly_{\tau}^*$, we have the identification $\tinclude_*(\tanpoly_{B_0}^*)_{x} \cong \pi_{v\tau}^{-1}(\tanpoly_{\rho})$. There is a direct sum decomposition $\tinclude_*(\tanpoly_{B_0}^*)_{x} = \tanpoly_{\rho} \oplus \norpoly_{\tau}^*$, depending on a decomposition $T_{v} = \tanpoly_{\tau} \oplus \norpoly_{\tau}$. 
 This gives the map%a \emph{monodromy invariant affine coordinate chart}
 \begin{equation}\label{eqn:monodromy_invariant_affine_functions_near_x_o}
 \mathtt{x} \colon U_{x} \rightarrow \pi_{v\tau}^{-1}(\tanpoly_{\rho})_{\real}^*
 \end{equation} 
 in a sufficiently small neighborhood $U_{x}$, locally defined up to a translation in $\pi_{v\tau}^{-1}(\tanpoly_{\rho})^*_{\real}$. We need to describe the compatibility between the map associated to a point $x \in \tsing_{e,\omega,\rho}$ and that to a point $\tilde{x} \in \tsing_{e,\tau,\tilde{\rho}}$ such that $\tsing_{e,\omega,\rho} \subset \overline{\tsing_{e,\tau,\tilde{\rho}}}$.  
 
 %For a nearby point $y \in \reint(\rho_1) \times \rho_2 \cap U_x$ in $\tsing_e$, $\mathtt{x}$ can also be used as a chart near $y$.
 
The first case is when $\omega =\tau$. We let $\tilde{x} \in \reint(\tilde{\rho}) \times \{o\} \cap U_{x}$ for some $\rho \subset \tilde{\rho}$. Then, after choosing suitable translations in $\pi_{v\tau}^{-1}(\tanpoly_{\rho})^*_{\real}$ for the maps $\mathtt{x}$ and $\tilde{\mathtt{x}}$, we have the following commutative diagram:
\begin{equation}\label{eqn:monodromy_invariant_differential_form_change_of_chart}
 \xymatrix@1{ U_{\tilde{x}} \cap U_{x} \ar[d] \ar[rr]^{\tilde{\mathtt{x}}} &  & \pi_{v\tau}^{-1}(\tanpoly_{\tilde{\rho}})^*_{\real} \ar[d]^{\mathtt{p}}\\ 
 	U_{x} \ar[rr]^{\mathtt{x}}  & & \pi_{v\tau}^{-1}(\tanpoly_{\rho})^*_{\real}.}
\end{equation}
The second case is when $\omega \subsetneq \tau$. Making use of the change of charts $\gimel$ in equation \eqref{eqn:affine_local_chart_change_of_strata}, and the description in the proof of Lemma \ref{lem:stratification_structure_on_essential_singular_locus}, we write 
$$\tilde{x} \in \reint(\tilde{\rho}) \times \{o\}$$
for some cone $\tilde{\rho} = \mathtt{i}_{\tau}^{-1}(\prod_{j=1}^p \tilde{\rho}_j) \in \mathtt{i}_{\tau}^{-1}\big(\prod_{j=1}^p \tanpoly_{\Delta_{j}(\tau)}^* \big)$ of positive codimension. In $\gimel^{-1}(W)$, we may assume $\tilde{\rho}$ is the pullback of a cone $\breve{\rho}$ via $\Pi_{\omega\subset\tau} \circ \mathtt{p}_{\tau}$ as in equation \eqref{eqn:compatibility_of_essential_singular_locus}. Since $\tsing_{e,\omega,\rho} \subset \overline{\tsing_{e,\tau,\tilde{\rho}}}$, we have $\rho \subset \mathtt{p}_{\omega}^{-1}(\breve{\rho})$ and hence $\mathtt{p}_{\omega\subset \tau}^{-1}(\tanpoly_{\rho}) \subset \tanpoly_{\tilde{\rho}}$. Therefore, from $\mathtt{p}_{\omega\subset \tau} \circ \pi_{v\tau} = \pi_{v\omega}$, we obtain $\pi_{v\omega}^{-1}(\tanpoly_{\rho}) \subset \pi_{v\tau}^{-1}(\tanpoly_{\tilde{\rho}})$, inducing the map $\mathtt{p}\colon \pi_{v\tau}^{-1}(\tanpoly_{\tilde{\rho}})_{\real}^* \rightarrow  \pi_{v\omega}^{-1}(\tanpoly_{\rho})_{\real}^*$. As a result, we still have the commutative diagram \eqref{eqn:monodromy_invariant_differential_form_change_of_chart} for a point $\tilde{x}$ sufficiently close to $x$. 
 
\begin{definition}\label{def:monodromy_invariant_differential_form_near_x_o}
    Given $x \in \tsing_e$ as above, the stalk of $\mdga^*$ at $x$ is defined as the stalk $\mdga^*_{x}:= (\mathtt{x}^{-1}\mdga^*)_{x}$ of the pullback of the sheaf of smooth de Rham forms on $\pi_{v\tau}^{-1}(\tanpoly_{\rho})^*_{\real}$, which is equipped with the de Rham differential $\md$. This defines the \emph{complex $(\mdga^*, \md)$ of monodromy invariant smooth differential forms} on $B$. A section $\alpha \in \mdga^*(W)$ is a collection of elements $\alpha_{x} \in \mdga^*_{x}$, $x \in W$ such that each $\alpha_{x}$ can be represented by $\mathtt{x}^{-1}\beta_{x}$ in a small neighborhood $U_{x} \subset \mathtt{p}^{-1}(\mathtt{U}_{x})$ for some smooth form $\beta_{x}$ on $\mathtt{U}_{x}$, and satisfies the relation $\alpha_{\tilde{x}} = \tilde{\mathtt{x}}^{-1}(\mathtt{p}^* \beta_{x})$ in $\mdga^*_{\tilde{x}}$ for every $\tilde{x} \in U_{x}$. 
\end{definition}
 
\begin{example}
 	In the $2$-dimensional case in Example \ref{eg:2d_monodromy}, we consider a singular point 
        $$\{ x \} = \tsing_e \cap \reint(\tau)$$ 
        for some $\tau \in \pdecomp^{[1]}$. In this case, we can take $\rho$ to be the $0$-dimenisonal stratum in $\mathscr{N}_{\tau}=\mathtt{i}_{\tau}^{-1}(\mathscr{N}_{\Delta_1(\tau)})$ and we have $\tinclude_*(\tanpoly_{B_0}^*)_{x} = \norpoly_{\tau}^*$. Taking a generator of $\norpoly_{\tau}^*$, we get an invariant affine coordinate $\mathtt{x}\colon U_{x} \rightarrow \real$ which is the normal affine coordinate of $\tau$. The stalk $\mdga^*_{x}$ is then identified with the pullback of the space of germs of smooth differential forms from $(\real,0)$ via $\mathtt{x}$. In particular, $\mdga^2_{x} =0$. 
 	
 	For the $Y$-vertex of type II in Example \ref{eg:3d_monodromy}, the situation is similar to the $2$-dimensional case. For $\{ x \} = \tsing_e \cap \reint(\tau)$, we still have $\tinclude_*(\tanpoly_{B_0}^*)_{x} = \norpoly_{\tau}^*$, and in this case, $\mathtt{x} \colon U_{x} \rightarrow \real^2$ are the two invariant affine coordinates. We can identify $\mdga^*_{x}$ as the pullback of the space of germs of smooth differential forms from $(\real^2,0)$ via $\mathtt{x}$. 
 	
 	For the $Y$-vertex of type I in Example \ref{eg:3d_monodromy}, we use the identification $ \tanpoly_{\tau,\real}^* \cong \reint(\tau)$ via $\Upsilon_{\tau}$ for the $2$-dimensional cell $\tau$ separating two maximal cells $\sigma_+$ and $\sigma_-$. 
    In this case, $\tsing_e$ is as shown (in blue color) in Figure \ref{fig:contraction_at_rho} and $\mathscr{N}=\mathtt{i}_{\tau}^{-1}(\mathscr{N}_{\Delta_1(\tau)})$ is the fan of $\mathbb{P}^2$. 
    If $x$ is the $0$-dimensional stratum of $\tsing_e \cap \reint(\tau)$, we have $\tinclude_*(\tanpoly_{B_0}^*)_{x} = \norpoly_{\tau}^*$ and $\mathtt{x}\colon U_{x} \rightarrow \real$ as an invariant affine coordinate. If $x$ is a point on a leg of the $Y$-vertex, we have $\mathtt{x} =(\mathtt{x}_1,\mathtt{x}_2)\colon U_{x} \rightarrow \real^2$ with $\mathtt{x}_1$ coming from a generator of $\tanpoly_{\rho}$ and $\mathtt{x}_2$ coming from a generator of $\norpoly_{\tau}^*$.
\end{example}
 
 It follows from the definition that $\underline{\real} \rightarrow \mdga^*$ is a resolution. We shall also prove the existence of a partition of unity.
 %First of all, from its construction, we see that $\mdga^0$ is a subsheaf of the sheaf $\cu{C}^{0}$ of continuous functions on $B$, so we can talk about the value $f(x)$ for $f \in \mdga^{0}(W)$ and $x \in W$. In particular, we can say whether $0\leq f \leq 1$ if it is satisfied for $f(x)$ at every point $x$. The existence of partition of unity relies on the following lemma.
 
 \begin{lemma}\label{lem:for_contruction_of_partition_of_unity}
 	Given any $x \in B$ and a sufficiently small neighborhood $U$, there exists $\varrho \in \mdga^{0}(U)$ with compact support in $U$ such that $0 \leq \varrho \leq 1$ and $\varrho \equiv 1$ near $x$. (Since $\mdga^0$ is a subsheaf of the sheaf $\cu{C}^{0}$ of continuous functions on $B$, we can talk about the value $f(x)$ for $f \in \mdga^{0}(W)$ and $x \in W$.)
 \end{lemma}
 
 \begin{proof}
 	If $x \notin \tsing_e$, the statement is a standard fact. So we assume that $x \in \reint(\tau) \cap \tsing_e$ for some $\tau \in \pdecomp$. As above, we can write $x \in \reint(\rho) \times \{o\}$. Furthermore, since $\rho$ is a cone in the fan $\mathtt{i}_{\tau}^{-1}(\mathscr{N}_{\Delta_1(\tau)} \times \cdots \times \mathscr{N}_{\Delta_p(\tau)})$, $\tanpoly_{\tau}^*$ has $\tanpoly_{\Delta_1(\tau)}^* \oplus \cdots \oplus \tanpoly_{\Delta_p(\tau)}^*$ as a direct summand, and the description of $\tinclude_*(\tanpoly_{B_0}^*)_x$ is compatible with the direct sum decomposition of $\tanpoly_{\tau}^*$. We may further assume that $p=1$ and $\tau = \Delta_1(\tau)$ is a simplex. 
 	
 	If $\rho$ is not the smallest cone (i.e. the one consisting of just the origin in $\mathscr{N}_{\tau}$), we have a decomposition $\tanpoly_{\tau}^* = \tanpoly_{\rho} \oplus \norpoly_{\rho}$ and the natural projection $\mathtt{p}\colon \tanpoly_{\tau}^* \rightarrow \norpoly_{\rho}$. Then, locally near $x_0$, we can write the normal fan $\mathscr{N}_{\tau}$ as $\mathtt{p}^{-1}(\Sigma_{\rho})$ for some normal fan $\Sigma_{\rho} \subset \norpoly_{\rho}$ of a lower dimensional simplex. For any vector $v$ tangent to $\rho$ at $x_0$ and the corresponding affine function $l_{v}$ locally near $x_0$, we always have $\ddd{l_{v}}{v}>0$. This allows us to construct a bump function $\varrho = \sum_{v_i} (l_{v_i}(x) - l_{v_i}(x_0))^2$ along the $\tanpoly_{\rho,\real}$-direction. So we are reduced to the case when $\rho = \{o\}$ is the smallest cone in the fan $\mathscr{N}_{\tau}$. 
 	
 	Now we construct the function $\varrho$ near the origin $o \in \mathscr{N}_{\tau}$ by induction on the dimension of the fan $\mathscr{N}_{\tau}$.
    %and try to construct the function $\varrho$ near origin $o$. 
    When $\dim_{\real}(\mathscr{N}_{\tau}) = 1$, it is the fan of $\mathbb{P}^1$ consisting of three cones $\real_-$, $\{o\}$ and $\real_+$. One can construct the bump function which is equal to $1$ near $o$ and supported in a sufficiently small neighborhood of $o$. 
 	For the induction step, we consider an $n$-dimensional fan $\mathscr{N}_\tau$. For any point $x$ near but not equal to $o$, we have $x \in \reint(\rho)$ for some $\rho \neq \{o\}$. Then we can decompose $\mathscr{N}_\tau$ locally as $\tanpoly_{\rho} \oplus \norpoly_{\rho}$. Applying the induction hypothesis to $\norpoly_{\rho}$ gives a bump function $\varrho_x$ compactly supported in any sufficiently small neighborhood of $x$ (for the $\tanpoly_{\rho}$ directions, we do not need the induction hypothesis to get the bump function). This produces a partition of unity $\{\varrho_i\}$ \emph{outside} $o$. Finally, letting $\varrho := 1 - \sum_i \varrho_i$ and extending it continuously to the origin $o$ gives the desired function.  
 	%From induction we learn that for every point $x \in \reint(\rho)$ for some $\rho \neq \{o\}$ we can decompose locally as $\tanpoly_{\rho} \oplus \norpoly_{\rho}$ and work on lower dimensional fan in $\norpoly_{\rho}$ by induction. Therefore, there is a bump function $\varrho_x$ compactly supported in any sufficiently small neighborhood of $x$. Therefore we have a partition of unity $\{\varrho_i\}$ outside $o$, and we can let $\varrho = 1-\sum_i \varrho_i$ and extend it to take constant value $1$ near the origin $o$ to get the desired function.  
 \end{proof}
 
 Lemma \ref{lem:for_contruction_of_partition_of_unity} produces a partition of unity for the complex $(\mdga^*, \md)$ of monodromy invariant differential forms on $B$, which satisfies the requirement in Condition \ref{cond:requirement_of_the_de_rham_dga} below. In particular, the cohomology of $(\mdga^*(B),\md)$ computes $R\Gamma(B,\underline{\real})$. 
 Given a point $x \in B \setminus \tsing_e$, we can take an element $\varrho_x \in \mdga^n(B)$, compactly supported in an arbitrarily small neighborhood $U_x \subset B \setminus \tsing_e$, to represent a non-zero element in the cohomology $H^n(\mdga^*,\md) = H^n(B,\comp) \cong \comp$. 
 
\section{Smoothing of maximally degenerate Calabi--Yau varieties via dgBV algebras}\label{sec:deformation_via_dgBV}

%We begin with summarizing the geometric structures that we have introduced.  With strongly simple positive integral tropical manifold $(B,\pdecomp)$ with open gluing data $s$ satisfying Condition \ref{cond:open_gluing_data_lifting_condition} and strictly convex multi-valued piecewise affine function $\varphi$, we construction a complex analytic log scheme $\centerfiber{0}(B,\pdecomp,s)^{\dagger}$ which is further assumed to be projective. This allows us to construct a generalized moment map $\moment\colon \centerfiber{0} \rightarrow B$. The singular locus of log structure $Z \subset \centerfiber{0}^{\dagger}$ will maps to an amoeba $\amoeba \subset B$ under $\moment$. We compose $\mu$ with a contraction map $\mathscr{C}$ to create a modified map $\modmap$. 

In this section, we review and refine the results in \cite{chan2019geometry} concerning smoothing of the maximally degenerate Calabi--Yau log variety $\centerfiber{0}^{\dagger}$ over $\logsf = \spec(\hat{\cfr})^{\dagger} = \spec(\comp[[q]])^{\dagger}$ using the local smoothing models $V^{\dagger} \rightarrow \localmod{k}^{\dagger}$'s specified in \S \ref{subsec:log_structure_and_slab_function}. In order to relate with tropical geometry on $B$, we will choose $V$ so that it is the pre-image $\modmap^{-1}(W)$ of an open subset $W$ in $B$. 

%\subsection{Smoothing via gluing of dgBV algebras}\label{subsec:deformation_via_gluing_dgBV}
\subsection{Good covers and local smoothing data}\label{subsubsec:local_deformation_data}

Given $\tau \in \pdecomp$ and a point $x \in \reint(\tau) \subset B$, we take a sufficiently small open subset $W \in \mathscr{B}_x$. We need to construct a local smoothing model on the Stein open subset $V = \modmap^{-1}(W)$. 
%First of all, we consider the compact subset $\mathscr{C}^{-1}(x)$ for some $x \in \reint(\tau)$ and take a sufficiently small open subset $W \in \mathscr{B}_x$ for constructing a local model on $V = \modmap^{-1}(W)$.
\begin{itemize}
    \item
    If $x \notin \tsing_e$, then we can simply take the local smoothing $\localmod{}^{\dagger}$ introduced in \eqref{eqn:cartesian_diagram_of_log_spaces} in \S \ref{subsec:log_structure_and_slab_function}.

    \item 
    If $x \in \tsing_e \cap \reint(\tau)$, we assume that $\mathscr{C}^{-1}(W) \cap \amoeba^{\tau}_i \neq \emptyset$ for $i = 1,\dots,r$ and $\mathscr{C}^{-1}(W) \cap \amoeba^{\tau}_i = \emptyset $ for other $i$'s. Note that $\mathscr{C}^{-1}(W) \cap \reint(\tau)$ may not be a small open subset in $\reint(\tau)$ as we may contract a polytope $\mathtt{P}$ via $\mathscr{C}$ (Figure \ref{fig:contraction_at_rho}). If we write $\tanpoly_{\Delta_{1}(\tau)}\oplus \cdots \oplus \tanpoly_{\Delta_{p}(\tau)} \oplus A_{\tau} = \tanpoly_{\tau}$ as lattices, then for each direct summand $\tanpoly_{\Delta_{i}(\tau)}$, we have a commutative diagram
    $$
    \xymatrix@1{  \tanpoly_{\tau,\comp^*}^*\ar[rr]^{\mathtt{i}_{\tau,i,\comp^*}} \ar[d]^{\log}  & & \tanpoly_{\Delta_{i}(\tau),\comp^*}^* \ar[d]^{\log}\\
	\tanpoly_{\tau,\real}^* \ar[rr]^{\mathtt{i}_{\tau,i,\real}} & & \tanpoly_{\Delta_{i}(\tau),\real}^*,}
    $$
    so that both $Z^{\tau}_i$ and $\amoeba^{\tau}_i$ are coming from pullbacks of some subsets  under the projection maps $\mathtt{i}_{\tau,i,\comp^*}$ and $\mathtt{i}_{\tau,i,\real}$ respectively. 
    %which are compatible with the vertical log maps. 
    From this, we see that $\mathscr{C}^{-1}(W) \cap \amoeba^{\tau}_1 \cap \cdots\cap \amoeba^{\tau}_r \neq \emptyset$ and $\modmap^{-1}(W) \cap Z^{\tau}_1 \cap \cdots \cap Z^{\tau}_r \neq \emptyset$ while $\modmap^{-1}(W) \cap Z^{\tau}_i = \emptyset $ for other $i$'s. Now we take $\psi_{x,i} = \psi_i$ for $1\leq i\leq r$ and $\psi_{x,i} =0$ otherwise accordingly. Then we can take $P_{\tau,x}$ introduced in \eqref{eqn:local_monoid_near_singularity} and the map $V = \modmap^{-1}(W) \rightarrow \spec(\comp[\Sigma_{\tau}\oplus \mathbb{N}^{l}])$ defined by
    \begin{equation}\label{eqn:modified_local_mod_for_singularity}
		\begin{cases}
		z^m \mapsto  h_m \cdot z^m & \text{if } m\in \Sigma_{\tau} ;\\
		u_i \mapsto  f_{v,i} & \text{if } 1\leq i \leq r;\\
		u_i \mapsto  z_i & \text{if } r<i\leq l.
	   \end{cases}
    \end{equation} 
    Note that the third line of this formula is different from that of equation \eqref{eqn:singular_locus_local_model} because we do not specify a point $x \in  Z^{\tau}_1 \cap \cdots \cap Z^{\tau}_r$. By shrinking $W$ if necessary, one can show that it is an embedding using an argument similar to \cite[Thm. 2.6]{Gross-Siebert-logII}. This is possible because we can check that the Jacobian appearing in the proof of \cite[Thm. 2.6]{Gross-Siebert-logII} is invertible for all point in $\modmap^{-1}(x) = \moment^{-1}(\mathscr{C}^{-1}(x))$, which is a connected compact subset by property $(5)$ in Assumption \ref{assum:existence_of_contraction}.
\end{itemize}
 
\begin{condition}\label{cond:good_cover_of_B}
An open cover $\{ W_{\alpha} \}_{\alpha}$ of $B$ is said to be \emph{good} if
\begin{enumerate}
	\item for each $W_{\alpha}$, there exists a unique $\tau_{\alpha} \in \pdecomp$ such that $W_{\alpha} \in \mathscr{B}_x$ for some $x \in \reint(\tau)$;
	\item $W_{\alpha\beta}=W_{\alpha} \cap W_{\beta} \neq \emptyset$ only when $\tau_{\alpha} \subset \tau_{\beta}$ or $\tau_{\beta} \subset \tau_{\alpha}$, and if this is the case, we have either $\reint(\alpha) \cap W_{\alpha\beta} \neq \emptyset$ or $\reint(\beta) \cap W_{\alpha\beta} \neq \emptyset$.

\end{enumerate}
\end{condition}
Given a good cover $\{ W_{\alpha} \}_{\alpha}$ of $B$, we have the corresponding Stein open cover $\cu{V} := \{V_\alpha\}_\alpha$ of $\centerfiber{0}$ given by $V_{\alpha} := \modmap^{-1}(W_{\alpha})$ for each $\alpha$. For each $V_{\alpha}^{\dagger}$, the infinitesimal local smoothing model is given as a log space $\localmod{}_{\alpha}^{\dagger}$ over $\logsf$ (see \eqref{eqn:cartesian_diagram_of_log_spaces}). Let $\localmod{k}_{\alpha}$ be the $k^{\text{th}}$-order thickening over $\logsk{k}= \spec(\cfr/\mathbf{m}^{k+1})^{\dagger}$ and $j \colon V_{\alpha} \setminus Z \hookrightarrow V_{\alpha}$ be the open inclusion. 
As in \cite[\S 8]{chan2019geometry}, we obtain coherent sheaves of BV algebras (and modules) over $V_\alpha$ from these local smoothing models. But for the purpose of this paper, we would like to push forward these coherent sheaves to $B$ and work with the open subsets $W_{\alpha}$'s. 
This leads to the following modification of \cite[Def. 7.6]{chan2019geometry} (see also \cite[Def. 2.14 and 2.20]{chan2019geometry}):

\begin{definition}\label{def:higher_order_thickening_data_from_gross_siebert}
	For each $k \in \inte_{\geq 0}$, we define
	\begin{itemize}
		\item
		the \emph{sheaf of $k^{\text{th}}$-order polyvector fields} to be $\bva{k}_{\alpha}^* := \modmap_* j_* ( \bigwedge^{-*} \Theta_{\localmod{k}_{\alpha}^{\dagger}/\logsk{k}})$ (i.e. push-forward of relative log polyvector fields on $\localmod{k}_\alpha^{\dagger}$);
		
		\item
		the \emph{$k^{\text{th}}$-order log de Rham complex} to be $\tbva{k}{}^*_\alpha :=\modmap_* j_* (\Omega^*_{\localmod{k}_{\alpha}^{\dagger}/\comp})$ (i.e. push-forward of log de Rham differentials) equipped with the de Rham differential $\dpartial{k}_{\alpha} = \dpartial{}$ which is naturally a dg module over $\logsdrk{k}{*}$; 
		
		\item
		the \emph{local log volume form} $\volf{}_\alpha$ as a nowhere vanishing element in $\modmap_* j_*(\Omega^n_{\localmod{}_{\alpha}^{\dagger}/\logsf})$ and the \emph{$k^{\text{th}}$-order volume form} to be $\volf{k}_{\alpha} = \volf{}_{\alpha} \ (\text{mod $\mathbf{m}^{k+1}$})$.
	\end{itemize}
\end{definition}

Given $k > l$, there are natural maps $\rest{k,l}\colon j_* ( \bigwedge^{-*} \Theta_{\localmod{k}_{\alpha}^{\dagger}/\logsk{k}}) \rightarrow  j_* ( \bigwedge^{-*} \Theta_{\localmod{l}_{\alpha}^{\dagger}/\logsk{l}})$ which induce the maps $\rest{k,l}\colon \bva{k}^*_{\alpha} \rightarrow \bva{l}^*_{\alpha}$. Before taking the push-forward $\moment_*$, each $j_* ( \bigwedge^{r} \Theta_{\localmod{k}_{\alpha}^{\dagger}/\logsk{k}})$ is a sheaf of flat $\cfrk{k}$-modules with the property that $j_* ( \bigwedge^{r} \Theta_{\localmod{k}_{\alpha}^{\dagger}/\logsk{k}}) \cong j_* ( \bigwedge^{r} \Theta_{\localmod{k+1}_{\alpha}^{\dagger}/\logsk{k+1}}) \otimes_{\cfrk{k+1}} \cfrk{k}$ by \cite[Cor. 7.4 and 7.9]{Felten-Filip-Ruddat}. In other words, we have a short exact sequence of coherent sheaves
$$
\xymatrix@1{
0 \ar[r] & j_* ( \bigwedge^{r} \Theta_{\localmod{0}_{\alpha}^{\dagger}/\logsk{0}}) \ar[r]^{\cdot q^{k+1}} \ar[r] & j_* ( \bigwedge^{r} \Theta_{\localmod{k+1}_{\alpha}^{\dagger}/\logsk{k+1}})  \ar[r] & j_* ( \bigwedge^{r} \Theta_{\localmod{k}_{\alpha}^{\dagger}/\logsk{k}}) \ar[r] & 0.
}
$$
Applying $\moment_*$, which is exact, we get
$$
\xymatrix@1{
	0 \ar[r] & \bva{0}_{\alpha}^{-r} \ar[r]^{\cdot q^{k+1}} \ar[r] & \bva{k+1}_{\alpha}^{-r}  \ar[r] & \bva{k}_{\alpha}^{-r} \ar[r] & 0.
}
$$
As a result, we see that $\bva{k}_{\alpha}^{-r}$ is a sheaf of flat $\cfrk{k}$-modules on $W_{\alpha}$, so we have $\bva{k+1}_{\alpha}^{-r} \otimes_{\cfrk{k+1}} \cfrk{k} \cong \bva{k}^{-r}_{\alpha}$ for each $r$; a similar statement holds for $\tbva{k}{}^{r}_{\alpha}$. 

A natural filtration $\tbva{k}{\bullet}^*_{\alpha}$ is given by $\tbva{k}{s}^*_{\alpha}:= \logsdrk{k}{\geq s} \wedge \tbva{k}{}^*_{\alpha}[s]$ and taking wedge product defines the natural sheaf isomorphism $\gmiso{k}{r}^{-1} \colon \logsdrk{k}{r} \otimes_{\cfrk{k}} (\tbva{k}{0}^*_\alpha/ \tbva{k}{1}^*_\alpha[-r]) \rightarrow \tbva{k}{r}^*_\alpha/ \tbva{k}{r+1}^*_\alpha$. We have the space $\tbva{k}{\parallel}^*_{\alpha} :=\tbva{k}{0}^*_{\alpha}/\tbva{k}{1}^*_{\alpha} \cong \modmap_* j_* (\Omega^*_{\prescript{k}{}{\mathbf{V}}_{\alpha}^{\dagger}/\logsk{k}})$ of \emph{relative log de Rham differentials}. 

There is a natural action $v \mathbin{\lrcorner} \varphi$ for $v \in \bva{k}_{\alpha}^*$ and $\varphi \in \tbva{k}{}^*$ given by contracting a logarithmic holomorphic vector field $v$ with a logarithmic holomorphic form $\varphi$. 
To simplify notations, for $v \in \bva{k}_{\alpha}^0$, we often simply write $v\varphi$, suppressing the contraction $\mathbin{\lrcorner}$.
We define the \emph{Lie derivative} via the formula $\cu{L}_{v} := (-1)^{|v|} \partial \circ (v\mathbin{\lrcorner}) - (v\mathbin{\lrcorner}) \circ \partial$ (or equivalently, $(-1)^{|v|}\cu{L}_{v}:= [\dpartial{},v\mathbin{\lrcorner}]$). By contracting with $\volf{k}_{\alpha}$, we get a sheaf isomorphism $\mathbin{\lrcorner} \volf{k}_{\alpha} \colon \bva{k}_{\alpha}^* \rightarrow \tbva{k}{\parallel}^*_{\alpha}$, which defines the \emph{BV operator} $\bvd{k}_{\alpha}$ by $\bvd{k}_{\alpha}(\varphi) \mathbin{\lrcorner} \volf{k} := \dpartial{k}_{\alpha} (\varphi \mathbin{\lrcorner} \volf{k})$. We call it the BV operator because the BV identity:
\begin{equation}\label{eqn:BV_identity}
(-1)^{|v|}[v,w] : = \Delta(v\wedge w ) - \Delta(v) \wedge w -(-1)^{|v|} v\wedge \Delta(w)
\end{equation}
for $v, w \in \bva{k}_{\alpha}^*$, where we put $\Delta = \bvd{k}_{\alpha}$, defines a graded Lie bracket. This gives $\bva{k}^*_{\alpha}$ the structure of a sheaf of BV algebras.

\subsection{An explicit description of the sheaf of log de Rham forms}\label{subsubsec:local_description_for_log_de_rham_forms}

Here we apply the calculations in \cite{Gross-Siebert-logII,Felten-Filip-Ruddat} to give an explicit description of 
the stalk $\tbva{k}{}_{\alpha,x}^*$.

Let us consider $K=\modmap^{-1}(x)$ and the local model near $K$ described in \S \ref{subsubsec:local_deformation_data}, with $P_{\tau,x}$ and $Q_{\tau,x}$ as in \eqref{eqn:local_monoid_near_singularity}, \eqref{eqn:central_fiber_local_model} and an embedding $V \rightarrow \spec(\comp[Q_{\tau,x}])$. We may treat $K \subset V$ as a compact subset of $\comp^{l} = \spec(\comp[\bb{N}^l]) \hookrightarrow \spec(\comp[Q_{\tau,x}])$ via the identification $\spec(\comp[\Sigma_{\tau}\oplus\bb{N}^l]) \cong \spec(\comp[Q_{\tau,x}])$. For each $m \in \Sigma_{\tau}$, we denote the corresponding element $(m,\psi_{x,0}(m),\dots,\psi_{x,l}(m)) \in P_{\tau,x}$ by $\hat{m}$ % to avoid any confusion, 
and the corresponding function by $z^{\hat{m}} \in \comp[P_{\tau,x}]$. Similar to \cite[Lem. 7.14]{Felten-Filip-Ruddat}, the germs of holomorphic functions $\cu{O}_{\localmod{k},K}$ near $K$ in the space $\localmod{k} = \spec(\comp[P_{\tau,x}/q^{k+1}])$ can be written as 
\begin{equation}\label{eqn:local_germ_of_holomorphic_functions_near_K}
\cu{O}_{\localmod{k},K} = \Bigg\{ \sum_{m \in \Sigma_{\tau},\ 0\leq i \leq k }  \alpha_{m,i} q^{i} z^{\hat{m}} \,\Big|\, \alpha_{m,i} \in \cu{O}_{\comp^{l}}(U)\ \text{for some neigh. $U\supset K$}, \ \sup_{m \in \Sigma_{\tau} \setminus \{0\}} \frac{\log|\alpha_{m,i}|}{\mathtt{d}(m)} < \infty \Bigg\}, 
\end{equation}
where $\mathtt{d}\colon \Sigma_{\tau} \rightarrow \bb{N}$ is a monoid morphism such that $\mathtt{d}^{-1}(0) = 0$, and it is equipped with the product $z^{\hat{m}_1} \cdot z^{\hat{m}_2} := z^{\hat{m}_1+\hat{m}_2}$ (but note that $\widehat{m_1+m_2}\neq \hat{m}_2 + \hat{m}_2$ in general). Thus we have $\tbva{k}{}^0_{\alpha,x} \cong \bva{k}_{\alpha,x}^0 \cong \cu{O}_{\localmod{k},K}$. 

To describe the differential forms, we consider the vector space $\mathscr{E} = P_{\tau,x,\comp}$, regarded as the space of $1$-forms on $\spec(\comp[P_{\tau,x}^{\mathrm{gp}}]) \cong (\comp^*)^{n+1}$. Write $d\log z^{p}$ for $p \in P_{\tau,x,\comp}$ and set $\mathscr{E}_1 := \comp \langle d\log u_i \rangle_{i=1}^l$, as a subset of $\mathscr{E}$. For an element $m\in \norpoly_{\tau,\comp}$, we have the corresponding $1$-form $d\log z^{\hat{m}} \in P_{\tau,x,\comp}$ under the association between $m$ and $z^{\hat{m}}$. Let $\mathtt{P}$ be the power set of $\{1,\dots,l\}$ and write $u^{I} = \prod_{i\in I} u_i$ for $I \in \mathtt{P}$. A computation for sections of the sheaf $j_*(\Omega^r_{\localmod{k}^{\dagger}/\comp})$ from \cite[Prop. 1.12]{Gross-Siebert-logII} and \cite[Lem. 7.14]{Felten-Filip-Ruddat} can then be rephrased as the following lemma.

\begin{lemma}[\cite{Gross-Siebert-logII,Felten-Filip-Ruddat}]\label{lem:local_computation_of_log_de_rham_forms}
    The space of germs of sections of $j_*(\Omega^*_{\localmod{k}^{\dagger}/\comp})_K$ near $K$ is a subspace of $\cu{O}_{\localmod{k},K} \otimes \bigwedge^* \mathscr{E}$ given by elements of the form
	$$\displaystyle
	\alpha = \sum_{\substack{m \in \Sigma_{\tau}\\ 0\leq i \leq k }} \sum_{I} \alpha_{m,i,I} q^i z^{\hat{m}} u^I \otimes \beta_{m,I}, \quad \beta_{m,I} \in \bigwedge\nolimits^* \mathscr{E}_{m,I} = \bigwedge\nolimits^*(\mathscr{E}_{1,m,I}\oplus \mathscr{E}_{2,m,I} \oplus \langle d\log q\rangle),
	$$
	where $\mathscr{E}_{1,m,I} = \langle d\log u_i \rangle_{i\in I} \subset \mathscr{E}_1$ and the subspace $\mathscr{E}_{2,m,I} \subset \mathscr{E}$ is given as follows: we consider the pullback of the product of normal fans $\prod_{i \notin I}\mathscr{N}_{\check{\Delta}_i(\tau)}$ to $ \norpoly_{\tau,\real}$ and take $\mathscr{E}_{2,m,I} = \langle d\log z^{\hat{m}'} \rangle$ for $m' \in \sigma_{m,I}$, where $\sigma_{m,I}$ is the smallest cone in $\prod_{i \notin I}\mathscr{N}_{\check{\Delta}_i(\tau)} \subset \norpoly_{\tau,\real}$ containing $m$.
\end{lemma} 

Here we can treat $\prod_{i\notin I} \mathscr{N}_{\check{\Delta}_i(\tau)}\subset \norpoly_{\tau,\real}$ since $\bigoplus_{i} \tanpoly_{\check{\Delta}_i(\tau)}$ is a direct summand of $\norpoly_{\tau}^*$. A similar description for $j_*(\Omega^*_{\localmod{k}^{\dagger}/\comp^{\dagger}})_K$ is simply given by quotienting out the direct summand $\langle d\log q\rangle$ in the above formula for $\alpha$. In particular, if we restrict ourselves to the case $k=0$, a general element $\alpha$ can be written as
$$\displaystyle
\alpha = \sum_{m \in \Sigma_{\tau}} \sum_{I} \alpha_{m,I}  z^{\hat{m}} u^I \otimes \beta_{m,I}, \quad \beta_{m,I} \in \bigwedge\nolimits^* \mathscr{E}_{m,I} = \bigwedge\nolimits^*(\mathscr{E}_{1,m,I}\oplus \mathscr{E}_{2,m,I}).
$$
One can choose a nowhere vanishing element 
$$\Omega= du_1\cdots du_l \otimes \eta  \in u_1\cdots u_l \otimes \wedge^l \mathscr{E}_1 \otimes \wedge^{n-\dim_{\real}(\tau)} \mathscr{E}_2 \subset j_*(\Omega^n_{\localmod{0}^{\dagger}/\comp^{\dagger}})_K$$
for some nonzero element $\eta \in \wedge^{n-\dim_{\real}(\tau)} \mathscr{E}_2$, which is well defined up to rescaling. Any element in $j_*(\Omega^n_{\localmod{0}^{\dagger}/\comp^{\dagger}})_K$ can be written as $f \Omega$ for some $f =\sum_{m \in \Sigma_{\tau}} f_{m}z^{\hat{m}} \in \cu{O}_{\localmod{0},K}$. 

Recall that the subset $K \subset \comp^{l}$ is intersecting the singular locus $Z^{\tau}_1,\dots,Z^{\tau}_{r}$ (as in \S \ref{subsubsec:local_deformation_data}), where $u_i$ is the coordinate function of $\comp^{l}$ with simple zeros along $Z_i^{\tau}$ for $i=1,\dots,r$. There is a change of coordinates between a neighborhood of $K$ in $\comp^{l}$ and that of $K$ in $(\comp^*)^l$ given by 
\begin{equation*}
\begin{cases}
u_i  \mapsto  f_{v,i}|_{(\comp^{*})^l} & \text{if } 1\leq i \leq r;\\
u_i  \mapsto  z_i &\text{if } r<i\leq l.
\end{cases}
\end{equation*}
Under the map $\log \colon (\comp^*)^l \rightarrow \real^l$, we have $K = \log^{-1}(\mathscr{C})$ for some connected compact subset $\mathscr{C}\subset \real^l$. In the coordinates $z_1,\dots,z_l$, we find that $d\log z_1 \cdots d\log z_l\otimes \eta$ can be written as $f \Omega$ near $K$ for some nowhere vanishing function $f \in \cu{O}_{\localmod{0},K}$. 

\begin{lemma}\label{lem:local_computation_for_top_cohomology}
	When $K \cap Z = \emptyset$ (i.e. $r=0$ in the above discussion), the top cohomology group $\cu{H}^n(j_*(\Omega^n_{\localmod{0}^{\dagger}/\comp^{\dagger}})_K,\dpartial{}) :=j_*(\Omega^n_{\localmod{0}^{\dagger}/\comp^{\dagger}})_K/\mathrm{Im}(\dpartial{})$ is isomorphic to $\comp$, which is generated by the element $d\log z_1 \cdots d\log z_l\otimes \eta$. 
\end{lemma}

\begin{proof}
	Given a general element $f \Omega$ as above, first observe that we can write $f = f_0 + f_{+}$, where $f_+ = \sum_{m \in \Sigma_{\tau}\setminus \{0\}} f_m z^{\hat{m}}$ and $f_0 \in \cu{O}_{\comp^l,K}$. We take a basis $e_1,\dots,e_s$ of $\norpoly_{\tau,\real}^*$, and also a partition $I_1,\dots,I_s$ of the lattice points in $\Sigma_{\tau}\setminus \{0\}$ such that $\langle e_j,m\rangle \neq 0$ for $m \in I_j$. Letting
	$$
	\alpha =(-1)^l \sum_{j} \sum_{m \in I_j} \frac{ f_m}{\langle e_j,m\rangle} z^{\hat{m}} du_1 \cdots du_l \otimes \iota_{e_j} \eta,
	$$
	we have $\dpartial{}(\alpha) = f_+ \Omega$. So we only need to consider elements of the form $f_0 \Omega$. If $f_0 \Omega = \dpartial{}(\alpha)$ for some $\alpha$, we may take $\alpha = \sum_j \alpha_j du_1\cdots \widehat{du_j} \cdots du_l \otimes \eta$ for some $\alpha_j \in \cu{O}_{\comp^{l},K}$. Now this is equivalent to $f_0 du_1 \cdots du_l = \dpartial{} \big(\sum_j \alpha_j du_1\cdots \widehat{du_j} \cdots du_l \big) $ as forms in $\Omega^l_{\comp^l,K}$. This reduces the problem to $\comp^l$.

	Working in $(\comp^*)^l$ with coordinates $z_i$'s, we can write 
	$$
	\cu{O}_{(\comp^*)^l,K} = \left\{ \sum_{m\in \inte^l} a_m z^{m} \ \Big| \ \sum_{m \in \inte^l} |a_m| e^{\langle v, m \rangle } < \infty, \ \text{for all $v \in W$, for some open $W \supset \mathscr{C}$} \right\},
	$$
	using the fact that $K$ is multi-circular. By writing $\Omega^*_{(\comp^*)^l,K} = \cu{O}_{(\comp^*)^l,K} \otimes \bigwedge^* \mathscr{F}_1$ with $\mathscr{F}_1 = \langle d\log z_i \rangle_{i=1}^l$, we can see that any element can be represented as $c d\log z_1 \cdots d\log z_l$ in the quotient $\Omega^l_{(\comp^*)^l,K}/\mathrm{Im}(\dpartial{})$, for some constant $c$. 
\end{proof}

From this lemma, we conclude that the top cohomology sheaf $\cu{H}^n(\tbva{0}{\parallel}^*,\dpartial{})$ is isomorphic to the locally constant sheaf $\underline{\comp}$ over $B \setminus \tsing_e$. 

\begin{lemma}\label{lem:holomorphic_volume_form_non_vanishing_in_cohomology}
	The volume element $\volf{0}$ is non-zero in $\cu{H}^n(\tbva{0}{\parallel}^*,\dpartial{})_x$ for every $x \in B$. 
\end{lemma}

\begin{proof}
	We first consider the case when $x \in \reint(\sigma)$ for some maximal cell $\sigma \in \pdecomp^{[n]}$. The toric stratum $\centerfiber{0}_{\sigma}$ associated to $\sigma$ is equipped with the natural divisorial log structure induced from its boundary divisor. Then the sheaf $\Omega^*_{\centerfiber{0}_{\sigma}^{\dagger}/\comp^{\dagger}}$ of log derivations for $\centerfiber{0}^{\dagger}$ is isomorphic to $\bigwedge^n \tanpoly_{\sigma} \otimes_{\inte} \cu{O}_{\centerfiber{0}_{\sigma}}$.
	By \cite[Lem. 3.12]{Gross-Siebert-logII}, we have $\volf{0}_{x} = c (\mu_{\sigma})_{\modmap^{-1}(x)}$ in $\modmap_*(\Omega^n_{\centerfiber{0}_{\sigma}^{\dagger}/\comp^{\dagger}})_x \cong \tbva{0}{\parallel}^n_x$, where $\mu_{\sigma} \in \bigwedge^n \tanpoly_{\sigma,\comp}$ is nowhere vanishing and $c$ is a non-zero constant $c$. Thus $\centerfiber{0}|_x$ is non-zero in the cohomology as the same is true for $\mu_{\sigma}\in\modmap_*(\Omega^n_{\centerfiber{0}_{\sigma}^{\dagger}/\comp^{\dagger}})_x$.
	Next we consider a general point $x \in \reint(\tau)$. If the statement is not true, we will have $\volf{0}_x = \dpartial{0} (\alpha)$ for some $\alpha \in \tbva{0}{\parallel}^{n-1}_x$. Then there is an open neighborhood $U \supset \mathscr{C}^{-1}(x)$ such that this relation continues to hold. As $U \cap \reint(\sigma) \neq \emptyset$, for those maximal cells $\sigma$ which contain the point $x$, we can take a nearby point $y \in U \cap \reint(\sigma)$ and conclude that $c \mu_{\sigma} = \dpartial{0}(\alpha) $ in $\modmap_*(\Omega^n_{\centerfiber{0}_{\sigma}^{\dagger}/\comp^{\dagger}})_y$. This contradicts the previous case.
\end{proof}

\begin{lemma}\label{lem:preserving_volume_element_by_vector_fields}
	Suppose that $x \in W_{\alpha} \setminus \tsing_e$. For an element of the form 
    $$e^{f } (\volf{k}_{\alpha}) \in \tbva{k}{\parallel}^n_{\alpha,x}$$ 
    with $f \in \bva{k}^{0}_{\alpha,x} \cong \cu{O}_{\localmod{k}_{\alpha},x}$ satisfying $f \equiv 0 \text{(mod $\mathbf{m}$)}$, there exist $h(q) \in \cfrk{k}=\comp[q]/(q^{k+1})$ and $v \in \bva{k}^{-1}_{\alpha,x}$ with $h,v \equiv 0 \text{(mod $\mathbf{m}$)}$ such that
	\begin{equation}\label{eqn:preserving_volume_element_by_vector_fields}
	e^{f} (\volf{k}_{\alpha}) = e^{h} e^{\cu{L}_{v}} (\volf{k}_{\alpha})
	\end{equation}
	in $\tbva{k}{\parallel}^n_{\alpha,x}$, where we recall that $\cu{L}_{v} := (-1)^{|v|} \partial \circ (v\mathbin{\lrcorner}) - (v\mathbin{\lrcorner}) \circ \partial$.
\end{lemma}

\begin{proof}
	To simplify notations in this proof, we will drop the subscript $\alpha$. 
	We prove the first statement by induction on $k$. The initial case is trivial. Assuming that this has been done for the $(k-1)^{\text{st}}$-order, then, by taking an arbitrary lifting $\tilde{v}$ of $v$ to the $k^{\text{th}}$-order, we have 
	$$
	e^{-h + f +q^{k}\epsilon}(\volf{k}) = e^{\cu{L}_{\tilde{v}}} (\volf{k})
	$$
	for some $\epsilon \in \cu{O}_{\localmod{0}_{x}}$. By Lemmas \ref{lem:local_computation_for_top_cohomology} and \ref{lem:holomorphic_volume_form_non_vanishing_in_cohomology}, we have $\epsilon \volf{0}= c \volf{0} + \dpartial{}(\gamma)$ for some $\gamma$ and some suitable constant $c$. Letting $\theta \mathbin{\lrcorner} (\volf{0}) = \gamma$ and $\breve{v} = \tilde{v} + q^{k} \theta$, we have 
	$$e^{\cu{L}_{\breve{v}}} (\volf{k}) = e^{\cu{L}_{v}} (\volf{k}) - q^k \dpartial{}( \theta \mathbin{\lrcorner} (\volf{0})) = e^{-h+f + c q^k } (\volf{k}).
	$$
	By defining $\tilde{h}(q) := h(q) - cq^k$ in $\comp[q]/(q^{k+1})$, we obtain the desired expression.
\end{proof}

\subsection{A global pre-dgBV algebra from gluing}\label{subsubsec:gluing_construction}
%Since we have a global log Calabi-Yau space $\centerfiber{0}$, the sheaves $\bva{0}_{\alpha}$'s and $\tbva{0}{}_{\alpha}$'s and the relative holomorphic volume forms $\volf{0}_{\alpha}$'s glue together to give the $0^{\text{th}}$-order data $\bva{0}^*$, $\tbva{0}{}^*$ and $\volf{0} \in \Gamma(\centerfiber{0}, \tbva{0}{\parallel}^n)$.

One approach for smoothing $\centerfiber{0}$ is to look for gluing morphisms $\patch{k}_{\alpha\beta}\colon \localmod{k}_{\alpha}^{\dagger}|_{V_{\alpha\beta}} \rightarrow \localmod{k}_{\beta}^{\dagger}|_{V_{\alpha\beta}}$ between the local smoothing models which satisfy the cocycle condition, from which one obtain a $k^{\text{th}}$-order thickening $\centerfiber{k}$ over $\logsk{k}$.
This was done by Kontsevich--Soibelman \cite{kontsevich-soibelman04} (in 2d) and Gross--Siebert \cite{gross2011real} (in general dimensions) using {\em consistent scattering diagrams}.
If such gluing morphisms $\patch{k}_{\alpha\beta}$'s are available, one can certainly glue the global $k^{\text{th}}$-order sheaves $\bva{k}^*$, $\tbva{k}{}^*$ and the volume form $\volf{k}$. 

In \cite{chan2019geometry}, we instead took suitable dg-resolutions $\polyv{k}^{*,*}_{\alpha}:=\mdga^*(\bva{k}_{\alpha}^*)$'s of the sheaves $\bva{k}_{\alpha}^*$'s (more precisely, we used the Thom--Whitney resolution in \cite[\S 3]{chan2019geometry}) to construct gluings 
$$\glue{k}_{\alpha\beta} \colon \mdga^*(\bva{k}_{\alpha}^*)|_{V_{\alpha\beta}} \rightarrow \mdga^*(\bva{k}_{\beta}^*)|_{V_{\alpha\beta}}$$
of sheaves which only preserve the \emph{Gerstenhaber algebra} structure but not the differential. The key discovery in \cite{chan2019geometry} was that, as the sheaves $\mdga^*(\bva{k}_{\alpha}^*)$'s are soft, such a gluing problem could be solved \emph{without} any information from the complicated scattering diagrams. What we obtained is a \emph{pre-dgBV algebra}\footnote{This was originally called an \emph{almost dgBV algebra} in \cite{chan2019geometry}, but we later found the name \emph{pre-dgBV algebra} from \cite{felten2020log} more appropriate.} $\polyv{k}^{*,*}(\centerfiber{})$, in which the differential squares to zero only modulo $\mathbf{m} = (q)$. Using well-known algebraic techniques \cite{terilla2008smoothness, KKP08}, we can solve the {\em Maurer--Cartan equation} and construct the thickening $\centerfiber{k}$. In this subsection, we will summarize the whole procedure, incorporating the nice reformulation by Felten \cite{felten2020log} in terms of deformations of Gerstenhaber algebras.

To begin with, we assume the following condition holds:
\begin{condition}\label{cond:requirement_of_the_de_rham_dga}
	There is a sheaf $(\mdga^*,\md)$ of unital differential graded algebras (abbrev.\ as dga) (over $\real$ or $\comp$) over $B$, with degrees $0\leq * \leq L$ for some $L$, such that 
	\begin{itemize}
		\item the natural inclusion $\underline{\real} \rightarrow \mdga^*$ (or $\underline{\comp} \rightarrow \mdga^*$) of the locally constant sheaf (concentrated at degree $0$) gives a resolution, and
		
		\item for any open cover $\cu{U} = \{ U_i \}_{i \in \cu{I}}$, there is a partition of unity subordinate to $\cu{U}$, i.e. we have $\{ \rho_i\}_{i\in \cu{I}}$ with $\rho_i \in \Gamma(U_i,\mdga^0)$ and $\overline{\mathrm{supp}(\rho_i)} \subset U_i$ such that $\{\overline{\mathrm{supp}(\rho_i)} \}_i$ is locally finite and $\sum_i \rho_i \equiv 1$. 
	\end{itemize}
\end{condition}

%For instant, using the projectivity of $\centerfiber{0}$ we have an embedding $\iota \colon \centerfiber{0} \hookrightarrow \mathbb{P}^{N}$ as closed subvariety. One can simply take the inverse image of the sheaf of de Rham forms $(\iota^{-1}(\mdga^*_{\mathbb{P}^N}), \md)$, or the sheaf of Dolbeaut forms $(\iota^{-1}(\mdga^{0,*}_{\mathbb{P}^{N}}),\pdb)$ on $\mathbb{P}^N$. Alternatively, if there is a smooth manifold structure on $B$, then we can take the usual de Rham forms on $B$ which will give $\mdga^*$ with the degree bound $0 \leq 0 \leq n$. 

It is easy to construct such an $\mdga^*$ and there are many natural choices. For instance, if $B$ is a smooth manifold, then we can simply take the usual de Rham complex on $B$. 
In \S \ref{subsec:derham_for_B}, the sheaf of monodromy invariant differential forms we constructed using the (singular) integral affine structure on $B$ is another possible choice for $\mdga^*$ (with degrees $0 \leq * \leq n$). Yet another variant, namely the sheaf of \emph{monodromy invariant tropical differential forms}, will be constructed in \S \ref{sec:asymptotic_support}; this links tropical geometry on $B$ with the smoothing of the maximally degenerate Calabi--Yau variety $\centerfiber{0}$.

Let us recall how to obtain a gluing of the dg resolutions of the sheaves $\bva{k}_{\alpha}^*$ and $\tbva{k}{}_{\alpha}^*$ using any possible choice of such an $\mdga^*$. We fix a good cover $\cu{W} := \{W_{\alpha}\}_{\alpha}$ of $B$ and the corresponding Stein open cover $\cu{V} := \{V_\alpha\}_\alpha$ of $\centerfiber{0}$, where $V_{\alpha} = \modmap^{-1}(W_{\alpha})$ for each $\alpha$.

\begin{definition}\label{def:local_dgBV_from_resolution}
	We define $\polyv{k}_{\alpha}^{p,q}= \mdga^q(\bva{k}_{\alpha}^p):= \mdga^q|_{W_{\alpha}} \otimes_{\real} \bva{k}^p_{\alpha}$ and $\polyv{k}_{\alpha}^{*,*} = \bigoplus_{p,q} \polyv{k}_{\alpha}^{p,q}$, which gives a sheaf of dgBV algebras over $W_{\alpha}$. The dgBV structure $(\wedge,\pdb_{\alpha}, \bvd{}_{\alpha})$ is defined componentwise by
	\begin{align*}
	(\varphi \otimes v) \wedge ( \psi \otimes w) & := (-1)^{|v||\psi|} (\varphi \wedge \psi) \otimes (v \wedge w),\\
	\pdb_{\alpha} (\varphi \otimes v) := (\md\varphi) \otimes v ,&\quad 
	\bvd{}_{\alpha}(\varphi \otimes v) := (-1)^{|\varphi|} \varphi \otimes (\bvd{} v) ,
	\end{align*}
	for $\varphi, \psi \in \mdga^*(U)$ and $v, w \in \bva{k}_{\alpha}^*(U)$ for each open subset $U \subset W_{\alpha}$.
\end{definition}

\begin{definition}\label{def:local_dga_from_resolution}
	We define $\totaldr{k}{}_{\alpha}^{p,q}= \mdga^q(\tbva{k}{}_{\alpha}^p):= \mdga^q|_{W_{\alpha}} \otimes_{\real} \tbva{k}{}^p_{\alpha}$ and $\totaldr{k}{}_{\alpha}^{*,*} = \bigoplus_{p,q} \totaldr{k}{}_{\alpha}^{p,q}$, which gives a sheaf of dgas over $W_{\alpha}$ equipped with the natural filtration $\totaldr{k}{\bullet}_{\alpha}^{*,*}$ inherited from $\tbva{k}{\bullet}^*_{\alpha}$. The structures $(\wedge,\pdb_{\alpha},\dpartial{}_{\alpha} )$ are defined componentwise by 
	\begin{align*}
	(\varphi \otimes v) \wedge ( \psi \otimes w) & := (-1)^{|v||\psi|} (\varphi \wedge \psi) \otimes (v \wedge w),\\
	\pdb_{\alpha} (\varphi \otimes v) := (\md\varphi) \otimes v ,&\quad  \dpartial{}_{\alpha}(\varphi \otimes v) = (-1)^{|\varphi|} \varphi \otimes (\dpartial{}v),
	\end{align*}
	for $\varphi, \psi \in \mdga^*(U)$ and $v, w \in \tbva{k}{}_{\alpha}^*(U)$ for each open subset $U \subset W_{\alpha}$.
\end{definition}

There is an action of $\polyv{k}_{\alpha}^{*,*}$ on $\totaldr{k}{}_{\alpha}^{*,*}$ by contraction $\mathbin{\lrcorner}$ defined by the formula
$$
(\varphi \otimes v) \mathbin{\lrcorner} (\psi \otimes w):= (-1)^{|v||\psi|} (\varphi \wedge \psi) \otimes (v\mathbin{\lrcorner} w),
$$
for $\varphi, \psi \in \mdga^*(U)$, $v\in \bva{k}_{\alpha}^*(U)$ and $ w \in \tbva{k}{}_{\alpha}^*(U)$ for each open subset $U \subset W_{\alpha}$.
Note that the local holomorphic volume form $\volf{k}_{\alpha} \in \totaldr{k}{\parallel}_{\alpha}^{n,0}(W_{\alpha})$ satisfies $\pdb_{\alpha}(\volf{k}_{\alpha}) = 0 $, and we have the identity $\dpartial{k}_{\alpha} (\phi \mathbin{\lrcorner} \volf{k}_{\alpha}) = \bvd{k}_{\alpha}(\phi) \mathbin{\lrcorner} \volf{k}_{\alpha}$ of operators.

The next step is to consider gluing of the local sheaves $\polyv{k}_{\alpha}$'s for higher orders $k$. Similar constructions have been done in \cite{chan2019geometry,felten2020log} using the combinatorial Thom--Whitney resolution for the sheaves $\bva{k}_{\alpha}$'s. We make suitable modifications of those arguments to fit into our current setting. 

First, since $\localmod{k}^{\dagger}_{\alpha}|_{V_{\alpha\beta}}$ and $\localmod{k}^{\dagger}_{\beta}|_{V_{\alpha\beta}}$ are divisorial deformations (in the sense of \cite[Def. 2.7]{Gross-Siebert-logII}) of the intersection $V^{\dagger}_{\alpha\beta}:= V^{\dagger}_{\alpha} \cap V^{\dagger}_{\beta}$, we can use \cite[Thm. 2.11]{Gross-Siebert-logII} and the fact that $V_{\alpha\beta}$ is Stein to obtain an isomorphism $\patch{k}_{\alpha\beta} \colon \localmod{k}^{\dagger}_{\alpha}|_{V_{\alpha\beta}} \rightarrow \localmod{k}^{\dagger}_{\beta}|_{V_{\alpha\beta}}$ of divisorial deformations which induces the gluing morphism $\patch{k}_{\alpha\beta} \colon \bva{k}_{\alpha}^*|_{W_{\alpha\beta}} \rightarrow \bva{k}_{\beta}^*|_{W_{\alpha\beta}}$ that in turn gives $\patch{k}_{\alpha\beta} \colon \polyv{k}_{\alpha}|_{W_{\alpha\beta}} \rightarrow \polyv{k}_{\beta}|_{W_{\alpha\beta}}$.

\begin{definition}\label{def:deformation_of_Gerstenhaber_algebra}
	A \emph{$k^{\text{th}}$-order Gerstenhaber deformation} of $\polyv{0}$ is a collection of gluing morphisms $\glue{k}_{\alpha\beta} \colon \polyv{k}_{\alpha}|_{W_{\alpha\beta}} \rightarrow \polyv{k}_{\beta}|_{W_{\alpha\beta}}$ of the form 
    $$\glue{k}_{\alpha\beta} = e^{[\vartheta_{\alpha\beta},\cdot]} \circ \patch{k}_{\alpha\beta}$$ 
    for some $\theta_{\alpha\beta} \in \polyv{k}_{\beta}^{-1,0}(W_{\alpha\beta})$ with $\theta_{\alpha\beta} \equiv 0  \ (\text{mod $\mathbf{m}$})$, such that the cocycle condition 
    $$\glue{k}_{\gamma\alpha} \circ \glue{k}_{\beta\gamma} \circ \glue{k}_{\alpha\beta} = \mathrm{id}$$
    is satisfied. 
	
	An \emph{isomorphism between two $k^{\text{th}}$-order Gerstenhaber deformations} $\{\glue{k}_{\alpha\beta}\}_{\alpha\beta}$ and $\{\glue{k}_{\alpha\beta}'\}_{\alpha\beta}$ is a collection of automorphisms $\prescript{k}{}{h}_{\alpha} \colon \polyv{k}_{\alpha} \rightarrow \polyv{k}_{\alpha}$ of the form 
    $$\prescript{k}{}{h}_{\alpha} = e^{[\mathbf{b}_{\alpha},\cdot]}$$
    for some $\mathtt{b}_{\alpha} \in \polyv{k}_{\alpha}^{-1,0}(W_{\alpha})$ with $\mathtt{b}_{\alpha} \equiv 0 (\text{mod $\mathbf{m}$})$, such that $$\glue{k}_{\alpha\beta}'\circ \prescript{k}{}{h}_{\alpha} = \prescript{k}{}{h}_{\beta} \circ \glue{k}_{\alpha\beta}.$$
\end{definition}

A slight modification of \cite[Lem. 6.6]{felten2020log}, with essentially the same proof, gives the following:
\begin{prop}\label{prop:classification_of_gerstenhaber_deformation}
	Given a $k^{\text{th}}$-order Gerstenhaber deformation $\{\glue{k}_{\alpha\beta}\}_{\alpha\beta}$, the obstruction to the existence of a lifting to a $(k+1)^{\text{st}}$-order deformation $\{\glue{k+1}_{\alpha\beta}\}_{\alpha\beta}$ lies in the \v{C}ech cohomology (with respect to the cover $\cu{W} = \{W_{\alpha}\}_{\alpha}$)
	$$
	\check{H}^{2}(\cu{W},\polyv{0}^{-1,0}) \otimes_{\mathbb{C}} (\mathbf{m}^{k+1}/\mathbf{m}^{k+2}). 
	$$
	The isomorphism classes of $(k+1)^{\text{st}}$-order liftings are in
	$$
	\check{H}^{1}(\cu{W},\polyv{0}^{-1,0}) \otimes_{\mathbb{C}} (\mathbf{m}^{k+1}/\mathbf{m}^{k+2}). 
	$$
	Fixing a $(k+1)^{\text{st}}$-order lifting $\{\glue{k+1}_{\alpha\beta}\}_{\alpha\beta}$, the automorphisms fixing $\{\glue{k}_{\alpha\beta}\}_{\alpha\beta}$ are in
	$$
	\check{H}^{0}(\cu{W},\polyv{0}^{-1,0}) \otimes_{\mathbb{C}} (\mathbf{m}^{k+1}/\mathbf{m}^{k+2}). 
	$$
\end{prop}

Since $\mdga^i$ satisfies Condition \ref{cond:requirement_of_the_de_rham_dga}, we have $\check{H}^{>0}(\cu{W},\polyv{0}^{-1,0}) = 0$. In particular, we always have a set of compatible Gerstenhaber deformations $\glue{} = (\glue{k})_{k \in \mathbb{N}}$ where $\glue{k} = \{\glue{k}_{\alpha\beta} \}_{\alpha\beta}$ and any two of them are equivalent. Fixing such a set $\glue{}$, we obtain a set $\{\polyv{k}\}_{k \in \mathbb{N}}$ of Gerstenhaber algebras which is compatible, in the sense that there are natural identifications $\polyv{k+1} \otimes_{\cfrk{k+1}} \cfrk{k} = \polyv{k}$.  

We can also glue the local sheaves $\totaldr{k}{}^{*,*}_{\alpha}$'s of dgas using $\glue{} = (\glue{k})_{k \in \mathbb{N}}$. First, we can define $\patch{k}_{\alpha\beta} \colon \tbva{k}{}_{\alpha}^*|_{W_{\alpha\beta}} \rightarrow \tbva{k}{}_{\beta}^*|_{W_{\alpha\beta}}$ using $\patch{k}_{\alpha\beta} \colon \localmod{k}^{\dagger}_{\alpha}|_{V_{\alpha\beta}} \rightarrow \localmod{k}^{\dagger}_{\beta}|_{V_{\alpha\beta}}$. For each fixed $k$, we can write $\glue{k}_{\alpha\beta} = e^{[\vartheta_{\alpha\beta},\cdot]} \circ  \patch{k}_{\alpha\beta}$ as before. Then
\begin{equation}\label{eqn:gluing_formula_for_local_de_rham_dga}
	\glue{k}:=e^{\cu{L}_{\vartheta_{\alpha\beta}}} \circ \patch{k}_{\alpha\beta} \colon \totaldr{k}{}_{\alpha}^{*,*}|_{W_{\alpha\beta}}  \rightarrow \totaldr{k}{}_{\beta}^{*,*}|_{W_{\alpha\beta}},
\end{equation}
where we recall that $\cu{L}_{v} := (-1)^{|v|} \partial \circ (v\mathbin{\lrcorner}) - (v\mathbin{\lrcorner}) \circ \partial$, preserves the dga structure $(\wedge,\dpartial{}_{\alpha} )$ and the filtration on $\totaldr{k}{\bullet}_{\alpha}^{*,*}$'s. As a result, we obtain a set of compatible sheaves $\{(\totaldr{k}{}^{*,*}, \wedge,\dpartial{})\}_{k \in \mathbb{N}}$ of dgas. The contraction action $\mathbin{\lrcorner}$ is also compatible with the gluing construction, so we have a natural action $\mathbin{\lrcorner}$ of $\polyv{k}^{*,*}$ on $\totaldr{k}{}^{*,*}$.

Next, we glue the operators $\pdb_{\alpha}$'s and $\bvd{}_{\alpha}$'s.	
\begin{definition}\label{def:predifferential}
	A \emph{$k^{\text{th}}$-order pre-differential} $\pdb$ on $\polyv{k}^{*,*}$ is a degree $(0,1)$ operator obtained from gluing the operators $\pdb_{\alpha}+[\eta_{\alpha},\cdot]$ specified by a collection of elements $\eta_\alpha \in \polyv{k}^{-1,1}_{\alpha}(W_{\alpha})$ such that $\eta_{\alpha} \equiv 0 \ (\text{mod $\mathbf{m}$})$ and
	$$
	\glue{k}_{\beta\alpha }\circ (\pdb_{\beta} + [\eta_{\beta},\cdot]) \circ \glue{k}_{\alpha\beta} =  (\pdb_{\alpha} + [\eta_{\alpha},\cdot]).
	$$
	Two pre-differentials $\pdb$ and $\pdb'$ are \emph{equivalent} if there is a Gerstenhaber automorphism (for the deformation $\glue{k}$) $h \colon \polyv{k}^{*,*} \rightarrow \polyv{k}^{*,*}$ such that $h^{-1}\circ \pdb \circ h = \pdb'$. 
\end{definition}
	
Notice that we only have $\pdb^2 \equiv  0$ $(\text{mod $\mathbf{m}$})$, which is why we call it a pre-differential. Using the argument in \cite[Thm. 3.34]{chan2019geometry} or \cite[Lem. 8.1]{felten2020log}, we can always lift any $k^{\text{th}}$-order pre-differential $\prescript{k}{}{\pdb}$ to a $(k+1)^{\text{st}}$-order pre-differential. Furthermore, any two such liftings differ by a global element $\mathfrak{d} \in \polyv{0}^{-1,1} \otimes \mathbf{m}^{k+1}/\mathbf{m}^{k+2}$. We fix a set $\pdb := \{\prescript{k}{}{\pdb}\}_{k \in \mathbb{N}}$ of such compatible pre-differentials. 
For each $k$, the action of $\prescript{k}{}{\pdb}$ on $\totaldr{k}{}^{*,*}$ is given by gluing of the action of $\pdb_{\alpha} + \cu{L}_{\eta_{\alpha}}$ on $\totaldr{k}{}^{*,*}_{\alpha}$. 
On the other hand, the elements
\begin{equation}\label{eqn:initial_obstruction_for_pdb_square_to_zero}
	\mathfrak{l}_{\alpha}:= \pdb_{\alpha}(\eta_{\alpha}) + \half [\eta_{\alpha},\eta_{\alpha} ] \in \polyv{k}^{-1,2}_{\alpha}(W_{\alpha})
\end{equation}
glue to give a global element $\mathfrak{l} \in \polyv{k}^{-1,2}(B)$, and for different $k$'s, these elements are compatible. Computation shows that $\pdb^{2} = [\mathfrak{l},\cdot]$ on $\polyv{k}^{*,*}$ and $\pdb^{2} = \cu{L}_{\mathfrak{l}}$ on $\totaldr{k}{}^{*,*}$. 
	
To glue the operators $\bvd{}_{\alpha}$'s, we need to glue the local volume elements $\volf{k}_{\alpha}$'s to a global $\volf{k}$. We consider an element of the form $e^{\mathfrak{f}_{\alpha} \mathbin{\lrcorner}} \cdot \volf{k}_{\alpha}$, where $\mathfrak{f}_{\alpha} \in \polyv{k}^{0,0}(W_{\alpha})$ satisfies $\mathfrak{f}_{\alpha} \equiv 0 \ (\text{mod $\mathbf{m}$})$. Given a $k^{\text{th}}$-order global volume element $e^{\mathfrak{f}_{\alpha} \mathbin{\lrcorner}} \cdot \volf{k}_{\alpha}$, we take a lifting $e^{\tilde{\mathfrak{f}}_{\alpha} \mathbin{\lrcorner}}\cdot \volf{k+1}_{\alpha}$ such that
$$
\glue{k+1}_{\alpha\beta} \big( e^{\tilde{\mathfrak{f}}_{\alpha} \mathbin{\lrcorner}} \cdot \volf{k+1}_{\alpha} \big)= e^{
		(\tilde{\mathfrak{f}}_{\beta} - \mathfrak{o}_{\alpha\beta}) \mathbin{\lrcorner}} \cdot \volf{k+1}_{\beta},
$$
for some element $\mathfrak{o}_{\alpha\beta} \in \polyv{0}^{0,0}(W_{\beta}) \otimes \mathbf{m}^{k+1}/\mathbf{m}^{k+2}$. By construction, $\{\mathfrak{o}_{\alpha\beta}\}_{\alpha\beta}$ gives a \v{C}ech $1$-cycle in $\polyv{0}^{0,0}$ which is exact. So there exist $\mathfrak{u}_{\alpha}$'s such that $\mathfrak{u}_{\beta}|_{W_{\alpha\beta}} - \mathfrak{u}_{\alpha}|_{W_{\alpha\beta}} = \mathfrak{o}_{\alpha\beta}$, and we can modify $\tilde{\mathfrak{f}}_{\alpha}$ as $\tilde{\mathfrak{f}}_{\alpha} + \mathfrak{u}_{\alpha}$, which gives the desired $(k+1)^{\text{st}}$-order volume element. Inductively, we can construct compatible volume elements $\volf{k} \in \totaldr{k}{\parallel}^{n,0}(B)$, $k\in \mathbb{N}$. Any two such volume elements $\volf{k}$ and $\volf{k}'$ differ by $\volf{k} = e^{\mathfrak{f} \mathbin{\lrcorner}} \cdot \volf{k}'$, where $\mathfrak{f} \in \polyv{k}^{0,0}(B)$ is some global element. Notice that $\prescript{k}{}{\pdb}(\volf{k}) \neq 0$ unless $\text{mod $\mathbf{m}$}$. 
	
Using the volume element $\volf{}$ (we omit the dependence on $k$ if there is no confusion), we may now define the \emph{global BV operator} $\bvd{}$ by 
\begin{equation}\label{eqn:defining_global_BV_operator}
	(\bvd{} \varphi) \mathbin{\lrcorner} \volf{} = \dpartial{} (\varphi\mathbin{\lrcorner} \volf{}),
\end{equation}
which can locally be written as $\bvd{k}_{\alpha}  + [\mathfrak{f}_{\alpha},\cdot]$. We have $\bvd{}^2 = 0$.
The local elements
\begin{equation}\label{eqn:bv_operator_obstruction}
	\mathfrak{n}_{\alpha}:= \bvd{k}_{\alpha}(\eta_{\alpha}) + \pdb_{\alpha}(\mathfrak{f}_{\alpha}) + [\eta_{\alpha},\mathfrak{f}_{\alpha}]
\end{equation}
glue to give a global element $\mathfrak{n} \in \polyv{k}^{0,1}(B)$ which satisfies $\pdb\bvd{}+\bvd{} \pdb = [\mathfrak{n},\cdot]$. Also, the elements $\mathfrak{l}$ and $\mathfrak{n}$ satisfy the relation $\pdb(\mathfrak{n}) + \bvd{}(\mathfrak{l}) = 0$ by a local calculation.

In summary, we obtain pre-dgBV algebras $(\polyv{k},\pdb,\bvd{},\wedge)$ and pre-dgas $(\totaldr{k}{},\pdb,\dpartial{},\wedge)$ with a natural contraction action $\mathbin{\lrcorner}$ of $\prescript{k}{}{\pdb}$ on $\totaldr{k}{}^{*,*}$, and also volume elements $\volf{}$. We set 
$$\polyv{} := \varprojlim_{k} \polyv{k},\ \totaldr{}{}:= \varprojlim_{k} \totaldr{k}{},$$
and define a \emph{total de Rham operator} $\mathbf{d} \colon \totaldr{}{}^{*,*} \rightarrow \totaldr{}{}^{*,*}$ by
\begin{equation}\label{eqn:total_de_rham_operator}
	\mathbf{d}:= \pdb + \dpartial{} + \mathfrak{l}\mathbin{\lrcorner},
\end{equation}
which preserves the filtration $\totaldr{}{\bullet}^{*,*}$. Using the contraction $\volf{}\mathbin{\lrcorner} \colon \polyv{}^{*,*} \rightarrow \totaldr{}{\parallel}^{*+n,*}$ to pull back the operator, we obtain the operator
$\mathbf{d} = \pdb + \bvd{} + (\mathfrak{l} + \mathfrak{n})\wedge$ acting on $\polyv{}^{*,*}$. Direct computation shows that $\mathbf{d}^2 = 0$, and indeed it plays the role of the de Rham differential on a smooth manifold. Readers may consult \cite[\S 4.2]{chan2019geometry} for the computations and more details. 
	
\begin{definition}\label{def:global_polyvector_and_de_rham}
	We call $\polyv{}^{*,*}$ (resp. $\polyv{k}^{*,*}$) the \emph{sheaf of (resp. $k^{\text{th}}$-order) smooth relative polyvector fields over $\logs$}, and $\totaldr{}{}^{*,*}$ (resp. $\totaldr{k}{}^{*,*}$) the \emph{sheaf of (resp. $k^{\text{th}}$-order) smooth forms over $\logs$}. We denote the corresponding total complexes by $\polyv{}^{*} = \bigoplus_{p+q=*} \polyv{}^{p,q}$ (resp. $\polyv{k}^{*} $) and $\totaldr{}{}^{*} = \bigoplus_{p+q=*} \totaldr{}{}^{p,q}$ (resp. $\totaldr{k}{}^{*} $).
\end{definition}

\subsection{Smoothing by solving the Maurer--Cartan equation}\label{subsubsec:smoothing_via_maurer_cartan}
With the sheaf $\polyv{}^{*}$ of pre-dgBV algebras defined, we can now consider the \emph{extended Maurer--Cartan equation}
\begin{equation}\label{eqn:extended_maurer_cartan_equation}
(\pdb+t\bvd{})\varphi + \half [\varphi,\varphi] + \mathfrak{l} + t \mathfrak{n} = 0
\end{equation}
for $\varphi = (\prescript{k}{}{\varphi})_k$, where $\prescript{k}{}{\varphi} \in \polyv{k}^{0}(B)[[t]] := \polyv{k}^{0}(B)\otimes_{\comp} \comp[[t]]$.
Setting $t = 0$ gives the \emph{(classical) Maurer--Cartan equation} 
\begin{equation}\label{eqn:classical_maurer_cartan_equation}
\pdb \varphi + \half[\varphi,\varphi] + \mathfrak{l} = 0
\end{equation}
for $\varphi \in \polyv{}^0(B)$. To inductively solve these equations, we need two conditions, namely the \emph{holomorphic Poincar\'{e} Lemma} and the \emph{Hodge-to-de Rham degeneracy}. 
   
We begin with the holomorphic Poincar\'{e} Lemma, which is a local condition on the sheaves $j_* (\Omega^*_{\localmod{k}_{\alpha}^{\dagger}/\comp})$'s. We consider the complex $(j_* (\Omega^*_{\localmod{k}_{\alpha}^{\dagger}/\comp})[u], \widetilde{\dpartial{}_{\alpha}})$, where 
$$\widetilde{\dpartial{}_{\alpha}}\left(\sum_{s=0}^l \nu_s u^s\right) := \sum_{s} (\dpartial{}_{\alpha} \nu_s) u^s + s d\log(q) \wedge \nu_s u^{s-1}.$$
There is a natural exact sequence
\begin{equation}\label{eqn:equation_for_holomorphic_poincare_lemma}
\xymatrix@1{
0 \ar[r] & \prescript{k}{}{\bar{\mathfrak{K}}}^*_{\alpha} \ar[r] & j_* (\Omega^*_{\localmod{k}_{\alpha}^{\dagger}/\comp})[u] \ar[r]^{\widetilde{\rest{k,0}}} & j_* (\Omega^*_{\localmod{0}_{\alpha}^{\dagger}/\logsk{0}}) \ar[r] & 0,
}
\end{equation}
where $\widetilde{\rest{k,0}} (\sum_{s=0}^l \nu_s u^s) := \rest{k,0}(\nu_0)$ as elements in $j_* (\Omega^*_{\localmod{0}_{\alpha}^{\dagger}/\logsk{0}}) $.  
\begin{condition}\label{cond:holomorphic_poincare_lemma}
	We say that the {\em holomorphic Poincar\'{e} Lemma} holds if at every point $x \in \centerfiber{0}^{\dagger}$, the complex $(\prescript{k}{}{\bar{\mathfrak{K}}}^*_{\alpha,x},\widetilde{\dpartial{}_{\alpha}})$ is acyclic, where $\prescript{k}{}{\bar{\mathfrak{K}}}^*_{\alpha,x}$ denotes the stalk of $\prescript{k}{}{\bar{\mathfrak{K}}}^*_{\alpha}$ at $x$.
\end{condition}

The holomorphic Poincar\'{e} Lemma for our setting was proved in \cite[proof of Thm. 4.1]{Gross-Siebert-logII}, but a gap was subsequently pointed out by Felten--Filip--Ruddat in \cite{Felten-Filip-Ruddat}, who used a different strategy to close the gap and give a correct proof in \cite[Thm. 1.10]{Felten-Filip-Ruddat}. 
From this condition, we can see that the cohomology sheaf $\cu{H}^*(\tbva{k}{\parallel}^*_{\alpha},\dpartial{k}_{\alpha})$ is free over $\cfrk{k} = \comp[q]/(q^{k+1})$ (cf. \cite[Lem. 4.1]{Kawamata-Namikawa}).
We will need freeness of the cohomology $H^*(\totaldr{k}{\parallel}^*(B),\mathbf{d})$ over $\cfrk{k}$, which can be seen by the following lemma (see \cite{Kawamata-Namikawa} and \cite[\S 4.3.2]{chan2019geometry} for similar arguments). 

\begin{lemma}
    Under Condition \ref{cond:holomorphic_poincare_lemma} (the holomorphic Poincar\'{e} Lemma),
	the natural map 
    $$\rest{k,0}\colon H^*(\totaldr{k}{\parallel}^*(B),\mathbf{d}) \rightarrow H^*(\totaldr{0}{\parallel}^*(B),\mathbf{d})$$
    is surjective for each $k \geq 0$. 
\end{lemma}
	
\begin{proof}
First of all, applying the functor $\nu_*$ to the exact sequence
	$$
	\xymatrix@1{
		0 \ar[r] &  \prescript{k}{}{\bar{\mathfrak{K}}}^*_{\alpha} \ar[r] & j_* (\Omega^*_{\localmod{k}_{\alpha}^{\dagger}/\comp})[u] \ar[r]^{\widetilde{\rest{k,0}}} & j_* (\Omega^*_{\localmod{0}_{\alpha}^{\dagger}/\logsk{0}}) \ar[r] & 0
	}
	$$
gives the following exact sequence of sheaves on $B$:
	$$
	\xymatrix@1{
	0 \ar[r] &  \prescript{k}{}{\mathfrak{K}}^*_{\alpha} \ar[r] & \tbva{k}{}^*_{\alpha}[u] \ar[r]^{\widetilde{\rest{k,0}}} & \tbva{0}{}^*_{\alpha} \ar[r]& 0.  
	}
	$$
This is true because 
every sheaf in the first exact sequence is a direct limit of coherent analytic sheaves, 
$R\nu_{!}$ commutes with direct limits of sheaves, and 
$R\nu_{!} = R\nu_* $ as the fiber $\nu^{-1}(x)$ is a compact Hausdorff topological space; see e.g. \cite{Kashiwara-Schapira94}.
By taking a Cartan--Eilenberg resolution, we have the implication:
	$$
	(\prescript{k}{}{\bar{\mathfrak{K}}}^*_{\alpha,x},\widetilde{\dpartial{k}_{\alpha}})  \text{ is acyclic} \Longrightarrow R\Gamma_U((\prescript{k}{}{\bar{\mathfrak{K}}}^*_{\alpha},\widetilde{\dpartial{k}_{\alpha}})) = 0
	$$
for any open subset $U$, where $R\Gamma_U$ is the derived global section functor in the derived category of sheaves.
In our case, $U = \nu^{-1}(W)$ and we have $R\Gamma_{\nu^{-1}(W)} = R\Gamma_{W} \circ R\nu_*$. Furthermore, we see that 
$$R\nu_* (\prescript{k}{}{\bar{\mathfrak{K}}}^*_{\alpha},\widetilde{\partial_{\alpha}}) = (\prescript{k}{}{\mathfrak{K}}^*_{\alpha},\widetilde{\partial_{\alpha}}).$$
This can be seen by taking a double complex $C^{*,*}$ resolving $(\prescript{k}{}{\bar{\mathfrak{K}}}^*_{\alpha},\widetilde{\partial_{\alpha}})$ such that $\nu_*(C^{*,*})$ computes $R\nu_* (\prescript{k}{}{\bar{\mathfrak{K}}}^*_{\alpha},\widetilde{\partial_{\alpha}})$. 
The spectral sequence associated to the double complex has the $E_1$-page given by $R^q \nu_* (\prescript{k}{}{\bar{\mathfrak{K}}}^p_{\alpha})$, which is $0$ if $q >0$ because $\prescript{k}{}{\bar{\mathfrak{K}}}^p_{\alpha}$ is a direct limit of coherent analytic sheaves. 
Therefore, $\nu_* (\prescript{k}{}{\bar{\mathfrak{K}}}^*_{\alpha},\widetilde{\partial_{\alpha}}) \rightarrow \nu_*(C^{*,*})= R\nu_* (\prescript{k}{}{\bar{\mathfrak{K}}}^*_{\alpha},\widetilde{\partial_{\alpha}})$ is a quasi-isomorphism. Combining these, we obtain that $R\Gamma^i_W (\prescript{k}{}{\mathfrak{K}}^*_{\alpha},\widetilde{\partial_{\alpha}}) = 0$ for each $i$. 

%We use the composition of derived functors, $R\Gamma_{\nu^{-1}(W)} = R\Gamma_{W} \circ R\nu_*$, on the derived category of sheaves over the open subset $\nu^{-1}(W)$. 
%For the specific complexes $(\prescript{k}{}{\bar{\mathfrak{K}}}_{\alpha}^*,\widetilde{\partial_{\alpha}})$, we have . This can be seen by taking a double complex $C^{*,*}$ resolving $(\prescript{k}{}{\bar{\mathfrak{K}}}^*_{\alpha},\widetilde{\partial_{\alpha}})$ such that $\nu_*(C^{*,*})$ computes $R\nu_* (\prescript{k}{}{\bar{\mathfrak{K}}}^*_{\alpha},\widetilde{\partial_{\alpha}})$. 
%The spectral sequence associated to the double complex has the $E_1$-page given by $R^q \nu_* (\prescript{k}{}{\bar{\mathfrak{K}}}^p_{\alpha})$, which is $0$ if $q >0$ because $\prescript{k}{}{\bar{\mathfrak{K}}}^p_{\alpha}$ is a direct limit of coherent analytic sheaves. Therefore $\nu_* (\prescript{k}{}{\bar{\mathfrak{K}}}^*_{\alpha},\widetilde{\partial_{\alpha}}) \rightarrow \nu_*(C^{*,*})= R\nu_* (\prescript{k}{}{\bar{\mathfrak{K}}}^*_{\alpha},\widetilde{\partial_{\alpha}})$ is a quasi-isomorphism. We conclude that $R\Gamma^i_W (\prescript{k}{}{\mathfrak{K}}^*_{\alpha},\widetilde{\partial_{\alpha}}) = 0$ for each $i$. 

Next, by Condition \ref{cond:requirement_of_the_de_rham_dga}, $(\Omega^*|_{W_{\alpha}} \otimes_{\real}  \prescript{k}{}{\mathfrak{K}}^*_{\alpha})$ is a resolution with a partition of unity, so the cohomology of the complex 
	$$\left(\prescript{k}{}{\mathcal{B}}^{*}_{\alpha}(W), \pdb_{\alpha}+\widetilde{\dpartial{}_{\alpha}} \right) :=\left((\Omega^*|_{W_{\alpha}} \otimes_{\real}  \prescript{k}{}{\mathfrak{K}}^*_{\alpha}) (W), \pdb_{\alpha}+\widetilde{\dpartial{}_{\alpha}} \right) $$
computes $R\Gamma_{W} (\prescript{k}{}{\mathfrak{K}}^*_{\alpha})$. Through an isomorphism $e^{\eta_{\alpha}\mathbin{\lrcorner}} \colon \prescript{k}{}{\mathcal{B}}^{*}_{\alpha} \rightarrow \prescript{k}{}{\mathcal{B}}^{*}_{\alpha}$, we can identify the operator:
	$$
	\mathbf{d}_{\alpha}:=\pdb_{\alpha} + \cu{L}_{\eta_{\alpha}} + \widetilde{\dpartial{}_{\alpha}} + \iota_{\pdb_{\alpha} (\eta_{\alpha}) + \frac{1}{2} [ \eta_{\alpha},\eta_{\alpha}]} 
	$$ 
with $\pdb_{\alpha}+\widetilde{\dpartial{}_{\alpha}}$, and hence deduce that $(\prescript{k}{}{\mathcal{B}}_{\alpha}^{*}(W),\mathbf{d}_{\alpha})$ is acyclic for any open subset $W$.

Now, we consider the global sheaf $(\prescript{k}{}{\mathcal{B}}^{*},\mathbf{d})$ of complexes on $B$ obtained by gluing the local sheaves $(\prescript{k}{}{\mathcal{B}}_{\alpha}^{*},\mathbf{d}_{\alpha})$.
We also have $(\widetilde{\prescript{k}{}{\mathcal{A}}^{*}},\mathbf{d})$ obtained by gluing $(\Omega^*|_{W_{\alpha}} \otimes \prescript{k}{}{\mathcal{K}}^*_{\alpha}[u],\mathbf{d}_{\alpha})$, and $(\prescript{0}{\parallel}{\mathcal{A}}^{*},\mathbf{d})$ obtained by gluing $(\Omega^*|_{W_{\alpha}} \otimes \prescript{0}{\parallel}{\mathcal{K}}^*_{\alpha},\mathbf{d}_{\alpha})$. 
Then there is an exact sequence of complexes of sheaves 
	$$
	\xymatrix@1{
	0 \ar[r] &  \prescript{k}{}{\mathcal{B}}^{*} \ar[r] & \widetilde{\prescript{k}{}{\mathcal{A}}^{*}} \ar[r] & \prescript{0}{\parallel}{\mathcal{A}}^{*} \ar[r]& 0.  
	}
	$$
To see that the complex $(\prescript{k}{}{\mathcal{B}}^{*}(B),\mathbf{d})$ is acyclic, we consider the total \v{C}ech complex associated to the cover $\{W_{\alpha}\}_{\alpha}$.
%:
%	$$	(\check{\mathcal{C}}(\prescript{k}{}{\mathcal{B}}^*), \mathbf{d} + \delta_{\check{\mathcal{C}}}) \rightarrow (\check{\mathcal{C}}(\prescript{0}{}{\mathcal{B}}^*), \mathbf{d} + \delta_{\check{\mathcal{C}}}).$$
The associated spectral sequence has zero $E_1$ page, thus $(\prescript{k}{}{\mathcal{B}}^{*}(B),\mathbf{d})$ is indeed acyclic.
%In order for the complex to be zero after taking cohomology, we employ a spectral sequence argument. Specifically, we need to demonstrate that the $E_1$-page is zero, which corresponds to $(\prescript{k}{}{\mathcal{B}}_{\alpha}^{*}(W),\mathbf{d}_{\alpha}) = 0$ obtained above, where $W = W_{\alpha_0\cdots\alpha_l} = W_{\alpha_0} \cap \cdots \cap W_{\alpha_l}$.
As a result, the map $H^i(\widetilde{\prescript{k}{}{\mathcal{A}}^{*}_{\alpha}}(B),\mathbf{d}_{\alpha}) \rightarrow H^i(\prescript{0}{\parallel}{\mathcal{A}}^{*}_{\alpha}(B),\mathbf{d}_{\alpha})$ is an isomorphism. Finally, surjectivity of the map $\rest{k,0}$ follows from the fact that the isomorphism $H^i(\widetilde{\prescript{k}{}{\mathcal{A}}^{*}_{\alpha}}(B),\mathbf{d}_{\alpha}) \rightarrow H^i(\prescript{0}{\parallel}{\mathcal{A}}^{*}_{\alpha}(B),\mathbf{d}_{\alpha})$ factors through $\rest{k,0}$.
%thus surjectivity of the map $\rest{k,0}$ follows before the above isomorphism factors through it.
\end{proof}
			
The Hodge-to-de Rham degeneracy is a global Hodge-theoretic condition on $\centerfiber{0}^{\dagger}$. We consider the Hodge filtration $F^{\geq r} j_* (\Omega^*_{\centerfiber{0}^{\dagger}/\logsk{0}}) =\bigoplus_{p\geq r} j_* (\Omega^{p}_{\centerfiber{0}^{\dagger}/\logsk{0}})$; the spectral sequence associated to it computes the hypercohomology of the complex of sheaves $(j_* (\Omega^*_{\centerfiber{0}^{\dagger}/\logsk{0}}),\dpartial{0})$

\begin{condition}\label{cond:Hodge-to-deRham}
	We say that the \emph{Hodge-to-de Rham degeneracy} holds for $\centerfiber{0}^{\dagger}$ if the spectral sequence associated to the above Hodge filtration degenerates at $E_1$.
\end{condition}

Under the assumption that $(B,\pdecomp)$ is strongly simple (Definition \ref{def:strongly simple}), the Hodge-to-de Rham degeneracy for the maximally degenerate Calabi--Yau scheme $\centerfiber{0}^{\dagger}$ was proved in \cite[Thm. 3.26]{Gross-Siebert-logII}. This was later generalized to the case when $(B,\pdecomp)$ is only simple (instead of strongly simple)\footnote{The subtle difference between the log Hodge group and the affine Hodge group when $(B,\pdecomp)$ is just simple, instead of strongly simple, was studied in details by Ruddat in his thesis \cite{Ruddat10}.} and further to log toroidal spaces in Felten--Filip--Ruddat \cite{Felten-Filip-Ruddat} using different methods.

We consider the dgBV algebra $\polyv{0}^*(B)[[t]]$ equipped with the operator $\pdb + t \bvd{}$. 

\begin{lemma}
	Under Condition \ref{cond:Hodge-to-deRham} (the Hodge-to-de Rham degeneracy), $H^*(\polyv{0}^*(B)[[t]],\pdb + t \bvd{})$ is a free $\comp[[t]]$-module. 
\end{lemma}
	
\begin{proof}
    Recall that we are working with a good cover $\cu{W} = \{ W_\alpha\}_{\alpha}$, so that the inverse image $V_{\alpha} = \modmap^{-1}(W_{\alpha})$ is Stein for each $\alpha$. 
    We have $R\Gamma_{\modmap^{-1}(W)} = R\Gamma_{W} \circ R\modmap_*$ and 
    $$R\modmap_* (j_* (\Omega^*_{\centerfiber{0}^{\dagger}/\logsk{0}}),\dpartial{})= (\tbva{0}{\parallel}^*,\dpartial{}).$$ 
    If $\modmap^{-1}(W)$ is Stein, then $R\Gamma_{\modmap^{-1}(W)}(j_* (\Omega^r_{\centerfiber{0}^{\dagger}/\logsk{0}})) = \Gamma_{\modmap^{-1}(W)} (j_* (\Omega^r_{\centerfiber{0}^{\dagger}/\logsk{0}}))$ and hence 
    $$R\Gamma_{W}(\tbva{0}{\parallel}^{r}) = \Gamma_W(\tbva{0}{\parallel}^{r}).$$
    The hypercohomology of $(j_*(\Omega^*_{\centerfiber{0}^{\dagger}/\logsk{0}}),\dpartial{})$ is computed using the \v{C}ech double complex 
    $$\check{\cu{C}}^*(\cu{V},j_* (\Omega^*_{\centerfiber{0}^{\dagger}/\logsk{0}}))$$
    with respect to the Stein open cover $\cu{V} = \{\modmap^{-1}(W_{\alpha})\}_{\alpha}$. 
    Similarly, the hypercohomology of the complex $(\tbva{0}{\parallel}^*,\dpartial{})$ is computed using the \v{C}ech double complex $\check{\cu{C}}^*(\cu{W},\tbva{0}{\parallel}^*)$ with respect to the cover $\cu{W} = \{ W_\alpha\}_{\alpha}$; here, the Hodge filtration is induced from the filtration $F^{\geq r}\tbva{0}{\parallel}^* = \bigoplus_{p\geq r} \tbva{0}{\parallel}^{\geq p}$. 
    
    These two \v{C}ech complexes, as well as their corresponding Hodge filtrations, are identified as $\tbva{0}{\parallel}^*(W) =  j_* (\Omega^r_{\centerfiber{0}^{\dagger}/\logsk{0}})(\modmap^{-1}(W))$ for each $W = W_{\alpha_1} \cap \cdots \cap W_{\alpha_k}$.
    Hence, under Condition \ref{cond:Hodge-to-deRham}, we have $E_1$ degeneracy also for $\check{\cu{C}}^*(\cu{W},\tbva{0}{\parallel}^*)$, or equivalently, that $(\check{\cu{C}}^*(\cu{W},\tbva{0}{\parallel}^*)[[t]], \delta + t \dpartial{})$ is a free $\comp[[t]]$-module. In view of the isomorphisms $( \bva{0}^*,\bvd{})\cong (\tbva{0}{\parallel},\dpartial{})$ and 
    $$H^*(\polyv{0}^*(B)[[t]],\pdb + t \bvd{}) \cong H^*(\check{\cu{C}}^*(\cu{W},\tbva{0}{\parallel}^*)[[t]], \delta + t \dpartial{}),$$
    we conclude that $H^*(\polyv{0}^*(B)[[t]],\pdb + t \bvd{}) $ is a free $\comp[[t]]$-module as well.
\end{proof}

For the purpose of this paper, we restrict ourselves to the case that 
$$\prescript{k}{}{\varphi} = \prescript{k}{}{\phi} + t (\prescript{k}{}{f}),$$
where $\prescript{k}{}{\phi} \in \polyv{k}^{-1,1}(B)$ and $\prescript{k}{}{f} \in \polyv{k}^{0,0}(B)$. The extended Maurer-Cartan equation \eqref{eqn:extended_maurer_cartan_equation} can be decomposed, according to orders in $t$, into the (classical) Maurer--Cartan equation \eqref{eqn:classical_maurer_cartan_equation} for $\prescript{k}{}{\phi}$ and the equation
\begin{equation}\label{eqn:volume_form_equation}
\pdb(\prescript{k}{}{f})  + [\prescript{k}{}{\phi},\prescript{k}{}{f} ] +  \bvd{}(\prescript{k}{}{\phi}) + \mathfrak{n} = 0.
\end{equation}
%As in classical deformation theory, $\prescript{k}{}{\phi}$ can be interpreted as deforming the complex structure to the $k^{\text{th}}$-order and $e^{\prescript{k}{}{f}} (\volf{k})$ is a holomorphic volume form which comes along. 

\begin{theorem}\label{prop:Maurer_cartan_equation_unobstructed}
	Suppose that both Conditions \ref{cond:holomorphic_poincare_lemma} and \ref{cond:Hodge-to-deRham} hold. Then for any $k^{\text{th}}$-order solution $\prescript{k}{}{\varphi} = \prescript{k}{}{\phi} + t (\prescript{k}{}{f}) $ to the extended Maurer--Cartan equation \eqref{eqn:extended_maurer_cartan_equation}, there exists a $(k+1)^{\text{st}}$-order solution $\prescript{k+1}{}{\varphi} = \prescript{k+1}{}{\phi} + t (\prescript{k+1}{}{f})$ to \eqref{eqn:extended_maurer_cartan_equation} lifting $\prescript{k}{}{\varphi}$. The same statement holds for the Maurer--Cartan equation \eqref{eqn:classical_maurer_cartan_equation} if we restrict to $\prescript{k}{}{\phi} \in \polyv{k}^{-1,1}(B)$.
\end{theorem}

\begin{proof}
		The first statement follows from \cite[Thm. 5.6]{chan2019geometry} and \cite[Lem. 5.12]{chan2019geometry}:
		Starting with a $k^{\text{th}}$-order solution $\prescript{k}{}{\varphi} = \prescript{k}{}{\phi} + t (\prescript{k}{}{f})$ for \eqref{eqn:extended_maurer_cartan_equation}, one can always use \cite[Thm. 5.6]{chan2019geometry} to lift it to a general $\prescript{k+1}{}{\varphi} \in \polyv{k+1}^0(B)[[t]]$. The argument in \cite[Lem. 5.12]{chan2019geometry} shows that we can choose $\prescript{k+1}{}{\varphi}$ such that the component of $\prescript{k+1}{}{\varphi}|_{t=0}$ in $\polyv{k+1}^{0,0}(B)$ is zero. As a result, the component of $\prescript{k+1}{}{\phi} + t (\prescript{k+1}{}{f})$ in $\polyv{k+1}^{-1,1}(B) \otimes t (\polyv{k+1}^{0,0}(B))$ is again a solution to \eqref{eqn:extended_maurer_cartan_equation}.
		
		For the second statement, we argue that, given $\prescript{k}{}{\phi}$, there always exists $\prescript{k}{}{f} \in \polyv{k}^{0,0}(B)$ such that $\prescript{k}{}{\phi}+t (\prescript{k}{}{f})$ is a solution to \eqref{eqn:extended_maurer_cartan_equation}. We need to solve the equation \eqref{eqn:volume_form_equation}
		by induction on the order $k$. The initial case is trivial by taking $\prescript{0}{}{f} = 0$. 
		Suppose the equation can be solved for $\prescript{j-1}{}{f}$. Then we take an arbitrary lifting ${\prescript{j}{}{\tilde{f}}}$ to the $j^{\text{th}}$-order. We can define an element $\mathfrak{o} \in \polyv{0}^{0,0}(B)$ by	
		$$
		q^j \mathfrak{o} = \pdb(\prescript{j}{}{\tilde{f}})  + [\prescript{j}{}{\phi},\prescript{j}{}{\tilde{f}} ] +  \bvd{}(\prescript{j}{}{\phi}) + \mathfrak{n},
		$$
		which satisfies $\pdb(\mathfrak{o}) = 0$. Therefore, the class $[\mathfrak{o}]$ lies in the cohomology 
        $$H^1(\polyv{0}^{0,*},\pdb) \cong H^1(\centerfiber{0},\cu{O}) \cong H^1(B,\comp),$$
        where the last equivalence is from \cite[Prop. 2.37]{Gross-Siebert-logI}. By our assumption in \S \ref{sec:gross_siebert}, we have $H^1(B,\comp)=0$, and hence we can find an element $\breve{f}$ such that $\pdb(\breve{f}) = \mathfrak{o}$. Letting $\prescript{j}{}{f} = \prescript{j}{}{\tilde{f}} + q^j \cdot \breve{f} \ (\text{mod $q^{j+1}$})$ proves the induction step from the $(j-1)^{\text{st}}$-order to the $j^{\text{th}}$-order. 
		Now, applying the first statement, we can lift the solution $\prescript{k}{}{\varphi}:=\prescript{k}{}{\phi}+t (\prescript{k}{}{f})$ to $\prescript{k+1}{}{\varphi} = \prescript{k+1}{}{\phi}+t (\prescript{k+1}{}{f})$ which satisfies equation \eqref{eqn:extended_maurer_cartan_equation}, and hence $\prescript{k+1}{}{\phi}$ solves \eqref{eqn:classical_maurer_cartan_equation}. 
\end{proof}

From Theorem \ref{prop:Maurer_cartan_equation_unobstructed}, we obtain a solution $\phi \in \polyv{}^{-1,1}(B)$ to the Maurer--Cartan equation \eqref{eqn:classical_maurer_cartan_equation}, from which we obtain the sheaves $\ker(\pdb+[\phi,\cdot])\subset \polyv{k}^{*,*}$ and $\ker(\pdb+\cu{L}_{\phi}) \subset \totaldr{k}{\parallel}^{*,*}$ over $B$. 
These sheaves are locally isomorphic to $\bva{k}^*_{\alpha}$ and $\tbva{k}{\parallel}^*_{\alpha}$, so we may treat them as obtained from gluing of the local sheaves $\bva{k}^*_{\alpha}$'s and $\tbva{k}{\parallel}^*_{\alpha}$'s. From these, we can extract consistent and compatible gluings $\prescript{k}{}{\varPhi}_{\alpha\beta} \colon \localmod{k}^{\dagger}_{\alpha}|_{V_{\alpha\beta}} \rightarrow \localmod{k}^{\dagger}_{\beta}|_{V_{\alpha\beta}}$ satisfying the cocycle condition, and hence obtain a $k$-th order thichening $\centerfiber{k}$ of $\centerfiber{0}$ over $\logsk{k}$; see \cite[\S 5.3]{chan2019geometry}.
Also, $e^{f} \mathbin{\lrcorner} \volf{}$, as a section of $\ker(\pdb+\cu{L}_{\phi})$ over $B$, defines a holomorphic volume form on the $k$-th order thickening $\centerfiber{k}$. 

%Given a log structure on the maximally degenerate Calabi-Yau variety $\centerfiber{0}$ specified by \emph{slab functions} (or \emph{initial wall-crossing factors}), the above theorem says that we can construct a solution to the Maurer-Cartan equation \eqref{eqn:classical_maurer_cartan_equation}, which in turn gives rise to a smoothing of $\centerfiber{0}$. On the other hand, a consistent scattering diagram can be obtained from slab functions by Kontsevich-Soibelman's lemma \cite{kontsevich-soibelman04}. In the next section, we will see that a consistent scattering diagram can also be obtained from the Maurer-Cartan solution in Theorem \ref{prop:Maurer_cartan_equation_unobstructed}. So our results can be regarded as completing the equivalences:

%``consistent scattering diagrams $\Longleftrightarrow$ MC solutions $\Longleftrightarrow$ smoothings of $\centerfiber{0}$''.

\subsubsection{Normalized volume form}\label{subsec:normalization_condition}

For later purposes, we need to further normalize the holomorphic volume 
$$\varOmega := e^{f} \mathbin{\lrcorner} \volf{} \in \ker(\pdb+\cu{L}_{\phi})(B) \subset \totaldr{k}{\parallel}^{n,0}(B)$$
by adding a suitable power series $h(q) \in (q) \subset \comp[[q]]$ to $f$, so that the condition that $\int_{T} e^{f} \mathbin{\lrcorner} \volf{} = 1$, where $T$ is a nearby $n$-torus in the smoothing, is satisfied.

The \emph{$k^{\text{th}}$-order Hodge bundle} over $\spec(\comp[q]/q^{k+1})$ is defined as the cohomology
$$\prescript{k}{}{\cu{H}}:= H^n(\totaldr{k}{\parallel}^*,\mathbf{d}),$$
equipped with a Gauss--Manin connection $\gmc{k}$, where $\gmc{k}_{\dd{\log q}}$ is the connecting homomorphism of the long exact sequence associated to 
\begin{equation}\label{eqn:Gauss_Manin_connection_definition}
 0 \rightarrow \totaldr{k}{\parallel}^{*-1} \otimes_{\comp} \comp \langle d\log q \rangle \rightarrow \totaldr{k}{}^* \rightarrow \totaldr{k}{\parallel}^* \rightarrow 0;
\end{equation}
here $\comp \langle d\log q \rangle $ is the $1$-dimensional graded vector space spanned by the degree $1$ element $d\log q$. We denote $\widehat{\cu{H}} := \varprojlim_{k} \prescript{k}{}{\cu{H}}$. Restricting to the $0^{\text{th}}$-order, we have $N=\gmc{0}_{\dd{\log q}}$, which is a nilpotent operator acting on $\prescript{0}{}{\cu{H}} = H^{n}(\totaldr{0}{\parallel}^*) \cong \mathbb{H}^n(\centerfiber{},j_* \Omega^*_{\centerfiber{}^{\dagger}/\comp^{\dagger}})$, where $\centerfiber{} = \centerfiber{0}$. If we consider the top cohomoloy $H^{2n}(\totaldr{0}{\parallel}^*)$, which is $1$-dimensional, we see that $N=\gmc{0}_{\dd{\log q}} = 0$. So $\gmc{k}_{\dd{\log q}}$ is a flat connection without log poles at $q=0$. Hence, we can find a basis (order by order in $q$) to identify $H^{2n}(\totaldr{k}{\parallel}^*) \cong H^{2n}(\totaldr{0}{\parallel}^*) \otimes \comp[q]/q^{k+1}$, which also trivializes the flat connection $\gmc{}$ as $\dd{\log q}$.

Since $H^n(B,\comp)\cong \comp$, we can fix a non-zero generator and choose a representative $\varrho \in \mdga^n(B)$. Then the element $\varrho\otimes 1 \in \totaldr{k}{\parallel}^n(B)$ (which may simply be written as $\varrho$) represents a section $[\varrho]$ in $\widehat{\cu{H}}$. A direct computation shows that $\gmc{} [\varrho] = 0$, i.e. it is a flat section to all orders. The pairing with the $0^{\text{th}}$-order volume form $\volf{0}$ gives a non-zero element $[\volf{0}\wedge \varrho ]$ in $H^{2n}(\totaldr{0}{\parallel}^*)$. 

\begin{definition}\label{def:normalized_volume_form}
	The volume form $\varOmega=e^{f} \mathbin{\lrcorner} \volf{}$ is said to be \emph{normalized} if $[\varOmega \wedge \varrho]$ is flat under $\gmc{}$. 
\end{definition}
In other words, we can write $[\varOmega \wedge \varrho]= [\volf{0}\wedge \varrho]$ under the identification 
$$H^{2n}(\totaldr{k}{\parallel}^*) \cong H^{2n}(\totaldr{0}{\parallel}^*) \otimes \comp[q]/q^{k+1}.$$
By modifying $f$ to $f+h(q)$, this can always be achieved. Further, after the modification, $\varphi = \phi + t f$ still solves \eqref{eqn:extended_maurer_cartan_equation}.

\section{From smoothing of Calabi--Yau varieties to tropical geometry}\label{sec:tropical_geometry_and_mc_equation}

\subsection{Tropical differential forms}\label{sec:asymptotic_support}

To tropicalize the pre-dgBV algebra $\polyv{}^{*,*}$, we need to replace the Thom--Whitney resolution used in \cite{chan2019geometry} by a geometric resolution.
To do so, we first need to recall some background materials from our previous works \cite[\S 4.2.3]{kwchan-leung-ma} and \cite[\S 3.2]{kwchan-ma-p2}. 
Of crucial importance is the notion of \emph{differential forms with asymptotic support} (which  will be called \emph{tropical differential forms} in this paper) that originated from multi-valued Morse theory and Witten deformations. 
Such differential forms can be regarded as distribution-valued forms supported on tropical polyhedral subsets. This key notion allows us to develop tropical intersection theory via differential forms, and in particular, define the intersection pairing between possibly non-transversal tropical polyhedral subsets simply using the wedge product.

%develop tropical intersection theory via differential forms
%We use this as a replacement for distributions supported on tropical polyhedral subsets for defining the product of two forms which corresponding to possibly non-transversal intersection of tropical polyhedral subsets.

Let $U$ be an open subset of $M_\real$. We consider the space $\Omega^k_\hp(U) := \Gamma(U \times \mathbb{R}_{>0}, \bigwedge^{\raisebox{-0.4ex}{\scriptsize $k$}} T^{\vee} U)$, where we take $\cu{C}^{\infty}$ sections of $\bigwedge^{\raisebox{-0.4ex}{\scriptsize $k$}} T^{\vee} U$ and $\hp$ is a coordinate on $\mathbb{R}_{>0}$. 
Let $\mathcal{W}^{k}_{-\infty}(U) \subset \Omega^k_\hp(U)$ be the subset of $k$-forms $\alpha$ such that, for each $q \in U$, there exist a neighborhood $q \in V \subset U$, constants $D_{j,V}$, $c_V$ and a sufficiently small real number $\hp_0 > 0$ such that $\|\nabla^j \alpha\|_{L^\infty(V)} \leq D_{j,V} e^{-c_V/\hp}$ for all $j \geq 0$ and for $0<\hp < \hp_0$; 
here, the $L^\infty$-norm is defined by $\| \alpha \|_{L^{\infty}(V)} = \sup_{x\in V} \|\alpha(x)\|$ for any section $\alpha$ of the tensor bundle $TU^{\otimes k} \otimes T^{\vee}U^{\otimes l}$, where
we fix a constant metric on $M_{\real}$ and use the induced metric on $TU^{\otimes k} \otimes T^{\vee}U^{\otimes l}$;
$\nabla^j$ denotes an operator of the form $\nabla_{\dd{x_{l_1}}}\cdots \nabla_{\dd{x_{l_j}}}$, where $\nabla$ is a torsion-free, flat connection defining an affine structure on $U$ and $x = (x_1,\dots, x_n)$ is an affine coordinate system (note that $\nabla$ is {\it not} the Gauss--Manin connection in the previous section).
Similarly, let $\mathcal{W}^{k}_{\infty}(U) \subset \Omega^k_\hp(U)$ be the set of $k$-forms $\alpha$ such that, for each $q \in U$, there exist a neighborhood $q \in V \subset U$, a constant $D_{j,V}$, $N_{j,V} \in \inte_{>0}$ and a sufficiently small real number $\hp_0 > 0$ such that $\|\nabla^j \alpha\|_{L^\infty(V)} \leq D_{j,V} \hp^{-N_{j,V}}$ for all $j \geq 0$ and for $0<\hp < \hp_0$. 

The assignment $U \mapsto \mathcal{W}^{k}_{-\infty}(U)$ (resp. $U \mapsto \mathcal{W}^{k}_{\infty}(U)$) defines a sheaf $\mathcal{W}^{k}_{-\infty}$ (resp. $\mathcal{W}^{k}_{\infty}$) on $M_\real$ (\cite[Defs. 4.15 \& 4.16]{kwchan-leung-ma}). Note that $\mathcal{W}^{k}_{-\infty}$ and $\mathcal{W}^{k}_{\infty}$ are closed under the wedge product,  $\nabla_{\dd{x}}$ and the de Rham differential $d$. Since $\mathcal{W}^{k}_{-\infty}$ is a dg ideal of $\mathcal{W}^{k}_{\infty}$, the quotient $\mathcal{W}^{*}_{\infty}/\mathcal{W}^{*}_{-\infty}$ is a sheaf of dgas when equipped with the de Rham differential.

Now suppose $U$ is a convex open set. By a {\em tropical polyhedral subset} of $U$, we mean a connected convex subset of $U$ which is defined by finitely many affine equations or inequalities over $\mathbb{Q}$ of the form $a_1 x_1 + \cdots + a_n x_n \leq b$.

\begin{definition}[\cite{kwchan-leung-ma}, Def. 4.19]\label{def:asypmtotic_support_pre}
	A $k$-form $\alpha \in \mathcal{W}_{\infty}^k(U)$ is said to have {\em asymptotic support on a closed codimension $k$ tropical polyhedral subset $P \subset U$ with weight $s \in \inte$}, denoted as $\alpha \in \mathcal{W}_{P,s}(U)$, if the following conditions are satisfied:
	\begin{enumerate}
		\item
		For any $p \in U \setminus P$, there is a neighborhood $p \in V \subset U \setminus P$ such that $\alpha|_V \in \mathcal{W}_{-\infty}^k(V)$.
		
		\item
		There exists a neighborhood $W_P \subset U$ of $P$ such that $\alpha =  h(x,\hp) \nu_P + \eta$ on $W_P$, where $\nu_P \in \bigwedge^k N_{\real}$ is a non-zero affine $k$-form (defined up to non-zero constant) which is normal to $P$, $h(x,\hp) \in C^\infty(W_P \times \real_{>0})$ and $\eta \in \mathcal{W}_{-\infty}^k(W_P)$.
		
		\item
		For any $p \in P$, there exists a convex neighborhood $p \in V \subset U$ equipped with an affine coordinate system $x = (x_1,\dots, x_n)$ such that $x' := (x_1, \dots, x_k)$ parametrizes codimension $k$ affine linear subspaces of $V$ parallel to $P$, with $x' = 0$ corresponding to the subspace containing $P$. With the foliation $\{(P_{V, x'})\}_{x' \in N_V}$, where $P_{V,x'} = \{ (x_1,\dots,x_n) \in V  \ | \ (x_1,\dots,x_k) = x' \}$ and $N_V$ is the normal bundle of $V$, we require that, for all $j \in \inte_{\geq 0}$ and multi-indices $\beta = (\beta_1,\dots,\beta_k) \in \inte_{\geq 0}^k$, the estimate
		\[
		%\label{estimate_order_s}
		\int_{x'}  (x')^\beta \left(\sup_{P_{V,x'}}|\nabla^j (\iota_{\nu_P^\vee} \alpha)| \right) \nu_P \leq D_{j,V,\beta} \hp^{-\frac{j+s-|\beta|-k}{2}}
		\]
		holds for some constant $D_{j,V,\beta}$ and $s \in \inte$, where $|\beta| = \sum_l \beta_l$ and $\nu_P^\vee = \dd{x_1}\wedge\cdots \wedge\dd{x_k}$.\footnote{For $k=0$, we use the convention that $\nu_P = 1 \in \bigwedge^0 N_{\real} = \real$ and also set $\nu_P^\vee =1$.}
	\end{enumerate}
\end{definition}

Observe that $\nabla_{\dd{x_l}} \mathcal{W}_{P,s}(U) \subset \mathcal{W}_{P,s+1}(U)$ and $(x')^{\beta}\mathcal{W}_{P,s}(U) \subset \mathcal{W}_{P,s-|\beta|}(U)$. It follows that
\[
%\label{asy_support_basic_property}
(x')^{\beta} \nabla_{\dd{x_{l_1}}}\cdots \nabla_{\dd{x_{l_j}}} \mathcal{W}_{P,s}(U) \subset \mathcal{W}_{P,s+j-|\beta|}(U).
\]
The weight $s$ defines a filtration of $\mathcal{W}^k_{\infty}$ (we drop the $U$ dependence from the notation whenever it is clear from the context):\footnote{Note that $k$ is equal to the codimension of $P \subset U$.}
\[
%\label{filtration}
\mathcal{W}_{-\infty}^k \subset \cdots \subset \mathcal{W}_{P,-1}\subset \mathcal{W}_{P,0} \subset \mathcal{W}_{P,1} \subset \cdots \subset \mathcal{W}_{\infty}^k \subset \Omega^k_\hp(U).
\]
This filtration, which keeps track of the polynomial order of $\hp$ for $k$-forms with asymptotic support on $P$, provides a convenient tool to express and prove results in asymptotic analysis.

\begin{definition}[\cite{kwchan-ma-p2}, Def. 3.10]\label{def:asymptotic_support}\label{def:asy_support_algebra}
	A differential $k$-form $\alpha$ is \emph{in $\tilde{\mathcal{W}}_{s}^k(U)$} if there exist polyhedral subsets $P_1, \dots, P_l \subset U$ of codimension $k$ such that $\alpha \in \sum_{j=1}^l \mathcal{W}_{P_j,s}(U)$.
	If, moreover, $d \alpha \in \tilde{\mathcal{W}}_{s+1}^{k+1}(U)$, then we write $\alpha \in \mathcal{W}_s^k(U)$. For every $s \in \inte$, let $\mathcal{W}_s^*(U) = \bigoplus_k \mathcal{W}_{s+k}^k(U)$.
\end{definition}

\begin{example}\label{example:bump_form}
    Let $U = \real$ and $x$ be an affine coordinate on $U$. Then we consider the $\hp$-dependent $1$-form
	$$
	\delta:= \left( \frac{1}{\hp \pi} \right)^{\half} e^{-\frac{x^2}{\hp}} dx.
	$$
	Direct calculations in \cite[Lem 4.12]{kwchan-leung-ma} showed that $\delta \in \mathcal{W}_1^1(U)$ has asymptotic support on the hyperplane $P$ defined by $x=0$. 
	
    The hyperplane $P$ separates $U$ into two chambers $H_+$ and $H_-$. If we fix a base point in $H_-$ and apply the integral operator $I$ in \cite[Lem. 4.23]{kwchan-leung-ma}, we obtain $I(\delta) \in W^0_0(U)$ which has asymptotic support on $H_+ \cup P$, playing the role of a step function. 
	
	Taking finite products of elements of the above form, we obtain $\alpha \in \mathcal{W}^k_k(U)$ with asymptotic support on arbitrary tropical polyhedral subsets of $U$. Any forms obtained from a finite number of steps of applying the differential $d$, applying the integral operator $I$ and taking wedge product are in $W^*_0(U)$. 
\end{example}

We say that two closed tropical polyhedral subsets $P_1, P_2 \subset U$ of codimension $k_1, k_2$ {\em intersect transversally} if the affine subspaces of codimension $k_1$ and $k_2$ which contain $P_1$ and $P_2$, respectively, intersect transversally. This definition applies also when $P_1 \cap P_2 = \emptyset $ or $\partial P_i \neq \emptyset$.

\begin{lemma}[{\cite[Lem. 4.22]{kwchan-leung-ma}}]
	\phantomsection
	\label{lem:support_product}\label{prop:support_product}\label{prop:asy_support_algebra_ideal}
	\begin{enumerate}
		\item Let $P_1, P_2, P \subset U$ be closed tropical polyhedral subsets of codimension $k_1$, $k_2$ and $k_1+k_2$, respectively, such that $P$ contains $P_1 \cap P_2$ and is normal to $\nu_{P_1} \wedge \nu_{P_2}$. Then $\mathcal{W}_{P_1,s}(U) \wedge \mathcal{W}_{P_2,r}(U) \subset \mathcal{W}_{P,r+s}(U)$ if $P_1$ and $P_2$ intersect transversally with $P_1 \cap P_2 \neq \emptyset$, and $\mathcal{W}_{P_1,s}(U) \wedge \mathcal{W}_{P_2,r}(U) \subset \mathcal{W}_{-\infty}^{k_1 + k_2}(U)$ otherwise.
		
		\item We have $\mathcal{W}_{s_1}^{k_1}(U) \wedge \mathcal{W}_{s_2}^{k_2}(U) \subset \mathcal{W}_{s_1+s_2}^{k_1+k_2}(U)$. In particular, $\mathcal{W}_0^*(U) \subset \mathcal{W}_{\infty}^*(U)$ is a dg subalgebra and $\mathcal{W}_{-1}^*(U) \subset \mathcal{W}_0^*(U)$ is a dg ideal.
	\end{enumerate}
\end{lemma}

\begin{definition}\label{def:sheaf_of_tropical_dga}
	Let $\mathcal{W}_s^*$ be the sheafification of the presheaf defined by $U \mapsto \mathcal{W}_s^*(U)$. We call the quotient sheaf $\tform^*:=\mathcal{W}_0^*/\mathcal{W}_{-1}^*$ the \emph{sheaf of tropical differential forms}, which is a sheaf of dgas on $M_{\real}$ with structures $(\wedge,\md)$. 
\end{definition}

From \cite[Lem. 3.6]{kwchan-ma-p2}, we learn that $\underline{\real} \rightarrow \tform^*$ is a resolution. Furthermore, given any point $x \in U$ and a sufficiently small neighborhood $x \in W \subset  U$, we can show that there exists $f \in \mathcal{W}_0^0(W)$ with compact support in $W$ and satisfying $f \equiv 1$ near $x$ (using an argument similar to the proof of Lemma \ref{lem:for_contruction_of_partition_of_unity}). Therefore, $\tform^*$ has a partition of unity subordinate to a given open cover. 
%\subsubsection{The global sheaf of tropical differential forms}\label{subsec:global_sheaf_of_forms}
Replacing  the sheaf of de Rham differential forms on $\tanpoly_{\rho_1,\real}^* \oplus \norpoly_{\tau,\real}$ by the sheaf $\tform^*$ of tropical differential forms, we can construct a particular complex on the integral tropical manifold $B$ satisfying Condition \ref{cond:requirement_of_the_de_rham_dga}, which dictates the tropical geometry of $B$. 

\begin{definition}\label{def:global_sheaf_of_monodromy_invariant_tropical_forms}
		Given a point $x$ as in \S\ref{subsec:derham_for_B} (with a chart as in equation \eqref{eqn:monodromy_invariant_affine_functions_near_x_o}), the stalk of $\tform^*$ at $x$ is defined as $\tform^*_{x}:= (\mathtt{x}^{-1}\tform^*)_{x}$. This defines the \emph{complex $(\tform^*,\md)$ (or simply $\tform^*$) of monodromy invariant tropical differential forms on $B$}. A section $\alpha \in \tform^*(W)$ is a collection of elements $\alpha_{x} \in \tform^*_{x}$, $x \in W$ such that each $\alpha_{x}$ can be represented by $\mathtt{x}^{-1}\beta_{x}$ in a small neighborhood $U_{x} \subset \mathtt{p}^{-1}(\mathtt{U}_{x})$ for some tropical differential form $\beta_{x}$ on $\mathtt{U}_{x}$, and satisfies the relation $\alpha_{\tilde{x}} = \tilde{\mathtt{x}}^{-1}(\mathtt{p}^* \beta_{x})$ in $\tform^*_{\tilde{x}}$ for every $\tilde{x} \in U_{x}$. %We write the sheaf of tropical forms as  .  
	\end{definition}

Notice that the definition of $\tform^*$ requires the projection map $\mathtt{p}$ in equation \eqref{eqn:monodromy_invariant_differential_form_change_of_chart} to be affine, while that of $\mdga^*$ in \S \ref{subsec:derham_for_B} does not. But like $\mdga^*$, $\tform^*$ satisfies Condition \ref{cond:requirement_of_the_de_rham_dga} and can be used for the purpose of gluing the sheaf $\polyv{}^*$ of dgBV algebras in \S \ref{subsubsec:gluing_construction}. In the rest of this section, we shall use the notations $\polyv{}^*$ and $\totaldr{}{}^*$ to denote the complexes of sheaves constructed using $\tform^*$. 

\subsection{The semi-flat dgBV algebra and its comparison with the pre-dgBV algebra $\polyv{}^{*,*}$}\label{subsec:semi-flat_dgBV}
In this section, we define a twisting of the semi-flat dgBV algebra by the slab functions (or initial wall-crossing factors) in \S \ref{subsec:log_structure_and_slab_function}, and compare it with the dgBV algebra we constructed in \S \ref{subsubsec:gluing_construction} using gluing of local smoothing models. The key result is Lemma \ref{lem:comparing_sheaf_of_dgbv}, which is an important step in the proof of our main result.

We start by recalling some notations from \S \ref{subsec:log_structure_and_slab_function}. Recall that for each vertex $v$, we fix a representative $\varphi_v\colon U_v \rightarrow \real$ of the strictly convex multi-valued piecewise linear function $\varphi \in H^0(B,\cu{MPL}_{\pdecomp})$ to define the cone $C_v$ and the monoid $P_v$. The natural projection $T_v \oplus \inte \rightarrow T_v$ induces a surjective ring homomorphism $\comp[\rho^{-1}P_v] \rightarrow \comp[\rho^{-1}\Sigma_v]$; we denote by $\bar{m} \in \rho^{-1}\Sigma_v$ the image of $m \in \rho^{-1} P_v$ under the natural projection.
We consider $\mathbf{V}(\tau)_v := \spec(\comp[\tau^{-1}P_v])$ for some $\tau$ containing $v$, and write $z^m$ for the function corresponding to $m \in \tau^{-1} P_v$. The element $\varrho$ together with the corresponding function $z^{\varrho}$ determine a family $\spec(\comp[\tau^{-1}P_v]) \rightarrow \comp$, whose central fiber is given by $\spec(\comp[\tau^{-1}\Sigma_v])$. The variety $\mathbf{V}(\tau)_v = \spec(\comp[\tau^{-1}P_v])$ is equipped with the divisorial log structure induced by $\spec(\comp[\tau^{-1}\Sigma_v])$, which is log smooth. We write $\mathbf{V}(\tau)_v^{\dagger}$ if we need to emphasize the log structure. 

Since $B$ is orientable, we can choose a nowhere vanishing integral element $\mu \in \Gamma(B\setminus \tsing_e,\bigwedge^n T_{B,\inte})$. We fix a local representative $\mu_v \in \bigwedge^n T_{v}$ for every vertex $v$ and $\mu_{\sigma} \in \bigwedge^n \tanpoly_{\sigma}$ for every maximal cell $\sigma$. Writing $\mu_v = m_1 \wedge \cdots \wedge  m_n$, we have the corresponding relative volume form 
$$\mu_v = d\log z^{m_1} \wedge \cdots \wedge d\log z^{m_n}$$ 
in $\Omega^n_{\mathbf{V}(\tau)_v^{\dagger}/\comp^{\dagger}}$. Now the relative log polyvector fields can be written as
$$
\bigwedge\nolimits^{-l} \Theta_{\mathbf{V}(\tau)_v^{\dagger}/\comp^{\dagger}} = \bigoplus_{m \in \tau^{-1}P_v} z^m \partial_{n_1} \wedge \cdots \wedge \partial_{n_l}.
$$ 
The volume form $\mu_v$ defines a BV operator via contraction $(\bvd{}\alpha) \mathbin{\lrcorner} \mu_v := \partial(\alpha \mathbin{\lrcorner} \mu_v)$, which is given explicitly by
$$
\bvd{}(z^m \partial_{n_1} \wedge \cdots \wedge \partial_{n_l}) = \sum_{j=1}^l (-1)^{j-1} \langle m,n_j\rangle z^m  \partial_{n_1} \wedge \cdots \widehat{\partial}_{n_j}  \cdots \wedge \partial_{n_l}.
$$
A Schouten–-Nijenhuis--type bracket is given by extending the following formulae skew-symmetrically:
\begin{align*}
	[z^{m_1} \partial_{n_1},z^{m_2}\partial_{n_2}] & = z^{m_1+m_2} \partial_{\langle \bar{m}_1, n_2  \rangle n_{1} - \langle \bar{m}_2, n_1 \rangle n_{2}},\\
	[z^m , \partial_n] & = \langle \bar{m}, n \rangle z^m. 
\end{align*}
This gives $\bigwedge^{-*} \Theta_{\mathbf{V}(\tau)_v^{\dagger}/\comp^{\dagger}}$ the structure of a BV algebra.% which can also be passed to the associated analytic scheme $\mathbf{V}(\tau)_v = \spec(\comp[\tau^{-1}P_v])$. 

\subsubsection{Construction of the semi-flat sheaves}\label{subsubsec:semi-flat}

For each $k \in \mathbb{N}$, we shall define a sheaf $\sfbva{k}^*_{\mathrm{sf}}$ (resp. $\sftbva{k}{}^*_{\mathrm{sf}}$) of $k^{\text{th}}$-order semi-flat log vector fields (resp. semi-flat log de Rham forms) over the open dense subset $W_0 \subset B$ defined by
$$
W_0 := \bigcup_{\sigma \in \pdecomp^{[n]}} \reint(\sigma) \cup \bigcup_{\rho \in \pdecomp^{[n-1]}_0} \reint(\rho)  \cup \bigcup_{\rho \in \pdecomp^{[n-1]}_1} \big( \reint(\rho) \setminus (\tsing \cap \reint(\rho))  \big),
$$
where $\pdecomp^{[n-1]}_0$ consists of $\rho$'s such that $\reint(\rho) \cap \tsing_e = \emptyset$ and $\pdecomp^{[n-1]}_1$ of $\rho$'s that intersect with $\tsing_e$. These sheaves use the natural divisorial log structure on $\mathbf{V}(\rho)_v^{\dagger}$ and will \emph{not} depend on the slab functions $f_{v\rho}$'s. This construction is possible because we are using the much more flexible Euclidean topology on $W_0$, instead of the Zariski topology on $\centerfiber{0}$.

For $\sigma \in \pdecomp^{[n]}$, recall that we have $V(\sigma) = \spec(\comp[\sigma^{-1}\Sigma_{v}])$ for some $v \in \sigma^{[0]}$. We also have $\spec(\comp[\sigma^{-1}\Sigma_{v}]) = \tanpoly^*_{\sigma,\comp}/\tanpoly^*_{\sigma} $, which is isomorphic to $(\comp^{*})^n$, because $\sigma^{-1}\Sigma_v = \tanpoly_{\sigma,\real} = T_{v,\real}$. The local $k^{\text{th}}$-order thickening 
$$\localmod{k}(\sigma)^{\dagger}:= \spec(\comp[\sigma^{-1}P_{v}/q^{k+1}]) \cong (\comp^{*})^n \times \spec(\comp[q]/q^{k+1})$$
is obtained by identifying $\sigma^{-1}P_v$ as $\tanpoly_{\sigma} \times \mathbb{N}$. Choosing a different vertex $v'$, we can use the parallel transport $T_{v} \cong T_{v'}$ from $v$ to $v'$ within $\reint(\sigma)$ and the difference $\varphi_v|_{\sigma} - \varphi_{v'}|_{\sigma}$ between two affine functions to identify the monoids $\sigma^{-1}P_{v}\cong \sigma^{-1}P_{v'}$. We take 
$$\sfbva{k}^*_{\mathrm{sf}}|_{\reint(\sigma)}:= \modmap_*\Big(\bigwedge\nolimits^{-*} \Theta_{\localmod{k}(\sigma)^{\dagger}/\logsk{k}}\Big) \cong \modmap_*(\mathcal{O}_{\localmod{k}(\sigma)^{\dagger}}) \otimes_{\real} \bigwedge\nolimits^{-*} \tanpoly_{\sigma,\real}^*.$$ 

Next, we need to glue the sheaves $\sfbva{k}^*_{\mathrm{sf}}|_{\reint(\sigma)}$'s along neighborhoods of codimension one cells $\rho$'s. For each codimension one cell $\rho$, we fix a primitive normal $\check{d}_{\rho}$ to $\rho$ and label the two adjacent maximal cells $\sigma_+$ and $\sigma_-$ so that $\check{d}_{\rho}$ is pointing into $\sigma_+$. There are two situations to consider. 

The simpler case is when $\tsing_e \cap \reint(\rho) = \emptyset$, where the monodromy is trivial. In this case, we have $V(\rho) = \spec(\comp[\rho^{-1}\Sigma_{v}])$, with the gluing $V(\sigma_{\pm}) \hookrightarrow V(\rho)$ as described below Definition \ref{def:open_gluing_data} using the open gluing data $s_{\rho\sigma_{\pm}}$. We take the $k^{\text{th}}$-order thickening given by
$$\localmod{k}(\rho)^{\dagger}:= \spec(\comp[\rho^{-1}P_{v}/q^{k+1}])^{\dagger},$$ 
equipped with the divisorial log structure induced by $V(\rho)$. 
We extend the open gluing data 
$$s_{\rho\sigma_{\pm}} \colon \tanpoly_{\sigma_{\pm}} \rightarrow \comp^{*}$$
to 
$$s_{\rho\sigma_{\pm}} \colon \tanpoly_{\sigma_{\pm}} \oplus \inte \rightarrow \comp^{*}$$
so that $s_{\rho\sigma_{\pm}}(0,1) = 1$, which acts as an automorphism of $\spec(\comp[\sigma^{-1}\Sigma_{v}])$. In this way we can extend the gluing $V(\sigma_{\pm}) \hookrightarrow V(\rho)$ to 
$$\spec(\comp[\sigma_{\pm}^{-1}P_{v}/q^{k+1}]) \rightarrow \spec(\comp[\rho^{-1}P_{v}/q^{k+1}])$$
by twisting with the ring homomorphism induced by $z^{m} \rightarrow s_{\rho\sigma_{\pm}}(m)^{-1}z^{m}$. On a sufficiently small neighborhood $\mathscr{W}_{\rho}$ of $\reint(\rho)$, we take $$\sfbva{k}^*_{\mathrm{sf}}|_{\mathscr{W}_{\rho}}:= \modmap_* \Big(\bigwedge\nolimits^{-*} \Theta_{\localmod{k}(\rho)^{\dagger}/\logsk{k}}\Big)\Big|_{\mathscr{W}_{\rho}}.$$
Choosing a different vertex $v'$, we may use parallel transport to identify the fans $\rho^{-1} \Sigma_{v} \cong \rho^{-1}  \Sigma_{v'}$, and further use the difference $\varphi_v|_{\mathscr{W}_{\rho}} - \varphi_{v'}|_{\mathscr{W}_{\rho}}$ to identify the monoids $\rho^{-1}P_v \cong \rho^{-1}P_{v'}$. One can check that the sheaf $\sfbva{k}^*_{\mathrm{sf}}|_{\mathscr{W}_{\rho}}$ is well-defined.  

The more complicated case is when $\tsing_e \cap \reint(\rho) \neq \emptyset$, where the monodromy is non-trivial. We write $\reint(\rho) \setminus \tsing = \bigcup_{v} \reint(\rho)_v$, where $\reint(\rho)_v$ is the unique component which contains the vertex $v$ in its closure. We fix one $v$, the corresponding $\reint(\rho)_v$, and a sufficiently small open subset $\mathscr{W}_{\rho,v}$ of $\reint(\rho)_v$. We assume that the neighborhood $\mathscr{W}_{\rho,v}$ of $\reint(\rho)_v$ intersects neither $\mathscr{W}_{v',\rho'}$ nor $\mathscr{W}_{\rho'}$ for any possible $v'$ and $\rho'$. 
%Forgetting the log structure, which is described by the slab functions $f_{v,\rho}$'s, we have $V(\rho) = \spec(\comp[\rho^{-1}\Sigma_{v}])$. The semi-flat construction of the sheaf $\sfbva{k}^*_{\mathrm{sf}}$ itself does not depend on the log structure. 
Then we consider the scheme-theoretic embedding 
$$V(\rho) = \spec(\comp[\rho^{-1}\Sigma_v]) \rightarrow \spec(\comp[\rho^{-1}P_v])$$
given by 
$$
z^{m} \mapsto \begin{cases}
	z^{\bar{m}} & \text{if $m$ lies on the boundary of the cone $\rho^{-1}P_{v}$,}\\
	0 & \text{if $m$ lies in the interior of the cone $\rho^{-1}P_{v}$.}
	\end{cases}
$$
We denote by $\sflocmod{k}(\rho)_v^{\dagger}$ the $k^{\text{th}}$-order thickening of $V(\rho)|_{\modmap^{-1}(\mathscr{W}_{\rho,v})}$ in $\spec(\comp[\rho^{-1}P_{v}])$ and equip it with the divisorial log structure which is log smooth over $\logsk{k}$ (note that it is {\em different} from the local model $\localmod{k}(\rho)^{\dagger}$
%(in particular $\sflocmod{0}(\rho)_v^{\dagger}$ is different from $V(\rho)^{\dagger}|_{\modmap^{-1}(\mathscr{W}_{\rho,v})}$)
introduced earlier in \S \ref{sec:deformation_via_dgBV} because the latter depends on the slab functions $f_{v,\rho}$, as we can see explicitly in \S \ref{subsubsec:explicit_gluing_away_from_codimension_2}, while the former doesn't). We take
$$\sfbva{k}^*_{\mathrm{sf}}|_{\mathscr{W}_{\rho,v}} := \bigwedge\nolimits^{-*} \Theta_{\sflocmod{k}(\rho)_v^{\dagger}/\logsk{k}}.$$

The gluing with nearby maximal cells $\sigma_{\pm}$ on the overlap $\reint(\sigma_{\pm}) \cap \mathscr{W}_{\rho,v}$ is given by parallel transporting through the vertex $v$ to relate the monoids $\sigma_{\pm}^{-1}P_v$ and $\rho^{-1}P_v$ constructed from $P_v$, and twisting the map $\spec(\comp[\sigma^{-1}_{\pm}P_v]) \rightarrow \spec(\comp[\rho^{-1}P_v])$ with the open gluing data
$$
z^{m} \mapsto s_{\rho \sigma_{\pm}}^{-1}(m) z^{m}, 
$$
using previous liftings of $s_{\rho\sigma_{\pm}}$ to $\tanpoly_{\sigma_{\pm}}\oplus \inte$. 
We obtain a commutative diagram of holomorphic maps
$$
\xymatrix@1{
	V(\sigma_{\pm})|_{\mathscr{D}} \ar[r] \ar[d] &  \localmod{k}(\sigma_{\pm})^{\dagger}|_{\mathscr{D}} \ar[d]\\
	V(\rho)|_{\mathscr{D}} \ar[r] & \sflocmod{k}(\rho)^{\dagger}|_{\mathscr{D}}
},
$$
where $\mathscr{D} =\modmap^{-1}( \mathscr{W}_{\rho,v} \cap \reint(\sigma_{\pm}))$ and the vertical arrow on the right hand side respects the log structures. The induced isomorphism
$$\modmap_* \Big(\bigwedge\nolimits^{-*} \Theta_{\sflocmod{k}(\rho)_v^{\dagger}/\logsk{k}}\Big) \cong \modmap_* \Big(\bigwedge\nolimits^{-*} \Theta_{\localmod{k}(\sigma_{\pm})_v^{\dagger}/\logsk{k}}\Big)$$
of sheaves on the overlap $\mathscr{W}_{\rho,v} \cap \reint(\sigma_{\pm})$ then gives the desired gluing for defining the sheaf $\sfbva{k}^*_{\mathrm{sf}}$ on $W_0$. Note that the cocycle condition is trivial here as there is no triple intersection of any three open subsets from $\reint(\sigma)$, $\mathscr{W}_\rho$ and $\mathscr{W}_{\rho,v}$. 

Similarly, we can define the sheaf $\sftbva{k}{}^*_{\mathrm{sf}}$ of semi-flat log de Rham forms, together with a relative volume form $\volf{k}_0 \in \sftbva{k}{\parallel}^n_{\mathrm{sf}}(W_0)$ obtained from gluing the local $\mu_v$'s specified by the element $\mu$ as described in the beginning of \S \ref{subsec:semi-flat_dgBV}. 

It would be useful to write down elements of the sheaf $\sfbva{k}^*_{\mathrm{sf}}$ more explicitly. For instance, fixing a point $x \in \reint(\rho)_v$, we may write 
\begin{equation}\label{eqn:explicit_description_of_semi_flat_bva}
	\sfbva{k}^*_{\mathrm{sf},x}= \modmap_*(\cu{O}_{\sflocmod{k}(\rho)_v})_x \otimes_{\real} \bigwedge\nolimits^{-*} T^*_{v,\real},
\end{equation}
and use $\partial_n$ to stand for the semi-flat holomorphic vector field associated to an element $n \in T^*_{v,\real}$. 

Note that analytic continuation around the singular locus $\tsing_e \cap \reint(\rho)$ acts non-trivially on the semi-flat sheaf $\sfbva{k}^*_{\mathrm{sf}}$ due to the presence of non-trivial monodromy of the affine structure. Below is a simple example.

\begin{example}\label{eg:monodromy_action_on_semi_flat_sheaf}
	We consider the local affine charts which appeared in Example \ref{eg:K3_example}, equipped with a strictly convex piecewise linear affine function $\varphi$ on $\Sigma_{\rho}$ whose change of slopes is $1$. Let us study the analytic continuation of a local section along the loop $\gamma$ which starts at a point $b_+$, as shown in Figure \ref{fig:affine_chart_2}.
	\begin{figure}[h!]
		\includegraphics[scale=0.5]{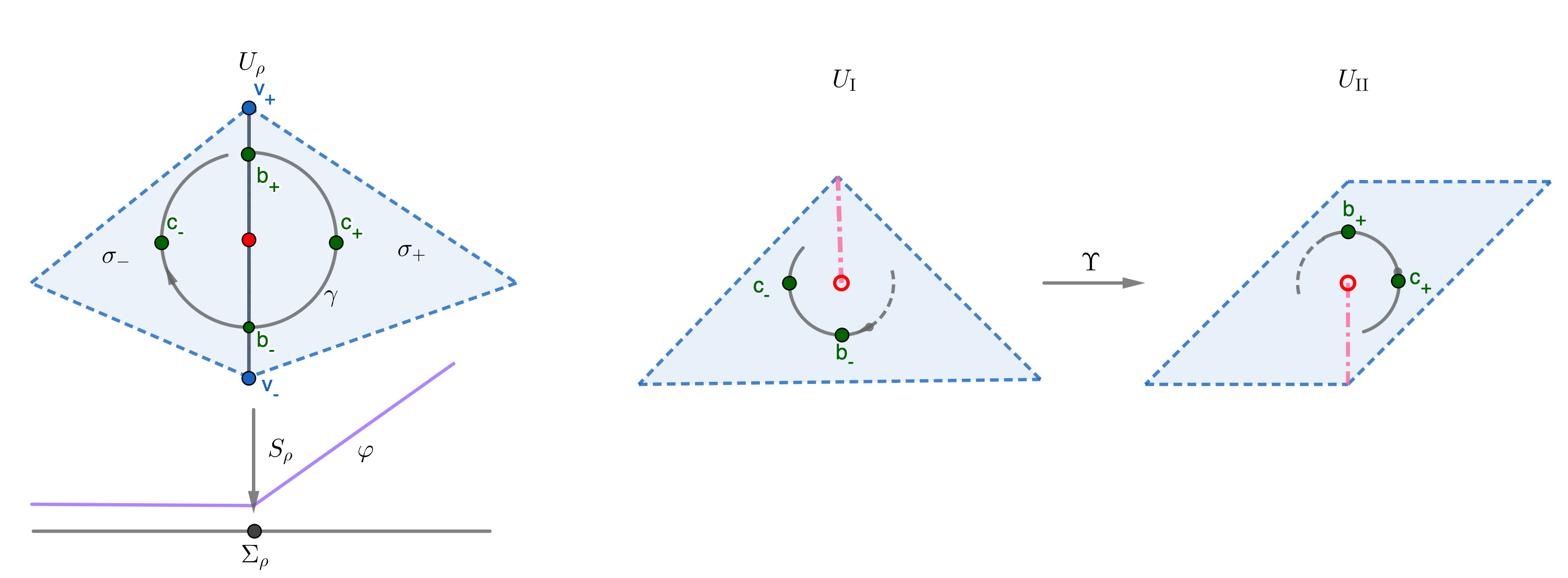}
		\caption{Analytic continuation along $\gamma$}\label{fig:affine_chart_2}
	\end{figure}
	First, we can identify both $\rho^{-1}P_{v_+}$ and $\rho^{-1}P_{v_-}$ with the monoid in the cone $P = \{ (x,y,z) \ | \ z\geq \varphi(x) \}$ via parallel transport through $\sigma_+$. Writing $u = z^{(1,0,1)}$, $v= z^{(-1,0,0)}$, $w = z^{(0,-1,0)}$ and $q = z^{(0,0,1)}$, we have $\comp[P] \cong \comp[u,v,w^{\pm}, q] / (uv-q)$ . 
    Now the analytic continuation of $u \in \modmap_*(\cu{O}_{\sflocmod{k}(\rho)_{v_+}})_{b_+}$ along $\gamma$ (going from the chart $U_{\mathrm{II}}$ to the chart $U_{\mathrm{I}}$ and then back to $U_{\mathrm{II}}$) is given by as a sequence of elements:
	$$
	\xymatrix@1{ 
		u \ar[r] & s_{\rho\sigma_+}((1,0))^{-1}u  \ar[r] & uw \ar[r] & s_{\rho\sigma_-}((1,0))^{-1} qv^{-1}w \ar[r] & wu,
	}
	$$ 
	via the following sequence of maps between the stalks over $b_+, c_+ \in U_{\mathrm{II}}$ and $b_-, c_- \in U_{\mathrm{I}}$:
	$$
	\xymatrix@1{ 
		 \modmap_*(\cu{O}_{\sflocmod{k}(\rho)_{v_+}})_{b_+} \ar[r] &  \modmap_*(\cu{O}_{\localmod{k}(\sigma_{+})^{\dagger}})_{c_+} \ar[r] &  \modmap_*(\cu{O}_{\sflocmod{k}(\rho)_{v_-}})_{b_-} \ar[r] &  \modmap_*(\cu{O}_{\localmod{k}(\sigma_{-})^{\dagger}})_{c_-} \ar[r] &  \modmap_*(\cu{O}_{\sflocmod{k}(\rho)_{v_+}})_{b_+}}.
	$$
	So we see that the analytic continuation along $\gamma$ maps $u$ to $wu$. 
\end{example}

$\sfbva{k}^*_{\mathrm{sf}}$ is equipped with the BV algebra structure inherited from $\spec(\comp[\rho^{-1}P_v])^{\dagger}$ (as described in the beginning of \S \ref{subsec:semi-flat_dgBV}), which agrees with the one induced from the volume form $\volf{k}_0$. This allows us to define the \emph{sheaf of semi-flat tropical vertex Lie algebras} as
\begin{equation}\label{eqn:tropical_vertex_lie_algebra}
\sftvbva{k}:= \mathrm{Ker}(\bvd{})|_{\sfbva{k}^{-1}_{\mathrm{sf}}}[-1]. 
\end{equation}
\begin{remark}
	The sheaf can actually be extended over the non-essential singular locus $\tsing  \setminus \tsing_e$ because the monodromy around that locus acts trivially, but this is not necessary for our later discussion.%, so we will simply work over $\tsing$ instead. 
\end{remark}

\subsubsection{Explicit gluing away from codimension $2$}\label{subsubsec:explicit_gluing_away_from_codimension_2}

When we define the sheaves $\bva{k}^{*}_{\alpha}$'s in \S \ref{subsubsec:local_deformation_data}, the open subset $W_{\alpha}$ is taken to be a sufficiently small neighborhood of $x \in \reint(\tau)$ for some $\tau \in \pdecomp$. In fact, we can choose one of these open subsets to be the large open dense subset $W_0$. In this subsection, we construct the sheaves $\bva{k}^*_0$ and $\tbva{k}{}^*_{0}$ on $W_0$ using an explicit gluing of the underlying complex analytic space.

Over $\reint(\sigma)$ for $\sigma \in \pdecomp^{[n]}$ or over $\mathscr{W}_{\rho}$ for $\rho \in \pdecomp^{[n-1]}$ with $\tsing_e \cap \reint(\rho) = \emptyset$, we have $\bva{k}^*_0 = \sfbva{k}^*_{\mathrm{sf}}$, which was just constructed in \S \ref{subsubsec:semi-flat}. So it remains to consider $\rho \in \pdecomp^{[n-1]}$ such that $\tsing_e \cap \reint(\rho) \neq \emptyset$.
The log structure of $V(\rho)^{\dagger}$ is prescribed by the slab functions $f_{v,\rho} \in \Gamma(\cu{O}_{V_{\rho}(v)})$'s, which restrict to functions $s^{-1}_{v,\rho}(f_{v,\rho})$'s on the torus $\spec(\comp[\tanpoly_{\rho}]) \cong (\comp^*)^{n-1}$. Each of these can be pulled back via the natural projection $\spec(\comp[\rho^{-1}\Sigma_{v}]) \rightarrow \spec(\comp[\tanpoly_{\rho}])$ to give a function on $\spec(\comp[\rho^{-1}\Sigma_{v}])$. In this case, we may fix the log chart $V(\rho)^{\dagger}|_{\modmap^{-1}(\mathscr{W}_{\rho,v})} \rightarrow \spec(\comp[\rho^{-1}P_{v}])^{\dagger}$ given by the equation 
$$
z^m \mapsto
\begin{dcases*} 
 z^{\bar{m}} & \text{if $\langle \check{d}_{\rho}, \bar{m} \rangle \geq 0$ }, \\
 z^{\bar{m}} \big( s^{-1}_{v\rho}(f_{v,\rho})\big)^{\langle \check{d}_{\rho} , \bar{m} \rangle } & \text{if $\langle \check{d}_{\rho}, \bar{m} \rangle \leq 0$ }.
\end{dcases*}
$$
Write $\localmod{k}(\rho)_{v}^{\dagger}$ for the corresponding $k^{\text{th}}$-order thickening in $\spec(\comp[\rho^{-1}P_{v}])$, which gives a local model for smoothing $V(\rho)|_{\modmap^{-1}(\mathscr{W}_{\rho,v})}$ (as in \S \ref{sec:deformation_via_dgBV}). We take 
$$\bva{k}_0^*|_{\mathscr{W}_{\rho,v}} :=\modmap_* \Big(\bigwedge\nolimits^{-*} \Theta_{\localmod{k}(\rho)_v^{\dagger}/\logsk{k}}\Big).$$

We have to specify the gluing on the overlap $\mathscr{W}_{\rho,v} \cap \reint(\sigma_\pm)$ with the adjacent maximal cells $\sigma_{\pm}$. This is given by first using parallel transport through $v$ to relate the monoids $\sigma_{\pm}^{-1}P_v$ and $\rho^{-1}P_v$ as in the semi-flat case, and then an embedding $\spec(\comp[\sigma_{\pm}^{-1}P_{v}/q^{k+1}]) \rightarrow \spec(\comp[\rho^{-1}P_{v}/q^{k+1}])$ via the formula 
\begin{equation}\label{eqn:correction_by_wall_crossing}
z^m \mapsto
\begin{dcases*} 
	s_{\rho\sigma_+}^{-1}(m) z^{m} & \text{for $\sigma_+$ }, \\
	 s_{\rho\sigma_-}^{-1}(m) z^{m} \big( s^{-1}_{v\sigma_-}(f_{v,\rho})\big)^{\langle \check{d}_{\rho} , \bar{m} \rangle } & \text{for $\sigma_-$ }, 
\end{dcases*}
\end{equation}
where $s_{v\sigma_\pm}$, $s_{\rho\sigma_\pm}$ are treated as maps $\tanpoly_{\sigma_{\pm}} \oplus \inte \rightarrow \comp^*$ as before. We observe that there is a commutative diagram of log morphisms
$$
\xymatrix@1{
V(\sigma_{\pm})^{\dagger}|_{\mathscr{D}} \ar[r] \ar[d] &  \localmod{k}(\sigma_{\pm})^{\dagger}|_{\mathscr{D}} \ar[d]\\
V(\rho)^{\dagger}|_{\mathscr{D}} \ar[r] & \localmod{k}(\rho)^{\dagger}|_{\mathscr{D}}
},
$$
where $\mathscr{D} = \modmap^{-1}(\mathscr{W}_{\rho,v} \cap \reint(\sigma_{\pm}))$. The induced isomorphism 
$$\modmap_* \Big(\bigwedge\nolimits^{-*} \Theta_{\localmod{k}(\rho)_v^{\dagger}/\logsk{k}}\Big) \cong \modmap_* \Big(\bigwedge\nolimits^{-*} \Theta_{\localmod{k}(\sigma_{\pm})_v^{\dagger}/\logsk{k}}\Big)$$
of sheaves on the overlap $\mathscr{D}$ then provides the gluing for defining the sheaf $\bva{k}^*_0$ on $W_0$. Hence, we obtain a sheaf $\bva{k}^*_0$ of BV algebras, where the BV structure is inherited from the local models $\spec(\comp[\sigma^{-1}P_v])$ and $\spec(\comp[\rho^{-1}P_v])$. Similarly, we can define the sheaf $\tbva{k}{}^*_0$ of log de Rham forms over $W_0$, together with a relative volume form $\volf{k}_0 \in \tbva{k}{\parallel}^n_0(W_0)$ by gluing the local $\mu_v$'s.

\subsubsection{Relation between the semi-flat dgBV algebra and the log structure}\label{subsubsec:relating_semi_flat_and_actual_construction}

The difference between $\bva{k}^*_0$ and $\sfbva{k}^*_{\mathrm{sf}}$ is that analytic continuation along a path $\gamma$ in $\reint(\sigma_{\pm}) \cup \reint(\rho)$, where $\rho = \sigma_+ \cap \sigma_-$, induces a non-trivial action on $\sfbva{k}^*_{\mathrm{sf}}$ (the semi-flat sheaf) but not on $\bva{k}^*_0$ (the corrected sheaf).
This is because, near a singular point $p \in \Gamma$ of the affine structure on $B$, there is another local model $\bva{k}^*_{\alpha}$ for $p \in W_{\alpha}$ constructed in \ref{subsubsec:local_deformation_data}, where restrictions of sections are invariant under analytic continuation (cf. Example \ref{eg:monodromy_action_on_semi_flat_sheaf}).
This is in line with the philosophy that monodromy is being cancelled by the slab functions $f_{v,\rho}$'s (which we also call \emph{initial wall-crossing factors}). In view of this, we should be able to relate the sheaves $\bva{k}^*_0$ and $\sfbva{k}^*_{\mathrm{sf}}$ by adding back the initial wall-crossing factors $f_{v,\rho}$'s. 

Recall that the slab function $f_{v,\rho}$ is a function on $V_\rho(v) \subset \centerfiber{0}_{\rho}$, whose zero locus is $Z^{\rho}_1 \cap V_{\rho}(v)$ for $\rho$ such that $\tsing_e \cap \reint(\rho) \neq \emptyset$. 
Also recall that, for $\rho$ containing $v$, $\rho_{v}$ is the unique contractible component in $\rho \cap \mathscr{C}^{-1}(B\setminus \tsing)$ such that $v \in \rho_{v}$, as defined in Assumption \ref{assum:existence_of_contraction}. Note that the inverse image $\mu^{-1}(\rho_v) \subset V_{\rho}(v)$ under the generalized moment map $\mu$ is also a contractible open subset. It contains the $0$-dimensional stratum $x_v$ in $V_\rho(v)$ that corresponds to $v$.
Since $f_{v,\rho} (x_v) =1$, we can define $\log(f_{v,\rho})$ in a small neighborhood of $x_v$, and it can further be extended to the whole of $\mu^{-1}(\rho_v) \subset V_{\rho}(v)$ because this subset is contractible. 
Restricting to the open dense torus orbit $\spec(\comp[\tanpoly_{\rho}]) \cap \mu^{-1}(\rho_v)$, we obtain $\log(s_{v\rho}^{-1}(f_{v,\rho}))$, which can in addition be lifted to a section in $\sfbva{k}^0_{\mathrm{sf}}(\mathscr{W}_{\rho,v}) = \Gamma(\mathscr{W}_{\rho,v}, \cu{O}_{\sflocmod{k}(\rho)_{v}})$ for a sufficiently small $\mathscr{W}_{\rho,v}$. 

Now we resolve the sheaves $\bva{k}^*_0$ and $\sfbva{k}^*_{\mathrm{sf}}$ by the complex $\tform^*$ introduced in \S \ref{sec:asymptotic_support}. We let 
$$\sfpolyv{k}^{*,*}_{\mathrm{sf}} := \tform^*|_{W_0} \otimes_{\real} \sfbva{k}^*_{\mathrm{sf}}$$ 
and equip it with $\pdb_{\circ}=d \otimes 1$, $\bvd{}$ and $\wedge$, making it a sheaf of dgBV algebras. Over the open subset $\mathscr{W}_{\rho,v}$, using the explicit description of $\sfbva{k}^*_{\mathrm{sf}}|_{\mathscr{W}_{\rho,v}}$, we consider the element 
\begin{equation}\label{eqn:defining_wall}
	\phi_{v,\rho}:= -\delta_{v,\rho} \otimes \log(s_{v\rho}^{-1}(f_{v,\rho}))\partial_{\check{d}_\rho} \in \sfpolyv{k}^{-1,1}_{\mathrm{sf}}(\mathscr{W}_{\rho,v}),
	\end{equation}
where $\delta_{v,\rho}$ is any $1$-form with asymptotic support in $\reint(\rho)_v$ and whose integral over any curve transversal to $\reint(\rho)_v$ going from $\sigma_-$ to $\sigma_+$ is asymptotically $1$;
such a 1-form can be constructed using a family of bump functions in the normal direction of $\reint(\rho)_v$ similar to Example \ref{example:bump_form} (see also \cite[\S 4]{kwchan-leung-ma}). We can further extend the section $\phi_{v,\rho}$ to the whole $W_0$ by setting it to be $0$ outside a small neighborhood of $\reint(\rho)_v$ in $\mathscr{W}_{\rho,v}$. 

\begin{definition}\label{def:sheaf_of_sf_polyvector}
	The \emph{sheaf of semi-flat polyvector fields} is defined as 
    $$\sfpolyv{k}^{*,*}_{\mathrm{sf}} := \tform^*|_{W_0} \otimes_{\real} \sfbva{k}^*_{\mathrm{sf}},$$ 
    which is equipped with a BV operator $\bvd{}$, a wedge product $\wedge$ (and hence a Lie bracket $[\cdot,\cdot]$) and the operator
	$$
	\pdb_{\mathrm{sf}} := \pdb_{\circ} + [\phi_{\mathrm{in}},\cdot] = \pdb_{\circ} + \sum_{v,\rho} [\phi_{v,\rho}, \cdot],
	$$ 
	where $\pdb_{\circ} = d\otimes 1$ and $\phi_{\mathrm{in}} := \sum_{v,\rho} \phi_{v,\rho}$.
	We also define the \emph{sheaf of semi-flat log de Rham forms} as 
    $$\sftotaldr{k}{}^{*,*}_{\mathrm{sf}}:= \tform^*|_{W_0} \otimes_{\real} \sftbva{k}{}^*_{\mathrm{sf}},$$
    equipped with $\dpartial{}$, $\wedge$, 
	$$
	\pdb_{\mathrm{sf}} := \pdb_{\circ} + \sum_{v,\rho} \cu{L}_{\phi_{v,\rho}},
	$$
	and a contraction action $\mathbin{\lrcorner}$ by elements in $\sfpolyv{k}^{*,*}_{\mathrm{sf}}$. 
\end{definition}

It can be easily checked that $\pdb_{\mathrm{sf}}^2 = [\pdb_{\mathrm{sf}},\bvd{}] = 0$, so we have a sheaf of dgBV algebras. 

On the other hand, we write 
$$\polyv{k}^{*,*}_0 := \tform^*|_{W_0} \otimes_{\real} \bva{k}^*_0,$$
which is equipped with the operators $\pdb_0 = d\otimes 1$, $\bvd{}$ and $\wedge$. The following important lemma is a comparison between the two sheaves of dgBV algebras.

\begin{lemma}\label{lem:comparing_sheaf_of_dgbv}
	There exists a set of compatible isomorphisms
	$$
	\varPhi \colon \polyv{k}^{*,*}_0 \rightarrow \sfpolyv{k}^{*,*}_{\mathrm{sf}},\ k \in \mathbb{N}
	$$
	of sheaves of dgBV algebras such that $\varPhi\circ \pdb_0 = \pdb_{\mathrm{sf}} \circ \varPhi$ for each $k \in \mathbb{N}$.
	
	There also exists a set of compatible isomorphisms 
	$$
	\varPhi \colon \totaldr{k}{}^{*,*}_0 \rightarrow \sftotaldr{k}{}^{*,*}_{\mathrm{sf}},\ k \in \mathbb{N}
	$$
	of sheaves of dgas preserving the contraction action $\mathbin{\lrcorner}$ and such that $\varPhi\circ \pdb_0 = \pdb_{\mathrm{sf}} \circ \varPhi$ for each $k \in \mathbb{N}$. Furthermore, the relative volume form $\volf{k}_0$ is identified via $\varPhi$.
\end{lemma}

\begin{proof}
	Outside those $\reint(\rho)$'s such that $\tsing_e \cap \reint(\rho) \neq \emptyset$, the two sheaves are identical. So we will take a component $\reint(\rho)_v$ of $\reint(\rho) \setminus \tsing$ and compare the sheaves on a neighborhood $\mathscr{W}_{\rho,v}$. 
	
	We fix a point $x \in \reint(\rho)_v$ and describe the map $\Phi$ at the stalks of the two sheaves. First, the preimage $K:=\modmap^{-1}(x) \cong \tanpoly_{\rho,\real}^*/\tanpoly_{\rho}^*$ can be identified as a real $(n-1)$-dimensional torus in $\spec(\comp[\tanpoly_{\rho}]) \cong (\comp^*)^{n-1}$. We have an identification $\rho^{-1}\Sigma_v \cong  \Sigma_{\rho} \times \tanpoly_{\rho}$, and we choose the unique primitive element $m_{\rho} \in \Sigma_{\rho}$ in the ray pointing into $\sigma_+$. 
    As analytic spaces, we write 
    $$\spec(\comp[\Sigma_\rho]) = \{ uv=0\} \subset \comp^2,$$
    where $u = z^{m_\rho}$ and $v = z^{-m_{\rho}}$, and 
    $$\spec(\comp[\rho^{-1}\Sigma_v]) = (\comp^*)^{n-1} \times \{uv = 0 \}.$$ 
    The germ $\cu{O}_{V(\rho),K}$ of analytic functions can be written as 
	$$
	\cu{O}_{V(\rho),K} = \left\{a_0 +  \sum_{i=1}^{\infty} a_i u^i + \sum_{i=-1}^{-\infty}a_i v^{-i} \ \Big| \ a_i \in \cu{O}_{(\comp^*)^{n-1}}(U) \ \text{for neigh. $U\supset K$}, \  \sup_{i\neq 0} \frac{\log|a_i|}{|i|} < \infty  \right\}. 
	$$
	Using the embedding $V(\rho)|_{\modmap^{-1}(\mathscr{W}_{\rho,v})} \rightarrow \localmod{k}(\rho)_v^{\dagger}$ in \S \ref{subsubsec:explicit_gluing_away_from_codimension_2}, we can write
	\begin{align*}
	&\bva{k}^0_{0,x}=\cu{O}_{\localmod{k}(\rho)_v,K}=\\
	& \left\{\sum_{j=0}^k (a_{0,j} +  \sum_{i=1}^{\infty} a_{i,j} u^i + \sum_{i=-1}^{-\infty}a_{i,j} v^{-i}) q^j 
	 \Big| \ a_{i,j} \in \cu{O}_{(\comp^*)^{n-1}}(U)  \ \text{for neigh. $U\supset K$}, \  \sup_{i\neq 0} \frac{\log|a_{i,j}|}{|i|} < \infty  \right\}, 
	 \end{align*}
	with the relation $uv = q^l s_{v\rho}^{-1}(f_{v,\rho})$ (here $l$ is the change of slopes for $\varphi_v$ across $\rho$). For the elements $(m_{\rho},\varphi_v(m_{\rho}))$ and $(-m_{\rho},\varphi_v(-m_{\rho}))$ in $\rho^{-1}P_v$, we have the identities (we omit the dependence on $k$ when we write elements in the stalks of sheaves):
	\begin{align*}
	z^{(m_{\rho},\varphi_v(m_{\rho}))} & = u, \\ z^{-(-m_{\rho},\varphi_v(-m_{\rho}))} & = s_{v\rho}^{-1}(f_{v,\rho})^{-1} v, 
	\end{align*}
	describing the embedding $\localmod{k}(\rho)_v^{\dagger} \hookrightarrow \spec(\comp[\rho^{-1}P_v])^{\dagger}$. 
	For polyvector fields, we can write 
    $$\bva{k}^*_{0,x} = \bva{k}^0_{0,x} \otimes_{\real} \bigwedge^{-*} T_{v,\real}^*.$$ 
    The BV operator is described by the relations $\bvd{}(\partial_n) = 0$, $[\partial_{n_1},\partial_{n_2}] =0$, and 
	\begin{equation}\label{eqn:BV-relations}
		\begin{cases}
	[z^m,\partial_n] = \bvd{}(z^{m}\partial_n)  = \langle m,n\rangle z^m & \text{for $m$ with $\bar{m} \in \tanpoly_{\rho}, \ n \in T^*_{v,\real}$;}\\
	[u,\partial_n] =\bvd{} (u\partial_n)  = \langle m_{\rho},n\rangle u & \text{for $n \in T^*_{v,\real}$};\\
	[v,\partial_n] =\bvd{}(v\partial_n)  =  \langle -m_{\rho},n \rangle v + \partial_n(\log s_{v\rho}^{-1}(f_{v,\rho}))  v & \text{for $n \in T^*_{v,\real}$}.	
	\end{cases}
	\end{equation}

Similarly, we can write down the stalk for $\sfbva{k}^*_{\mathrm{sf},x} = \sfbva{k}^*_{\mathrm{sf},x} \otimes_{\real} \bigwedge^{-*} T_{v,\real}^*$. As a module over $\cu{O}_{(\comp^*)^{n-1},K}\otimes_{\comp} \comp[q]/(q^{k+1})$, we have $\sfbva{k}^*_{\mathrm{sf},x} = \bva{k}^*_{0,x}$; the ring structure on $\sfbva{k}^0_{\mathrm{sf},x}$ differs from that on $\bva{k}^0_{0,x}$ and is determined by the relation $uv = q^l$. The embedding $\sflocmod{k}(\rho)_v^{\dagger} \hookrightarrow \spec(\comp[\rho^{-1}P_v])^{\dagger}$ is given by 
	\begin{align*}
	z^{(m_{\rho},\varphi_v(m_{\rho}))} & = u, \\
	z^{-(-m_{\rho},\varphi_v(-m_{\rho}))} & =  v.
\end{align*}
The formulae for the BV operator are the same as that for $\bva{k}^*_{0,x}$, except that for the last equation in \eqref{eqn:BV-relations}, we have $[v,\partial_n] = \bvd{}(v\partial_n) = \langle -m_\rho,n\rangle v$ instead. 

We apply the argument in \cite[\S 4]{kwchan-leung-ma}, where we considered a scattering diagram consisting of only one wall, to relate these two sheaves. We can find a set of compatible elements $\theta = (\prescript{k}{}{\theta})_{k\in \mathbb{N}}$, where $\prescript{k}{}{\theta} \in \sfpolyv{k}^{-1,0}_{\mathrm{sf}}(\mathscr{W}_{\rho,v})$ for $k\in \mathbb{N}$, such that $e^{\theta}*\pdb_{\circ}  = \pdb_{\mathrm{sf}}$ and $\bvd{}(\theta) = 0$. Explicitly, $\theta$ is a step-function-like section of the form
$$
\theta = \begin{dcases}
	\log(s_{v\rho}^{-1}(f_{v,\rho})) \partial_{\check{d}_{\rho}} & \text{on $\reint(\sigma_+) \cap \mathscr{W}_{\rho,v}$,}\\
	0 & \text{on $\reint(\sigma_-) \cap \mathscr{W}_{\rho,v}$.}
	\end{dcases}
$$
For each $k\in \mathbb{N}$, we also define $\theta_0:= \log(s_{v\rho}^{-1}(f_{v,\rho})) \partial_{\check{d}_{\rho}}$, as an element in $\sfbva{k}^{-1}_{\mathrm{sf}}(\mathscr{W}_{\rho,v})$. Now we define the map $\varPhi_{x} \colon \polyv{k}^{*,*}_{0,x} \rightarrow \sfpolyv{k}_{\mathrm{sf},x}^{*,*}$ at the stalks by writing 
$$\polyv{k}^{*,*}_{0,x} = \tform_x^* \otimes_{\real} \bva{k}^0_{0,x} \otimes_{\real} \bigwedge^{-*} T_{v,\real}^*,$$ 
(and similarly for $\sfpolyv{k}^{*,*}_{\mathrm{sf},x}$), and extending the formulae 
\begin{equation*}
	\begin{cases}
	\varPhi_x(\alpha) = \alpha & \text{for $\alpha \in \tform_x$},\\
	\varPhi_x(f) = e^{[\theta,\cdot]}f = f  & \text{for $f \in \cu{O}_{(\comp^*)^{n-1},K}$}, \\
	\varPhi_x(u) = e^{[\theta-\theta_0,\cdot]} u, & \\
	\varPhi_x(v) = e^{[\theta,\cdot]} v ,& \\
	\varPhi_x(\partial_n) = e^{[\theta-\theta_0,\cdot]} \partial_n & \text{for $n \in T_{v,\real}^*$}
	\end{cases}
\end{equation*}
through the tensor product $\otimes_{\real}$ and skew-symmetrically in $\partial_n$'s.

To see that $\varPhi$ is the desired isomorphism, we check all the relations by computations:
\begin{itemize}
	\item Since $e^{[\theta,\cdot]} \circ \pdb_{\circ} \circ e^{-[\theta,\cdot]} = \pdb_{\mathrm{sf}}$, we have
	$$
	\pdb_{\mathrm{sf}} \varPhi_x(u) = e^{[\theta,\cdot]} \pdb_{\circ} ( e^{-[\theta_0,\cdot]} u) = 0;
	$$
	similarly, we have $\pdb_{\mathrm{sf}}( \varPhi_x(v) ) = 0 = \pdb_{\mathrm{sf}}(\varPhi_x (\partial_n))$. Hence, we have $\varPhi_x \circ \pdb_0 = \pdb_{\mathrm{sf}} \circ \varPhi_x$.
	
	\item We have $e^{-[\theta_0,\cdot]} u = s^{-1}_{v\rho}(f_{v,\rho}) u$ and  
	$$\varPhi_x(u) \varPhi_x(v) = e^{[\theta,\cdot]} (s^{-1}_{v\rho}(f_{v,\rho}) u) e^{[\theta,\cdot]}v =s^{-1}_{v\rho}(f_{v,\rho}) e^{[\theta,\cdot]} (uv) = q^l s^{-1}_{v\rho}(f_{v,\rho}) = \varPhi_x(uv ),
	$$ 
	i.e. the map $\varPhi_x$ preserves the product structure.
	
	\item From the fact that $\bvd{}(\theta) = 0 = \bvd{}(\theta_0)$, we see that $e^{[\theta-\theta_0,\cdot]}$ commutes with $\bvd{}$, and hence $\bvd{}(\varPhi_x(\partial_n)) = e^{[\theta-\theta_0,\cdot]} \bvd{}(\partial_n)=0$. We also have $[\varPhi_x(\partial_{n_1}),\varPhi_x(\partial_{n_2})] = e^{[\theta-\theta_0,\cdot]} [\partial_{n_1},\partial_{n_2}] = 0$.
	
	\item Again from $\bvd{}(\theta) = 0 = \bvd{}(\theta_0)$, we have
	\begin{align*}
		\bvd{}(\varPhi_x(u) \varPhi_x(\partial_n)) & = \bvd{}(e^{[\theta-\theta_0,\cdot]} (u\partial_n)) =e^{[\theta-\theta_0,\cdot]} \left( \bvd{}(u\partial_n) \right) \\
		& = \langle m_{\rho} , n \rangle e^{[\theta-\theta_0,\cdot]}(u) = \langle m_{\rho} , n \rangle \varPhi_x(u) = \varPhi_x(\bvd{}(u\partial_n)).
	\end{align*}
	
	\item Finally, we have 
	\begin{align*}
		\bvd{}(\varPhi_x(v) \varPhi_x(\partial_n)) & = \bvd{}\big(e^{[\theta-\theta_0,\cdot]} ( (e^{[\theta_0,\cdot]}v) \partial_n)\big) 
		=e^{[\theta-\theta_0,\cdot]} \big( \bvd{}(s_{v\rho}^{-1}(f_{v,\rho}) v \partial_n) \big)\\
		& = e^{[\theta-\theta_0,\cdot]} \big(\langle -m_{\rho}, n \rangle s^{-1}_{v\rho}(f_{v,\rho}) v + \partial_n(s^{-1}_{v\rho}(f_{v,\rho})) v \big) \\
		& = \langle -m_{\rho}, n \rangle (e^{[\theta,\cdot]}v) + \partial_n \big(\log s^{-1}_{v\rho}(f_{v,\rho}) \big) (e^{[\theta,\cdot]}v)\\ 
		& = \langle -m_{\rho}, n \rangle \varPhi_x(v) + \partial_n \big(\log s^{-1}_{v\rho}(f_{v,\rho}) \big) \varPhi_x(v)\\
		& = \varPhi_x(\bvd{}(v\partial_n)).
	\end{align*}
\end{itemize}
We conclude that  $\varPhi_x \colon \polyv{k}^{*,*}_{0,x} \rightarrow \sfpolyv{k}_{\mathrm{sf},x}^{*,*}$ is an isomorphism of dgBV algebras. 
We need to check that the map $\varPhi_x$ agrees with the isomorphism $\polyv{k}^{*,*}_{0}|_{\mathscr{C}} \rightarrow \sfpolyv{k}^{*,*}_{\mathrm{sf}}|_{\mathscr{C}}$ induced simply by the identity $\bva{k}^*_0|_{\mathscr{C}} \cong \sfbva{k}^*_{\mathrm{sf}}|_{\mathscr{C}}$, where $\mathscr{C} = W_0 \setminus \bigcup_{\tsing_e \cap \reint(\rho) \neq \emptyset} \reint(\rho)$. For this purpose, we consider two nearby maximal cells $\sigma_{\pm}$ such that $\sigma_+ \cap \sigma_- = \rho$. We have $\localmod{k}(\sigma_{\pm}) = \spec(\comp[\sigma^{-1}_{\pm} P_v]/q^{k+1})$, and the gluing of $\bva{k}^*_0$ over $\mathscr{W}_{\rho,v} \cap \sigma_+$ is given by parallel transporting through $v$, and then by the formulae
\begin{equation}\label{eqn:PV-gluing}
	\begin{cases}
	z^{m} \mapsto s^{-1}_{\rho\sigma_{+}}(m) z^m  & \text{for $m \in \tanpoly_{\rho}$},\\
	u \mapsto s^{-1}_{\rho\sigma_+}(m_{\rho}) z^{m_{\rho}}, & \\
	v \mapsto q^{l} s^{-1}_{v\sigma_+}(f_{v,\rho}) s^{-1}_{\rho\sigma_+}(-m_{\rho}) z^{-m_{\rho}}.
	\end{cases} 
\end{equation}
The only difference for gluing of $\sfbva{k}^*_{\mathrm{sf}}$ is the last equation in \eqref{eqn:PV-gluing}, which is now replaced by the formula $v \mapsto  q^{l} s^{-1}_{\rho\sigma_+}(-m_{\rho}) z^{-m_{\rho}}$. Since we have 
$$
\varPhi_x(v) = \begin{dcases}
	s_{v\rho}^{-1}(f_{v,\rho}) v & \text{on $U_x \cap \reint(\sigma_+)$}, \\
	v  & \text{on $U_x \cap \reint(\sigma_-)$}
\end{dcases}
$$
on a sufficiently small neighborhood $U_x$ of $x$, we see that $\varPhi_x(v) \mapsto  q^{l} s^{-1}_{v\sigma_+}(f_{v,\rho}) s^{-1}_{\rho\sigma_+}(-m_{\rho}) z^{-m_{\rho}}$ under the gluing map of $\sfbva{k}^*_{\mathrm{sf}}$ on $U_x \cap \reint(\sigma_+)$. This shows the compatibility of $\varPhi_x$ with the gluing of $\bva{k}^*_0$ and $\sfbva{k}^*_{\mathrm{sf}}$ over $U_x \cap \reint(\sigma_+)$. A similar argument applies for $U_x \cap \reint(\sigma_-)$.  

The proof for $\varPhi\colon \totaldr{k}{}^{*,*}_0 \rightarrow \sftotaldr{k}{}^{*,*}_{\mathrm{sf}}$ is similar and will be omitted. The volume form is preserved under $\varPhi$ because we have $\bvd{}(\theta) = 0 = \bvd{}(\theta_0)$. This completes the proof of the lemma.
\end{proof}

\subsubsection{A global sheaf of dgLas from gluing of the semi-flat sheaves}\label{subsubsec:gluing_using_semi_flat}

We shall apply the procedure described in \S \ref{subsubsec:gluing_construction} to the semi-flat sheaves to glue a global sheaf of dgLas. First of all, we choose an open cover $\{W_{\alpha}\}_{\alpha \in \mathscr{I}}$ satisfying the Condition \ref{cond:good_cover_of_B}, together with a decomposition $\mathscr{I} = \mathscr{I}_1 \sqcup \mathscr{I}_2$ such that $\mathcal{W}_1 = \{W_{\alpha}\}_{\alpha \in \mathscr{I}_1}$ is a cover of the semi-flat part $W_0$, and $\mathcal{W}_2 = \{W_{\alpha}\}_{\alpha \in \mathscr{I}_2}$ is a cover of a neighborhood of $\big( \bigcup_{\tau \in \pdecomp^{[n-2]}}\tau \big) \cup  \big( \bigcup_{\rho \cap \tsing_e \neq \emptyset} \tsing \cap \reint(\rho) \big)$. 

For each $W_{\alpha}$, we have a compatible set of local sheaves $\bva{k}^*_{\alpha}$ of BV algebras, local sheaves $\tbva{k}{}^*_{\alpha}$ of dgas, and relative volume elements $\volf{k}_{\alpha}$, $k \in \mathbb{N}$ (as in \S \ref{subsubsec:local_deformation_data}). We can further demand that, over the semi-flat part $W_0$, we have $\bva{k}^*_{\alpha} = \bva{k}^*_0|_{W_{\alpha}}$, $\tbva{k}{}^*_{\alpha} = \tbva{k}{}^*_0|_{W_{\alpha}}$ and $\volf{k}_{\alpha} = \volf{k}_0 |_{W_{\alpha}}$, and hence $\polyv{k}^{*,*}_{\alpha} = \polyv{k}^{*,*}_0|_{W_{\alpha}}$ and $\totaldr{k}{}^{*,*}_{\alpha} = \totaldr{k}{}^{*,*}_0|_{W_{\alpha}}$ for $\alpha \in \mathscr{I}_1$.

Using the construction in \S \ref{subsubsec:gluing_construction}, we obtain a Gerstenhaber deformation $\glue{k}_{\alpha\beta} = e^{[\theta_{\alpha\beta},\cdot]} \circ \patch{k}_{\alpha\beta}$ specified by $\theta_{\alpha\beta} \in \polyv{k}^{-1,0}_\beta(W_{\alpha\beta})$, which give rise to sets of compatible global sheaves $\polyv{k}^{*,*}$ and $\totaldr{k}{}^{*,*}$, $k \in \mathbb{N}$. Restricting to the semi-flat part, we get two Gerstenhaber deformations $\polyv{k}^{*,*}_0$ and $\polyv{k}^{*,*}|_{W_{0}}$, which must be equivalent as $\check{H}^{>0}(\mathcal{W}_1, \polyv{0}^{-1,0}|_{W_0}) = 0$. So we have a set of compatible isomorphisms locally given by $h_{\alpha} = e^{[\mathbf{b}_\alpha,\cdot]}\colon \polyv{k}^{*,*}_0|_{W_{\alpha}} \rightarrow \polyv{k}^{*,*}|_{W_{\alpha}} \cong \polyv{k}^{*,*}_{\alpha} $ for some $\mathbf{b}_{\alpha} \in \polyv{k}^{-1,0}_0(W_{\alpha})$, for each $k \in \mathbb{N}$, and they fit into the following commutative diagram
$$
\xymatrix@1{ \polyv{k}^{*,*}_0|_{W_{\alpha\beta}} \ar[r]^{\mathrm{id}} \ar[d]^{h_{\alpha}} &\polyv{k}^{*,*}_0|_{W_{\alpha\beta}} \ar[d]^{h_{\beta}} \\
	\polyv{k}^{*,*}_{\alpha}|_{W_{\alpha\beta}} \ar[r]^{\glue{k}_{\alpha\beta}} & \polyv{k}^{*,*}_{\beta}|_{W_{\alpha\beta}}.
}
$$
%For each $k \in \mathbb{N}$, 
Since the pre-differential on $\polyv{k}^{*,*}|_{W_{0}}$ obtained from the construction in \S \ref{subsubsec:gluing_construction} is of the form $\pdb_\alpha + [\eta_{\alpha},\cdot]$ for some $\eta_{\alpha} \in \polyv{k}^{-1,1}_{0}(W_{\alpha})$, pulling back via $h_{\alpha}$ gives a global element $\eta \in \polyv{k}^{-1,1}_{0}(W_0)$ such that
$$
h_{\alpha}^{-1}\circ (\pdb_\alpha + [\eta_{\alpha},\cdot]) \circ h_{\alpha} = \pdb_0 + [\eta,\cdot]. 
$$
Theorem \ref{prop:Maurer_cartan_equation_unobstructed} gives a Maurer--Cartan solution $\phi \in \polyv{k}^{-1,1}(B)$ such that $(\pdb+[\phi,\cdot])^2=0$, together with a holomorphic volume form $e^{f} \volf{}$, compatible for each $k$. We denote the pullback of $\phi$ under $h_{\alpha}$'s to $\polyv{k}^{-1,1}_0(W_0)$ as $\phi_0$, and that of volume form to $\totaldr{k}{\parallel}^{n,0}_0(W_0)$ as $e^{g} \volf{}_0$. We see that the equation
$$(\pdb_0 + \cu{L}_{\eta+\phi_0}) e^{g} \volf{}_0=0$$
is satisfied, or equivalently, that $\eta+\phi_0 + tg$ is a solution to the extended Maurer--Cartan equation \ref{eqn:extended_maurer_cartan_equation}.

\begin{lemma}\label{lem:vector field}
	If the holomorphic volume form $e^{f} \volf{}$ is normalized in the sense of Definition \ref{def:normalized_volume_form}, then we can find a set of compatible $\cu{V} \in \polyv{k}^{-1,0}_0(W_0)$, $k \in \mathbb{N}$ such that 
	$$
	e^{-\cu{L}_{\cu{V}}} \volf{}_0 = e^{g} \volf{}_0. 
	$$
	As a consequence, the Maurer--Cartan solution $\eta+\phi_0 + tg$ is gauge equivalent to a solution of the form $\zeta_0 + t\cdot 0$ for some $\zeta_0 \in \polyv{k}^{-1,1}_0(W_0)$, via the gauge transformation $e^{[\cu{V},\cdot]} \colon \polyv{k}^{*,*}_0 \rightarrow \polyv{k}^{*,*}_0$.  
\end{lemma}

\begin{proof}
	We should construct $\cu{V}$ by induction on $k$ as in the proof of Lemma \ref{lem:preserving_volume_element_by_vector_fields}. Namely, suppose $\cu{V}$ is constructed for the $(k-1)^{\text{st}}$-order, then we shall lift it to the $k^{\text{th}}$-order. 
	We prove the existence of a lifting $\cu{V}_x \in \polyv{k}^{-1,0}_{0,x}$ at every stalk $x \in W_0$ and use partition of unity to glue a global lifting $\cu{V}$. 
 
    First of all, we can always find a gauge transformation $\theta \in \polyv{k}^{-1,0}_{0,x}$ such that 
	$$
	e^{-[\theta,\cdot]} \circ \pdb_0 \circ e^{[\theta,\cdot]} = \pdb_0 + [\eta+\phi_0,\cdot]. 
	$$
	So we have $\pdb_0 (e^{\cu{L}_\theta} e^{g} \volf{}_0)=0$, which implies that $e^{\cu{L}_\theta} e^{g} \volf{}_0 \in \tbva{k}{\parallel}^n_{0,x}$. We can write $e^{\cu{L}_\theta} e^{g} \volf{}_0 = e^{h} \volf{}_0$ in the stalk at $x$ for some germ $h\in \bva{k}^0_{0,x}$ of holomorphic functions. Applying Lemma \ref{lem:preserving_volume_element_by_vector_fields}, we can further choose $\theta$ so that $h=h(q) \in (q)\subset \comp[q]/q^{k+1}$. In a sufficiently small neighborhood $U_x$, we find an element $\varrho_x \in \tform^n(U_x)$ as in Definition \ref{def:normalized_volume_form}. The fact that the volume form is normalized forces $e^{h(q)} [\volf{}_0\wedge \varrho_x]$ to be constant with respect to the Gauss--Manin connection $\gmc{k}$. Tracing through the exact sequence \eqref{eqn:Gauss_Manin_connection_definition} on $U_x$, we can lift $\volf{}_0$ to $\tbva{k}{}^n_{0}(U_x)$ which is closed under $\dpartial{}$. As a consequence, we have $\gmc{k}_{\dd{\log q}}[\volf{}_0\wedge \varrho_x] = 0$, and hence we conclude that $h(q) =0$.  
	
	Now we have to solve for a lifting $\cu{V}_x$ such that $e^{\cu{L}_{\theta}}e^{-\cu{L}_{\cu{V}_x}} \volf{}_0 = \volf{}_0$ up to the $k^{\text{th}}$-order. This is equivalent to solving for a lifting $u$ satisfying $e^{\cu{L}_{u}} \volf{}_0 = \volf{}_0$ for the $k^{\text{th}}$-order once the $(k-1)^{\text{st}}$-order is given. Take an arbitrary lifting $\tilde{u}$ to the $k^{\text{th}}$-order, and making use of the formula in \cite[Lem. 2.8]{chan2019geometry}, we have
	$$
	e^{\cu{L}_{\tilde{u}}} \volf{}_0 = \exp\left( \sum_{s=0}^{\infty} \frac{\delta_{\tilde{u}}^{s}}{(s+1)!} \bvd{}(\tilde{u}) \right) \volf{}_0,
	$$  
	where $\delta_{\tilde{u}} = -[\tilde{u},\cdot]$. From $e^{\cu{L}_{\tilde{u}}} \volf{}_0 = \volf{}_0 \ (\text{mod $\mathbf{m}^k$})$, we use induction on the order $j$ to prove that $\bvd{}(\tilde{u})=0$ up to order $(k-1)$. Therefore we can write 
    $$\bvd{}(\tilde{u}) =q^{k} \bvd{}(\breve{u}) \ (\text{mod $\mathbf{m}^k$})$$ 
    for some $\breve{u} \in \polyv{0}^{-1,0}_{0,x}$, by the fact that the cohomology sheaf under $\bvd{}$ is free over $\cfrk{k}= \comp[q]/(q^{k+1})$ (see the discussion right after Condition \ref{cond:holomorphic_poincare_lemma}). Setting $u = \tilde{u} - q^{k} \breve{u}$ will then solve the equation. 
\end{proof}

The element $\cu{V}$ obtained in Lemma \ref{lem:vector field} can be used to conjugate the operator $\pdb_0+[\eta+\phi_0,\cdot]$ to get $\pdb_0 + [\zeta_0,\cdot]$, i.e.
$$
e^{-[\cu{V},\cdot]}\circ (\pdb_0 + [\zeta_0,\cdot]) \circ e^{[\cu{V},\cdot]} = \pdb_0+[\eta+\phi_0,\cdot].
$$
The volume form $\volf{}_0$ will be holomorphic under the operator $\pdb_0 + [\zeta_0,\cdot]$. From the equation \eqref{eqn:volume_form_equation}, we observe that $\bvd{}(\zeta_0) = 0$. 
Furthermore, the image of $\zeta_0$ under the isomorphism $\varPhi \colon \polyv{k}^{*,*}_0 \rightarrow \sfpolyv{k}^{*,*}_{\mathrm{sf}}$ in Lemma \ref{lem:comparing_sheaf_of_dgbv} gives $\phi_{\mathrm{s}} \in \sfpolyv{k}^{-1,1}_{\mathrm{sf}}(W_0)$, and an operator of the form 
\begin{equation}\label{eqn:dbar_operator_on_semi-flat}
	\pdb_{\circ} + [\phi_{\mathrm{in}} + \phi_{\mathrm{s}},\cdot] = \pdb_{\circ}+ \sum_{v,\rho} [\phi_{v,\rho},\cdot] + [\phi_{\mathrm{s}},\cdot],
\end{equation}
where $\phi_{\mathrm{in}} = \sum_{v,\rho} \phi_{v,\rho}$, that acts on $\sfpolyv{k}^{*,*}_{\mathrm{sf}}$.

Equipping with this operator, the semi-flat sheaf $\sfpolyv{k}^{*,*}_{\mathrm{sf}}$ can be glued to the sheaves $\polyv{k}^{*,*}_{\alpha}$'s for $\alpha \in \mathscr{I}_2$, preserving all the operators. More explicitly, on each overlap $W_{0\alpha}:=W_0 \cap W_{\alpha}$, we have 
\begin{equation}\label{eqn:gluing_of_semi-flat_sheaf_to_others}
\glue{k}_{0\alpha} \colon \sfpolyv{k}^{*,*}_{\mathrm{sf}}|_{W_{0\alpha}}\rightarrow \polyv{k}^{*,*}|_{W_{0\alpha}}
\end{equation}
defined by 
$$\glue{k}_{\alpha\beta}\circ \glue{k}_{0\alpha}|_{W_{\alpha\beta}} := h_{\beta}\circ e^{-[\cu{V},\cdot]} \circ  \varPhi^{-1}|_{W_{\alpha\beta}}$$ 
for $\beta \in \mathscr{I}_1$, which sends the operator $\pdb_{\circ} + [\phi_{\mathrm{in}} + \phi_{\mathrm{s}},\cdot]$ to $\pdb_\alpha + [\eta_{\alpha}+\phi,\cdot]$. 

\begin{definition}\label{def:semi_flat_tropical_vertex_lie_algebra}
	We call $\sftropv{k}^*_{\mathrm{sf}}:= \mathrm{Ker}(\bvd{})[-1] \subset \sfpolyv{k}^{-1,*}_{\mathrm{sf}}[-1]$, equipped with the structure of a dgLa using $\pdb_{\circ}$ and $[\cdot,\cdot]$ inherited from $\sfpolyv{k}^{-1,*}_{\mathrm{sf}}$, the \emph{sheaf of semi-flat tropical vertex differential graded Lie algebras (abbrev.\ as sf-TVdgLa)}. 
\end{definition}
Note that $\sftropv{k}^*_{\mathrm{sf}} \cong \tform^*|_{W_0} \otimes_{\real} \sftvbva{k}$.
Also, we have $\bvd{}(\phi_{\mathrm{s}})=0$ since $\bvd{}(\zeta_0)=0$, and a direct computation shows that $\bvd{}(\phi_{\mathrm{in}})=0$. Thus $\phi_{\mathrm{in}}, \phi_{\mathrm{s}} \in \sftropv{k}^{1}_{\mathrm{sf}}(W_0)$, and the operator $\pdb_{\circ} + [\phi_{\mathrm{in}}+\phi_{\mathrm{s}},\cdot]$ preserves the sub-dgLa $\sftropv{k}^*_{\mathrm{sf}}$.

From the description of the sheaf $\tform^*$, we can see that locally on $U \subset W_0$, $\phi_{\mathrm{s}}$ is supported on finitely many codimension one polyhedral subsets, called \emph{walls} or \emph{slabs}, which are constituents of a scattering diagram. This is why we use the subscript `s' in $\phi_{\mathrm{s}}$, which stands for `scattering'.%We should study the relation between consistency of wall-crossing and Maurer-Cartan equation by as we did in \cite{kwchan-leung-ma}. 

\subsection{Consistent scattering diagrams and Maurer--Cartan solutions}\label{subsec:scattering_diagram_from_MC_solution}

\subsubsection{Scattering diagrams}\label{subsubsec:scattering_diagram}

In this subsection, we recall the notion of scattering diagrams introduced by Kontsevich--Soibelman \cite{kontsevich-soibelman04} and Gross--Siebert \cite{gross2011real}, and make modifications to suit our needs. We begin with the notion of walls from \cite[\S 2]{gross2011real}. Let 
$$\nsf = \left( \bigcup_{\tau \in \pdecomp^{[n-2]}}\tau \right) \cup  \left( \bigcup_{\substack{\rho \in \pdecomp^{[n-1]}\\ \rho \cap \tsing_e \neq \emptyset}} \tsing \cap \reint(\rho) \right)$$
be equipped with a polyhedral decomposition induced from $\pdecomp$ and $\tsing$. For the exposition below, we will always fix $k>0$ and consider all these structures modulo $\mathbf{m}^{k+1} = (q^{k+1})$.   

\begin{definition}\label{def:walls}
	A \emph{wall} $(\mathbf{w},\sigma_{\mathbf{w}}, \check{d}_{\mathbf{w}}, \Theta_{\mathbf{w}})$ consists of
    \begin{itemize}
    \item a maximal cell $\sigma_\mathbf{w}\in \pdecomp^{[n]}$, 
    \item a closed $(n-1)$-dimensional tropical polyhedral subset $\mathbf{w}$ of $\sigma_{\mathbf{w}}$ such that 
	$$\reint(\mathbf{w}) \cap \left( \bigcup_{\substack{\rho \in \pdecomp^{[n-1]}\\ \rho \cap \tsing_e \neq \emptyset}} \reint(\rho) \right) = \emptyset,$$ 
    \item a choice of a primitive normal $\check{d}_{\mathbf{w}}$, and 
    \item a section $\Theta_{\mathbf{w}}$ of the {\em tropical vertex group} $\exp(q \cdot \sftvbva{k})$ over a sufficiently small neighborhood of $\mathbf{w}$. 
    \end{itemize}
    We call $\Theta_{\mathbf{w}}$ the \emph{wall-crossing factor associated to the wall $\mathbf{w}$}.
    We may write a wall as $(\mathbf{w},\Theta_{\mathbf{w}})$ for simplicity. 
\end{definition}

A wall cannot be contained in $\rho$ with $\rho \cap \tsing_e \neq \emptyset$. We define a notion of slabs for these subsets of codimension one strata $\rho$ intersecting $\tsing_e$. The difference is that we have an extra term $\varTheta_{v,\rho}$ coming from the slab function $f_{v,\rho}$. 

\begin{definition}\label{def:slabs}
	A \emph{slab} $(\mathbf{b}, \rho_{\mathbf{b}}, \check{d}_{\rho}, \varXi_{\mathbf{b}})$ consists of 
    \begin{itemize}
    \item an $(n-1)$-cell $\rho_{\mathbf{b}} \in \pdecomp^{[n-1]}$ such that $\rho_{\mathbf{b}} \cap \tsing_e \neq \emptyset$, 
    \item a closed $(n-1)$-dimensional tropical polyhedral subset $\mathbf{b}$ of $\rho_{\mathbf{b}} \setminus (\rho_{\mathbf{b}} \cap \tsing)$,
    \item a choice of a primitive normal $\check{d}_{\rho}$, and 
    \item a section $\varXi_{\mathbf{b}}$ of $\exp(q\cdot \sftvbva{k})$ over a sufficiently small neighborhood of $\mathbf{b}$. 
    \end{itemize}
    The \emph{wall-crossing factor associated to the slab $\mathbf{b}$} is given by
	$$
	\Theta_{\mathbf{b}}:=\varTheta_{v,\rho} \circ \varXi_{\mathbf{b}},
	$$
	where $v$ is the unique vertex such that $\reint(\rho)_v$ contains $\mathbf{b}$ and
	$$
	\varTheta_{v,\rho}  =   \exp([\log(s_{v\rho}^{-1} (f_{v,\rho})) \partial_{\check{d}_{\rho}},\cdot])
	$$
	(cf. equation \eqref{eqn:defining_wall}). 
    We may write a slab as $(\mathbf{b},\Theta_{\mathbf{b}})$ for simplicity. 
\end{definition}

\begin{remark}
	In the above definition, a slab is \emph{not} allowed to intersect the singular locus $\tsing$. This is different from the situation in \cite[\S 2]{gross2011real}. However, in our definition of consistent scattering diagrams, we will require consistency around each stratum of $\tsing$.
\end{remark}

\begin{example}\label{eg:wall_and_slab}
    We consider the 3-dimensional example shown in Figure \ref{fig:wall_and_slab}, from which we can see possible supports of the walls and slabs.
    There are two adjacent maximal cells intersecting at $\rho  \in \pdecomp^{[n-1]}$ with $\tsing_e \cap \rho= \tsing \cap \rho$ colored in red. 
    The $2$-dimensional polyhedral subsets colored in blue can support walls and the polyhedral subset colored in green can support a slab because it is lying inside $\rho$ with $\tsing_e \cap \rho \neq \emptyset$. 
	\begin{figure}[h!]
		\includegraphics[scale=0.2]{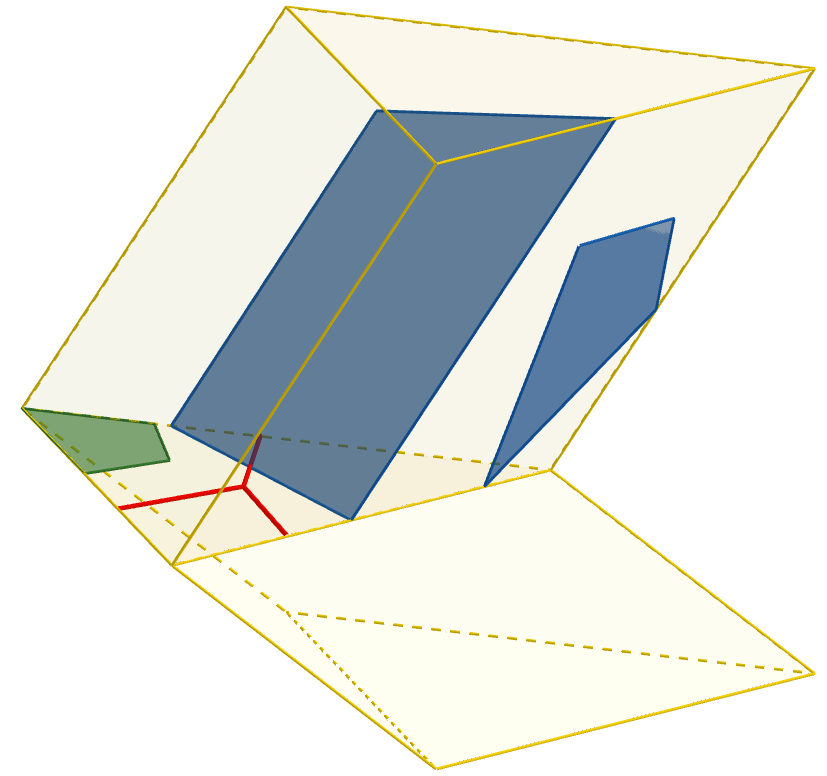}
		\caption{Supports of walls/slabs}\label{fig:wall_and_slab}
	\end{figure}
\end{example}

\begin{definition}\label{def:scattering_diagram}
	A \emph{($k^{\text{th}}$-order) scattering diagram} is a countable collection 
    $$\mathscr{D} = \{(\mathbf{w}_i,\Theta_i)\}_{i\in \mathbb{N}} \cup \{(\mathbf{b}_j,\Theta_{j})\}_{j\in \mathbb{N}}$$ 
    of walls or slabs such that the intersections of any two walls/slabs is at most an $(n-2)$-dimensional tropical polyhedral subset, and $\{\mathbf{w}_i \cap W_0 \}_{i\in \mathbb{N}} \cup \{ \mathbf{b}_j\cap W_0\}_{j\in \mathbb{N}}$ is locally finite in $W_0$.%\footnote{Recall that our notion of scattering diagrams is more relaxed than the usual one defined in \cite{kontsevich-soibelman04, gross2011real}, as explained in Remark \ref{rmk:relaxed_scattering_diagram}.}
\end{definition}

Our notion of scattering diagrams is more flexible than the one defined in \cite{kontsevich-soibelman04, gross2011real} in two ways: 
First, there is no relation between the affine direction orthogonal to a wall $\mathbf{w}$ or a slab $\mathbf{b}$ and its wall crossing factor. As a result, we cannot allow overlapping of walls/slabs in their relative interior because in that case their associated wall crossing factors are not necessarily commuting.
Second, we only require that the intersection of $\mathscr{D}$ with $W_0$ is a locally finite collection of $W_0$, which implies that we allow a possibly infinite number of walls/slabs approaching strata of $\hat{\tsing}$. In the construction of the scattering diagram $\mathscr{D}(\varphi)$ associated to a Maurer--Cartan solution $\varphi$ below, all the walls/slabs will be compact subsets of $W_0$. These walls will not intersect $\hat{\tsing}$, as illustrated in Figure \ref{fig:wall_and_slab}. However, there could be a union of infinitely many walls limiting to some strata of $\hat{\tsing}$.
See also Remark \ref{rmk:relaxed_scattering_diagram}.

\begin{example}
    For the 2-dimensional example shown in Figure \ref{fig:infinite_many_walls}, we see a vertex $v$ and its adjacent cells, and the singular locus $\tsing_e$ consists of the red crosses. In our version of scattering diagrams, we allow infinitely many intervals limiting to $\{v\}$ or $\tsing_e$. 
	\begin{figure}[h!]
		\includegraphics[scale=0.4]{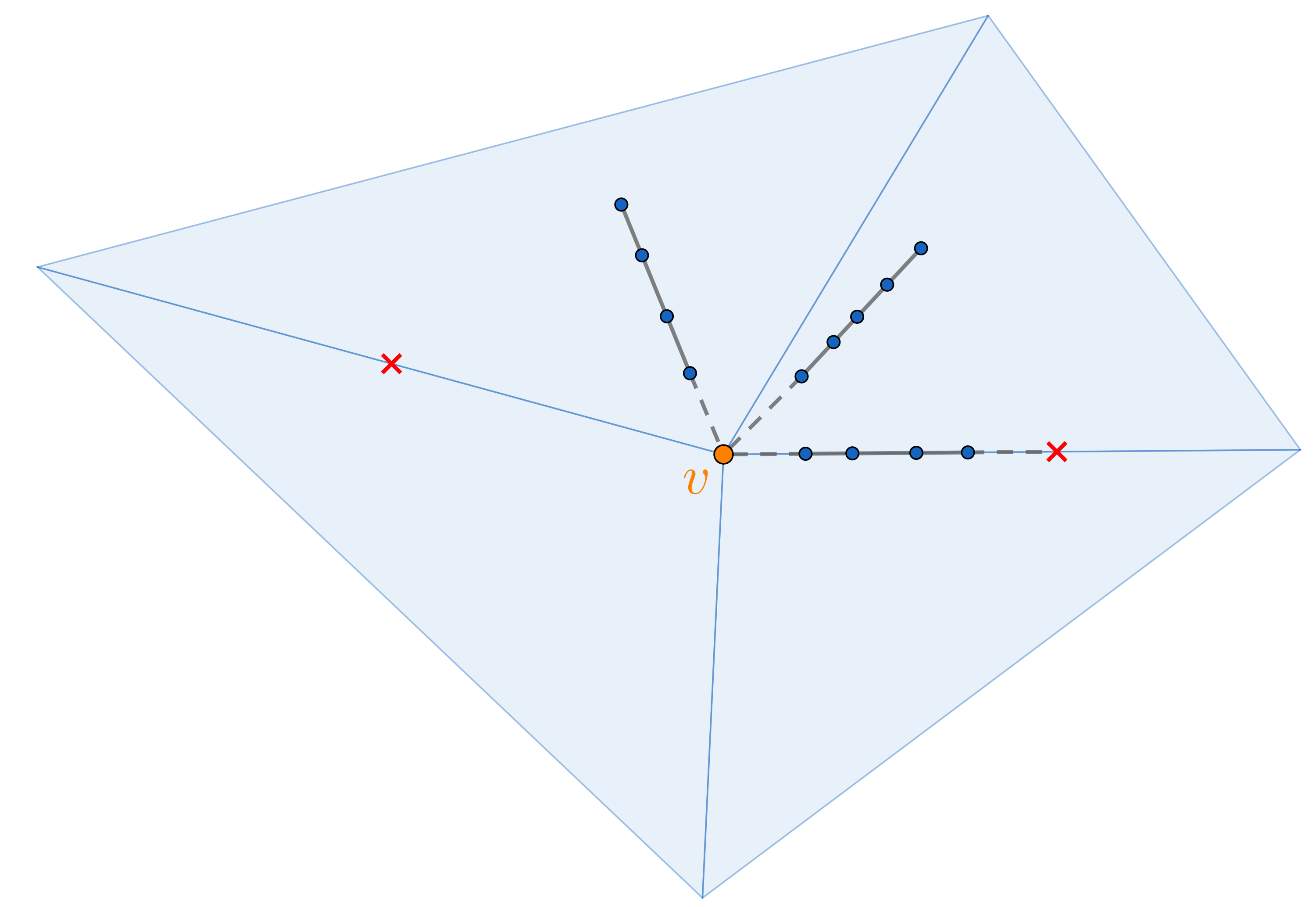}
		\caption{Walls/slabs around $\hat{\tsing}$}\label{fig:infinite_many_walls}
	\end{figure}
\end{example}

Given a scattering diagram $\mathscr{D}$, we can define its \emph{support} as $|\mathscr{D}|:= \bigcup_{i \in \bb{N}} \mathbf{w}_i \cup \bigcup_{j \in \bb{N}} \mathbf{b}_j$. There is an induced polyhedral decomposition on $|\mathscr{D}|$ such that its $(n-1)$-cells are closed subsets of some walls or slabs, and all intersections of walls or slabs are lying in the union of the $(n-2)$-cells. We write $|\mathscr{D}|^{[i]}$ for the collection of all the $i$-cells in this polyhedral decomposition. We may assume, after further subdividing the walls or slabs in $\mathscr{D}$ if necessary, that every wall or slab is an $(n-1)$-cell in $|\mathscr{D}|$. We call an $(n-2)$-cell $\mathfrak{j} \in |\mathscr{D}|^{[n-2]}$ a \emph{joint}, and a connected component of $W_0 \setminus |\mathscr{D}|$ a \emph{chamber}. 

Given a wall or slab, we shall make sense of \emph{wall crossing} in terms of jumping of holomorphic functions across it. Instead of formulating the definition in terms of path-ordered products of elements in the tropical vertex group as in \cite{gross2011real}, we will express it in terms of the action by the tropical vertex group on the local sections of $\sfbva{k}^0_{\mathrm{sf}}$. There is no harm in doing so since we have the inclusion $\sfbva{k}^{-1}_{\mathrm{sf}} \hookrightarrow \mathrm{Der}(\sfbva{k}^0_{\mathrm{sf}},\sfbva{k}^0_{\mathrm{sf}})$, i.e. a relative vector field is determined by its action on functions. 

In this regard, we would like to define the \emph{($k^{\text{th}}$-order) wall-crossing sheaf} $\wcs{k}_{\mathscr{D}}$ on the open set
$$W_{0}(\mathscr{D}):= W_0 \setminus \bigcup_{\mathfrak{j} \in |\mathscr{D}|^{[n-2]}} \mathfrak{j},$$
which captures the jumping of holomorphic functions described by the wall-crossing factor when crossing a wall/slab. We first consider the sheaf $\sfbva{k}^{0}_{\mathrm{sf}}$ of holomorphic functions over the subset $W_0 \setminus |\mathscr{D}|$, and let 
$$\wcs{k}_{\mathscr{D}}|_{W_0 \setminus |\mathscr{D}|} := \sfbva{k}^{0}_{\mathrm{sf}}|_{W_0 \setminus |\mathscr{D}|}.$$
To extend it through the walls/slabs, we will specify the analyic continuation through $\reint(\mathbf{w})$ for each $\mathbf{w} \in |\mathscr{D}|^{[n-1]}$. 
Given a wall/slab $\mathbf{w}$ with two adjacent chambers $\cu{C}_+$, $\cu{C}_-$ and $\check{d}_{\mathbf{w}}$ pointing into $\cu{C}_+$, and a point $x \in \reint(\mathbf{w})$ with the germ $\Theta_{\mathbf{w},x}$ of wall-crossing factors near $x$, we let 
$$\wcs{k}_{\mathscr{D},x} := \sfbva{k}^0_{\mathrm{sf},x},$$
but with a different gluing to nearby chambers $\cu{C}_{\pm}$: in a sufficiently small neighborhood $U_x$ of $x$, the gluing of a local section $f \in \wcs{k}_{\mathscr{D},x}$ is given by
\begin{equation}\label{eqn:wall_crossing_gluing}
	f|_{U_x \cap \cu{C}_{\pm}} := \begin{dcases}
		\Theta_{\mathbf{w},x}(f)|_{U_x \cap \cu{C}_+} & \text{on $U_x \cap \cu{C}_+$,}\\
		f|_{U_x \cap \cu{C}_-} & \text{on $U_x \cap \cu{C}_-$.}
	\end{dcases}
\end{equation}
In this way, the sheaf $\wcs{k}_{\mathscr{D}}|_{W_0 \setminus |\mathscr{D}|}$ extends to $W_0(\mathscr{D})$. 

Now we can formulate consistency of a scattering diagram $\mathscr{D}$ in terms of the behaviour of the sheaf $\wcs{k}_{\mathscr{D}}$ over the joints $\mathfrak{j}$'s and $(n-2)$-dimensional strata of $\nsf$. More precisely, we consider the push-forward $\mathfrak{i}_*(\wcs{k}_{\mathscr{D}})$ along the embedding $\mathfrak{i} \colon W_0(\mathscr{D}) \rightarrow B$, and its stalk at $x \in \reint(\mathfrak{j})$ and $x \in \reint(\tau)$ for strata $\tau \subset \nsf$. Similar to above, we can define the ($l^{\text{th}}$-order) sheaf $\wcs{l}_{\mathscr{D}}$ by using $\sfbva{l}^0_{\mathrm{sf}}$ and considering equation \eqref{eqn:wall_crossing_gluing} modulo $(q)^{l+1}$. There is a natural restriction map $\rest{k,l} \colon \mathfrak{i}_*(\wcs{k}_{\mathscr{D}}) \rightarrow \mathfrak{i}_*(\wcs{l}_{\mathscr{D}})$. Taking tensor product, we have $\rest{k,l}\colon \mathfrak{i}_*(\wcs{k}_{\mathscr{D}}) \otimes_{\cfrk{k}} \cfrk{l} \rightarrow \mathfrak{i}_*(\wcs{l}_{\mathscr{D}})$, where $\cfrk{k} = \comp[q]/(q^{k+1})$. 

%The above extension results will allow us to prove the following extension lemma.
The proof of the following lemma will be given in Appendix \S \ref{sec:hartogs}.
\begin{lemma}[Hartogs extension property]\label{lem:hartogs_extension_2}
	We have 
    $$\iota_* (\bva{0}^0|_{W_0}) = \bva{0}^0,$$
    where $\iota \colon W_0 \rightarrow B$ is the inclusion. Moreover, for any scattering diagram $\mathscr{D}$, we have 
    $$\mathfrak{i}_*(\bva{0}^0|_{W_0(\mathscr{D})}) = \bva{0}^0,$$
    where $\mathfrak{i} \colon W_0(\mathscr{D}) \rightarrow B$ is the inclusion.
\end{lemma}

\begin{lemma}\label{lem:0_order_wall_crossing_sheaf_vs_structure_sheaf}
The $0^{\text{th}}$-order sheaf $\mathfrak{i}_*(\wcs{0}_{\mathscr{D}})$ is isomorphic to the sheaf $\bva{0}^0$. 
\end{lemma}
\begin{proof}
 	In view of Lemma \ref{lem:hartogs_extension_2}, we only have to show that the two sheaves are isomorphic on the open subset $W_0(\mathscr{D})$. Since we work modulo $(q)$, only the wall-crossing factor $\varTheta_{v,\rho}$ associated to a slab matters. So we take a point $x \in \reint(\mathbf{b}) \subset \reint(\rho)_v$ for some vertex $v$, and compare $\wcs{0}_{\mathscr{D},x}$ with $\bva{0}^0_{x} = \sfbva{0}^0_{\mathrm{sf},x}$. From the proof of Lemma \ref{lem:comparing_sheaf_of_dgbv}, we have
 	\begin{align*}
 	\bva{0}^0_{x} & = \sfbva{0}^0_{\mathrm{sf},x}=\cu{O}_{\localmod{k}(\rho)_v,K}\\
        & = \left\{a_{0,j} +  \sum_{i=1}^{\infty} a_{i} u^i + \sum_{i=-1}^{-\infty}a_{i} v^{-i} \ \Big| \ a_{i} \in \cu{O}_{(\comp^*)^{n-1}}(U) \ \text{for some neigh. $U\supset K$}, \  \sup_{i\neq 0} \frac{\log|a_i|}{|i|} < \infty  \right\}, 
 	\end{align*}
 	with the relation $uv= 0$. The gluings with nearby maximal cells $\sigma_{\pm}$ of both $\bva{0}^0$ and $\sfbva{0}^0_{\mathrm{sf}}$ are simply given by the parallel transport through $v$ and the formulae
 	\begin{equation*}
 		\sigma_+\colon\begin{cases}
 		z^{m} \mapsto  s^{-1}_{\rho\sigma_{+}}(m) z^m  & \text{for $m \in \tanpoly_{\rho}$},\\
 		u \mapsto s^{-1}_{\rho\sigma_+}(m_{\rho}) z^{m_{\rho}}, & \\
 		v \mapsto 0, &
 		\end{cases}
 		\qquad\quad
 	%\end{equation*}
 	%\begin{equation*}
 		\sigma_-\colon\begin{cases}
 			z^{m} \mapsto  s^{-1}_{\rho\sigma_{-}}(m) z^m  & \text{for $m \in \tanpoly_{\rho}$},\\
 			u \mapsto 0, & \\
 			v \mapsto s^{-1}_{\rho\sigma_-}(-m_{\rho})z^{-m_{\rho}} &
 		\end{cases}
 	\end{equation*}
in the proof of Lemma \ref{lem:comparing_sheaf_of_dgbv}.% For $\sigma_-$, we will have $u \mapsto 0$ and $v \mapsto s^{-1}_{\rho\sigma_-}(-m_{\rho})z^{-m_{\rho}}$ instead.
 
Now for the wall-crossing sheaf $\wcs{0}_{\mathscr{D},x} \cong \sfbva{0}^0_{\mathrm{sf},x}$, the wall-crossing factor $\varTheta_{v,\rho}$ acts trivially except on the two coordinate functions $u,v$ because $\langle m ,\check{d}_{\rho} \rangle = 0$ for $m \in \tanpoly_{\rho}$. The gluing of $u$ to the nearby maximal cells which obeys wall crossing is given by
$$
 	u|_{U_x \cap \sigma_{\pm}} := 
 	\begin{dcases}
 	u|_{U_x \cap \sigma_+} & \text{on $U_x \cap \sigma_+$,}\\
 	\varTheta^{-1}_{v,\rho,x}(u)|_{U_x \cap \sigma_-} = 0 & \text{on $U_x \cap \sigma_-$,}
 	\end{dcases}
$$
in a sufficiently small neighborhood $U_x$ of $x$. Here, the reason that we have $\varTheta^{-1}_{v,\rho,x}(u)|_{U_x \cap \sigma_-} = 0$ on $U_x \cap \sigma_-$ is simply because we have $u \mapsto 0$ in the gluing of $\sfbva{0}^0_{\mathrm{sf}}$. For the same reason, we see that the gluing of $v$ agrees with that of $\bva{0}^0$ and $\sfbva{0}^0_{\mathrm{sf}}$. 
\end{proof}

\begin{definition}\label{def:consistency_of_scattering_diagram}
	A ($k^{\text{th}}$-order) scattering diagram $\mathscr{D}$ is said to be \emph{consistent} if there is an isomorphism $\mathfrak{i}_*(\wcs{k}_{\mathscr{D}})|_{W_{\alpha}} \cong \bva{k}^0_{\alpha}$ as sheaves of $\comp[q]/(q^{k+1})$-algebras on each open subset $W_{\alpha}$.
\end{definition}

The above consistency condition would imply that $\rest{k,l}\colon \mathfrak{i}_*(\wcs{k}_{\mathscr{D}}) \rightarrow \mathfrak{i}_*(\wcs{l}_{\mathscr{D}})$ is surjective for any $l < k$ and hence $\mathfrak{i}_*(\wcs{k}_{\mathscr{D}})$ is a sheaf of free $\comp[q]/(q^{k+1})$-modules on $B$. We are going to see that $\mathfrak{i}_*(\wcs{k}_{\mathscr{D}})$ agrees with the push-forward of the sheaf of holomorphic functions on a ($k^{\text{th}}$-order) thickening $\centerfiber{k}$ of the central fiber $\centerfiber{0}$ under the modified moment map $\modmap$.  

Let us elaborate a bit on the relation between this definition of consistency and that in \cite{gross2011real}. Assuming we have a consistent scattering diagram in the sense of \cite{gross2011real}, then we obtain a $k^{\text{th}}$-order thickening $\centerfiber{k}$ of $\centerfiber{0}$ which is locally modeled on the thickenings $\localmod{k}_{\alpha}$'s by \cite[Cor. 2.18]{Gross-Siebert-logII}. Pushing forward via the modified moment map $\modmap$, we obtain a sheaf of algebras over $\comp[q]/(q^{k+1})$ lifting $\bva{0}^0$, which is locally isomorphic to the $\bva{k}^0_{\alpha}$'s. This consequence is exactly what we use to formulate our definition of consistency.  

\begin{lemma}\label{lem:ring_hom_induce_space_hom}
	Suppose we have $W \subset W_{\alpha} \cap W_{\beta}$ such that $V = \modmap^{-1}(W)$ is Stein, and an isomorphism $ h \colon \bva{k}^0_{\beta}|_W \rightarrow \bva{k}^0_{\alpha}|_W$ of sheaves of $\comp[q]/(q^{k+1})$-algebras which is the identity modulo $(q)$. Then there is a unique isomorphism $\psi\colon \localmod{k}_{\alpha}|_V \rightarrow \localmod{k}_{\beta}|_V$ of analytic spaces inducing $h$. 
\end{lemma}
\begin{proof}
	From the description in \S \ref{subsec:log_structure_and_slab_function}, we can embed both families $\localmod{k}_{\alpha}$, $\localmod{k}_{\beta}$ over $\spec(\comp[q]/(q^{k+1}))$ as closed analytic subschemes of $\comp^{N+1} = \comp^{N}\times \comp_{q}$ and $\comp^{L+1} = \comp^{L} \times \comp_q$ respectively, where projection to the second factor defines the family over $\comp[q]/(q^{k+1})$. Let $\cu{J}_{\alpha}$ and $\cu{J}_{\beta}$ be the corresponding ideal sheaves, which can be generated by finitely many elements. We can take Stein open subsets $U_{\alpha} \subseteq \comp^{N+1}$ and $U_{\beta} \subseteq \comp^{L+1}$ such that their intersections with the subschemes give $\localmod{k}_{\alpha}|_{V}$ and $\localmod{k}_{\beta}|_{V}$ respectively. 
	By taking global sections of the sheaves over $W$, we obtain the isomorphism $h\colon \cu{O}_{\localmod{k}_{\beta}}(V) \rightarrow \cu{O}_{\localmod{k}_{\alpha}}(V)$. 
	Using the fact that $U_{\alpha}$ is Stein, we can lift $h(z_i)$'s, where $z_i$'s are restrictions of coordinate functions to $\localmod{k}_{\beta}|_V \subset U_{\beta}$, to holomorphic functions on $U_{\alpha}$. In this way, $h$ can be lifted as a holomorphic map $\psi\colon U_{\alpha} \rightarrow U_{\beta}$. Restricting to $\localmod{k}_{\alpha}|_{V}$, we see that the image lies in $\localmod{k}_{\beta}|_{V}$, and hence we obtain the isomorphism $\psi$. The uniqueness follows from the fact the $\psi$ is determined by $\psi^*(z_i) = h(z_i)$. 
\end{proof}

Given a consistent scattering diagram $\mathscr{D}$ (in the sense of Definition \ref{def:consistency_of_scattering_diagram}), the sheaf $\mathfrak{i}_*(\wcs{k}_{\mathscr{D}})$ can be treated as a gluing of the local sheaves $\bva{k}^0_{\alpha}$'s. Then from Lemma \ref{lem:ring_hom_induce_space_hom}, we obtain a gluing of the local models $\localmod{k}_{\alpha}$'s yielding a thickening $\centerfiber{k}$ of $\centerfiber{0}$. This justifies Definition \ref{def:consistency_of_scattering_diagram}.

\subsubsection{Constructing consistent scattering diagrams from Maurer--Cartan solutions}\label{subsubsec:consistent_diagram_from_solution}

We are finally ready to demonstrate how to construct a consistent scattering diagram $\mathscr{D}(\varphi)$ in the sense of Definition \ref{def:consistency_of_scattering_diagram} from a Maurer--Cartan solution $\varphi = \phi + t f$ obtained in Theorem \ref{prop:Maurer_cartan_equation_unobstructed}. As in \S \ref{subsubsec:gluing_using_semi_flat}, we obtain a $k^{\text{th}}$-order Maurer--Cartan solution $\zeta_0$ and define its scattered part as $\phi_{\mathrm{s}} \in \sftropv{k}^{1}_{\mathrm{sf}}(W_0)$. From this, we want to construct a $k^{\text{th}}$-order scattering diagram $\mathscr{D}(\varphi)$. 

We take an open cover $\{U_i\}_{i}$ by pre-compact convex open subsets of $W_0$ such that, locally on $U_i$, $\phi_{\mathrm{in}}+\phi_{\mathrm{s}}$ can be written as a finite sum
$$
(\phi_{\mathrm{in}}+\phi_{\mathrm{s}})|_{U_i} = \sum_{j} \alpha_{ij} \otimes v_{ij},
$$
where $\alpha_{ij} \in \tform^1(U_i)$ has asymptotic support on a codimension one polyhedral subset $P_{ij} \subset U_i$, and $v_{ij} \in \sftvbva{k}(U_i)$.  We take a partition of unity $\{\varrho_i\}_{i}$ subordinate to the cover $\{U_i\}_{i}$ such that $\mathrm{supp}(\varrho_i)$ has asymptotic support on a compact subset $C_i$ of $U_i$. As a result, we can write 
\begin{equation}
	\phi_{\mathrm{in}}+\phi_{\mathrm{s}} = \sum_i \sum_j (\varrho_i \alpha_{ij}) \otimes v_{ij},
\end{equation}  
where each $(\varrho_i \alpha_{ij})$ has asymptotic support on the compact codimension one subset $C_i \cap P_{ij} \subset U_i$. The subset $\bigcup_{ij}C_i \cap P_{ij}$ will be the support $|\mathscr{D}|$ of our scattering diagram $\mathscr{D} = \mathscr{D}(\varphi)$. 

We may equip $|\mathscr{D}| := \bigcup_{ij}C_i \cap P_{ij}$ with a polyhedral decomposition such that all the boundaries and mutual intersections of $C_i \cap P_{ij}$'s are contained in $(n-2)$-dimensional strata of $|\mathscr{D}|$. So, for each $(n-1)$-dimensional cell $\tau$ of $|\mathscr{D}|$, if $\reint(\tau) \cap (C_i \cap P_{ij} ) \neq \emptyset$ for some $i,j$, then we must have $\tau \subset C_i \cap P_{ij}$. Let $\mathtt{I}(\tau):= \{ (i,j) \ | \ \tau \subset C_i \cap P_{ij} \}$, which is a finite set of indices. We will equip the $(n-1)$-cells $\tau$'s of $|\mathscr{D}|$ with the structure of walls or slabs. 

We first consider the case of a wall. Take $\tau \in |\mathscr{D}|^{[n-1]}$ such that $\reint(\tau) \cap \reint(\rho) = \emptyset$ for all $\rho$ with $\rho \cap \tsing_e \neq \emptyset$. We let $\mathbf{w} = \tau$, choose a primitive normal $\check{d}_{\mathbf{w}}$ of $\tau$, and give the labels $\cu{C}_{\pm}$ to the two adjacent chambers $\cu{C}_{\pm}$ so that $\check{d}_{\mathbf{w}}$ is pointing into $\cu{C}_{+}$. In a sufficiently small neighborhood $U_\tau$ of $\reint(\tau)$, we have $\phi_{\mathrm{in}}|_{U_{\tau}} = 0$ and we may write 
$$
\phi_{\mathrm{s}}|_{U_{\tau}} = \sum_{(i,j)\in \mathtt{I}(\tau)}(\varrho_i \alpha_{ij}) \otimes v_{ij},
$$
where each $(\varrho_i \alpha_{ij})$ has asymptotic support on $\reint(\tau)$. Since locally on $U_{\tau}$ any Maurer--Cartan solution is gauge equivalent to $0$, there exists an element $\theta_{\tau} \in \tform^0(U_{\tau})\otimes q\cdot \sftvbva{k}(U_{\tau})$ such that 
$$e^{[\theta_{\tau},\cdot]} \circ \pdb_{\circ} \circ e^{-[\theta_{\tau},\cdot]} = \pdb_{\circ} + [\phi_{\mathrm{s}},\cdot].$$
Such an element can be constructed inductively using the procedure in \cite[\S 3.4.3]{matt-leung-ma}, and can be chosen to be of the form
\begin{equation}\label{eqn:construction_of_wall_crossing_factor_from_solution_wall_case}
	\theta_{\tau}|_{U_{\tau} \cap \cu{C}_{\pm}}	= \begin{dcases}
		\theta_{\tau,0}|_{U_{\tau} \cap \cu{C}_+}  & \text{on $U_{\tau} \cap \cu{C}_+$,}\\
		 0  & \text{on $U_{\tau} \cap \cu{C}_-$,}
	\end{dcases}
\end{equation}
for some $\theta_{\tau,0} \in q \cdot  \sftvbva{k}(U_{\tau})$.
% that is determined inductively by $\theta_{\tau}$.
From this we obtain the wall-crossing factor associated to the wall $\mathbf{w}$
\begin{equation}\label{eqn:construction_wall_crossing_factor_wall}
\Theta_{\mathbf{w}} := e^{[\theta_{\tau,0},\cdot]}.
\end{equation}

\begin{remark}
	Here we need to apply the procedure in \cite[\S 3.4.3]{matt-leung-ma}, which is a generalization of that in \cite{kwchan-leung-ma}, because of the potential non-commutativity: $[v_{ij},v_{ij'}]\neq 0$ for $j \neq j'$. 
\end{remark}

For the case where $\tau \subset \reint(\rho)_v$ for some $\rho$ with $\rho \cap \tsing_e \neq \emptyset$, we will define a slab. We take $U_{\tau}$ and $\mathtt{I}(\tau)$ as above, and let the slab $\mathbf{b} = \tau$. The primitive normal $\check{d}_{\rho}$ is the one we chose earlier for each $\rho$. Again we work in a small neighborhood $U_{\tau}$ of $\reint(\tau)$ with two adjacent chambers $\cu{C}_{\pm}$. As in the proof of Lemma \ref{lem:comparing_sheaf_of_dgbv}, we can find a step-function-like element $\theta_{v,\rho}$ of the form
$$
\theta_{v,\rho} = 
\begin{dcases}
	\log(s_{v\rho}^{-1}(f_{v,\rho})) \partial_{\check{d}_{\rho}} & \text{on $ U_{\tau} \cap \cu{C}_+$},\\
	0 & \text{on $ U_{\tau} \cap \cu{C}_-$}
\end{dcases}
$$
to solve the equation $e^{[\theta_{v,\rho},\cdot]}\circ \pdb_{\circ} \circ e^{-[\theta_{v,\rho},\cdot]} = \pdb_{\circ} + [\phi_{\mathrm{in}},\cdot]$ on $U_\tau$. In other words, 
$$\varPsi:=e^{-[\theta_{v,\rho},\cdot]} \colon (\sftropv{k}^*_{\mathrm{sf}}|_{U_{\tau}},\pdb_{\mathrm{sf}}) \rightarrow (\sftropv{k}^*_{\mathrm{sf}}|_{U_{\tau}},\pdb_{\circ})$$ 
is an isomorphism of sheaves of dgLas. Computations using the formula in \cite[Lem. 2.5]{chan2019geometry} then gives the identity
$$
\varPsi^{-1} (\pdb_{\circ} + [\varPsi(\phi_{\mathrm{s}}),\cdot] ) \circ \varPsi = \pdb_{\circ} + [\phi_{\mathrm{in}} + \phi_{\mathrm{s}},\cdot]. 
$$
Once again, we can find an element $\theta_{\tau}$ such that
$$
e^{[\theta_{\tau},\cdot]} \circ \pdb_{\circ} \circ e^{-[\theta_{\tau},\cdot]} = \pdb_{\circ} + [\varPsi(\phi_{\mathrm{s}}),\cdot],
$$
and hence a corresponding element $\theta_{\tau,0} \in q\cdot \sftvbva{k}(U_\tau)$ of the form \eqref{eqn:construction_of_wall_crossing_factor_from_solution_wall_case}. From this we get 
\begin{equation}\label{eqn:construction_wall_crossing_factor_slab}
\varXi_{\mathbf{b}}:= e^{[\theta_{\tau,0},\cdot]}
\end{equation}
and hence the wall-crossing factor $\Theta_{\mathbf{b}} :=\varTheta_{v,\rho} \circ \varXi_{\mathbf{b}} $ associated to the slab $\mathbf{b}$. 

Next we would like to argue that consistency of the scattering diagram $\mathscr{D}$ follows from the fact that $\phi$ is a Maurer--Cartan solution.
First of all, on the global sheaf $\polyv{k}^{*,*}$ over $B$, we have the operator $\pdb_{\phi}:=\pdb+[\phi,\cdot]$ which satisfies $[\bvd{},\pdb_{\phi}] = 0$ and $\pdb_{\phi}^2 = 0$. This allows us to define the \emph{sheaf of $k^{\text{th}}$-order holomorphic functions} as
$$\prescript{k}{}{\cu{O}}_{\phi}:=\mathrm{Ker}(\pdb_{\phi}) \subset \polyv{k}^{0,0},$$
for each $k\in \mathbb{N}$. It is a sequence of sheaves of commutative $\comp[q]/(q^{k+1})$-algebras over $B$, equipped with a natural map $\rest{k,l}\colon \prescript{k}{}{\cu{O}}_{\phi} \rightarrow \prescript{l}{}{\cu{O}}_{\phi}$ for $l<k$ that is induced from the maps for $\polyv{k}^{*,*}$. By construction, we see that $\prescript{0}{}{\cu{O}}_{\phi} \cong \bva{0}^0 \cong \modmap_*(\cu{O}_{\centerfiber{0}})$. 

We claim that the maps $\rest{k,l}$'s are surjective.
To prove this, we fix a point $x\in B$ and take an open chart $W_{\alpha}$ containing $x$ in the cover of $B$ we chose at the beginning of \S \ref{subsubsec:gluing_using_semi_flat}. There is an isomorphism $\varPhi_{\alpha}\colon \polyv{k}^{*,*}|_{W_{\alpha}} \cong \polyv{k}^{*,*}_{\alpha}$ identifying the differential $\pdb$ with $\pdb_{\alpha}+ [\eta_{\alpha},\cdot]$ by our construction. Write $\phi_{\alpha} = \varPhi_{\alpha}(\phi)$ and notice that $\pdb_{\alpha} + [\eta_{\alpha}+\phi_{\alpha},\cdot]$ squares to zero, which means that $\eta_{\alpha} + \phi_{\alpha}$ is a solution to the Maurer--Cartan equation for $\polyv{k}^{*,*}_{\alpha}(W_{\alpha})$. We apply the same trick as above to the local open subset $W_{\alpha}$, namely, any Maurer--Cartan solution lying in $\polyv{k}^{-1,1}_{\alpha}(W_{\alpha})$ is gauge equivalent to the trivial one, so there exists $\theta_{\alpha} \in \polyv{k}^{-1,0}_{\alpha}(W_{\alpha})$ such that 
$$
e^{[\theta_{\alpha},\cdot]} \circ \pdb_{\alpha} \circ e^{-[\theta_{\alpha},\cdot]} = \pdb_{\alpha} + [\eta_{\alpha}+\phi_{\alpha},\cdot]. 
$$
As a result, the map $e^{-[\theta_{\alpha},\cdot]} \circ \varPhi_{\alpha} \colon (\polyv{k}^{*,*}|_{W_{\alpha}},\pdb+[\phi,\cdot]) \cong (\polyv{k}^{*,*}_{\alpha},\pdb_{\alpha})$ is an isomorphism of dgLas, sending $\prescript{k}{}{\cu{O}}_{\phi}$ isomorphically onto $\bva{k}_{\alpha}^0$. 

We shall now prove the consistency of the scattering diagram $\mathscr{D} = \mathscr{D}(\varphi)$ by identifying the associated wall-crossing sheaf $\wcs{k}_{\mathscr{D}}$ with the sheaf $\prescript{k}{}{\cu{O}}_{\phi}|_{W_0(\mathscr{D})}$ of $k^{\text{th}}$-order holomorphic functions.   

\begin{theorem}\label{thm:consistency_of_diagram_from_mc}
	There is an isomorphism $\Phi \colon \prescript{k}{}{\cu{O}}_{\phi}|_{W_0(\mathscr{D})} \rightarrow \wcs{k}_{\mathscr{D}}$ of sheaves of $\comp[q]/(q^{k+1})$-algebras on $W_0(\mathscr{D})$. Furthermore, the scattering diagram $\mathscr{D} = \mathscr{D}(\varphi)$ associated to the Maurer--Cartan solution $\phi$ is consistent in the sense of Definition \ref{def:consistency_of_scattering_diagram}.
\end{theorem}

\begin{proof}
	To prove the first statement, we first notice that there is a natural isomorphism 
    $$\prescript{k}{}{\cu{O}}_{\phi}|_{W_0\setminus | \mathscr{D}|} \cong \wcs{k}_{\mathscr{D}}|_{W_0\setminus | \mathscr{D}|},$$
    so we only need to consider those points $x \in \reint(\tau)$ where $\tau$ is either a wall or a slab. Since $W_0(\mathscr{D}) \subset W_0$, we will work on the semi-flat locus $W_0$ and use the model $\sfpolyv{k}^{*,*}_{\mathrm{sf}}$, which is equipped with the operator $ \pdb_{\circ}+ [\phi_{\mathrm{in}} + \phi_{\mathrm{s}},\cdot]$. Via the isomorphism 
    $$\varPhi \colon (\polyv{k}^{*,*}_0,\pdb_{\phi}) \rightarrow (\sfpolyv{k}^{*,*}_{\mathrm{sf}}, \pdb_{\circ}+ [\phi_{\mathrm{in}} + \phi_{\mathrm{s}},\cdot])$$ 
    from Lemma \ref{lem:comparing_sheaf_of_dgbv}, we may write 
    $$\prescript{k}{}{\cu{O}}_{\phi}|_{W_0} = \mathrm{Ker}(\pdb_\phi) \subset \sfpolyv{k}^{0,0}_{\mathrm{sf}}.$$
    We fix a point $x \in W_0(\mathscr{D}) \cap |\mathscr{D}|$ and consider the stalk at $x$ for both sheaves. In the above construction of walls and slabs from the Maurer--Cartan solution $\phi$, we first take a sufficiently small open subset $U_x$ and then find a gauge transformation of the form $\varPsi = e^{[\theta_{\tau},\cdot]}$ in the case of a wall, and of the form $\varPsi =e^{[\theta_{v,\rho},\cdot]} \circ e^{[\theta_{\tau},\cdot]} $ in the case of a slab. We have 
    $$\varPsi^{} \circ \pdb_{\circ} \circ \varPsi^{-1} = \pdb_{\circ}+ [\phi_{\mathrm{in}} + \phi_{\mathrm{s}},\cdot]$$
    by construction, so this further induces an isomorphism 
    $$\varPsi \colon \sfbva{k}^0_{\mathrm{sf}}|_{U_x} \rightarrow \prescript{k}{}{\cu{O}}_{\phi}|_{U_x}$$
    of $\comp[q]/(q^{k+1})$-algebras. 
	
	It remains to see how the stalk $\varPsi\colon \sfbva{k}^0_{\mathrm{sf},x} \rightarrow  \prescript{k}{}{\cu{O}}_{\phi,x}$ is glued to nearby chambers $\cu{C}_{\pm}$. 
    For this purpose, we let 
    $$\Psi_0 := e^{[\theta_{\tau,0},\cdot]}$$ 
    as in equation \eqref{eqn:construction_wall_crossing_factor_wall} in the case of a wall, 
    and 
    $$\Psi_0 := \varTheta_{v,\rho} \circ e^{[\theta_{\tau,0},\cdot]}$$ 
    as in \eqref{eqn:construction_wall_crossing_factor_slab} in the case of a slab. 
    Then, the restriction of an element $f \in \sfbva{k}^0_{\mathrm{sf},x}$ to a nearby chamber is given by 
	$$
	\varPsi (f) = \begin{dcases}
		\Psi_0(f) & \text{on $ U_{x} \cap \cu{C}_+$},\\
		f & \text{on $ U_{x} \cap \cu{C}_-$}
	\end{dcases}
	$$
	in a sufficiently small neighborhood $U_x$. This agrees with the description of the wall-crossing sheaf $\wcs{k}_{\mathscr{D},x}$ in equation \eqref{eqn:wall_crossing_gluing}. Hence we obtain an isomorphism $\prescript{k}{}{\cu{O}}_{\phi}|_{W_0( \mathscr{D})} \cong \wcs{k}_{\mathscr{D}}$.
	
	To prove the second statement, we first apply pushing forward via $\mathfrak{i}\colon W_0(\mathscr{D}) \rightarrow B$ to the first statement to get the isomorphism 
    $$\mathfrak{i}_* (\prescript{k}{}{\cu{O}}_{\phi}|_{W_0( \mathscr{D})}) \cong \mathfrak{i}_*(\wcs{k}_{\mathscr{D}}).$$ 
	Now, by the discussion right before this proof, we may identify $\prescript{k}{}{\cu{O}}_{\phi}$ with $\bva{k}^0_{\alpha}$ locally. But the sheaf $\bva{k}^0_{\alpha}$, which is isomorphic to the restriction of $\bva{0}^0 \otimes_{\comp} \comp[q]/(q^{k+1})$ to $W_\alpha$ as sheaves of $\comp[q]/(q^{k+1})$-modules, satisfies the Hartogs extension property from $W_0(\mathscr{D})\cap W_{\alpha}$ to $W_{\alpha}$ by Lemma \ref{lem:hartogs_extension_2}. So we have $\mathfrak{i}_* (\prescript{k}{}{\cu{O}}_{\phi}|_{W_0( \mathscr{D})}) \cong \prescript{k}{}{\cu{O}}_{\phi}$. Hence, we obtain 
    $$\mathfrak{i}_*(\wcs{k}_{\mathscr{D}})|_{W_{\alpha}} \cong  (\prescript{k}{}{\cu{O}}_{\phi})|_{W_{\alpha}} \cong \bva{k}^0_{\alpha},$$ 
    from which follows the consistency of the diagram $\mathscr{D} = \mathscr{D}(\varphi)$.
\end{proof}

\begin{remark}
	From the proof of Theorem \ref{thm:consistency_of_diagram_from_mc}, we actually have a correspondence between step-function-like elements in the gauge group and elements in the tropical vertex group as follows. We fix a generic point $x$ in a joint $\mathfrak{j}$, and consider a neighborhood of $x$ of the form $U_x \times D_x$, where $U_x$ is a neighborhood of $x$ in $\reint(\mathfrak{j})$ and $D_x$ is a disk in the normal direction of $\mathfrak{j}$. We pick a compact annulus $A_x \subset D_x$ surrounding $x$, intersecting finitely many walls/slabs. We let $\tau_1,\dots,\tau_s$ be the walls/slabs in anti-clockwise direction. For each $\tau_i$, we take an open subset $\mathscr{W}_{i} $ just containing the wall $\tau_i$ such that $\mathscr{W}_i \setminus \tau_i =\mathscr{W}_{i,+}\cup \mathscr{W}_{i,-}$. The following Figure \ref{fig:annulus} below illustrates the situation.  
	
	As in the proof of Theorem \ref{thm:consistency_of_diagram_from_mc}, there is a gauge transformation on each $\mathscr{W}_i$ of the form
	$$
	\varPsi_i \colon (\sfpolyv{k}^{*,*}_{\mathrm{sf}}|_{\mathscr{W}_i}, \pdb_{\circ}) \rightarrow (\sfpolyv{k}^{*,*}_{\mathrm{sf}}|_{\mathscr{W}_i}, \pdb_{\circ} + [\phi_{\mathrm{in}} + \phi_{\mathrm{s}},\cdot]),
	$$
	where $\varPsi_i = e^{[\theta_{v,\rho},\cdot]} \circ e^{[\theta_{\tau},\cdot]}$ for a slab and $\varPsi_i = e^{[\theta_{\tau},\cdot]}$ for a wall. These are step-function-like elements in the gauge group satisfying
	$$
	\varPsi_i = \begin{dcases}
		\Theta_{i} & \text{on $\mathscr{W}_{i,+}$,}\\
		\mathrm{id} & \text{on $\mathscr{W}_{i,-}$,}
	\end{dcases}
	$$
	where $\Theta_{i}$ is the wall crossing factor associated to $\tau_i$. 
 
	On the overlap $\mathscr{W}_{i,+} = \mathscr{W}_i \cap \mathscr{W}_{i+1}$ (where we set $i+1=1$ if $i=s$), there is a commutative diagram
	$$
	\xymatrix@1{
	(\sfpolyv{k}^{*,*}_{\mathrm{sf}}|_{\mathscr{W}_{i,+}},\pdb_{\circ}) \ar[rr]^{\Theta_i} \ar[d]^{\varPsi_i} & & (\sfpolyv{k}^{*,*}_{\mathrm{sf}}|_{\mathscr{W}_{i,+}},\pdb_{\circ}) \ar[d]^{\varPsi_{i+1}} \\
	(\sfpolyv{k}^{*,*}_{\mathrm{sf}}|_{\mathscr{W}_{i,+}},\pdb_{\circ}+ [\phi_{\mathrm{in}} + \phi_{\mathrm{s}},\cdot]) \ar[rr]^{\mathrm{id}} & & (\sfpolyv{k}^{*,*}_{\mathrm{sf}}|_{\mathscr{W}_{i,+}},\pdb_{\circ}+ [\phi_{\mathrm{in}} + \phi_{\mathrm{s}},\cdot])
	}
	$$
	allowing us to interpret the wall crossing factor $\Theta_i$ as the gluing between the two sheaves $\sfpolyv{k}^{*,*}_{\mathrm{sf}}|_{\mathscr{W}_{i}}$ and $\sfpolyv{k}^{*,*}_{\mathrm{sf}}|_{\mathscr{W}_{i+1}}$ over $\mathscr{W}_{i,+}$. 
 
 	Notice that the Maurer--Cartan element $\phi$ is global. On a small neighborhood $W_{\alpha}$ containing $U_x\times D_x$, we have the sheaf $(\polyv{k}_{\alpha}^{*,*},\pdb_{\phi})$ on $W_{\alpha}$, and there is an isomorphism 
	$$
	e^{[\theta_{\alpha},\cdot]} \colon (\polyv{k}^{*,*}_{\alpha},\pdb_{\alpha}) \cong (\polyv{k}^{*,*}_{\alpha},\pdb_\phi).
	$$
	Composing with the isomorphism 
    $$(\polyv{k}^{*,*}_{\alpha}|_{ \mathscr{W}_i},\pdb_\phi) \cong (\sfpolyv{k}^{*,*}_{\mathrm{sf}}|_{\mathscr{W}_{i}},\pdb_{\circ}+ [\phi_{\mathrm{in}} + \phi_{\mathrm{s}},\cdot]),$$ 
    we have a commutative diagram of isomorphisms
	$$
	\xymatrix@1{
		(\sfpolyv{k}^{*,*}_{\mathrm{sf}}|_{\mathscr{W}_{i,+}},\pdb_{\circ}) \ar[rr]^{\varPsi_{i,0}} \ar[rd] & & (\sfpolyv{k}^{*,*}_{\mathrm{sf}}|_{\mathscr{W}_{i,+}},\pdb_{\circ}) \ar[ld]\\
		&(\polyv{k}^{*,*}_{\alpha}|_{\mathscr{W}_{i,+}},\pdb_{\alpha}) & 
	}.
	$$
	This is a \v{C}ech-type cocycle condition between the sheaves $\sfpolyv{k}^{*,*}_{\mathrm{sf}}|_{\mathscr{W}_i}$'s and $\polyv{k}^{*,*}_{\alpha}$, which can be understood as the original consistency condition defined using path-ordered products in \cite{kontsevich-soibelman04, gross2011real}. In particular, taking a local holomorphic function in $\bva{k}^{0}_{\alpha}(W_{\alpha})$ and restricting it to $U_x \times A_x$, we obtain elements in $\sfbva{k}^0_{\mathrm{sf}}(\mathscr{W}_i)$ that jump across the walls according to the wall crossing factors $\Theta_i$'s.
    \begin{center}
		\begin{figure}[h]
			\includegraphics[scale=0.8]{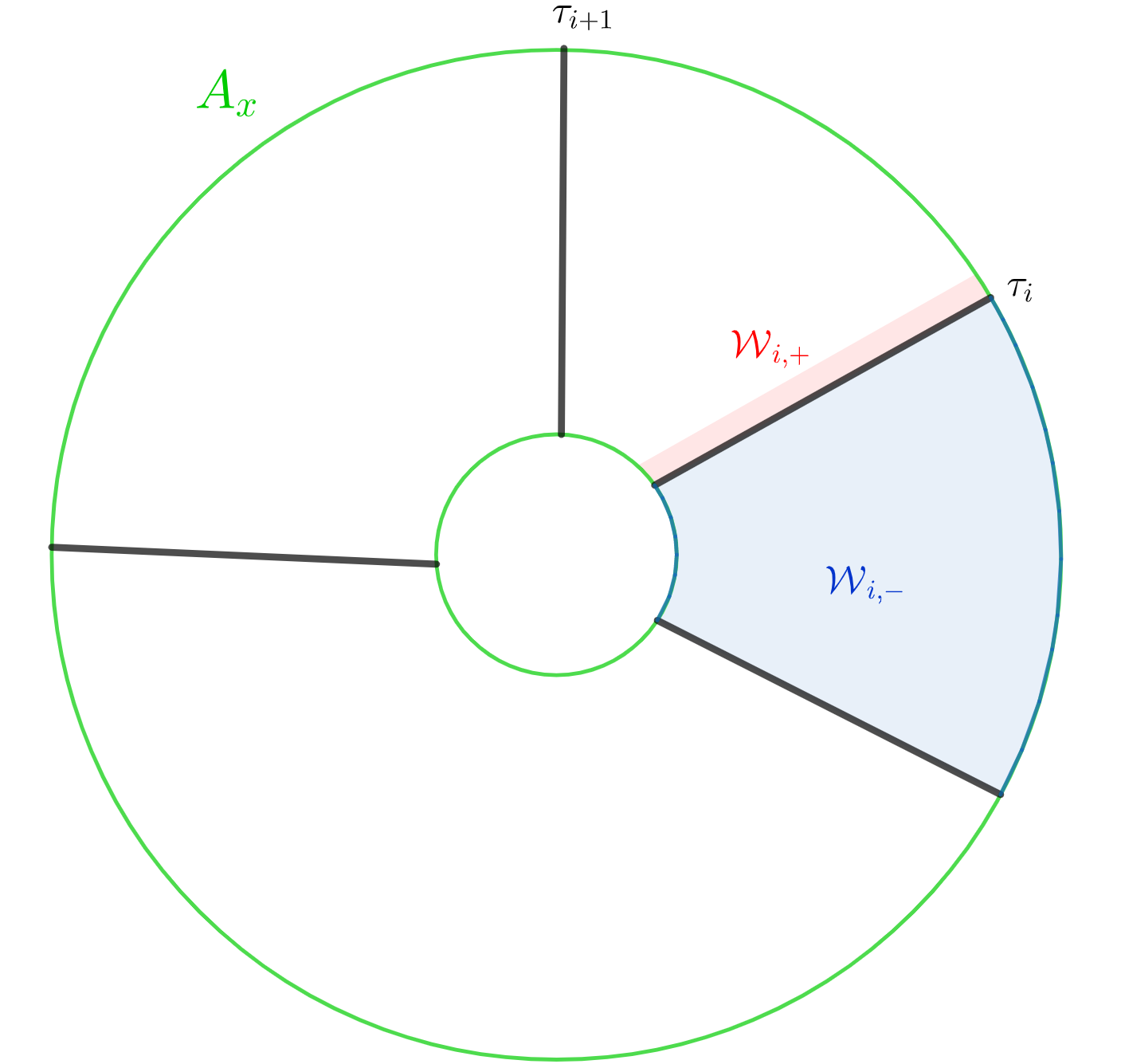}
			\caption{Wall crossing around a joint $\mathfrak{j}$}\label{fig:annulus}
		\end{figure}
	\end{center}
\end{remark}

\appendix

\section{The Hartogs extension property}\label{sec:hartogs}

The following lemma is an application of the Hartogs extension theorem \cite{rossi1963vector}. 

\begin{lemma}\label{lem:hartog_extension_1}
	Consider the analytic space $(\comp^*)^k \times \spec(\comp[\Sigma_{\tau}])$ for some $\tau$ and an open subset of the form $U \times V$, where $U \subset (\comp^*)^k$ and $V$ is a neighborhood of the origin $o \in \spec(\comp[\Sigma_{\tau}])$. Let $W := V \setminus \big( \bigcup_{\omega} V_{\omega} \big)$, where $\dim_{\real}(\omega)+2 \leq \dim_{\real}(\Sigma_{\tau})$ (i.e. $W$ is the complement of complex codimension $2$ orbits in $V$). Then the restriction $\cu{O}(U\times V) \rightarrow \cu{O}(U\times W)$ is a ring isomorphism. 
\end{lemma}

\begin{proof}
	We first consider the case where $\dim_{\real}(\Sigma_{\tau}) \geq 2$ and $W = V \setminus \{0\}$. We can further assume that $\Sigma_{\tau}$ consists of just one cone $\sigma$, because the holomorphic functions on $V$ are those on $V \cap \sigma$ that agree on the overlaps. So we can write 
	$$
	\cu{O}(U \times W) = \left\{ \sum_{m \in \tanpoly_{\sigma}} a_m z^m \ \Big| \ a_m \in \cu{O}_{(\comp^*)^k}(U) \right\},
	$$
	i.e. as Laurent series converging in $W$. We may further assume that $W$ is a sufficiently small Stein open subset. Take $f = \sum_{m \in \tanpoly_{\sigma}} a_m z^m \in \cu{O}(U\times W)$. We have the corresponding holomorphic function $\sum_{m \in \tanpoly_{\sigma}} a_m(u) z^m$ on $W$ for each point $u \in U$, which can be extended to $V$ using the Hartogs extension theorem \cite{rossi1963vector} because $\{0\}$ is a compact subset of $V$ such that $W = V \setminus \{0\}$ is connected. Therefore, we have $a_m(u) = 0$ for $m \notin \sigma \cap \tanpoly_{\sigma}$ for each $u$, and hence $f = \sum_{\sigma \cap \tanpoly_{\sigma}} a_m z^m$ is an element in $\cu{O}(U\times V)$. 
	
	For the general case, we use induction on the codimension of $\omega$ to show that any holomorphic function can be extended through $V_\omega \setminus \bigcup_{\tau} V_{\tau}$ with $\dim_{\real}(\tau) < \dim_{\real}(\omega)$. Taking a point $x \in V_{\omega} \setminus \bigcup_{\tau} V_\tau$, a neighborhood of $x$ can be written as $(\comp^*)^l \times \spec(\comp[\Sigma_{\omega}])$. By the induction hypothesis, we know that holomorphic functions can already be extended through $(\comp^*)^l \times \{0\}$. We conclude that any holomorphic function can be extended through $V_\omega \setminus \bigcup_{\tau} V_{\tau}$. 
\end{proof}

We will make use of the following version of the Hartogs extension theorem, which can be found in e.g. \cite[p. 58]{jarnicki2008first}, to handle extension within codimension one cells $\rho$'s and maximal cells $\sigma$'s.  

\begin{theorem}[Hartogs extension theorem, see e.g. \cite{jarnicki2008first}]\label{thm:hartogs_extension_C_n_version}
	Let $U \subset \comp^n$ be a domain with $n\geq 2$, and $A \subset U$ such that $U \setminus A$ is still a domain. Suppose $\pi(U) \setminus \pi(A)$ is a non-empty open subset, and $\pi^{-1}(\pi(x)) \cap A$ is compact for every $x\in A$, where $\pi \colon \comp^{n} \rightarrow \comp^{n-1}$ is projection along one of the coordinate direction. Then the natural restriction $\cu{O}(U) \rightarrow \cu{O}(U\setminus A)$ is an isomorphism.  
\end{theorem}

\begin{proof}[Proof of Lemma \ref{lem:hartogs_extension_2}]
	To prove the first statement, we apply Lemma \ref{lem:hartog_extension_1}. So we only need to show that, for $\rho \in \pdecomp^{[n-1]}$, a holomorphic function $f$ in $U_x \setminus \tsing \subset V(\rho)$ can be extended uniquely to $U_x$, where $U_x$ is some neighborhood of $x \in \reint(\rho) \cap \tsing$. Writing $V(\rho) = (\comp^*)^{n-1} \times \spec(\comp[\Sigma_{\rho}])$, we may simply prove that this is the case with $\Sigma_{\rho}$ consisting of a single ray $\sigma$ as in the proof of Lemma \ref{lem:hartog_extension_1}. Thus we can assume that $V(\rho) = (\comp^*)^{n-1} \times \comp$ and the open subset $U_x = U \times V$ for some connected $U$. We observe that extensions of holomorphic functions from $(U\setminus \tsing) \times V$ to $U \times V$ can be done by covering the former open subset with Hartogs' figures. 
	
	To prove the second statement, we need to further consider extensions through $\reint(\mathfrak{j})$ for a joint $\mathfrak{j}$. For those joints lying in some codimension one stratum $\rho$, the argument is similar to the above. So we assume that $\sigma_{\mathfrak{j}} = \sigma$ is a maximal cell. We take a point $x \in \reint(\mathfrak{j})$ and work in a sufficiently small neighborhood $U$ of $x$. In this case, we may find a codimension one rational hyperplane $\omega$ containing $\mathfrak{j}$, together with a lattice embedding $\tanpoly_{\omega} \hookrightarrow \tanpoly_{\sigma}$ which induces the projection $\pi \colon (\comp^*)^{n} \rightarrow (\comp^*)^{n-1}$ along one of the coordinate directions. Letting $A = \modmap^{-1}(A \cap U)$ and applying Theorem \ref{thm:hartogs_extension_C_n_version}, we obtain extensions of holomorphic functions in $U$. 
	\begin{comment}
	Writing $W = V \setminus \{0\}$ and ssing Laurent series expansion we may write 
	$$
	\cu{O}(U \times W) = \{ \sum_{m \in \tanpoly_{\sigma}} a_m z^m \ | \ a_m \in \cu{O}_{(\comp^*)^{n-1}}(U) \}.
	$$
	Take an element $\sum_{m \in \tanpoly_{\sigma}} a_m z^m$, we would like to prove that $a_m \equiv 0$ for those $m \notin \sigma \cap \tanpoly_{\sigma}$. This is the case because we can find some open subset of the form $D \times W $ not intersecting $\tsing$ with some smaller $D \subset U$. Therefore if $m \notin \sigma \cap \tanpoly_{\sigma}$ we have $a_m |_{D} \equiv 0$ which forces $a_m \equiv 0$ by the connectivity of $U$. This shows that $\sum_{m \in \tanpoly_{\sigma}} a_m z^m$ can be extended to $U_x$. 
	\end{comment}
\end{proof}

\bibliographystyle{amsplain}
\bibliography{geometry}

\providecommand{\bysame}{\leavevmode\hbox to3em{\hrulefill}\thinspace}
\providecommand{\MR}{\relax\ifhmode\unskip\space\fi MR }
\providecommand{\MRhref}[2]{%
  \href{http://www.ams.org/mathscinet-getitem?mr=#1}{#2}
}
\providecommand{\href}[2]{#2}
\begin{thebibliography}{10}

\bibitem{dbrane}
P.~Aspinwall, T.~Bridgeland, A.~Craw, M.~R. Douglas, M.~Gross, A.~Kapustin,
  G.~W. Moore, G.~Segal, B.~Szendr{\H{o}}i, and P.~M.~H. Wilson,
  \emph{Dirichlet branes and mirror symmetry}, Clay Mathematics Monographs,
  vol.~4, American Mathematical Society, Providence, RI; Clay Mathematics
  Institute, Cambridge, MA, 2009.

\bibitem{Barannikov99}
S.~Barannikov, \emph{{G}eneralized periods and mirror symmetry in dimensions
  $n>3$}, preprint,
  \href{http://arxiv.org/abs/math/9903124}{arXiv:math/9903124}.

\bibitem{Barannikov-Kontsevich98}
S.~Barannikov and M.~Kontsevich, \emph{Frobenius manifolds and formality of
  {L}ie algebras of polyvector fields}, Internat. Math. Res. Notices (1998),
  no.~4, 201--215.

\bibitem{BE-C-H-Lin21}
S.~Bardwell-Evans, M.-W. Cheung, H.~Hong, and Y.-S. Lin, \emph{Scattering
  diagrams from holomorphic discs in log {C}alabi-{Y}au surfaces}, preprint,
  \href{http://arxiv.org/abs/2110.15234}{arXiv:2110.15234}.

\bibitem{bridgeland2016scattering}
T.~Bridgeland, \emph{Scattering diagrams, {H}all algebras and stability
  conditions}, Algebr. Geom. \textbf{4} (2017), no.~5, 523--561.

\bibitem{cartan1957varietes}
H.~Cartan, \emph{Vari{\'e}t{\'e}s analytiques r{\'e}elles et vari{\'e}t{\'e}s
  analytiques complexes}, Bull. Soc. Math. France \textbf{85} (1957), no.~19.7,
  77.

\bibitem{kwchan-leung-ma}
K.~Chan, N.~C. Leung, and Z.~N. Ma, \emph{Scattering diagrams from asymptotic
  analysis on {M}aurer-{C}artan equations}, J. Eur. Math. Soc. (JEMS)
  \textbf{24} (2022), no.~3, 773--849.

\bibitem{chan2019geometry}
\bysame, \emph{Geometry of the {M}aurer-{C}artan equation near degenerate
  {C}alabi-{Y}au varieties}, J. Differential Geom. \textbf{125} (2023), no.~1,
  1--84.

\bibitem{kwchan-ma-p2}
K.~Chan and Z.~N. Ma, \emph{Tropical counting from asymptotic analysis on
  {M}aurer-{C}artan equations}, Trans. Amer. Math. Soc. \textbf{373} (2020),
  no.~9, 6411--6450.

\bibitem{Cheung-Lin21}
M.-W. Cheung and Y.-S. Lin, \emph{Some examples of family {F}loer mirror}, Adv.
  Theor. Math. Phys., to appear,
  \href{http://arxiv.org/abs/2101.07079}{arXiv:2101.07079}.

\bibitem{Cho-Hong-Lau18}
C.-H. Cho, H.~Hong, and S.-C. Lau, \emph{Gluing localized mirror functors}, J.
  Differential Geom., to appear,
  \href{http://arxiv.org/abs/1810.02045}{arXiv:1810.02045}.

\bibitem{Cho-Hong-Lau17}
\bysame, \emph{Localized mirror functor for {L}agrangian immersions, and
  homological mirror symmetry for {$\Bbb{P}^1_{a,b,c}$}}, J. Differential Geom.
  \textbf{106} (2017), no.~1, 45--126.

\bibitem{demailly1997complex}
J.-P. Demailly, \emph{Complex analytic and differential geometry}, 2012,
  \href{https://www-fourier.ujf-grenoble.fr/~demailly/manuscripts/agbook.pdf}{https://www-fourier.ujf-grenoble.fr/~demailly/manuscripts/agbook.pdf}.

\bibitem{dupont1976simplicial}
J.~L. Dupont, \emph{Simplicial de {R}ham cohomology and characteristic classes
  of flat bundles}, Topology \textbf{15} (1976), no.~3, 233--245.

\bibitem{Felten23}
S.~Felten, \emph{Global logarithmic deformation theory}, preprint (2023),
  \href{http://arxiv.org/abs/2310.07949}{arXiv:2310.07949}.

\bibitem{felten2020log}
\bysame, \emph{Log smooth deformation theory via {G}erstenhaber algebras},
  Manuscripta Math. \textbf{167} (2022), no.~1-2, 1--35.

\bibitem{Felten-Filip-Ruddat}
S.~Felten, M.~Filip, and H.~Ruddat, \emph{Smoothing toroidal crossing spaces},
  Forum Math. Pi \textbf{9} (2021), Paper No. e7, 36.

\bibitem{Floer88}
A.~Floer, \emph{Morse theory for {L}agrangian intersections}, J. Differential
  Geom. \textbf{28} (1988), no.~3, 513--547.

\bibitem{fukaya05}
K.~Fukaya, \emph{Multivalued {M}orse theory, asymptotic analysis and mirror
  symmetry}, Graphs and patterns in mathematics and theoretical physics, Proc.
  Sympos. Pure Math., vol.~73, Amer. Math. Soc., Providence, RI, 2005,
  pp.~205--278.

\bibitem{fukaya-oh}
K.~Fukaya and Y.-G. Oh, \emph{Zero-loop open strings in the cotangent bundle
  and {M}orse homotopy}, Asian J. Math. \textbf{1} (1997), no.~1, 96--180.

\bibitem{Fulton_toric_book}
W.~Fulton, \emph{Introduction to toric varieties}, Annals of Mathematics
  Studies, vol. 131, Princeton University Press, Princeton, NJ, 1993, The
  William H. Roever Lectures in Geometry.

\bibitem{Gammage-Shende17}
B.~Gammage and V.~Shende, \emph{Mirror symmetry for very affine hypersurfaces},
  Acta Math. \textbf{229} (2022), no.~2, 287--346.

\bibitem{gammage2021homological}
\bysame, \emph{Homological mirror symmetry at large volume}, Tunis. J. Math.
  \textbf{5} (2023), no.~1, 31--71.

\bibitem{Ganatra-Pardon-Shende18b}
S.~Ganatra, J.~Pardon, and V.~Shende, \emph{Microlocal {M}orse theory of
  wrapped {F}ukaya categories}, Ann. of Math., to appear,
  \href{http://arxiv.org/abs/1809.08807}{arXiv:1809.08807}.

\bibitem{Ganatra-Pardon-Shende20}
\bysame, \emph{Covariantly functorial wrapped {F}loer theory on {L}iouville
  sectors}, Publ. Math. Inst. Hautes \'{E}tudes Sci. \textbf{131} (2020),
  73--200.

\bibitem{Ganatra-Pardon-Shende18a}
\bysame, \emph{Sectorial descent for wrapped {F}ukaya categories}, J. Amer.
  Math. Soc. \textbf{37} (2024), no.~2, 499--635.

\bibitem{Gross-Siebert-logI}
M.~Gross and B.~Siebert, \emph{Mirror symmetry via logarithmic degeneration
  data. {I}}, J. Differential Geom. \textbf{72} (2006), no.~2, 169--338.
  \MR{2213573 (2007b:14087)}

\bibitem{Gross-Siebert-logII}
\bysame, \emph{Mirror symmetry via logarithmic degeneration data, {II}}, J.
  Algebraic Geom. \textbf{19} (2010), no.~4, 679--780. \MR{2669728
  (2011m:14066)}

\bibitem{gross2011real}
\bysame, \emph{From real affine geometry to complex geometry}, Ann. of Math.
  (2) \textbf{174} (2011), no.~3, 1301--1428. \MR{2846484}

\bibitem{gross2011invitation}
\bysame, \emph{An invitation to toric degenerations}, Surveys in differential
  geometry. {V}olume {XVI}. {G}eometry of special holonomy and related topics,
  Surv. Differ. Geom., vol.~16, Int. Press, Somerville, MA, 2011, pp.~43--78.
  \MR{2893676}

\bibitem{jarnicki2008first}
M.~Jarnicki and P.~Pflug, \emph{First steps in several complex variables:
  {R}einhardt domains}, EMS Textbooks in Mathematics, European Mathematical
  Society (EMS), Z\"{u}rich, 2008.

\bibitem{Kashiwara-Schapira94}
M.~Kashiwara and P.~Schapira, \emph{Sheaves on manifolds}, Grundlehren der
  mathematischen Wissenschaften [Fundamental Principles of Mathematical
  Sciences], vol. 292, Springer-Verlag, Berlin, 1994, With a chapter in French
  by Christian Houzel, Corrected reprint of the 1990 original.

\bibitem{KKP08}
L.~Katzarkov, M.~Kontsevich, and T.~Pantev, \emph{Hodge theoretic aspects of
  mirror symmetry}, From {H}odge theory to integrability and {TQFT}
  tt*-geometry, Proc. Sympos. Pure Math., vol.~78, Amer. Math. Soc.,
  Providence, RI, 2008, pp.~87--174.

\bibitem{Kawamata-Namikawa}
Y.~Kawamata and Y.~Namikawa, \emph{Logarithmic deformations of normal crossing
  varieties and smoothing of degenerate {C}alabi-{Y}au varieties}, Invent.
  Math. \textbf{118} (1994), no.~3, 395--409.

\bibitem{Kontsevich-sheaf}
M.~Kontsevich, \emph{Symplectic geometry of homological algebra}, preprint,
  \href{https://www.ihes.fr/~maxim/TEXTS/Symplectic_AT2009.pdf}{available
  online}.

\bibitem{kontsevich-soibelman04}
M.~Kontsevich and Y.~Soibelman, \emph{Affine structures and non-{A}rchimedean
  analytic spaces}, The unity of mathematics, Progr. Math., vol. 244,
  Birkh\"auser Boston, Boston, MA, 2006, pp.~321--385.

\bibitem{matt-leung-ma}
N.~C. Leung, Z.~N. Ma, and M.~B. Young, \emph{Refined scattering diagrams and
  theta functions from asymptotic analysis of {M}aurer-{C}artan equations},
  Int. Math. Res. Not. IMRN (2021), no.~5, 3389--3437.

\bibitem{Lin21}
Y.-S. Lin, \emph{Correspondence theorem between holomorphic discs and tropical
  discs on {K}3 surfaces}, J. Differential Geom. \textbf{117} (2021), no.~1,
  41--92.

\bibitem{Mikhalkin05}
G.~Mikhalkin, \emph{Enumerative tropical algebraic geometry in {$\Bbb R^2$}},
  J. Amer. Math. Soc. \textbf{18} (2005), no.~2, 313--377. \MR{2137980
  (2006b:14097)}

\bibitem{Nishinou-Siebert06}
T.~Nishinou and B.~Siebert, \emph{Toric degenerations of toric varieties and
  tropical curves}, Duke Math. J. \textbf{135} (2006), no.~1, 1--51.
  \MR{2259922 (2007h:14083)}

\bibitem{rossi1963vector}
H.~Rossi, \emph{Vector fields on analytic spaces}, Ann. of Math. (2)
  \textbf{78} (1963), 455--467.

\bibitem{Ruddat10}
H.~Ruddat, \emph{Log {H}odge groups on a toric {C}alabi-{Y}au degeneration},
  Mirror symmetry and tropical geometry, Contemp. Math., vol. 527, Amer. Math.
  Soc., Providence, RI, 2010, pp.~113--164.

\bibitem{ruddat2019period}
H.~Ruddat and B.~Siebert, \emph{Period integrals from wall structures via
  tropical cycles, canonical coordinates in mirror symmetry and analyticity of
  toric degenerations}, Publ. Math. Inst. Hautes \'{E}tudes Sci. \textbf{132}
  (2020), 1--82.

\bibitem{Ruddat-Zharkov22}
H.~Ruddat and I.~Zharkov, \emph{Topological {S}trominger-{Y}au-{Z}aslow
  fibrations}, in preparation.

\bibitem{Ruddat-Zharkov21}
\bysame, \emph{Compactifying torus fibrations over integral affine manifolds
  with singularities}, 2019--20 {MATRIX} annals, MATRIX Book Ser., vol.~4,
  Springer, Cham, [2021] \copyright 2021, pp.~609--622.

\bibitem{Seidel-ICM}
P.~Seidel, \emph{Fukaya categories and deformations}, Proceedings of the
  {I}nternational {C}ongress of {M}athematicians, {V}ol. {II} ({B}eijing, 2002)
  (Beijing), Higher Ed. Press, 2002, pp.~351--360.

\bibitem{syz96}
A.~Strominger, S.-T. Yau, and E.~Zaslow, \emph{Mirror symmetry is
  {$T$}-duality}, Nuclear Phys. B \textbf{479} (1996), no.~1-2, 243--259.

\bibitem{terilla2008smoothness}
J.~Terilla, \emph{Smoothness theorem for differential {BV} algebras}, J. Topol.
  \textbf{1} (2008), no.~3, 693--702.

\bibitem{whitney2012geometric}
H.~Whitney, \emph{Geometric integration theory}, Princeton University Press,
  Princeton, N. J., 1957.

\end{thebibliography}

\end{document}